\documentclass[11pt,reqno,a4paper]{amsart}
\usepackage{latexsym,bm}
\usepackage{amsmath,amssymb,cases}
\usepackage{amsthm}
\usepackage{xcolor}
\usepackage{enumerate}
\usepackage{caption}
\usepackage{graphicx,subfig}
\usepackage[numbers,sort&compress]{natbib}
\usepackage{accents}

\usepackage{geometry}
\usepackage{float}

\usepackage{ulem}
\usepackage[thicklines]{cancel}
\usepackage{relsize}
\usepackage{exscale}

\newtheorem{The}{Theorem}[section]
\newtheorem{Lem}[The]{Lemma}
\newtheorem{Cor}[The]{Corollary}
\newtheorem{Def}[The]{Definition}
\newtheorem{Pro}[The]{Proposition}

\theoremstyle{definition}
\newtheorem{Rem}{Remark}[section]
\newtheorem{Exa}[Rem]{Example}

\graphicspath{{fig/}}

\numberwithin{equation}{section}
\numberwithin{figure}{section}

\begin{document}
\captionsetup[figure]{labelfont={default},labelformat={default},labelsep=period,name={Fig.}}

\title[New Formula for Entropy Solutions for Conservation Laws]{New Formula for Entropy Solutions\\
for Scalar Hyperbolic Conservation Laws\\
with Flux Functions of Convexity Degeneracy\\
and Global Dynamic Patterns of Solutions}

\author{Gaowei Cao}
\address{Gaowei Cao: Wuhan Institute of Physics and Mathematics,
	Innovation Academy for Precision Measurement Science and Technology,
	Chinese Academy of Sciences, Wuhan 430071, China;
	\quad\quad Oxford Centre for Nonlinear Partial Differential Equations,
Mathematical Institute, University of Oxford, Oxford, OX2 6GG, UK}
\email{\tt gwcao@apm.ac.cn; caog@maths.ox.ac.uk}

\author{Gui-Qiang G. Chen$^{\dag}$}
\address{Gui-Qiang G. Chen: Oxford Centre for Nonlinear Partial Differential Equations,
Mathematical Institute, University of Oxford, Oxford, OX2 6GG, UK}
\email{\tt chengq@maths.ox.ac.uk}

\author{Xiaozhou Yang}
\address{Xiaozhou Yang: Wuhan Institute of Physics and Mathematics, 
Innovation Academy for Precision Measurement Science and Technology,
Chinese Academy of Sciences, Wuhan 430071, China}
\email{\tt xzyang@apm.ac.cn}
\date{\today}

\begin{abstract}
We are concerned with 
a new solution formula and its applications to the analysis of properties of
entropy solutions of the Cauchy problem for one-dimensional scalar hyperbolic conservation laws,
wherein the flux functions exhibit convexity degeneracy 
and the initial data are in  $L^\infty$.
We first introduce and validate the novel formula for entropy solutions for the Cauchy problem,
which generalizes the Lax--Oleinik formula.
Then, by employing this formula,
we obtain a series of fine properties of entropy solutions
and discover several new structures and phenomena, 
which are summarized in the following three aspects:
\begin{itemize}
\item [(i)] Series of results on the fine structures of entropy solutions, 
especially including the new criteria for all six types of initial waves 
for the Cauchy problem, 
the new structures of entropy solutions inside the backward characteristic triangle, 
and the new features of the formation and development of shocks 
such as all five types of continuous shock generation points, along with their criteria 
and the optimal regularities of the corresponding resulting shocks;

\item[(ii)] Series of results on the global structures of entropy solutions, 
including the four new invariants of entropy solutions, 
the new criteria for the locations and speeds of divides, 
and the exact determination of the global structures of entropy solutions; 

\item[(iii)] Series of new results on the asymptotic behaviors of entropy solutions,
including the asymptotic profiles and decay rates
of entropy solutions for initial data in $L^\infty$, respectively
in the $L^\infty$--norm and the $L^p_{{\rm loc}}$--norm: 
The asymptotic profile in the $L^\infty$--norm is either a rarefaction-constant solution (as a function only consisting of some constant regions,
centered rarefaction waves, and/or rarefaction regions) 
or a single shock, 
and the asymptotic profile in the $L^p_{{\rm loc}}$--norm is the generalized 
$N$--wave as we introduce.
\end{itemize}
Through the above results (i)--(iii),
we obtain the global dynamic patterns of entropy solutions
for scalar hyperbolic conservation laws with the flux functions
satisfying $\eqref{c1.2}$ and general initial data in $L^\infty$.
Moreover, the new solution formula is also extended to
more general scalar hyperbolic conservation laws.
\end{abstract}

\keywords{Formula for entropy solutions, Convexity degeneracy of flux functions,
Fine properties of entropy solutions,
Initial waves for the Cauchy problem,
Fine structures of entropy solutions,
Lifespans of characteristics,
Formation and development of shocks,
Invariants, Divides,
Global structures of entropy solutions,
Asymptotic behaviors.\\
$^\dag$Corresponding author.}

\maketitle

\tableofcontents

\section{Introduction}
We are concerned with 
a new solution formula
and its applications to the analysis of properties of entropy solutions of
the Cauchy problem for 
one-dimensional scalar hyperbolic conservation laws:
\begin{eqnarray}
&&u_t+f(u)_x=0\qquad\,\, \mbox{for $(x,t)\in \mathbb{R}\times\mathbb{R}^+:=({-}\infty,\infty)\times(0,\infty)$},\label{c1.1}\\[1mm]
&&u|_{t=0}=\varphi(x), \label{ID}
\end{eqnarray}
wherein the initial data function $\varphi(x)$ is in $L^\infty(\mathbb{R})$ and the flux function $f(u)$ exhibits convexity degeneracy:
\begin{equation}\label{c1.2}
f''(u)\geq 0 \,\,\,\,\mbox{on $u \in \mathbb{R}$},\qquad\,\,\, \mathcal{L}\{u\,:\, f''(u)=0\}=0,
\end{equation}
where $\mathcal{L}$ is the Lebesgue measure.

As is well-known, due to the nonlinearity of the flux function $f(u)$,
no matter how smooth the initial data function $\varphi(x)$ is,
the solution may develop singularities and form shock waves (shocks, for short)
generically in a finite time,
so that entropy solutions must be considered.
We are concerned with not only the existence and uniqueness, 
but also, more importantly, research for a more general solution formula 
and its applications to the analysis of properties of entropy solutions
so as to establish the global dynamic patterns of entropy solutions. 
This has been shown
to be also helpful for understanding
and analyzing
some important hyperbolic systems of conservation laws
(see, for example, \cite{HFM,HW,TB,WHD})
and related stochastic partial differential equations
(see, for example, \cite{HLB94,EKMS,SSI}), among others.

Owing to the nonlinearity of the flux functions
and the complicated interaction of characteristics,
representative formulas for general entropy solutions
have been
restricted to the case of the one-dimensional scalar conservation laws
with a uniformly convex flux function only.
The first successful attempt of solution formulas was given by Hopf \cite{HE},
in which the solution formula for the inviscid Burgers equation
was obtained by the vanishing viscosity limit.
For equation $\eqref{c1.1}$ with flux functions $f(u)$ satisfying that
$f''(u)\geq c_0>0$,
Lax \cite{LPD57} gave the implicit representation for entropy solutions
and Oleinik \cite{OOA} gave a formula for entropy solutions
in a different way.
However, for the general case $f''(u)\geq 0$, 
there is no representation formula of entropy solutions available, 
which is one of our main motivations of this paper.

For properties of entropy solutions of
the scalar conservation law $\eqref{c1.1}$,
Lax \cite{LPD57,LPD73} gave the invariants of entropy solutions
for initial data in $L^\infty \cap L^1$
and obtained the development of $N$--waves for initial data with compact support
and the formation of the sawtooth profiles for periodic initial data.
Dafermos \cite{DCM2,DCM3,DCM5} established the method of generalized characteristics
and obtained many confinement structures and several
cases of asymptotic behaviors of entropy solutions.
Quinn \cite{QBK} noted the $L^1$--contraction property for piecewise smooth solutions,
Holden--Holden \cite{HH} proved the $L^1$--stability via the numerical method
introduced by Dafermos \cite{DCM0},
and Holden--Risebro--Tveito \cite{HRT} proved that the maximum principle holds
for weak solutions as limits of approximate solutions generated by the Lax--Friedrichs scheme, 
the Godunov scheme,
and the Glimm scheme (even for the case of hyperbolic systems).
Moreover,
Lin \cite{LLW} gave a description of all the six types of
initial waves for the Cauchy problem 
but did not obtain the concrete criteria,
Huang \cite{HFM} gave a description of the backward characteristic triangles and the shock curves,
{Lebaud \cite{LMP} and Yin--Zhu \cite{YZ} proved that the singularity
around the first blowup point caused by the initial compression is actually a shock
under some specific degenerate information of sufficiently smooth initial data functions,}
Li--Wang \cite{LW} gave the global structures of entropy solutions
for several special types of the $L^\infty$ initial data functions,
and Li \cite{LBH} gave a description of the global structures of shocks.
Furthermore, Liu \cite{LTP} gave the stability of Riemann solutions
under the compact perturbation
and Chen--Frid \cite{CF,CF1} showed the large-time stability of the Riemann solutions
with respect to the $L^\infty \cap L^1$ perturbation
and the large-time decay of periodic entropy solutions
of hyperbolic conservation laws (even for the multidimensional case).
However, many important properties of entropy solutions have not yet been explored and/or understood 
when the uniform convexity of the flux function of $\eqref{c1.1}$ fails, especially
for the global dynamic patterns of the entropy solutions.
See also  Dafermos \cite{DCM1}
and the references cited therein.

The scalar hyperbolic conservation laws, serving as the foundational hyperbolic equations,
represent a crucial starting point for the analysis of hyperbolic systems of conservation laws.
Moreover, numerous ideas, concepts, and methodologies gleaned from analyzing scalar hyperbolic
conservation laws have played pivotal roles in advancing the theory
of hyperbolic systems of conservation laws. As examples, these especially include
the minimal entropy conditions
({\it cf.}\, Panov \cite{PEY94}, Bouchut--Perthame \cite{BP},
De Lellis--Otto--Westdickenberg \cite{DOW04},
Cao--Chen \cite{CC}, and Krupa--Vasseur \cite{KRV}),
the kinetic formulations
({\it cf.}\, Lions--Perthame--Tadmor \cite{LPT},
De Lellis--Otto--Westdickenberg \cite{DOW03},
Tadmor--Tao \cite{TT}, and Golse--Perthame \cite{GP}),
the entropy dissipation
({\it cf.}\, De Lellis--Riviere \cite{DR}, Bianchini--Marconi \cite{BM,MB},
and Silvestre \cite{SL}),
the BV/SBV and fractional BV
({\it cf.}\, Ambrosio--De Lellis \cite{ADL}, Bourdarias--Gisclon--Junca \cite{BGJ},
Bianchini--Yu \cite{BY}, and Marconi \cite{ME18}),
the divergence-measure fields and normal traces
({\it cf.}\, Chen--Frid \cite{CF2}, Chen--Rascle \cite{CR}, Chen--Perthame \cite{CP},
Vasseur \cite{VA}, Kwon--Vasseur \cite{KV}, and B\"{u}rger--Frid--Karlsen \cite{BFK}),
the contractive semigroups
({\it cf.}\, Quinn \cite{QBK}, Andreianov--Karlsen--Risebro \cite{AKR}),
Serre--Silvestre \cite{SS}, and Serre \cite{SD21,SD22}),
the front-tracking method
({\it cf.}\, Dafermos \cite{DCM0}, DiPerna \cite{DRJ},
Bressan \cite{BA}, and Bressan--Liu--Yang \cite{BLY}),
among others.
Also see
Ambrosio--Crippa--De Lellis--Otto--Westdickenberg \cite{ACD},
Chen--Holden--Karlsen \cite{CHK},
Holden--Risebro \cite{HR},
Perthame \cite{PB},
and the references cited therein.

\vspace{1pt}
For the reasons discussed above,  
achieving the global dynamic patterns
of the entropy solutions to scalar hyperbolic conservation laws
is paramount for advancing the theory of hyperbolic conservation laws.
This paper is dedicated to fulfilling this purpose.
One of the motivations of this paper 
that both formulas for entropy solutions given by Lax \cite{LPD73} and Oleinik \cite{OOA}
are not suitable for many convex functions $f(u)$,
such as the 
typical flux functions $f(u)=u^{2n}$ with $n>1$ and $f(u)=e^{ku}$ with $k\neq 0$.
Another motivation is that the complete description of
many properties of entropy solutions has not been achieved 
before,
even when the flux function is uniformly convex;
these include the following aspects:
The criteria for initial waves of the Cauchy problem, 
the general cases of the formation and development of shocks,
the structures of entropy solutions inside the backward characteristic triangles,
the necessary and sufficient conditions for the locations and speeds of divides,
and the asymptotic behaviors of entropy solutions with general initial data in $L^\infty$.
Our third motivation is that there are very few results
on the regularity and asymptotic behaviors of entropy solutions
when the convex flux function possesses
the state points of convexity degeneracy and/or asymptotic lines.

\vspace{3pt}
The aim of this paper is twofold:
Firstly, we introduce and validate a novel formula for entropy solutions
of scalar hyperbolic conservation laws with
the flux functions satisfying $\eqref{c1.2}$
({\it i.e.}, the uniform convexity is not required);
this new formula aligns with the well-known Lax--Oleinik formula \cite{LPD57,LPD73,OOA}
in the case that the flux function is of uniform convexity.
Importantly, the method used to derive this new formula is rooted in the understanding of
the intrinsic formula for smooth solutions {\rm (}see {\rm Remark} $\ref{rem:c1.2}${\rm )},
offering a new approach to identifying formulas
for entropy solutions in more general hyperbolic conservation laws
where such intrinsic formulas are known.
Indeed, as shown in \S 8,
the new formula for entropy solutions can be extended to
more general scalar hyperbolic conservation laws $\eqref{c7.1}$
with the flux pair $(U(u), F(u))$ satisfying $\eqref{c7.1c}$
and the initial data in $L^1_{{\rm loc}}$ satisfying $\eqref{c7.19}$.
Secondly, we 
leverage this newly derived formula to elucidate various fine properties of entropy solutions
and obtain the global dynamic patterns of entropy solutions
for the Cauchy problem with
{\it flux functions satisfying $\eqref{c1.2}$ and initial data only 
in $L^\infty$};
these especially include: 
\begin{enumerate}
\item[(i)] {\it Criteria for all the six types of initial waves for the Cauchy problem
and dynamic behaviors of characteristics in {\rm \S 3}.}
We prove the necessary and sufficient conditions
for the rays emitting from points on the $x$--axis to be characteristics at least local-in-time,
and obtain the criteria for all the six types of initial waves for the Cauchy problem
by providing the necessary and sufficient conditions for all the six 
types of initial waves for the Cauchy problem,
even when the point on the $x$--axis is not a Lebesgue point of the initial data function.
Furthermore, we introduce and prove the lifespan of characteristics,
and provide a classification of characteristics by the way of their terminations.

\vspace{1pt}
\item[(ii)] {\it New features of the formation and development of shocks in {\rm \S 4}
and fine structures of entropy solutions in {\rm\S 5}.}
In \S 4.1, we discover all the five types of continuous shock generation points
and prove the necessary and sufficient conditions of them;
in \S 4.2, we prove the optimal regularities of shock curves and entropy solutions
near all the five types of continuous shock generation points,
which are not necessarily assumed to be isolated.
In \S 5, the fine structures of entropy solutions are established
by providing and proving the new structures of entropy solutions inside the backward characteristic triangles,
proving the directional limits of entropy solutions at the discontinuous points
(even for the points at which there are infinitely many shocks collided together)
and the general structures of shocks (especially including the left-derivatives of them).

\vspace{1pt}
\item[(iii)] {\it Four new invariants, criteria for divides as well as their locations and speeds, and the exact determination of global structures of entropy solutions in {\rm\S 6}.}
We first introduce and prove the four invariants of entropy solutions,  
and then provide and prove 
the necessary and sufficient conditions respectively for the divides as well as their locations and speeds. 
Furthermore, by using the results on invariants and divides,
the global structures of entropy solutions are obtained by proving 
all the path-connected branches of the shock set of entropy solutions,
which are actually separated by the 
nearby
divides.

\vspace{2pt}
\item[(iv)] {\it Asymptotic profiles and decay rates respectively
in the $L^\infty$--norm and the $L^p_{{\rm loc}}$--norm in {\rm\S 7}.}
Using the advantages of the results on invariants and divides 
developed in this paper,
we prove that the entropy solutions decay to either a rarefaction-constant solution
(as a function only consisting of some constant regions,
centered rarefaction waves, and/or rarefaction regions)
or a single shock in the $L^\infty$--norm,  
and obtain the corresponding decay rates.
Moreover, introducing the concept of generalized $N$--waves, we prove that
the entropy solutions, which possess at least one divide,
decay to the generalized $N$--waves in the $L^p_{{\rm loc}}$--norm, 
and obtain the corresponding decay rates.

\vspace{2pt}
\item[(v)] {\it Influences of the convexity degeneracy of flux functions.}
As shown in this paper, the convexity degeneracy of the flux function $f(u)$
({\it i.e.}, the infinitesimal order of $f''(u)$ at degeneracy point)
influences the behaviors of entropy solutions substantially,
especially including the generation of characteristics and shocks, 
the formation and development of shocks,
and the decay rates respectively in the $L^\infty$--norm and the $L^p_{{\rm loc}}$--norm. 
\end{enumerate}
The global dynamic patterns of entropy solutions
with {\it general initial data in $L^\infty$} are new,
even for the case when the flux function $f(u)$ is uniformly convex;
these especially include: The criteria for all the six types
of initial waves for the Cauchy problem,
the structures of entropy solutions inside the backward characteristic triangles,
the directional limits of entropy solutions at the discontinuous points
and the left-derivatives of shock curves,
the formation and development of shocks for all the five types of continuous shock generation points,
the four invariants, 
the necessary and sufficient conditions respectively for the divides as well as their locations and speeds, 
the exact determination of the global structures of entropy solutions, 
and the asymptotic behaviors of entropy solutions 
respectively in the $L^\infty$--norm and the $L^p_{\rm loc}$--norm.
Our proofs are based on the new formula for entropy solutions
of the Cauchy problem with flux functions allowing for convexity degeneracy and initial data only in $L^\infty$.

Finally, we extend the new formula for entropy solutions to more general scalar hyperbolic conservation laws $U(u)_t+F(u)_x=0$ in \S 8, 
where the flux pair $(U(u), F(u))$ satisfies $\eqref{c7.1c}$
and the initial data function in $L^1_{{\rm loc}}$ satisfies $\eqref{c7.19}$.
The corresponding fine properties of entropy solutions are also presented.

\smallskip
\section{New Formula for Entropy Solutions and Its Main 
Application Results}
In this section, we first introduce a new formula
for entropy solutions of
the Cauchy problem $\eqref{c1.1}$--$\eqref{ID}$ with
the flux function $f(u)$ satisfying $\eqref{c1.2}$, 
and then validate it by proving that the function 
defined by this newly derived formula
solves the Cauchy problem $\eqref{c1.1}$--$\eqref{ID}$ uniquely
in Theorem $\ref{the:mt}$.
As the direct applications of this newly derived formula,
the refined $L^1-$contraction
and some pointwise properties such as the monotonicity in $L^\infty$
are proved in Corollary $\ref{cor:c2.1}$,
and the refined semigroup properties 
are proved in Corollary $\ref{cor:c2.2}$.

In \S 2.2,
we present a summary of the main results on the fine properties of entropy solutions,
which are elucidated in \S 3--\S 7
via leveraging this newly derived formula.

\smallskip
We first state the notions of weak solutions and entropy solutions, respectively.

\begin{Def}[Weak Solutions]\label{def:ws}
A bounded measurable function $u=u(x,t)$ on $\mathbb{R}\times \mathbb{R}^+$ is called
to be a weak solution of the Cauchy problem $\eqref{c1.1}$--$\eqref{ID}$ if $u=u(x,t)$ satisfies
\begin{eqnarray}
&&\iint_{\mathbb{R}\times \mathbb{R}^+}\big(u\phi_t+f(u)\phi_x\big)\,{\rm d}x{\rm d}t=0
\qquad\,\,\,\,\,
{\rm for\ any}\ \phi=\phi(x,t)\in C^\infty_{\rm c}(\mathbb{R}\times \mathbb{R}^+),\label{c1.3}\\[2mm]
&&\lim_{t\rightarrow 0{+}}\int_{x_1}^{x_2}u(x,t)\,{\rm d}x=\int_{x_1}^{x_2}\varphi(x)\,{\rm d}x
\qquad\,\,{\rm for\ any}\ x_1,x_2\in \mathbb{R}.\label{c1.3c}
\end{eqnarray}
\end{Def}

\begin{Def}[Entropy Solutions]\label{def:es}
A weak solution $u=u(x,t)$ on $\mathbb{R}\times \mathbb{R}^+$ is called
to be
an entropy solution of the Cauchy problem $\eqref{c1.1}$--$\eqref{ID}$
if it satisfies the Oleinik-type one-sided inequality{\rm :}
\begin{equation}\label{c1.4}
f'(u(x_2,t))-f'(u(x_1,t))\leq \frac{x_2-x_1}{t}\qquad\,\,\,
{\rm for\ any}\ x_2>x_1\ {\rm and}\ t>0.
\end{equation}
\end{Def}

\subsection{New solution formula and its direct corollaries}
We first introduce the new formula for entropy solutions.

For any fixed $(x,t)\in \mathbb{R}\times \mathbb{R}^+$,
we define a function $E(\cdot\,;x,t)$ in variable $u\in \mathbb{R}$ by
\begin{equation}\label{c2.1}
E(u;x,t):=t\int_0^u f''(s)\big(\varphi(x-tf'(s))-s\big)\,{\rm d}s.
\end{equation}
Without loss of generality,
we set $f(0)=0$.

\vspace{2pt}
From $\eqref{c1.2}$, $f'(u)$ is strictly increasing. 
Then, for any $(x,t)\in \mathbb{R}\times \mathbb{R}^+$,
\begin{equation*}
\varphi (x-tf'(s))<s \,\,\,\,{\rm if}\ s>\|\varphi\|_{L^\infty},
\qquad \,\,\, \varphi (x-tf'(s))>s \,\,\,\,{\rm if}\ s<-\|\varphi\|_{L^\infty}.
\end{equation*}
Thus, by $\eqref{c1.2}$,
if $s>\|\varphi\|_{L^\infty}$, then
$$
f''(s)\big(\varphi(x-tf'(s))-s\big)\leq 0
$$
so that $E(\cdot\,; x,t)$ is strictly decreasing;
and, if $s<-\|\varphi\|_{L^\infty}$, then
$$
f''(s)\big(\varphi(x-tf'(s))-s\big)\geq 0
$$
so that $E(\cdot\,; x,t)$ is strictly increasing.
Therefore, for any $(x,t)\in \mathbb{R}\times \mathbb{R}^+$,
$E(\cdot\,; x,t)$ always attains its maximum in the interval
$[-\|\varphi\|_{L^\infty},\|\varphi\|_{L^\infty}]$.

\smallskip
Let $\mathcal{U}(x,t)$ be the set of points at which $E(\cdot\,; x,t)$ attains its maximum:
\begin{equation}\label{c2.3}
\mathcal{U}(x,t):=\big\{w\in \mathbb{R}\,:\, E(w;x,t)=\max_{v\in \mathbb{R}}E(v;x,t)\big\}.
\end{equation}
Then, by the continuity of $E(\cdot\,;x,t)$, $\,\mathcal{U}(x,t)$ is a bounded closed set
and $\mathcal{U}(x,t)\neq \varnothing$.

\begin{Def}
We define $u^-(x,t)$ and $u^+(x,t)$ on $\mathbb{R} \times \mathbb{R}^+$ by
\begin{equation}\label{c2.4}
u^-(x,t)=\sup\, \mathcal{U}(x,t), \qquad\,\,\, u^+(x,t)=\inf\, \mathcal{U}(x,t).
\end{equation}
\end{Def}

Since $\mathcal{U}(x,t)$ is closed,
then $u^\pm(x,t)\in\mathcal{U}(x,t)$ (see {\rm Fig.} $\ref{figMaxU}$).
For any $(x,t)\in \mathbb{R}\times \mathbb{R}^+$, let
\begin{equation}\label{c2.5}
\quad \hat{E}(x,t):=E(u^\pm(x,t);x,t)-\int_0^{x-tf'(0)}\varphi(\xi)\,{\rm d}\xi
=\max_{v\in \mathbb{R}}E(v;x,t)-\int_0^{x-tf'(0)}\varphi(\xi)\,{\rm d}\xi.
\end{equation}

\begin{figure}[H]
	\begin{center}
		{\includegraphics[width=0.65\columnwidth]{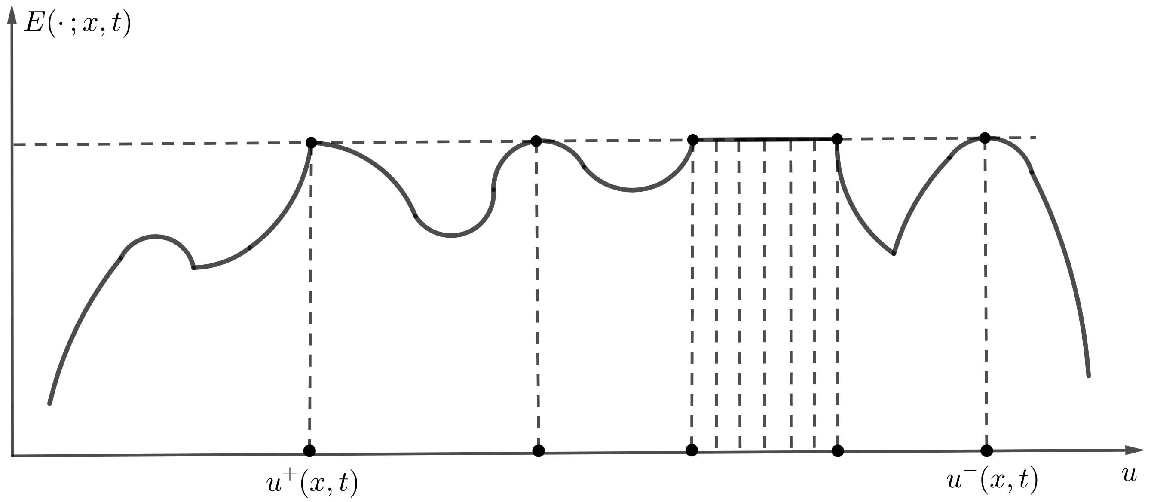}}
 \caption{$u^\pm(x,t)\in\mathcal{U}(x,t)\subset[u^+(x,t),u^-(x,t)]$.}\label{figMaxU}
	\end{center}
\end{figure}

We now present the main theorem on the new formula for entropy solutions.

\begin{The}[New formula for entropy solutions]\label{the:mt}
Let $f(u)\in C^2(\mathbb{R})$ be a convex flux function of $\eqref{c1.1}$ satisfying $\eqref{c1.2}$,
and let the initial data function $\varphi(x)$ be in $L^\infty(\mathbb{R})$.
Let $u=u(x,t)$ on $ \mathbb{R} \times \mathbb{R}^+$ be defined by
\begin{equation}\label{c2.50}
u(x,t):=u^+(x,t)=\inf \big\{w\in \mathbb{R}\,:\, E(w;x,t)=\max_{v\in \mathbb{R}}E(v;x,t)\big\}
\end{equation}
with $E(u;x,t)$ given by $\eqref{c2.1}$.
Then $\eqref{c2.50}$ is the solution formula of the Cauchy problem \eqref{c1.1}--\eqref{ID}
in the following sense{\rm :}
$u=u(x,t)$, defined by $\eqref{c2.50}$,
is the unique entropy solution of the Cauchy problem $\eqref{c1.1}$--$\eqref{ID}$.

Furthermore, the left- and right-traces of $u=u(x,t)$ exist pointwise and satisfy
 \begin{equation}\label{c2.50a}
u(x{-},t)=u^-(x,t),\qquad u(x{+},t)=u^+(x,t)
\qquad {\rm on}\ (x,t)\in \mathbb{R} \times \mathbb{R}^+;
\end{equation}
moreover, for each $t>0$,
$u(x-,t)=u(x+,t)$ holds except at most a countable set of $(x,t)$ at which $u(x-,t)>u(x+,t)$, and
\begin{equation}\label{c2.38}
f'(u(\cdot,t))\in {\rm BV}_{{\rm loc}}(\mathbb{R})
\qquad\,\,\, {\rm for\ any }\ t>0.
\end{equation}
\end{The}

\begin{Rem}
Theorem $\ref{the:mt}$ still holds if the solution is alternatively defined by
$$u(x,t):=u^-(x,t)=\sup \big\{w\in \mathbb{R}\,:\, E(w;x,t)=\max_{v\in \mathbb{R}}E(v;x,t)\big\}.$$
In fact, as shown in Theorem $\ref{the:mt}$,
we can prove that
$u(x-,t)=u(x+,t)$ holds except at most a countable set of $(x,t)$ for each $t>0$,
then it follows from $\eqref{c2.50a}$ that
$u^-(x,t)=u^+(x,t)$ almost everywhere on $\mathbb{R} \times \mathbb{R}^+$
(see also Lemma $\ref{lem:c2.3}$).
\end{Rem}

\begin{Cor}\label{cor:c2.1}
The entropy solutions, given by the solution formula $\eqref{c2.50}$,
satisfy the following properties{\rm :}
\begin{itemize}
\item [(i)] Let $u_1(x,t)$ and $u_2(x,t)$ as in $\eqref{c2.50}$ be determined by $E(u;x,t)$ in $\eqref{c2.1}$
with the $L^\infty$ initial data functions $\varphi_1(x)$ and $\varphi_2(x)$, respectively.
Then, for any $t>0$ and $x_2>x_1$,
 \begin{equation}\label{c2.51}
\int_{x_1}^{x_2}|u_2(x,t)-u_1(x,t)|\,{\rm d}x
\leq\int_{x_1-tf'(M(x_1,t))}^{x_2-tf'(m(x_2,t))}|\varphi_2(x)-\varphi_1(x)|\,{\rm d}x,
\end{equation}
where $M(x,t):=\max\{u_1(x{-},t),u_2(x{-},t)\}$ and $m(x,t):=\min\{u_1(x{+},t),u_2(x{+},t)\}$.

\smallskip
\item[(ii)] Let $u_i(x,t)$ as in $\eqref{c2.50}$ be determined by $E(u;x,t)$ in $\eqref{c2.1}$
with the $L^\infty$ initial data functions $\varphi_i(x)$ for $i=1,2$.
If $\varphi_1(x)\leq\varphi_2(x)$ holds almost everywhere on $\mathbb{R}$, then
 \begin{equation}\label{c2.50b}
 u_1(x\pm,t)\leq u_2(x\pm,t)
 \qquad\,\,\, { \rm for\ any}\ (x,t)\in \mathbb{R}\times \mathbb{R}^+.
\end{equation}

\item[(iii)] Let $u_n(x,t)$ and $u(x,t)$ as in $\eqref{c2.50}$ be determined by $E(u;x,t)$ in $\eqref{c2.1}$
with the $L^\infty$ initial data functions $\varphi_n(x)$ and $\varphi(x)$, respectively.
If $\{\varphi_n(x)\}_n$ is an increasing sequence such that
$\lim_{n\rightarrow \infty}\|\varphi_n-\varphi\|_{L^\infty}=0$, then
\begin{equation}\label{c7.30}
\lim_{n\rightarrow \infty}u_n(x+,t)=u(x+,t)
\qquad {\rm for\ any}\ (x,t)\in \mathbb{R}\times \mathbb{R}^+;
\end{equation}
and if $\{\varphi_n(x)\}_n$ is a decreasing sequence such that
$\lim_{n\rightarrow \infty}\|\varphi_n-\varphi\|_{L^\infty}=0$, then
\begin{equation}\label{c7.31}
\lim_{n\rightarrow \infty}u_n(x-,t)=u(x-,t)
\qquad {\rm for\ any}\ (x,t)\in \mathbb{R}\times \mathbb{R}^+.
\end{equation}
\end{itemize}
\end{Cor}

\begin{Rem}
In {\rm Corollary} $\ref{cor:c2.1}$,
the inequality in $\eqref{c2.51}$ provides a refined $L^1$--contraction property
with the exact upper and lower limits of the integral on the $x$--axis{\rm ;}
the inequalities in $\eqref{c2.50b}$ provide a refined monotonicity in $L^\infty$
which holds pointwise, instead of almost everywhere{\rm ;}
and the limits in $\eqref{c7.30}$--$\eqref{c7.31}$ provide
the pointwise convergence of entropy solutions
for the monotone sequence of the initial data,
which are not found in previous literature.
\end{Rem}

\smallskip
For any fixed $\tau>0$, consider the following Cauchy problem:
\begin{equation}\label{c2.59}
\begin{cases}
v_t+f(v)_x=0 \qquad & {\rm for}\ x\in \mathbb{R},\ t>\tau, \\[1mm]
v(x,t)|_{t=\tau}=u(x,\tau)\qquad & {\rm for}\ x\in \mathbb{R},\ t=\tau,
\end{cases}
\end{equation}
where $u(x,\tau)$ is obtained by restricting $u(x,t)$ in $\eqref{c2.50}$ to the line: $t=\tau$.

\vspace{1pt}
Similar to $\eqref{c2.1}$,
for any $(x,t)\in \mathbb{R}\times \mathbb{R}^+$ with $t>\tau$, we define
\begin{equation}\label{c2.57}
E(v;x,t;\tau):=(t-\tau)\int_0^v f''(s)\big(u(x-(t-\tau)f'(s),\tau)-s\big)\,{\rm d}s,
\end{equation}
and similar to $\eqref{c2.3}$--$\eqref{c2.5}$,
\begin{equation}\label{c2.58}
\begin{cases}
\displaystyle\mathcal{U}(x,t;\tau)
:=\big\{v\in \mathbb{R}\,:\, E(v;x,t;\tau)=\max_{w\in\mathbb{R}} E(w;x,t;\tau)\big\}, \\[2mm]
\displaystyle v^-(x,t):=\sup\,\mathcal{U}(x,t;\tau),
\qquad v^+(x,t):=\inf\, \mathcal{U}(x,t;\tau),\\[2mm]
\displaystyle \hat{E}(x,t;\tau):=\max_{w\in\mathbb{R}} E(w;x,t;\tau)+\hat{E}(x-(t-\tau)f'(0),\tau).
\end{cases}
\end{equation}

By Theorem $\ref{the:mt}$,
$v(x,t):=v^+(x,t)$ is the unique entropy solution of the Cauchy problem $\eqref{c2.59}$.
In fact, we have stronger results as follows:

\begin{Cor}[Refined semigroup properties of entropy solutions determined by the new solution formula]
\label{cor:c2.2}
Let $\mathcal{U}(x,t)$ and $\hat{E}(x,t)$ be defined by $\eqref{c2.3}$ and $\eqref{c2.5}$ respectively,
and let $\mathcal{U}(x,t;\tau)$ and $\hat{E}(x,t;\tau)$ for $t>\tau>0$
be defined by $\eqref{c2.58}$.
Then,
for any $(x,t)$ with $t>\tau$,
\begin{equation}\label{c2.61c}
\mathcal{U}(x,t;\tau)=\mathcal{U}(x,t),\qquad \hat{E}(x,t;\tau)=\hat{E}(x,t).
\end{equation}
In particular, $v(x,t)$ satisfies that
\begin{equation}\label{c2.61}
v(x\pm,t)=u(x\pm,t)\qquad\,\,\, {\rm for \ any}\ x\in \mathbb{R},\ t> \tau,
\end{equation}
and the map
$S_t:L^\infty(\mathbb{R})\times [0,\infty)\rightarrow L^\infty(\mathbb{R})$, associated with
$S_t\varphi=u(\cdot,t)$ for $u(x,0)=\varphi(x)\in L^\infty(\mathbb{R}),$
determines a contractive semigroup in the $L^1_{\rm loc}$--norm.
\end{Cor}

\begin{Rem}
As shown in Corollary $\ref{cor:c2.2}$,
$\eqref{c2.61}$ implies the semigroup property of entropy solutions, {\it i.e.},
for any $\varphi(x)\in L^\infty(\mathbb{R})$,
$S_t\varphi=S_{t-\tau}S_{\tau}\varphi$ for $\tau\in [0,t].$
Moreover, since $\eqref{c2.61}$ holds pointwise,
it can infer the fine properties of entropy solutions more than the semigroup property.
For examples,
by the arbitrariness of $\tau>0$ in $\eqref{c2.59}$,
{\rm Corollary} $\ref{cor:c2.2}$
{\rm (}especially the equality{\rm :} $v(x\pm,t)=u(x\pm,t)$ in $\eqref{c2.61}${\rm )}
is a powerful tool to prove the local structures of entropy solutions; in particular,
it has been used to establish
the development of shocks near all the five types of continuous shock generation
points in {\rm Theorem} $\ref{the:dsw1}$ and prove the left- and
right-derivatives of shock curves in Theorem 5.5.
\end{Rem}

\subsection{Main results on fine properties of entropy solutions via the solution formula}
In \S 3--\S 7,
we leverage the newly derived solution formula
$\eqref{c2.50}$ with $E(u;x,t)$ given by $\eqref{c2.1}$
to elucidate the fine properties of entropy solutions, and
obtain the global dynamic patterns of entropy solutions
for the Cauchy problem $\eqref{c1.1}$--$\eqref{ID}$ with
{\it flux functions allowing for convexity degeneracy and initial data only in $L^\infty$}, 
which can be summarized in four aspects{\rm :}

\smallskip
\noindent
{\bf 1. Global dynamic patterns of entropy solutions:} 
Initial waves for the Cauchy problem,
fine structures of entropy solutions, 
invariants and global structures of entropy solutions, 
and asymptotic behaviors and decay rates 
respectively in the $L^\infty$--norm and the $L^p_{\rm loc}$--norm.
\begin{itemize}
\item [{\rm (i)}] {\it Generation of initial waves for the Cauchy problem.}
The criteria for all the six types of initial waves for the Cauchy problem
are established by proving the necessary and sufficient conditions
for all the six types of initial waves for the Cauchy problem
in {\rm Theorem} $\ref{the:c4.1}$ associating with Theorem $\ref{the:c4.0}$,
which includes the case that the point is not a Lebesgue point of the initial data function
({\it cf.}\, Fig. $\ref{figPhiFab}$).

\begin{figure}[H]
\begin{center}
{\includegraphics[width=0.65\columnwidth]{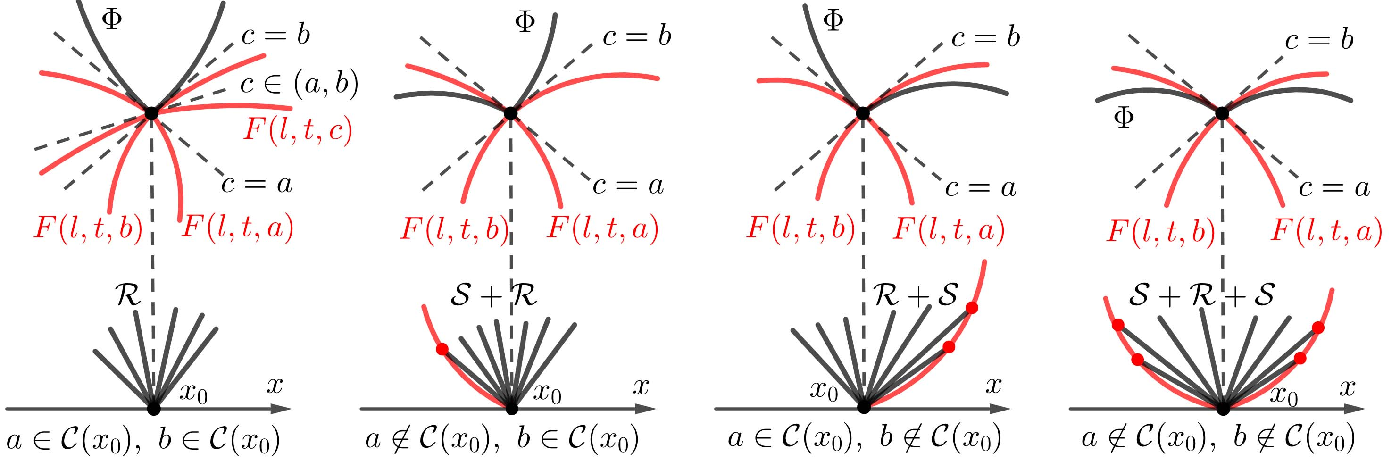}}
\caption{Initial waves for the Cauchy problem emitting from point $x_0$
when $a=\overline{{\rm D}}_- \Phi(x_0)<\underline{{\rm D}}_+ \Phi(x_0)=b.$
For such four cases, the initial waves 
are the combination of a rarefaction wave with/without one or two shocks.
}\label{figPhiFab}
	\end{center}
\end{figure}

\item [{\rm (ii)}]
{\it Fine structures of entropy solutions.}
The fine structures of entropy solutions are obtained
by proving the structures of entropy solutions inside the backward characteristic triangles in {\rm Theorem} $\ref{pro:c3.3}$,
the directional limits of entropy solutions at discontinuous points in {\rm Theorem} $\ref{pro:c3.4}$,
the general structures of shocks in Theorem $\ref{pro:c3.6}$,
and the fine structures of entropy solutions on the shock sets in Theorem $\ref{the:c4.2}$ ({\it cf.} Fig. $\ref{figStriangle}$).

\begin{figure}[H]
	\begin{center}
		{\includegraphics[width=0.6\columnwidth]{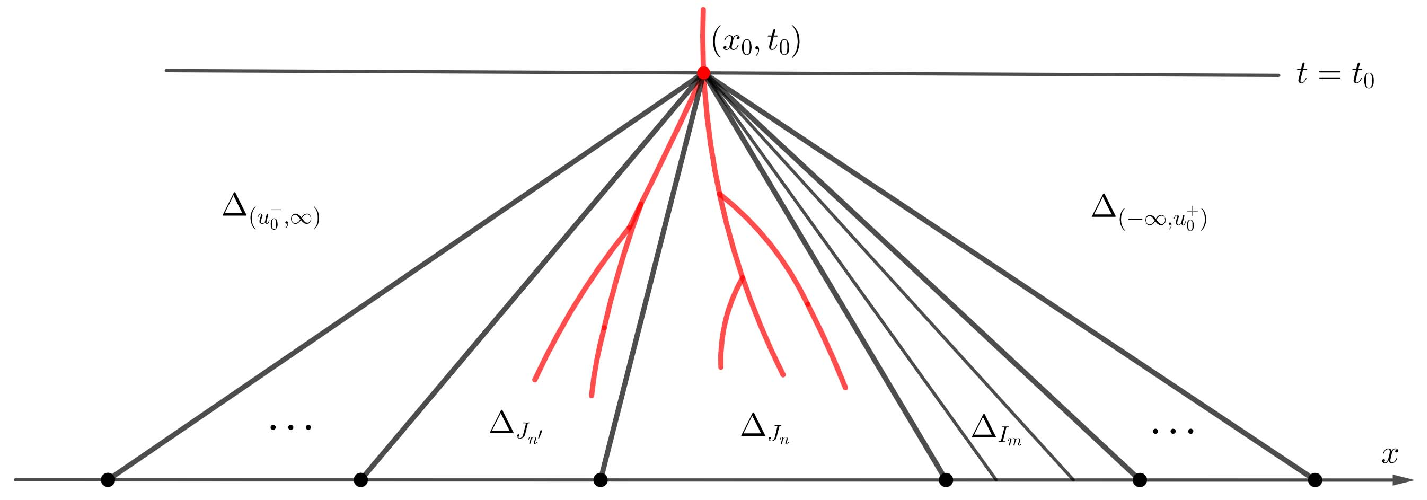}}
 \caption{Structures of entropy solutions inside the backward characteristic triangles
from discontinuous points.
In each $\Delta_{I_m}$, there is a centered compression wave; 
and, in each $\Delta_{J_n}$, any two shocks must coincide each other before time $t_0$.
}\label{figStriangle}
	\end{center}
\end{figure}

\item [{\rm (iii)}]
{\it Invariants and global structures of entropy solutions.}
The four invariants of entropy solutions, including $\overline{u}_l$ and $\underline{u}_r$,
are introduced and proved in Theorem $\ref{pro:c6.1}$.
Using the results on invariants and divides developed in this paper,
the global structures of entropy solutions are obtained in {\rm Theorem} $\ref{pro:c6.2}$ 
by proving all the path-connected branches of the shock set of entropy solutions
(see Figs. $\ref{figGSSD}$--$\ref{figGSSDKabab}$).

\vspace{1pt}
\item [{\rm (iv)}]
{\it Asymptotic profiles and decay rates in the $L^\infty$--norm.}
The asymptotic profiles of entropy solutions in the $L^\infty$--norm are proved to be
either a rarefaction-constant solution or a single shock
in {\rm Theorem} $\ref{pro:c6.3}$ with {\rm Lemmas} $\ref{the:c6.1}$--$\ref{lem:c6.3}$,
and the corresponding decay rates are obtained 
in {\rm Theorems} $\ref{the:c6.2}$--$\ref{the:c6.3}$
(see Figs. $\ref{figGSSD}$--$\ref{figGSSDKabab}$).

\begin{figure}[H]
	\begin{center}
		{\includegraphics[width=0.5\columnwidth]{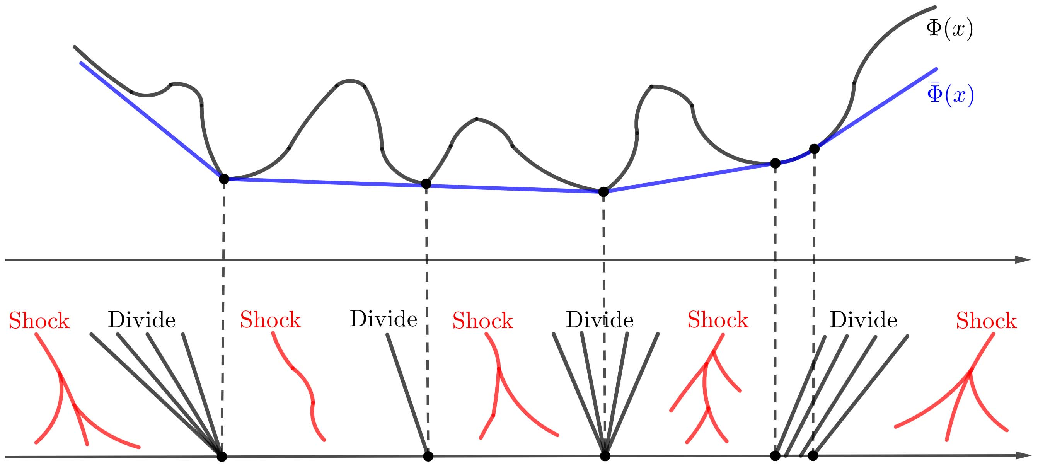}}
\caption{Divides, global structures of entropy solutions,
and asymptotic profiles in the $L^\infty$--norm
when $\overline{u}_l<\underline{u}_r$.
The divides form isolated characteristics,
centered rarefaction waves, and/or rarefaction regions;  
and, in each of the 
parallelogram regions, 
any two shocks must coincide each other in a finite time.
 }\label{figGSSD}
	\end{center}
\end{figure}

\begin{figure}[H]
	\begin{center}
		{\includegraphics[width=0.5\columnwidth]{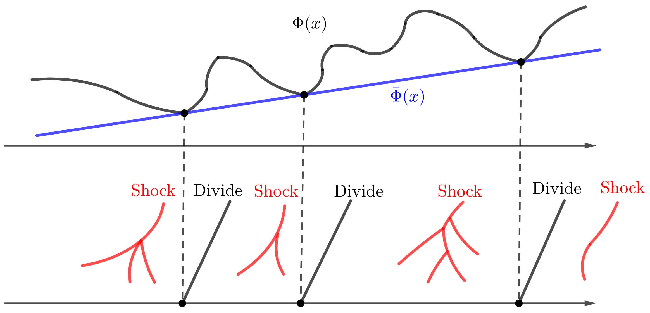}}
\caption{Divides, global structures of entropy solutions,
and asymptotic profiles in the $L^\infty$--norm
when $\overline{u}_l=\underline{u}_r$ with $\mathcal{K}_0\neq \varnothing$.
Every divide possesses the same speed $f'(\underline{u}_r)$; 
and, in each of the remaining parallelogram regions, any two shocks must coincide each other in a finite time.
 }\label{figGSSDKab}
	\end{center}
\end{figure}

\begin{figure}[H]
	\begin{center}
		{\includegraphics[width=0.5\columnwidth]{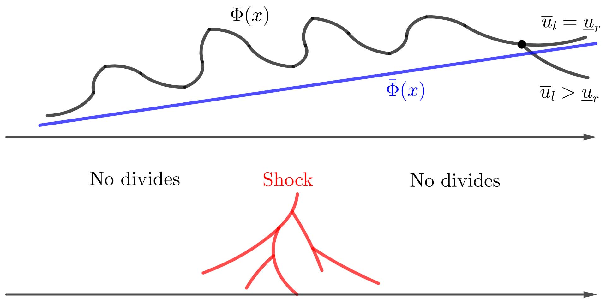}}
\caption{Divides, global structures of entropy solutions,
and asymptotic profiles in the $L^\infty$--norm
when $\overline{u}_l>\underline{u}_r$ or
$\overline{u}_l=\underline{u}_r$ with $\mathcal{K}_0=\varnothing$.
Any two shocks in $\mathbb{R}\times\mathbb{R}^+$ must coincide each other in a finite time 
(for this case, the entropy solution possesses no divides).
 }\label{figGSSDKabab}
	\end{center}
\end{figure}

\vspace{1pt}
\item [{\rm (v)}]
{\it Asymptotic profiles and decay rates in the $L^p_{{\rm loc}}$--norm.}
Introducing the generalized $N$--waves in {\rm Definition} $\ref{def:n.1}$,
we prove that the entropy solutions, which possess at least one divide,
decay to the generalized $N$--waves in the $L^p_{{\rm loc}}$--norm
in {\rm Lemma} $\ref{lem:n.1}$ and {\rm Theorem} $\ref{the:n.1}$,
in which the corresponding decay rates are also obtained
(see Fig. $\ref{figNwavth}$).

\begin{figure}[H]
	\begin{center}
		{\includegraphics[width=0.65\columnwidth]{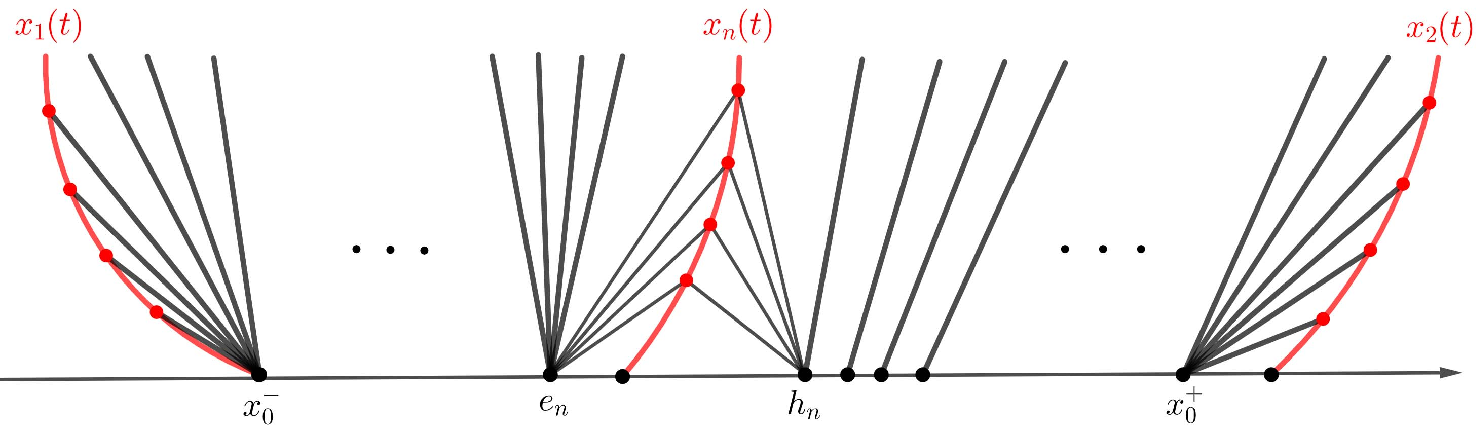}}
\caption{Generalized $N$--waves as the asymptotic profiles in the $L^p_{{\rm loc}}$--norm. 
The generalized $N$--wave is defined in the form of 
$w(x,t)=(f')^{-1}(\frac{x-\xi}{t})$ 
for $\xi$ belonging to the closed set of the divide generation points on the $x$--axis.
		}\label{figNwavth}
	\end{center}
\end{figure}
\end{itemize}

\noindent
{\bf 2. Dynamic behaviors of characteristics:} The generation, spread, and termination of characteristics;
and the divides as well as their locations and speeds.
\begin{itemize}
\item [{\rm (i)}]
{\it Generation of characteristics.} 
The generation of characteristics emitting from points on the $x$--axis
is completely solved in Theorem $\ref{the:c4.0}$ 
by providing and proving the criteria for all the cases of $\mathcal{C}(x_0)$  
(as the set of characteristic generation values of $x_0$) defined via the new solution formula, 
which includes the case that the point $x_0$ is not a Lebesgue point of the initial data function
({\it cf.}\, Fig. $\ref{figPhiFab}$).

\vspace{1pt}
\item [{\rm (ii)}]
{\it Spread of characteristics.}
The spread of characteristics is established in Theorems $\ref{lem:plc1}$--$\ref{pro:elc1}$
via introducing and proving the lifespan of characteristics determined by the new solution formula.

\vspace{1pt}
\item [{\rm (iii)}]
{\it Termination of characteristics.}
Using the upper bound of lifespan and the lifespan of characteristics,
we provide a classification of characteristics by the way of their terminations in Theorem $\ref{the:elc0}$,
especially including the way of terminations by the formation of the continuous shock generation points.

\vspace{1pt}
\item [{\rm (iv)}]
{\it Divides as well as their locations and speeds.}
The necessary and sufficient conditions respectively for the divides as well as their locations and speeds are provided and proved in Proposition $\ref{pro:c5.1}$ and {\rm Theorem} $\ref{the:c5.1}$, 
in which we determine the locations and speeds of divides 
by introducing the convex hull over $(-\infty,\infty)$ 
of the primitive function of the initial data function
(see Figs. $\ref{figGSSD}$--$\ref{figGSSDKabab}$). 
\end{itemize}

\noindent
{\bf 3. Dynamic behaviors of shocks:} The formation and development of shocks, the spread of shocks, and the asymptotic behaviors of shocks.
\begin{itemize}
\item [{\rm (i)}]
{\it Generation of shocks on the $x$--axis and formation of shocks.}
The generation of shocks emitting from points on the $x$--axis is established
by proving the necessary and sufficient conditions
for the shock generation points on the $x$--axis in {\rm Theorems} $\ref{the:c4.0}$--$\ref{the:c4.1}$,
and the formation of shocks is obtained by introducing and proving the necessary and sufficient conditions
for all the five types of the continuous shock generation points in {\rm Theorem} $\ref{the:c4.3}$
and for the discontinuous shock generation points in {\rm Theorem} $\ref{the:c4.2}${\rm (iii)}
associating with Theorem $\ref{pro:c3.3}${\rm (ii)}.

\vspace{1pt}
\item [{\rm (ii)}]
{\it Development of shocks.}
The development of shocks is obtained by proving the optimal regularities of shock curves
and entropy solutions near all the five types of continuous shock generation points in {\rm Theorem} $\ref{the:dsw1}$,
in which the continuous shock generation points are not necessarily assumed to be isolated
({\it cf}. Fig. $\ref{figUCSGPab}$).

\vspace{1pt}
\item [{\rm (iii)}]
{\it Spread of shocks.}
The spread of shocks is precisely described by both providing the two associated triangle sequences
and proving the left- and right-derivatives and the piecewise smoothness of shock curves in {\rm Theorem} $\ref{pro:c3.6}$.

\vspace{1pt}
\item [{\rm (iv)}]
{\it Asymptotic behaviors of shocks.}
The asymptotic behaviors of shocks are established by both providing 
all the path-connected branches of the shock sets of entropy solutions in {\rm Theorem} $\ref{pro:c6.2}$
and proving the large-time properties
especially including the limits of speeds and the left- and right-states of shocks
in {\rm Lemmas} $\ref{the:c6.1}$--$\ref{lem:c6.3}$ and {\rm Theorem} $\ref{pro:c6.3}$.
\end{itemize}

\begin{figure}[H]
	\begin{center}
		{\includegraphics[width=0.69\columnwidth]{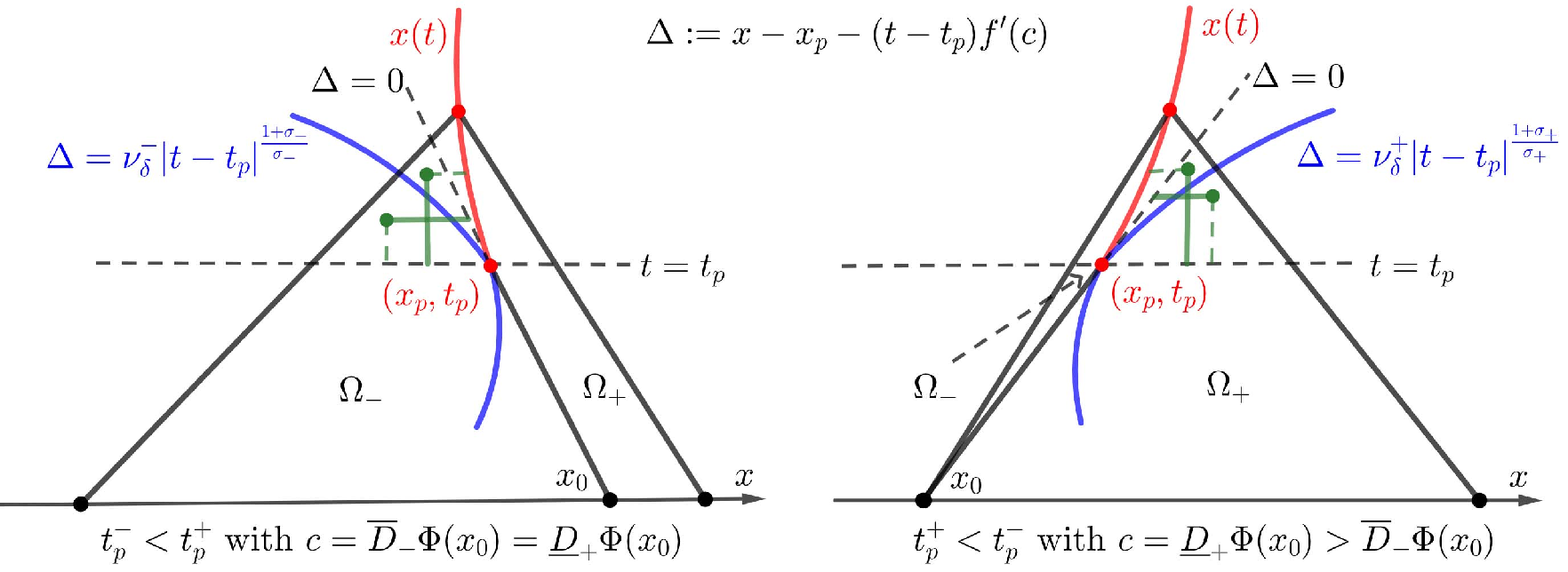}}
 \caption{The shock curve $x(t)$ for the case that $t_p^+\neq t_p^-$. 
 The left/right figure illustrates the case of shocks formatted by the compression of local characteristics 
 emitting from points on the left/right of $x_0$.
 }\label{figUCSGPab}
	\end{center}
\end{figure}

\noindent
{\bf 4. Influences of the convexity degeneracy of flux functions.}
{\it The convexity degeneracy of the flux function $f(u)$
{\rm(}the infinitesimal order of $f''(u)$ at degeneracy point{\rm)}
influences the behaviors of entropy solutions substantially
via reducing the strength of the compression of characteristics.}
These especially include:
\begin{itemize}
\item [{\rm (i)}] The generation of initial waves for the Cauchy problem,
including the types of characteristics and shocks,
as shown in Propositions $\ref{pro:c4.3}$--$\ref{pro:c4.4}$.

\vspace{1pt}
\item [{\rm (ii)}] The formation and development of shocks,
as shown in Theorems $\ref{the:c4.3}$ and $\ref{the:dsw1}$.

\vspace{1pt}
\item [{\rm (iii)}] The decay rates respectively in the $L^\infty$--norm and the $L^p_{{\rm loc}}$--norm,
as shown in Theorems $\ref{the:c6.2}$--$\ref{the:c6.3}$
and in Theorem $\ref{the:n.1}$ with Lemma $\ref{lem:n.1}$, respectively.
\end{itemize}

\smallskip
\noindent
{\bf Innovation points.} 
The innovation points on the new solution formula and the fine properties of entropy solutions 
can be summarized as follows:

\begin{itemize}
\item [(i)] {\it New solution formula and its derivation.}  
We first introduce and validate a novel formula for entropy solutions of scalar hyperbolic conservation laws 
with flux functions of convexity degeneracy 
that allow the state points of convexity degeneracy and asymptotic lines, 
which aligns with the well-known Lax--Oleinik formula \cite{LPD57,LPD73,OOA} 
in the case that the flux function is uniformly convex. 

Moreover, the derivation of this new formula is rooted in the understanding of
the intrinsic formula for smooth solutions,
which offers a new approach to identifying formulas
for entropy solutions in more general hyperbolic conservation laws
where such intrinsic formulas are available.

\vspace{2pt}
\item [(ii)] {\it Global dynamic patterns of entropy solutions.}
The global dynamic patterns of entropy solutions
with {\it general initial data in $L^\infty$} are new,
even for the case when the flux function $f(u)$ is uniformly convex. 
The global dynamic patterns of entropy solutions are established
by leveraging this newly derived formula and introducing new concepts, 
which especially include: 

\begin{itemize}
\item [(a)] {\it Fine structures of entropy solutions.}
By introducing the set $\mathcal{C}(x_0)$ of the speeds of characteristic lines 
emitting from any $x_0$ on the $x$--axis, the criteria for all six types of initial waves 
for the Cauchy problem are established.
Introducing the set $\mathcal{U}(x,t)$ of the speeds of the shock-free backward characteristics 
emitting from any point $(x,t)$ in the upper half-plane and the lifespan of characteristics, 
we obtain the structures of entropy solutions inside the backward characteristic triangles,
the directional limits of entropy solutions at discontinuous points, 
the left-derivatives of shock curves,
and the formation and development of shocks for all five types of
continuous shock generation points.

\vspace{1pt}
\item [(b)] {\it Global structures of entropy solutions.}
By introducing the invariants of entropy solutions and the convex hull over $({-}\infty,\infty)$ of 
the primitive function of the initial data function, 
the necessary and sufficient conditions respectively 
for the divides as well as their locations and speeds are provided, 
and the global structures of entropy solutions are established. 

\vspace{1pt}
\item [(c)] {\it Asymptotic behaviors of entropy solutions.}
By introducing the generalized $N$--waves and using the results on invariants and divides developed in this paper, 
the asymptotic behaviors of entropy solutions (including the asymptotic profiles and decay rates)
respectively in the $L^\infty$--norm and the $L^p_{\rm loc}$--norm are established.
\end{itemize}
\end{itemize}

\smallskip
\noindent
{\bf Comparison with the previous results.}
The comparison between the previously available results and the results in this paper
can be summarized as follows:
The previously available results, as discussed above, 
are mainly restricted to the Cauchy problem $\eqref{c1.1}$--$\eqref{ID}$
with the flux function $f(u)$ satisfying that
$$
f''(u)\geq c_0>0 \qquad {\rm on}\ u \in \mathbb{R}.
$$
While, in this paper,
the flux function $f(u)$ is only required to satisfy \eqref{c1.2}, $i.e.$,
$$
f''(u)\geq 0 \,\,\,\, {\rm on}\ u \in \mathbb{R},\qquad\,\,\, \mathcal{L}\{u\,:\, f''(u)=0\}=0,
$$
not only for the results of the new solution formula,
but also for the results on the fine properties of entropy solutions
of the Cauchy problem $\eqref{c1.1}$--$\eqref{ID}$ with
{\it initial data only in $L^\infty$}.

\vspace{2pt}
In \S 2, since the Lax--Oleinik formula \cite{LPD57,LPD73,OOA} requires that $f''(u)\geq c_0>0$,
it can not be applied directly to
the case of convex functions $f(u)$ with $f''(u)=0$ at some points and/or with asymptotic lines,
such as $f(u)=u^{2n}$ with $n>1$ and $f(u)=e^{ku}$ with $k\neq 0$.
From $\eqref{c1.2}$, our new solution formula is valid for such cases.
Moreover, as shown in Theorem $\ref{the:mt}$,
the entropy solutions of the Cauchy problem $\eqref{c1.1}$--$\eqref{ID}$ also possess strong traces, 
even when the flux function $f(u)$ in $\eqref{c1.2}$ has the state points of convexity degeneracy.

\vspace{2pt}
As in \S 3, for the initial waves for the Cauchy problem,
when $f''(u)\geq c_0>0$,
Lin \cite{LLW} showed that there exist six types of
initial waves for the Cauchy problem,
and for the special case that the initial data functions possess
discontinuity points only of the first kind,
and then
gave a sufficient condition of shocks emitting from
continuous points of the initial data function
and a necessary condition of that, respectively;
However, the necessary and sufficient condition to judge
the concrete type of initial waves emitting from
points on the $x$--axis have been opened.
In this paper, for the flux functions satisfying $\eqref{c1.2}$ 
and the initial data only in $L^\infty$,
the criteria for all six types of initial waves for the Cauchy problem
is established by providing the necessary and sufficient conditions
for all the six types of initial waves for the Cauchy problem in Theorems $\ref{the:c4.0}$--$\ref{the:c4.1}$,
even for the case that the point is not a Lebesgue point of the initial data function.
Furthermore, we introduce and prove the lifespan of characteristics in Theorems $\ref{lem:plc1}$--$\ref{pro:elc1}$,
and provide a classification of characteristics by the way of their terminations in Theorem $\ref{the:elc0}$.

\vspace{2pt}
As in \S 4, for the formation and development of shocks,
Lebaud \cite{LMP} and Yin--Zhu \cite{YZ} proved that the singularity
around the first blowup point caused by the initial compression is actually a shock,
under some specific degenerate information of sufficiently smooth initial data functions.
In \S 4.1, we discover all the five types of continuous shock generation points
and prove the necessary and sufficient conditions of them in Theorem $\ref{the:c4.3}$;
in \S 4.2, we prove the optimal regularities of shock curves and entropy solutions
near all the five types of continuous shock generation points in Theorems $\ref{the:dsw1}$,
in which the continuous shock generation points
are not necessarily assumed to be isolated.

\vspace{2pt}
As in \S 5, for the fine structures of entropy solutions, 
when $f''(u)\geq c_0>0$,
many results about the local structures of entropy solutions
by the method of generalized characteristics
are presented in Dafermos \cite[Chapter 11]{DCM1};
also see Huang \cite{HFM} for a description of the
backward characteristic triangles and the shock curves.
However, there exists no results on 
the structures of entropy solutions inside the backward characteristic triangles
and the directional limits of entropy solutions at the discontinuous points.
Even when the flux function $f(u)$ has the state points of convexity degeneracy as in $\eqref{c1.2}$, we provide and prove
the structures of entropy solutions inside the backward characteristic triangles
in Theorem $\ref{pro:c3.3}$ and
the directional limits of entropy solutions at the discontinuous points
(including the case for the points at which there are infinitely many shocks collided together)
in Theorem $\ref{pro:c3.4}$
via the vital exploitation of the new solution formula. 

\vspace{2pt}
As in \S 6, for the invariants, divides, and global structures of entropy solutions, when
$f''(u)\geq c_0>0$, Proposition $\ref{pro:c5.1}$
was given by Dafermos \cite[Theorem 11.4.1]{DCM1},
and some description of divides (without the locations and speeds)
and the global structures of shocks were given in Li \cite{LBH}.
As shown in Theorems $\ref{the:c5.1}$--$\ref{the:c5.2}$,
for general initial data in $L^\infty$,
the necessary and sufficient conditions respectively for the divides as well as their locations and speeds of entropy solutions are provided and proved.
As an illustration, the divides of entropy solutions for the $L^\infty$ initial data functions
with compact support, periodicity, or $L^1$-integrability, respectively,
are given in Examples $\ref{exa:c5.1}$--$\ref{exa:c5.3}$.
Furthermore, the global structures of entropy solutions with general initial data functions in $L^\infty$ are obtained in Theorem $\ref{pro:c6.2}$.

\vspace{2pt}
As in \S 7, for the asymptotic behaviors of entropy solutions, $f''(u)\geq c_0>0$,
Lax \cite{LPD57,LPD73} gave the invariants of entropy solutions
for the initial data in  $L^\infty\cap L^1$
and obtained the development of $N$--waves for the initial data with compact support
and the formation of sawtooth profiles for the periodic initial data;
Dafermos \cite{DCM2,DCM1} obtained several important
cases of
the asymptotic behaviors of entropy solutions by developing the method of generalized characteristics;
and Li--Wang \cite{LW} gave several asymptotic behaviors of entropy solutions
for some types of smooth initial data functions.
For the flux functions satisfying $\eqref{c1.2}$
and the initial data only in $L^\infty$,
we establish the asymptotic behaviors by proving
the detailed asymptotic profiles and the corresponding decay rates of entropy solutions
respectively in both the $L^\infty$--norm in Theorems $\ref{pro:c6.3}$--$\ref{the:c6.3}$
and the $L^p_{{\rm loc}}$--norm in Lemma $\ref{lem:n.1}$ and Theorem $\ref{the:n.1}$.

\subsection{Proofs of the solution formula and its corollaries}
In this subsection, we first prove Theorem $\ref{the:mt}$ on the new solution formula,
and then prove Corollaries $\ref{cor:c2.1}$--$\ref{cor:c2.2}$.

\vspace{1pt}
Before giving the proofs, we introduce some notions:
For any $(x,t)\in \mathbb{R} \times \mathbb{R}^+$ and $u\in \mathbb{R}$, we denote the set:
\begin{equation}\label{con}
{\rm Con}(x,t,u):=\big\{(\xi,\tau)\in \mathbb{R} \times \mathbb{R}^+\,:\,
\tau f'({-}\infty)<\xi-y(x,t,u)<\tau f'(\infty)\big\},
\end{equation}
where $y(x,t,u)$ is defined by
\begin{equation}\label{cony}
y(x,t,u):=x-tf'(u).
\end{equation}
Then, from $\eqref{con}$, ${\rm Con}(x,t,u)\subset \mathbb{R} \times \mathbb{R}^+$ is open
so that there exists $\delta>0$ such that $B_{\delta}(x,t)\subset {\rm Con}(x,t,u)$.
In particular,
$$
{\rm Con}(x,t,u)\equiv \mathbb{R} \times \mathbb{R}^+\qquad\, \mbox{ if $f'(\pm\infty)=\pm\infty$}.
$$

Since $f'(u)$ is strictly increasing, $\eqref{con}$--$\eqref{cony}$ imply that,
for any $(x',t')\in {\rm Con}(x,t,u)$,
there exists a unique $\hat{u}\in \mathbb{R}$ such that
\begin{equation}\label{cony1}
y(x,t,u)=x-tf'(u)=x'-t'f'(\hat{u})=y(x',t',\hat{u}),
\end{equation}
which, by $\eqref{con}$, implies
\begin{equation}\label{con1}
{\rm Con}(x,t,u)={\rm Con}(x',t',\hat{u}).
\end{equation}
See also {\rm Fig.} $\ref{figConxtu}$.

\vspace{1pt}
By $\eqref{cony1}$, it follows from $\eqref{c2.1}$ that
\begin{equation}\label{EE}
E(\hat{u};x',t')-E(u;x,t)=\int^{x'-t'f'(0)}_{x-tf'(0)}\varphi(\xi)\,{\rm d}\xi
-(t'-t)\int^{\hat{u}}_0sf''(s)\,{\rm d}s-t\int^{\hat{u}}_u s f''(s)\,{\rm d}s.
\end{equation}

\begin{figure}[H]
	\begin{center}
		{\includegraphics[width=0.6\columnwidth]{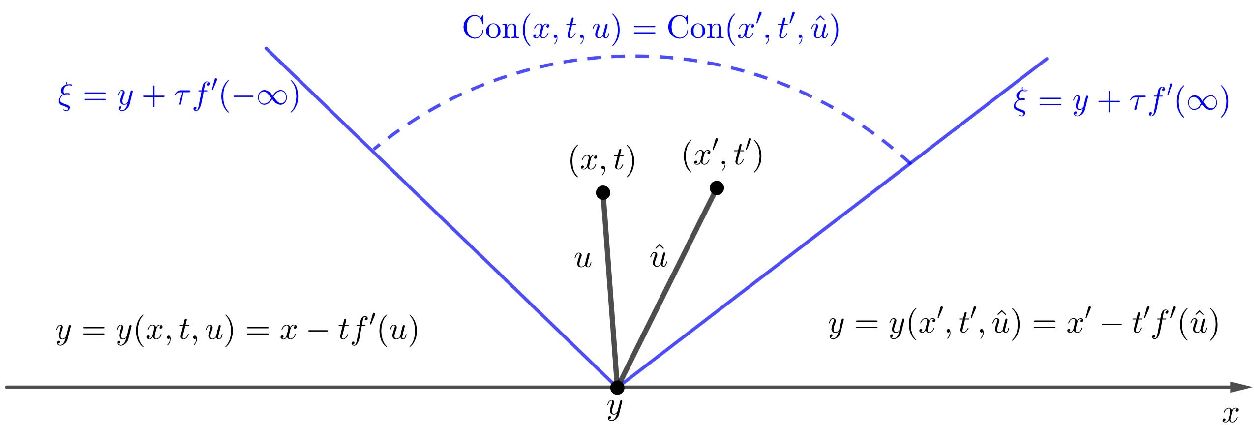}}
 \caption{${\rm Con}(x,t,u)={\rm Con}(x',t',\hat{u})$ as in $\eqref{con1}$.}\label{figConxtu}
	\end{center}
\end{figure}

\subsubsection{Proof of {\rm Theorem} $\ref{the:mt}$}
From $\eqref{c2.1}$--$\eqref{c2.3}$,
$\mathcal{U}(x,t) \subset [-\|\varphi\|_{L^\infty},\|\varphi\|_{L^\infty}]$
for any $(x,t)\in \mathbb{R} \times \mathbb{R}^+$.
Then it follows from $\eqref{c2.4}$ that,
for any $(x,t)\in \mathbb{R} \times \mathbb{R}^+$,
\begin{equation}\label{c2.7}
|u^\pm(x,t)|\leq \|\varphi\|_{L^\infty}.
\end{equation}

\smallskip
\noindent
The remaining proof is divided into five steps.

\smallskip
\noindent
{\bf 1.} For $\hat{E}(x,t)$ defined by $\eqref{c2.5}$, we have

\begin{Lem}\label{lem:c2.3} $\hat{E}(x,t)$ as in $\eqref{c2.5}$ is Lipschitz continuous
and satisfies the following properties{\rm :}
\begin{itemize}
\item[(i)] For any $t>0$ and $x_1,x_2\in \mathbb{R}$,
\begin{equation}\label{c2.13}
\int^{x_2}_{x_1}u^\pm(x,t)\,{\rm d}x=\hat{E}(x_1,t)-\hat{E}(x_2,t).
\end{equation}
\item[(ii)] For any $x\in \mathbb{R}$ and $t_2>t_1>0$,
\begin{equation}\label{c2.14}
\int^{t_2}_{t_1}f(u^\pm(x,t))\,{\rm d}t=\hat{E}(x,t_2)-\hat{E}(x,t_1).
\end{equation}
\end{itemize}
\end{Lem}

\begin{proof}[Proof of {\rm Lemma} $\ref{lem:c2.3}$.]
We divide the proof into three steps.

\smallskip
{\bf (a).} We first show that $\hat{E}(x,t)\in C(\mathbb{R} \times \mathbb{R}^+)$.
By the definition of $\hat{E}(x,t)$ in $\eqref{c2.5}$, it suffices to prove that $E(u^+(x,t);x,t)\in C(\mathbb{R} \times \mathbb{R}^+)$.

\smallskip
For any fixed $(x_0,t_0)\in \mathbb{R} \times\mathbb{R}^+$,
denote $u^+_0:=u^+(x_0,t_0)$.
It follows from $\eqref{c2.1}$ that
\begin{align*}
E(u^+_0;x,t)=-\int^{x-tf'(u^+_0)}_{x-tf'(0)}\varphi(\xi)\,{\rm d}\xi-t\int^{u^+_0}_0sf''(s)\,{\rm d}s,
\end{align*}
which means that $E(u^+_0;x,t)$ is continuous in $(x,t) \in \mathbb{R} \times \mathbb{R}^+$.
Thus, for any given $\varepsilon>0$, there exists $\delta_1>0$ such that,
for any $(x,t)\in B_{\delta_1}(x_0,t_0)$,
$$
E(u^+_0;x,t)>E(u^+_0;x_0,t_0)-\varepsilon,
$$
which, by $u^+(x,t)\in \mathcal{U}(x,t)$, implies that,
for any $(x,t)\in B_{\delta_1}(x_0,t_0)$,
\begin{align}\label{c2.9}
E(u^+(x,t);x,t)\geq E(u^+_0;x,t)> E(u^+_0;x_0,t_0)-\varepsilon.
\end{align}

On the other hand, since $f'(u)$ is strictly increasing,
from $\eqref{cony}$,
for any $(x,t)\in B_{\delta_1}(x_0,t_0)$,
\begin{align*}
\begin{cases}
\displaystyle\big|y(x,t,u^+(x,t))-y(x_0,t_0,u^+(x,t))\big|<N\delta_1,\\[2mm]
\displaystyle y(x_0,t_0,M)\leq y(x_0,t_0,u^+(x,t)) \leq y(x_0,t_0,-M),
\end{cases}
\end{align*}
where $N:=\sup_{|u|\leq M}\sqrt{1+f'(u)^2}$
for $M:=\|\varphi\|_{L^\infty}$.
Then $\delta_1$ can be further chosen sufficiently small,
if needed,
such that
$$
y(x_0,t_0,2M)\leq y(x,t,u^+(x,t)) \leq y(x_0,t_0,-2M) \qquad\,\,
{\rm for\ any\ } (x,t)\in B_{\delta_1}(x_0,t_0).
$$
This implies that,
for any $(x,t)\in B_{\delta_1}(x_0,t_0)$, there exists a unique $\hat{u}\in \mathbb{R}$ such that
\begin{equation}\label{c2.10}
x-tf'(u^+(x,t))=x_0-t_0f'(\hat{u}).
\end{equation}
Since $\varphi(x)$ and $u^+(x,t)$ are both bounded,
it follows from $\eqref{EE}$, $\eqref{c2.10}$, and $\eqref{aa5}$ that
\begin{align*}
&\ E(u^+(x,t);x,t)-E(\hat{u};x_0,t_0)\\
&=\int^{x-tf'(0)}_{x_0-t_0f'(0)}\varphi(\xi)\,{\rm d}\xi-(t-t_0)\int^{u^+(x,t)}_0sf''(s)\,{\rm d}s-t_0\int_{\hat{u}}^{u^+(x,t)}s f''(s)\,{\rm d}s\\[1mm]
&=\int^{x-tf'(0)}_{x_0-t_0f'(0)}\varphi(\xi)\,{\rm d}\xi-(t-t_0)\int^{u^+(x,t)}_0sf''(s)\,{\rm d}s
-t_0\,\rho(u^+(x,t),\hat{u})\big(f'(u^+(x,t))-f'(\hat{u})\big)
\\[1mm]
&\longrightarrow 0 \qquad\,\,\, {\rm as}\ (x,t)\rightarrow(x_0,t_0).
\end{align*}
This implies that, for sufficiently small $\delta \in (0,\delta_1)$,
if $(x,t)\in B_{\delta}(x_0,t_0)$, then
\begin{align}\label{c2.12}
E(u^+(x,t);x,t)<E(\hat{u};x_0,t_0)+\varepsilon\leq E(u^+_0;x_0,t_0)+\varepsilon.
\end{align}

Combining $\eqref{c2.9}$ with $\eqref{c2.12}$ yields that,
for any given $\varepsilon>0$, there exists $\delta>0$ such that
$$
|E(u^+(x,t);x,t)-E(u^+_0;x_0,t_0)|<\varepsilon \qquad{\rm for \ any}\ (x,t)\in B_{\delta}(x_0,t_0).
$$
By the arbitrariness of $(x_0,t_0)\in \mathbb{R}\times \mathbb{R}^+$,
we conclude that
$E(u^+(x,t);x,t)\in C(\mathbb{R} \times \mathbb{R}^+),$
which, by $\eqref{c2.5}$, implies that $\hat{E}(x,t)\in C(\mathbb{R} \times \mathbb{R}^+)$.

\smallskip
{\bf (b).} We now prove $\eqref{c2.13}$ for the case of $u^+(x,t)$ only, since
the case of $u^-(x,t)$ is similar.

\smallskip
For any $t>0$ and $x_1,x_2\in \mathbb{R}$, denote $u^+_i:=u^+(x_i,t)$ for $i=1,2$.
By the additivity of $\eqref{c2.13}$,
it suffices to prove $\eqref{c2.13}$ for small $|x_2-x_1|$.

\smallskip
Since $|u^+_i|\leq \|\varphi\|_{L^\infty}$,
it follows from $\eqref{con}$--$\eqref{cony}$ that,
for small $|x_2-x_1|$,
$$
(x_1,t)\in {\rm Con}(x_2,t,u^+_2), \qquad (x_2,t)\in {\rm Con}(x_1,t,u^+_1),
$$
so that, by $\eqref{cony1}$,
there exists a unique $(\hat{u}_1, \hat{u}_2)\in \mathbb{R}^2$ such that
\begin{equation}\label{c2.15}
x_1-tf'(\hat{u}_1)=x_2-tf'(u^+_2), \qquad x_1-tf'(u^+_1)=x_2-tf'(\hat{u}_2).
\end{equation}

Since $u^+_i\in \mathcal{U}(x_i,t)$,
by $\eqref{c2.5}$ and $\eqref{c2.15}$,
it follows from $\eqref{EE}$ and $\eqref{aa5}$ that
\begin{align}\label{c2.21}
\hat{E}(x_1,t)-\hat{E}(x_2,t)
&\geq E(\hat{u}_1;x_1,t)-E(u^+_2;x_2,t)-\int^{x_1-tf'(0)}_{x_2-tf'(0)}\varphi(\xi)\,{\rm d}\xi \nonumber \\
&=t \int^{u^+_2}_{\hat{u}_1} sf''(s)\,{\rm d}s
=(u^+_2+o(1))t\big(f'(u^+_2)-f'(\hat{u}_1)\big)
\nonumber \\[1mm]
&
=-(u^+_2+o(1))(x_1-x_2),
\end{align}
\begin{align}\label{c2.22}
\hat{E}(x_1,t)-\hat{E}(x_2,t)
&\leq E(u^+_1;x_1,t)-E(\hat{u}_2;x_2,t)-\int^{x_1-tf'(0)}_{x_2-tf'(0)}\varphi(\xi)\,{\rm d}\xi\nonumber \\
&=t \int_{u^+_1}^{\hat{u}_2} sf''(s)\,{\rm d}s
=(u^+_1+o(1))t\big(f'(\hat{u}_2)-f'(u^+_1)\big)\nonumber \\[1mm]
&=-(u^+_1+o(1))(x_1-x_2).
\end{align}

By the arbitrariness of $x_1$ and $x_2$,
it follows from $\eqref{c2.21}$--$\eqref{c2.22}$ that $\eqref{c2.13}$ holds.

\smallskip
{\bf (c).} We now prove $\eqref{c2.14}$ for $u^+(x,t)$ only,
since the case for $u^-(x,t)$ is similar.

\smallskip
For any $x\in\mathbb{R}$ and $t_2>t_1>0$, denote $u^+_i:=u^+(x,t_i)$ for $i=1,2$.
By the additivity of $\eqref{c2.14}$,
it suffices to prove $\eqref{c2.14}$ for small $t_2-t_1>0$.

\smallskip
Since $|u^+_i|\leq \|\varphi\|_{L^\infty}$,
it follows from $\eqref{con}$--$\eqref{cony}$ that,
for small $t_2-t_1>0$,
$$
(x,t_1)\in {\rm Con}(x,t_2,u^+_2), \qquad (x,t_2)\in {\rm Con}(x,t_1,u^+_1),
$$
so that, by $\eqref{cony1}$,
there exists a unique $(\check{u}_1, \check{u}_2)\in \mathbb{R}^2$ such that
\begin{equation}\label{c2.23}
x-t_1f'(\check{u}_1)=x-t_2f'(u^+_2),\qquad x-t_1f'(u^+_1)=x-t_2f'(\check{u}_2).
\end{equation}

Since $u^+_i\in \mathcal{U}(x_i,t)$,
from $\eqref{c2.5}$ and $\eqref{c2.23}$,
it follows from $\eqref{EE}$ and $\eqref{aa5}$ that
\begin{align}\label{c2.29}
\hat{E}(x,t_1)-\hat{E}(x,t_2)
&\geq E(\check{u}_1;x,t_1)-E(u^+_2;x,t_2)-\int^{x-t_1f'(0)}_{x-t_2f'(0)}\varphi(\xi)\,{\rm d}\xi \nonumber \\[1mm]
&=(t_2-t_1)\int^{u^+_2}_{0}sf''(s)\,{\rm d}s+t_1\int^{u^+_2}_{\check{u}_1} sf''(s)\,{\rm d}s\nonumber\\[1mm]
&=\big(u^+_2f'(u^+_2)-f(u^+_2)\big)(t_2-t_1)+(u^+_2+o(1))t_1\big(f'(u^+_2)-f'(\check{u}_1)\big)\nonumber \\[1mm]
&=-\big(u^+_2f'(u^+_2)-f(u^+_2)\big)(t_1-t_2)+(u^+_2+o(1))f'(u^+_2)(t_1-t_2)\nonumber \\[1mm]
&=\big(f(u^+_2)+o(1)\big)(t_1-t_2),
\end{align}
\begin{align}\label{c2.30}
\hat{E}(x,t_1)-\hat{E}(x,t_2)
&\leq E(u^+_1;x,t_1)-E(\check{u}_2;x,t_2)-\int^{x-t_1f'(0)}_{x-t_2f'(0)}\varphi(\xi)\,{\rm d}\xi \nonumber \\[1mm]
&=(t_2-t_1)\int^{u^+_1}_{0}sf''(s)\,{\rm d}s-t_2\int^{u^+_1}_{\check{u}_2} sf''(s)\,{\rm d}s\nonumber\\[1mm]
&=\big(u^+_1f'(u^+_1)-f(u^+_1)\big)(t_2-t_1)-(u^+_1+o(1))t_2\big(f'(u^+_1)-f'(\check{u}_2)\big)\nonumber \\[1mm]
&=-\big(u^+_1f'(u^+_1)-f(u^+_1)\big)(t_1-t_2)+(u^+_1+o(1))f'(u^+_1)(t_1-t_2)\nonumber \\[1mm]
&=\big(f(u^+_1)+o(1)\big)(t_1-t_2).
\end{align}
By the arbitrariness of $t_1$ and $t_2$,
it follows from $\eqref{c2.29}$--$\eqref{c2.30}$ that $\eqref{c2.14}$ holds.
\end{proof}

\smallskip
\noindent
{\bf 2.} We now prove the Oleinik-type one-sided inequality $\eqref{c1.4}$.

\begin{Lem}\label{lem:c2.4} For any $t>0$ and $x,x'\in \mathbb{R}$ with $x<x'$ $($see {\rm Fig.} $\ref{figOleinik})$,
\begin{equation}\label{c2.31}
x-tf'(u^+(x,t))\leq x'-tf'(u^-(x',t)).
\end{equation}
Furthermore, $y(x,t,u^\pm(x,t))$ defined by $\eqref{cony}$ are both nondecreasing in $x\in \mathbb{R}$.
\end{Lem}

\begin{figure}[H]
	\begin{center}
		{\includegraphics[width=0.7\columnwidth]{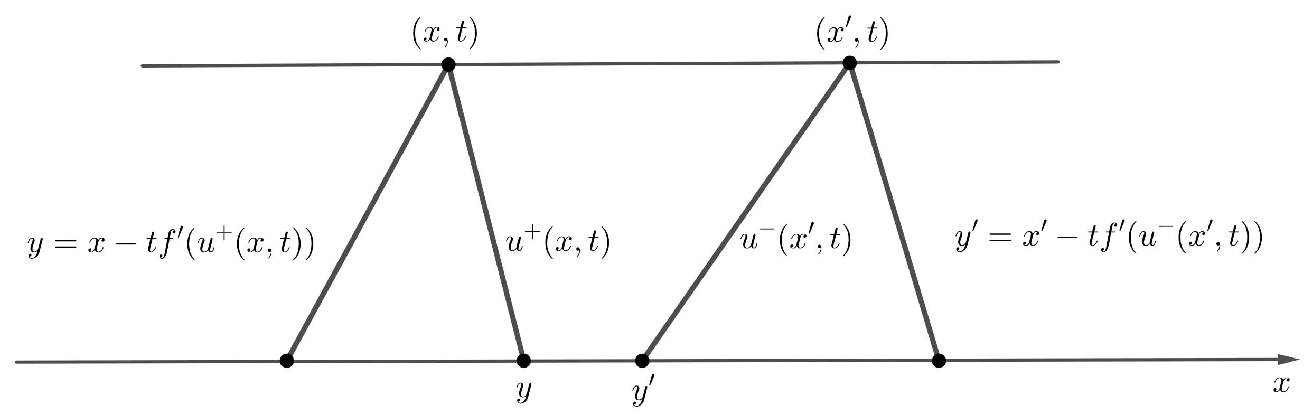}}\caption{The Oleinik-type one-sided inequality as in $\eqref{c2.31}$.
		}\label{figOleinik}
	\end{center}
\end{figure}

\begin{proof}[Proof of {\rm Lemma} $\ref{lem:c2.4}$.]
From $\eqref{c2.4}$, $u^+(x,t)\leq u^-(x,t)$, which implies
\begin{equation*}
x-tf'(u^-(x,t))\leq x-tf'(u^+(x,t)), \qquad x'-tf'(u^-(x',t))\leq x'-tf'(u^+(x',t)).
\end{equation*}
Then it suffices to prove $\eqref{c2.31}$ for small $x'-x>0$.

\vspace{2pt}
Fix $t>0$. For any $x,x'\in \mathbb{R}$ with $x<x'$ such that $(x',t)\in {\rm Con}(x,t, u^+(x,t))$,
denote $\delta:=x'-x>0$ and $u^+:=u^+(x,t)$.
Since $(x',t)\in {\rm Con}(x,t, u^+(x,t))$, from $\eqref{cony1}$,
there exists a unique $\hat{u}^+$ such that $y(x',t,\hat{u}^+)=y(x,t,u^+)$, $i.e.$,
\begin{equation}\label{xx+}
x'-tf'(\hat{u}^+)=x-tf'(u^+).
\end{equation}
Then, for any $u>\hat{u}^+$,
$y(x',t,u)<y(x',t,\hat{u}^+)=y(x,t,u^+).$
Since $x<x'$, then
$$
x-y(x,t,u^+)<x-y(x',t,u)<x'-y(x',t,u),
$$
which, by $\eqref{con}$, implies that $(x,t)\in {\rm Con}(x',t,u)$ for any $u>\hat{u}^+$.
This, by $\eqref{cony1}$, infers that there exists a unique $\hat{u}$ such that
\begin{equation}\label{xxu}
x'-tf'(u)=x-tf'(\hat{u}).
\end{equation}

Since $u^+\in \mathcal{U}(x,t)$,
then $E(\hat{u};x,t)-E(u^+;x,t)\leq 0.$
By $\eqref{EE}$ and $\eqref{xx+}$--$\eqref{xxu}$, for any $u>\hat{u}^+$,
\begin{align}\label{c2.33}
&\ E(u;x',t)-E(\hat{u}^+;x',t)\nonumber\\[1mm]
&=\big(E(u;x',t)-E(\hat{u};x,t)\big)+\big(E(\hat{u};x,t)-E(u^+;x,t)\big)
+\big(E(u^+;x,t)-E(\hat{u}^+;x',t)\big) \nonumber\\[1mm]
&\leq \big(E(u;x',t)-E(\hat{u};x,t)\big)+\big(E(u^+;x,t)-E(\hat{u}^+;x',t)\big) \nonumber\\
&=-t\int^u_{\hat{u}}sf''(s)\,{\rm d}s+t\int^{\hat{u}^+}_{u^+}sf''(s)\,{\rm d}s
\nonumber\\[1mm]
&
=:-t\big(L(u)-L(\hat{u}^+)\big).
\end{align}
It follows from $\eqref{xxu}$ that
$u>\hat{u}$ and $f''(u)\,{\rm d}u=f''(\hat{u})\,{\rm d}\hat{u}$ so that
\begin{align*}
{\rm d}L(u)=uf''(u)\,{\rm d}u-\hat{u}f''(\hat{u})\,{\rm d}\hat{u}=(u-\hat{u})f''(u)\,{\rm d}u,
\end{align*}
which implies that $L'(u)\geq 0$.

\smallskip
From $\eqref{c1.2}$, $f''(u)\geq 0$ with $\mathcal{L}\{u\,:\,f''(u)=0\}=0$.
Then $\mathcal{L}\{u\,:\,L'(u)=0\}=0$ so that
\begin{equation*}
L(u)-L(\hat{u}^+)=\int^u_{\hat{u}^+}L'(s)\,{\rm d}s>0 \qquad\,\, \mbox{for $u>\hat{u}^+$},
\end{equation*}
which, by $\eqref{c2.33}$, implies that, for any $u>\hat{u}^+$,
$$
E(u;x',t)-E(\hat{u}^+;x',t)\leq -t\big(L(u)-L(\hat{u}^+)\big)<0.
$$
This means that $E(\cdot\,; x',t)$ attains its maximum only if $u\leq \hat{u}^+$.
Then it follows from $\eqref{c2.3}$--$\eqref{c2.4}$ that $u^-(x',t)\leq \hat{u}^+$,
which, by $\eqref{xx+}$ with $x<x'$, yields $\eqref{c2.31}$.
\end{proof}

\smallskip
\noindent
{\bf 3.} We now prove the properties of $\mathcal{U}(x,t)$ in Lemma $\ref{lem:c2.5}$ below,
which include the traces of function $u(x,t)=u^+(x,t)$ and
the one-to-one correspondence between the elements of $\mathcal{U}(x,t)$
and the shock-free backward characteristics passing point $(x,t)$ in the upper half-space.

\begin{Lem}\label{lem:c2.5}
For any $(x,t)\in \mathbb{R}\times \mathbb{R}^+$, the following statements hold{\rm :}
\begin{itemize}
\item [(i)] For any sequence
$(x_n,t_n)\to (x,t)$ as $n\to \infty$, if
$\lim_{n\rightarrow \infty}u^\pm(x_n,t_n)$ exists, then
\begin{equation}\label{c2.39}
\lim\limits_{n\rightarrow \infty}u^\pm(x_n,t_n)\in \mathcal{U}(x,t).
\end{equation}
\item [(ii)] For $u^{\pm}(x{-},t):=\lim_{x' \rightarrow x{-}}u^{\pm}(x',t)$
and $u^{\pm}(x{+},t):=\lim_{x' \rightarrow x{+}}u^{\pm}(x',t)$,
\begin{equation}\label{c2.40}
u^{\pm}(x{-},t)=u^-(x,t),\qquad u^{\pm}(x{+},t)=u^+(x,t).
\end{equation}

\item [(iii)] Fix $(x_0,t_0)\in \mathbb{R}\times \mathbb{R}^+$.
For any $c\in \mathbb{R}$, define
\begin{equation}\label{c2.43}
 l_{c}(x_0,t_0):=\big\{(x,t)
\,:\, x=x_0+(t-t_0)f'(c)\quad {\rm for}\ t\in (0,t_0)\big\}.
\end{equation}
Then
\begin{equation}\label{uiff}
c\in \mathcal{U}(x_0,t_0)\qquad {\rm if\ and\ only\ if} \qquad
l_{c}(x_0,t_0)\,\, \mbox{is\ a\ shock-free\ characteristic}.
\end{equation}
\end{itemize}
\end{Lem}

\begin{proof}[Proof of {\rm Lemma} $\ref{lem:c2.5}$.]
We divide the proof into three steps.

\smallskip
{\bf (a)}.
Denote $u^*:=\lim_{n\rightarrow \infty}u^+(x_n,t_n)$.
Then it suffices to prove that $u^*\in \mathcal{U}(x,t)$.
By Lemma $\ref{lem:c2.3}$, $\hat{E}(x,t)\in C(\mathbb{R}\times \mathbb{R}^+)$ so that
\begin{align*}
\hat{E}(x,t)=\lim\limits_{n\rightarrow \infty}\hat{E}(x_n,t_n)
&=\lim\limits_{n\rightarrow \infty} \Big(E(u^+(x_n,t_n);x_n,t_n)-\int^{x_n-t_nf'(0)}_0\varphi (\xi)\,{\rm d}\xi\Big)\\[1mm]
& =\lim\limits_{n\rightarrow \infty}\Big({-}\int^{x_n-t_nf'(u^+(x_n,t_n))}_0\varphi (\xi)\,{\rm d}\xi
 -t_n\int^{u^+(x_n,t_n)}_0sf''(s)\,{\rm d}s\Big)\\[1mm]
& =-\int^{x-tf'(u^*)}_0\varphi (\xi)\,{\rm d}\xi-t\int^{u^*}_0sf''(s)\,{\rm d}s\\[1mm]
&=E(u^*;x,t)-\int^{x-tf'(0)}_0\varphi (\xi)\,{\rm d}\xi.
\end{align*}
This, by $\eqref{c2.5}$, implies that $E(u^*;x,t)=\max_{v\in \mathbb{R}}E(v;x,t)$
so that $u^*\in \mathcal{U}(x,t)$.
Similarly, it is direct to check that $\eqref{c2.39}$ holds for $u^-(x,t)$.

\smallskip
 {\bf (b).} From $\eqref{c2.4}$ and $\eqref{c2.39}$,
\begin{equation}\label{c2.42}
u^+(x,t)\leq \mathop{\underline{\lim}}\limits_{(x',t')\rightarrow (x,t)}u^\pm(x',t')
\leq \mathop{\overline{\lim}}\limits_{(x',t')\rightarrow (x,t)}u^\pm(x',t')\leq u^-(x,t).
\end{equation}

When $x'<x$, from Lemma $\ref{lem:c2.4}$,
$$
f'(u^-(x,t))+\frac{x'-x}{t}\leq f'(u^+(x',t))\leq f'(u^-(x',t)).
$$
As $x'\rightarrow x{-}$, it follows from $\eqref{c2.42}$ that
$$
f'(u^-(x,t))\leq \mathop{\underline{\lim}}\limits_{x'\rightarrow x{-}}f'(u^\pm(x',t))
\leq \mathop{\overline{\lim}}\limits_{x'\rightarrow x{-}}f'(u^\pm(x',t))\leq f'(u^-(x,t)),
$$
which, by the strictly increasing property of $f'(u)$, implies
$$
u^-(x,t)\leq\mathop{\underline{\lim}}\limits_{x'\rightarrow x{-}}u^\pm(x',t)
\leq\mathop{\overline{\lim}}\limits_{x'\rightarrow x{-}}u^\pm(x',t)\leq u^-(x,t).
$$
This means that $u^\pm(x{-},t)=u^-(x,t)$.

\smallskip
When $x'>x$, from Lemma $\ref{lem:c2.4}$,
$$
f'(u^+(x,t))+\frac{x'-x}{t}\geq f'(u^-(x',t))\geq f'(u^+(x',t)).
$$
As $x'\rightarrow x{+}$, it follows from $\eqref{c2.42}$ that
$$
f'(u^+(x,t))\geq\mathop{\overline{\lim}}\limits_{x'\rightarrow x{+}}f'(u^\pm(x',t))
\geq\mathop{\underline{\lim}}\limits_{x'\rightarrow x{+}}f'(u^\pm(x',t))\geq f'(u^+(x,t)),
$$
which, by the strictly increasing property of $f'(u)$, implies
$$
u^+(x,t)\geq\mathop{\overline{\lim}}\limits_{x'\rightarrow x{+}}u^\pm(x',t)
\geq\mathop{\underline{\lim}}\limits_{x'\rightarrow x{+}}u^\pm(x',t)\geq u^+(x,t).
$$
This means that $u^\pm(x{+},t)=u^+(x,t).$

\smallskip
{\bf (c).} If $c\in \mathcal{U}(x_0,t_0)$, then,
for any $(x,t)\in l_{c} (x_0,t_0)$,
\begin{equation}\label{c2.44}
x-tf'(c)=x_0-t_0f'(c),
\end{equation}
so that, for any $u\in \mathbb{R}$,
$$\frac{x_0-x}{t_0}+\frac{t}{t_0}f'(u)=\frac{t_0-t}{t_0}f'(c)+\frac{t}{t_0}f'(u)$$
lies between $f'(c)$ and $f'(u)$.
This implies that, for any $u\in \mathbb{R}$,
there exists a unique $\hat{u}$ between $c$ and $u$ such that
\begin{equation}\label{c2.45}
x-tf'(u)=x_0-t_0 f'(\hat{u}).
\end{equation}
Then it follows from $\eqref{c2.1}$ and $c\in \mathcal{U}(x_0,t_0)$ that,
if $u \neq c$,
\begin{align}\label{c2.47}
E(u;x,t)-E(c;x,t)
&=-\int^{x-tf'(u)}_{x-tf'(c)}\varphi(\xi)\,{\rm d}\xi-t\int^u_{c}sf''(s)\,{\rm d}s\nonumber\\[1mm]
&=-\int^{x_0-t_0f'(\hat{u})}_{x_0-t_0f'(c)}\varphi(\xi)\,{\rm d}\xi-t\int^u_{c}sf''(s)\,{\rm d}s\nonumber\\[1mm]
&=E(\hat{u};x_0,t_0)-E(c;x_0,t_0)
+t_0\int^{\hat{u}}_{c}sf''(s)\,{\rm d}s-t\int^{u}_{c}sf''(s)\,{\rm d}s\nonumber\\[1mm]
& \leq t_0\int^{\hat{u}}_{c}sf''(s)\,{\rm d}s-t\int^{u}_{c}sf''(s)\,{\rm d}s
=-l\big(\rho(\hat{u},c)-\rho(u,c)\big),
\end{align}
where $\rho(u,v)$ is given by $\eqref{aa5a}$, and
$l$ is determined by $\eqref{c2.44}$--$\eqref{c2.45}$ as
\begin{equation*}
l:=t\big(f'(c)-f'(u)\big)=t_0\big(f'(c)-f'(\hat{u})\big).
\end{equation*}

From $\eqref{aa6}$, $\rho(\cdot,c)$ is strictly increasing.
Since $\hat{u}$ lies between $c$ and $u$, it follows from $\eqref{c2.47}$ that
$E(u;x,t)-E(c;x,t)<0$ for any $u\neq c$.
Thus, $c$ is the unique maximum point of $E(\cdot\,; x,t)$. Therefore,
$u^\pm(x,t)=c$ for any $(x,t)\in l_{c}(x_0,t_0)$,
which implies that $l_{c}(x_0,t_0)$ is shock-free.

\smallskip
On the other hand, if $u^\pm(x,t)=c$ for any $(x,t)\in l_{c}(x_0,t_0)$,
it follows from $\eqref{c2.39}$ that
$c$,
as the limit of $u^+(x,t)$ as $(x,t)\rightarrow (x_0,t_0)$ along $l_{c}(x_0,t_0)$,
belongs to $\mathcal{U}(x_0,t_0)$.
\end{proof}

\smallskip
\noindent
{\bf 4.} We now show that $u(x,t)\in L^\infty(\mathbb{R} \times \mathbb{R}^+)$ defined by $\eqref{c2.50}$ solves
the Cauchy problem $\eqref{c1.1}$--$\eqref{ID}$ in the distributional sense.

\smallskip
{\bf (a).} We first prove $\eqref{c2.50a}$--$\eqref{c2.38}$.
From $\eqref{c2.50}$ and $\eqref{c2.40}$, $u(x\pm,t)=u^\pm(x,t)$ pointwise.
From $\eqref{c2.42}$, $u(x-,t)$ and $u(x+,t)$ are upper- and low-semicontinuous, respectively,
so that $u(x-,t)$ and $u(x+,t)$ are both continuous at point $(x,t)$ if $u^-(x,t)=u^+(x,t)$.
By Lemma $\ref{lem:c2.4}$, $y(\cdot,t,u^\pm(\cdot,t))$ is nondecreasing
so that it possesses at most countably discontinuous points for each $t>0$.
Since $f'(u)$ is strictly increasing, for each $t>0$, $u^-(x,t)=u^+(x,t)$ also holds
except at most countable points $(x,t)$ at which $u^-(x,t)>u^+(x,t)$.
This yields $\eqref{c2.50a}$.

Furthermore, from $\eqref{cony}$,
$$f'(u^{\pm}(x,t))=\frac{x}{t}-\frac{y(x,t,u^{\pm}(x,t))}{t},$$
so that, by Lemma $\ref{lem:c2.4}$,
$f'(u^{\pm}(\cdot,t))$ can be expressed by the difference of two nondecreasing functions.
This yields $\eqref{c2.38}$.

\smallskip
{\bf (b).} We now prove that $u(x,t)\in L^\infty(\mathbb{R} \times \mathbb{R}^+)$
defined by $\eqref{c2.50}$ is a weak solution of the Cauchy problem $\eqref{c1.1}$--$\eqref{ID}$
in the sense of Definition $\ref{def:ws}$.
From $\eqref{c2.50a}$ and Lemma $\ref{lem:c2.3}$, $\hat{E}(x,t)$ is Lipschitz continuous and satisfies
$$
\hat{E}_x(x,t)\mathop{=}\limits^{a.e.}-u(x,t),\quad
\hat{E}_t(x,t)\mathop{=}\limits^{a.e.}f(u(x,t))
\qquad {\rm on} \ \mathbb{R}\times \mathbb{R}^+.
$$
Then $\eqref{c1.3}$ follows since, for any test function $\phi\in C^\infty_{\rm c}(\mathbb{R}\times \mathbb{R}^+)$,
\begin{align*}
\iint_{\mathbb{R}\times \mathbb{R}^+}\big(u\phi_t+f(u)\phi_x\big)\,{\rm d}x{\rm d}t
=\iint_{\mathbb{R}\times \mathbb{R}^+} \big({-}\hat{E}_x\phi_t+\hat{E}_t\phi_x\big)\,{\rm d}x{\rm d}t
=0.
\end{align*}
Furthermore, since $u(x,t)\in L^\infty(\mathbb{R} \times \mathbb{R}^+)$,
it follows from $\eqref{c2.13}$ that, for any $x_1,x_2\in \mathbb{R}$,
\begin{align*}
\lim_{t\rightarrow 0{+}}\int^{x_2}_{x_1}u(x,t)\,{\rm d}x
&=\lim_{t\rightarrow 0{+}}\big(\hat{E}(x_1,t)-\hat{E}(x_2,t)\big)\nonumber\\
&=\lim_{t\rightarrow 0{+}}\bigg(t \int^{u(x_1,t)}_{u(x_2,t)}f''(s)\big(\varphi(x-tf'(s))-s\big)\,{\rm d}s
+\int^{x_2-tf'(0)}_{x_1-tf'(0)}\varphi(\xi)\,{\rm d}\xi\bigg)\nonumber\\[1mm]
&=\int^{x_2}_{x_1} \varphi(\xi)\,{\rm d}\xi.
\end{align*}

\smallskip
\noindent
{\bf 5.} We now prove the uniqueness of the entropy solutions given by the solution formula $\eqref{c2.50}$
in the sense of Definition $\ref{def:es}$
({\it cf.}\, Hoff \cite{HD}).

\smallskip
{\bf (a)}. Since the flux function $f(u)$ satisfies $\eqref{c1.2}$,
consider the first-order and second-order differences of $f(u)$:
For any three different $u,v$, and $w$,
\begin{align*}
&f[u,v]:=\frac{f(u)-f(v)}{u-v}=\int^1_0f'(\theta u+(1-\theta)v)\,{\rm d}\theta, \\[2mm]
&f[u,v,w]: =\dfrac{f[u,v]-f[v,w]}{u-w}
=\int^1_0\int^1_0 f''(\xi \theta u+(1-\theta)v+(1-\xi)\theta w)\theta\,{\rm d}\xi{\rm d}\theta\geq 0.
\end{align*}
Then $f[u,v]$ can be extended to a continuously differentiable
function in $(u,v)\in \mathbb{R}^2$,
and $f[u,v,w]$ can be extended to a continuous function in $(u,v,w)\in \mathbb{R}^3$.

\smallskip
{\bf (b)}. Let $u_1(x,t), u_2(x,t)\in L^\infty(\mathbb{R} \times \mathbb{R}^+)$
be two different weak solutions of the Cauchy problem $\eqref{c1.1}$--$\eqref{ID}$
with the same initial data function $\varphi(x)\in L^\infty(\mathbb{R})$
and satisfy the Oleinik-type one-sided inequality $\eqref{c1.4}$.
Then, for any fixed $t_1,t_2$ with $t_2>t_1>0$,
if $\phi\in C^\infty_{\rm c}(\mathbb{R}\times [t_1,t_2])$,
\begin{equation}\label{ac3}
\int_\mathbb{R} u_i(x,\cdot)\phi(x,\cdot)\big|^{t_2}_{t_1}\,{\rm d}x
=\int^{t_2}_{t_1}\int_\mathbb{R} \big(u_i\phi _t+f(u_i)\phi _x\big)\,{\rm d}x{\rm d}t
\qquad \mbox{for $i=1,2$}.
\end{equation}
Meanwhile,
it follows from $\eqref{c1.4}$ that, for $i=1,2$,
\begin{equation}\label{ac4}
  \partial_x f'(u_i(x,t))\leq \frac{1}{t} \qquad \mbox{in the distributional sense}.
\end{equation}
Denote $w=w(x,t):=u_1(x,t)-u_2(x,t)$. From $\eqref{ac3}$,
\begin{equation}\label{ac5}
\int_\mathbb{R} w(x,\cdot)\phi(x,\cdot)\big|^{t_2}_{t_1}\,{\rm d}x
=\int^{t_2}_{t_1}\int_\mathbb{R} w\big(\phi _t+f[u_1,u_2]\phi _x\big)\,{\rm d}x{\rm d}t.
\end{equation}

\smallskip{\bf (c)}. Now fix $\psi(x)\in C^\infty_{\rm c}(\mathbb{R})$, and set
\begin{equation*}
\psi_1(x):=\frac{\psi(x)}{2}+\frac{1}{2}\int^x_{-\infty}|\psi'(\xi)|\,{\rm d}\xi,\qquad \psi_2(x):=\frac{\psi(x)}{2}-\frac{1}{2}\int^x_{-\infty}|\psi'(\xi)|\,{\rm d}\xi.
\end{equation*}
Then we have
\begin{equation*}
\psi(x)=\psi_1(x)+\psi_2(x),\qquad \psi_1'(x)\geq 0\geq\psi_2'(x).
\end{equation*}
For any given $\varepsilon, \delta>0$,
let $\phi_i^{\varepsilon,\delta}, i=1,2,$ solve the following linear transport equation respectively:
\begin{equation}\label{ac8}
\begin{cases}
\displaystyle {\partial_t}v+\big(j_{\varepsilon}\ast f'(u_i)\big)\,\partial_x v=0,\\[2mm]
\displaystyle v(x,t_2)=(j_{\delta}\ast \psi_i)(x),
\end{cases}
\end{equation}
where $j_{\varepsilon},j_{\delta}>0$ are the modifiers satisfying that
$\int_\mathbb{R}j_{\varepsilon}(x)\,{\rm d}x=\int_\mathbb{R}j_{\delta}(x)\,{\rm d}x=1$
for $\varepsilon,\delta>0$.

\smallskip
Since $\psi$ is bounded and $ {\rm spt}(\psi)$ is a bounded set,
then $\phi_i^{\varepsilon,\delta}\in C^{\infty}_{\rm c}(\mathbb{R}\times[t_1,t_2])$.
Taking the derivative $\partial_x$ on $\eqref{ac8}$,
we obtain that $\partial_x\phi_i^{\varepsilon,\delta}$ solves the following linear equation of $v_x$:
\begin{equation*}
\begin{cases}
\displaystyle
\partial_t(v_x)+\big(j_\varepsilon *f'(u_i)\big)\,\partial_x(v_x)
 =-\big(j_\varepsilon * \partial_x(f'(u_i))\big)\,v_x,\\[2mm]
\displaystyle v_x(x,t_2)=(j_\delta *\psi'_i)(x).
\end{cases}
\end{equation*}
Then, along any characteristic $x(t):\, \frac{{\rm d}x(t)}{{\rm d}t}=(j_\varepsilon *f'(u_i))(x(t),t)$
with $x(t_2)=y$,
$\,\partial_x\phi_i^{\varepsilon,\delta}$ satisfies
\begin{equation*}
\partial_x\phi^{\varepsilon,\delta}_i(x(t),t)
=(j_\delta * \psi'_i)(y)
\exp\Big\{ \int^{t_2}_t \big(j_\varepsilon * \partial_x (f'(u_i))\big)(x(\tau),\tau)\,{\rm d}\tau\Big\},
\end{equation*}
and hence $\partial_x\phi^{\varepsilon,\delta}_1\geq 0\geq \partial_x\phi^{\varepsilon,\delta}_2$.
Thus, it follows from $\eqref{ac4}$ that
\begin{equation}\label{ac12}
\big|\partial_x\phi^{\varepsilon,\delta}_i(x,t)\big|
\leq\|\psi'_i\|_{L^\infty}\exp\Big\{\int^{t_2}_t\frac{1}{\tau}\,{\rm d}\tau\Big\}
= \frac{t_2}{t}\|\psi'_i\| _{L^\infty}.
\end{equation}

\smallskip
{\bf (d)}. Let $\phi^{\varepsilon,\delta}:=\phi^{\varepsilon,\delta}_1+\phi^{\varepsilon,\delta}_2$.
Then $\phi^{\varepsilon,\delta}\in C^\infty_{\rm c}(\mathbb{R}\times [t_1,t_2])$.
Taking $\phi^{\varepsilon,\delta}$ into $\eqref{ac5}$--$\eqref{ac8}$,
\begin{align}\label{ac13}
&\ \int_\mathbb{R} w(x,t_2)(j_\delta *\psi)(x)\,{\rm d}x
-\int_\mathbb{R}w(x,t_1) \,\phi^{\varepsilon,\delta}(x,t_1)\,{\rm d}x \nonumber\\[1mm]
& =\int^{t_2}_{t_1}\int_\mathbb{R} w\big((\phi^{\varepsilon,\delta}_1+\phi^{\varepsilon,\delta}_2)_t
+f[u_1,u_2](\phi^{\varepsilon,\delta}_1+\phi^{\varepsilon,\delta}_2)_x\big)\,{\rm d}x{\rm d}t\nonumber\\[1mm]
& =\int^{t_2}_{t_1}\int_\mathbb{R} w\,\partial_x\phi^{\varepsilon,\delta}_1\,
\big(f[u_1,u_2]-j_\varepsilon *f'(u_1)\big)\,{\rm d}x{\rm d}t
\nonumber\\[1mm]
& \quad\,
+\int^{t_2}_{t_1}\int_\mathbb{R} w\,\partial_x\phi^{\varepsilon,\delta}_2\,
\big(f[u_1,u_2]-j_\varepsilon *f'(u_2)\big)\,{\rm d}x{\rm d}t\nonumber\\[1mm]
& =: A^{\varepsilon,\delta}_1+A^{\varepsilon,\delta}_2.
\end{align}
For $A^{\varepsilon,\delta}_1$,
since $f[u_1,u_1]=f'(u_1)$ and $f[u_1,u_1,u_2]\geq 0$, then
\begin{align*}
&\ w\,\partial_x\phi^{\varepsilon,\delta}_1\,\big(f[u_1,u_2]-j_\varepsilon *f'(u_1)\big) \nonumber\\[1mm]
& =w\,\partial_x\phi^{\varepsilon,\delta}_1\,\big(f[u_1,u_2]-f[u_1,u_1]\big)
+w\,\partial_x\phi^{\varepsilon,\delta}_1\,\big(f'(u_1)- j_\varepsilon *f'(u_1)\big)\nonumber\\[1mm]
& =-w^2\,\partial_x\phi^{\varepsilon,\delta}_1\,f[u_1,u_1,u_2]
+w\,\partial_x\phi^{\varepsilon,\delta}_1\,\big(f'(u_1)- j_\varepsilon *f'(u_1)\big)\nonumber\\[1mm]
& \leq w\,\partial_x\phi^{\varepsilon,\delta}_1\,\big(f'(u_1)- j_\varepsilon *f'(u_1)\big).
\end{align*}
Since both functions $\psi_i(x), i=1,2$, are compact supported,
then there exists sufficiently large $R>0$ such that
${\rm spt}(\phi_i^{\varepsilon,\delta})\subset [-R,R]\times [0,t_2]=: \mathbb{D}$
for any $\varepsilon,\delta>0$ and $i=1,2$.
Thus, from $\eqref{ac12}$,
\begin{align}\label{ac15}
A^{\varepsilon,\delta}_1
&\leq\int^{t_2}_{t_1}\int^R_{-R} w\,\partial_x\phi^{\varepsilon,\delta}_1\,
\big(f'(u_1)-j_\varepsilon *f'(u_1)\big)\,{\rm d}x{\rm d}t\nonumber\\[1mm]
& \leq \int^{t_2}_{t_1}\,\int^R_{-R} \|w\|_{L^\infty}\, \frac{t_2}{t}\|\psi'_1\| _{L^\infty}\,
\big|f'(u_1)-j_\varepsilon *f'(u_1)\big| \,{\rm d}x{\rm d}t\nonumber\\[1mm]
& \leq \frac{t_2}{t_1}\,\|w\|_{L^\infty}\,\|\psi'_1\| _{L^\infty}\,
\|f'(u_1)-j_\varepsilon *f'(u_1)\|_{L^1(\mathbb{D})}.
\end{align}
Similarly, we have
\begin{align}\label{ac16}
A^{\varepsilon,\delta}_2
\leq \frac{t_2}{t_1}\,\|w\|_{L^\infty}\,\|\psi'_2\| _{L^\infty}\,
\|f'(u_2)-j_\varepsilon *f'(u_2)\|_{L^1(\mathbb{D})}.
\end{align}

From $\eqref{ac8}$, $\|\phi_i^{\varepsilon,\delta}\|_{L^{\infty}}\leq \|\psi_i\|_{L^{\infty}}$.
Combining $\eqref{ac13}$ with $\eqref{ac15}$--$\eqref{ac16}$ leads to
\begin{align*}
 \int_\mathbb{R} w(x,t_2)\,(j_\delta *\psi)(x)\,{\rm d}x
 &=A^{\varepsilon,\delta}_1+ A^{\varepsilon,\delta}_2
 +\int_\mathbb{R} w(x,t_1)\,\phi^{\varepsilon,\delta}(x,t_1)\,{\rm d}x\\[1mm]
& \leq A^{\varepsilon,\delta}_1+ A^{\varepsilon,\delta}_2
+\|w(\cdot,t_1)\|_{L^1(\mathbb{D})}\,\big(\|\psi_1\|_{L^{\infty}}+\|\psi_2\|_{L^{\infty}}\big)\\[1mm]
&\leq \frac{t_2}{t_1}\,\|w\|_{L^\infty}\,\|\psi'_1\| _{L^\infty}\,
\|f'(u_1)-j_\varepsilon *f'(u_1)\|_{L^1(\mathbb{D})}\\[1mm]
& \quad \, +\frac{t_2}{t_1}\,\|w\|_{L^\infty}\,\|\psi'_2\| _{L^\infty}\,
\|f'(u_2)-j_\varepsilon *f'(u_2)\|_{L^1(\mathbb{D})}\\[1mm]
&\quad\, +\|w(\cdot,t_1)\|_{L^1(\mathbb{D})}\,\big(\|\psi_1\|_{L^{\infty}}+\|\psi_2\|_{L^{\infty}}\big).
\end{align*}
This, by letting $\varepsilon\rightarrow 0{+}$, implies
\begin{equation*}
\int_\mathbb{R} w(x,t_2)\,(j_{\delta}\ast \psi)(x)\,{\rm d}x
\leq \|w(\cdot,t_1)\|_{L^1(\mathbb{D})}\,\big(\|\psi_1\|_{L^{\infty}}+\|\psi_2\|_{L^{\infty}}\big).
\end{equation*}
Letting $t_1\rightarrow 0{+}$ first and then
$\delta \rightarrow 0{+}$, we obtain
$$
\int_\mathbb{R}w(x,t_2)\,\psi(x)\,{\rm d}x\leq 0
\qquad\,\, {\rm for\ any }\ \psi(x)\in C_{\rm c}^{\infty}(\mathbb{R}).
$$
This means that, for any $t_2>0$, $w(x,t_2)=0$ holds almost everywhere for $x\in \mathbb{R}$.
By the arbitrariness of $t_2>0$ and the Fubini theorem,
$w(x,t)=0$ holds almost everywhere for $(x,t)\in \mathbb{R}\times \mathbb{R}^+$, $i.e.$,
$u_1(x,t)=u_2(x,t)$ holds almost everywhere on $ \mathbb{R}\times \mathbb{R}^+.$

\smallskip
After all the five steps above, we complete the proof of Theorem $\ref{the:mt}$.

\begin{Rem}
For any $x\neq x'$ and $t'>t$,
\begin{equation}\label{c7.24}
\frac{U(x,t',x')-U(x,t,x')}{t'-t}+\frac{F(x',t,t')-F(x,t,t')}{x'-x}=0,
\end{equation}
where $U(x,t,x'):=\frac{1}{x'-x}\int^{x'}_xu(\xi,t)\,{\rm d}\xi$ and
$F(x,t,t'):=\frac{1}{t'-t}\int^{t'}_t f(u(x,\tau))\,{\rm d}\tau.$

\smallskip
In fact, this can be seen as follows{\rm :}
By {\rm Lemma} $\ref{lem:c2.3}$, for any $x\neq x'$ and $t'> t$,
\begin{align*}
&\int^{x'}_x u(\xi,t) \,{\rm d}\xi=\hat{E}(x,t)-\hat{E}(x',t),
&&\int^{x'}_x u(\xi,t') \,{\rm d}\xi=\hat{E}(x,t')-\hat{E}(x',t'),\\[2mm]
&\int^{t'}_t f(u(x,\tau)) \,{\rm d}\tau=\hat{E}(x,t')-\hat{E}(x,t),
&&\int^{t'}_t f(u(x',\tau)) \,{\rm d}\tau=\hat{E}(x',t')-\hat{E}(x',t),
\end{align*}
so that
\begin{align*}
\int^{x'}_x\big(u(\xi,t')-u(\xi,t)\big)\,{\rm d}\xi
+\int^{t'}_t\big(f(u(x',\tau))-f(u(x,\tau))\big)\,{\rm d}\tau=0,
\end{align*}
which, by multiplying $\frac{1}{(x'-x)(t'-t)}$, implies $\eqref{c7.24}$.
\end{Rem}

\subsubsection{Proof of {\rm Corollary} $\ref{cor:c2.1}$}
The proof is divided into three steps.

\smallskip
\noindent
{\bf 1.} We first prove the $L^1-$contraction inequality $\eqref{c2.51}$.
Let the entropy solutions $u_i(x,t)$ be determined by $E_i(u;x,t)$ as in $\eqref{c2.1}$
with initial data functions $\varphi_i(x)\in L^\infty(\mathbb{R})$ for $i=1,2$.

\smallskip
{\bf (a).} We claim: For any $x_1,x_2\in \mathbb{R}$ and $t>0$,
\begin{equation}\label{c2.52}
\int_{x_1-tf'(u_{1,1})}^{x_2-tf'(u_{2,2})}(\varphi_2-\varphi_1)\,{\rm d}x
\leq\int_{x_1}^{x_2}(u_2-u_1)\,{\rm d}x
\leq\int_{x_1-tf'(u_{2,1})}^{x_2-tf'(u_{1,2})}(\varphi_2-\varphi_1)\,{\rm d}x,
\end{equation}
where $u_{i,j}\in \mathcal{U}_i(x_j,t)$ for $i,j=1,2$.

\smallskip
In fact, for any fixed $x_1,x_2\in \mathbb{R}$ and $t>0$,
$u_{2,1}\in \mathcal{U}_2(x_1,t)$ and $u_{1,2}\in \mathcal{U}_1(x_2,t)$.
From $\eqref{c2.5}$,
taking $u_i(x,t)$ into $\eqref{c2.13}$ respectively and subtracting them yield
\begin{align}\label{c2.53}
\int_{x_1}^{x_2}(u_2-u_1)\,{\rm d}x
&=\big(\hat{E}_2(x_1,t)-\hat{E}_1(x_1,t)\big)-\big(\hat{E}_2(x_2,t)-\hat{E}_1(x_2,t)\big)\nonumber\\
&\leq\big(E_2(u_{2,1};x_1,t)-E_1(u_{2,1};x_1,t)\big)-\big(E_2(u_{1,2};x_2,t)-E_1(u_{1,2};x_2,t)\big)\nonumber\\[1mm]
& \quad\,\,+\int^{x_2-tf'(0)}_{x_1-tf'(0)}\big(\varphi_2(\xi)-\varphi_1(\xi)\big)\,{\rm d}\xi\nonumber\\[1mm]
&= t\int_0^{u_{2,1}} f''(s)\big(\varphi_2 (x_1-tf'(s))-\varphi_1 (x_1-tf'(s))\big)\,{\rm d}s
\nonumber\\[1mm]
& \quad\,-t\int_0^{u_{1,2}} f''(s)\big(\varphi_2 (x_2-tf'(s))-\varphi_1 (x_2-tf'(s))\big)\,{\rm d}s
 \nonumber\\[1mm]
&\quad\,+\int^{x_2-tf'(0)}_{x_1-tf'(0)}\big(\varphi_2(\xi)-\varphi_1(\xi)\big)\,{\rm d}\xi\nonumber\\[1mm]
&=\bigg({-}\int^{x_1-tf'(u_{2,1})}_{x_1-tf'(0)}{+}\int_{x_2-tf'(0)}^{x_2-tf'(u_{1,2})}{+}
\int^{x_2-tf'(0)}_{x_1-tf'(0)}\bigg)\big(\varphi_2(\xi)-\varphi_1(\xi)\big)\,{\rm d}\xi\nonumber\\[1mm]
&=\int_{x_1-tf'(u_{2,1})}^{x_2-tf'(u_{1,2})}(\varphi_2-\varphi_1)\,{\rm d}x.
\end{align}
Similarly, from $u_{1,1}\in \mathcal{U}_1(x_1,t)$ and $u_{2,2}\in \mathcal{U}_2(x_2,t)$,
it is direct to check that
\begin{align}\label{c2.55}
\int_{x_1}^{x_2}(u_2-u_1)\,{\rm d}x
&=\big(\hat{E}_2(x_1,t)-\hat{E}_1(x_1,t)\big)-\big(\hat{E}_2(x_2,t)-\hat{E}_1(x_2,t)\big)\nonumber\\
&\geq \big(E_2(u_{1,1};x_1,t)-E_1(u_{1,1};x_1,t)\big)-\big(E_2(u_{2,2};x_2,t)-E_1(u_{2,2};x_2,t)\big)
\nonumber\\[1mm]
&\quad\,\,+\int^{x_2-tf'(0)}_{x_1-tf'(0)}\big(\varphi_2(\xi)-\varphi_1(\xi)\big)\,{\rm d}\xi
\nonumber\\[1mm]
&\geq\int_{x_1-tf'(u_{1,1})}^{x_2-tf'(u_{2,2})}(\varphi_2-\varphi_1)\,{\rm d}x.
\end{align}
Combining $\eqref{c2.53}$ with $\eqref{c2.55}$, we obtain $\eqref{c2.52}$.

\smallskip
{\bf (b).} From $\eqref{cony}$, for any $t>0$,
$$
f'(u_2(x\pm,t))-f'(u_1(x\pm,t))=\frac{y_1(x,t,u_1(x\pm,t))-y_2(x,t,u_2(x\pm,t))}{t}.
$$
Let $y^{\pm}_i(x,t):=y_i(x,t,u_i(x\pm,t))$ for $i=1,2$.
Since $f'(u)$ is strictly increasing, $f'(u_2(x\pm,t))-f'(u_1(x\pm,t))$ and hence
$y^{\pm}_1(x,t)-y^{\pm}_2(x,t)$ have the same signs
as $u_2(x\pm,t)-u_1(x\pm,t)$.

\smallskip
Since $y^{\pm}_1(\cdot,t)$ and $y^{\pm}_2(\cdot,t)$ are both nondecreasing in $x$,
then, for any $a,b\in \mathbb{Q}$ with $a<b$, the set:
$
\big\{x\in \mathbb{R}\,:\, y^{\pm}_1(x,t)<a\big\}\cap\big\{x\in \mathbb{R}\,:\, y^{\pm}_2(x,t)>b\big\}
$
is the intersection of two half-lines.
By removing at most countable points at which $y^{\pm}_1(\cdot,t)$ or $y^{\pm}_2(\cdot,t)$ is discontinuous,
we see that
$$
\big\{x\,:\,y^{\pm}_1(x,t)-y^{\pm}_2(x,t)<0\big\}
=\mathsmaller{\bigcup}_{a<b;\, a,b\in \mathbb{Q}}\big(\big\{x\,:\, y^{\pm}_1(x,t)<a\big\}
\cap\big\{x\,:\, y^{\pm}_2(x,t)>b\big\}\big)
$$
is an open set, so it is a countable union of disjoint open intervals
$\cup_n X_n:=\cup_n (x_{n,l},x_{n,r})$.
Similarly, the set: $\{x\,:\, y^{\pm}_1(x,t)-y^{\pm}_2(x,t)>0\}$
is a countable union of disjoint open intervals
$\cup_m Y_m:=\cup_m (y_{m,l},y_{m,r})$.
Furthermore, any two open intervals in $\{X_n,Y_m\}_{n,m}$ are disjoint.

\smallskip
For $x_1<x_2$, denote $[x_1(0),\,x_2(0)]:=[x_1-tf'(M(x_1,t)),\,x_2-tf'(m(x_2,t))]$.
Furthermore, for $X_n,Y_m$ in $\{X_n,Y_m\}_{n,m}$, denote
\begin{align*}
\begin{cases}
\displaystyle X_n(0):=(x_{n,l}-tf'(u_1(x_{n,l},t)),\,x_{n,r}-tf'(u_2(x_{n,r},t))),\\[2mm]
\displaystyle Y_m(0):=(y_{m,l}-tf'(u_2(y_{m,l},t)),\,y_{m,r}-tf'(u_1(y_{m,r},t))).
\end{cases}
\end{align*}
From $\eqref{c2.52}$, for any $t>0$ and $x_1,x_2\in \mathbb{R}$ with $x_1<x_2$,
\begin{align}\label{u2u1+}
\int_{[x_1,x_2]\bigcap \{u_2-u_1>0\}}(u_2-u_1)\,{\rm d}x
&=\int_{[x_1,x_2]\bigcap \{y^\pm_1(x,t)-y^\pm_2(x,t)>0\}}(u_2-u_1)\,{\rm d}x\nonumber\\[1mm]
& =\sum_m\int_{[x_1,x_2]\bigcap (y_{m,l},y_{m,r})}(u_2-u_1)\,{\rm d}x
\nonumber\\[1mm]
&
\leq\sum_m\int_{[x_1(0),x_2(0)]\bigcap Y_m(0)}(\varphi_2-\varphi_1)\,{\rm d}x,
\end{align}
where the last step is inferred by $\eqref{c2.52}$ with the following facts:
For the case that $x_1\in (y_{m,l},y_{m,r})$ for some $m$, $x_1(0)=x_1-tf'(u_2(x_1,t))$;
for the case that $x_2\in (y_{m',l},y_{m',r})$ for some $m'$, $x_2(0)=x_2-tf'(u_1(x_2,t))$.
Similarly, we have
\begin{equation}\label{u2u1-}
\int_{[x_1,x_2]\bigcap \{u_2-u_1<0\}}(u_2-u_1)\,{\rm d}x
\geq \sum_n\int_{[x_1(0),x_2(0)]\bigcap X_n(0)}(\varphi_2-\varphi_1)\,{\rm d}x.
\end{equation}

Since any two open intervals in $\{X_n,Y_m\}_{n,m}$ are disjoint,
by $\eqref{c2.31}$,
any two open intervals in $\{X_n(0),Y_m(0)\}_{n,m}$ do not intersect each other.
From $\eqref{u2u1+}$--$\eqref{u2u1-}$, we conclude
\begin{align*}
\int_{x_1}^{x_2}|u_2-u_1|\,{\rm d}x
& =-\int_{[x_1,x_2]\bigcap \{u_2-u_1<0\}}(u_2-u_1)\,{\rm d}x
+\int_{[x_1,x_2]\bigcap \{u_2-u_1>0\}}(u_2-u_1)\,{\rm d}x\\[1mm]
& \leq-\sum_n\int_{[x_1(0),x_2(0)]\bigcap X_n(0)}(\varphi_2-\varphi_1)\,{\rm d}x
+\sum_m\int_{[x_1(0),x_2(0)]\bigcap Y_m(0)}(\varphi_2-\varphi_1)\,{\rm d}x\\[1mm]
& \leq\bigg(\sum_n\int_{[x_1(0),x_2(0)]\bigcap X_n(0)}
+\sum_m\int_{[x_1(0),x_2(0)]\bigcap Y_m(0)}\bigg)|\varphi_2-\varphi_1|\,{\rm d}x\\[1mm]
& \leq\int_{[x_1(0),x_2(0)]}|\varphi_2-\varphi_1|\,{\rm d}x
=\int_{x_1-tf'(M(x_1,t))}^{x_2-tf'(m(x_2,t))}|\varphi_2-\varphi_1|\,{\rm d}x.
\end{align*}

\smallskip
\noindent
{\bf 2.} We now prove $\eqref{c2.50b}$.
Let the entropy solutions $u _i(x,t)$ be determined by $E_i(u;x,t)$ as in $\eqref{c2.1}$
with initial data functions $\varphi_i(x)\in L^\infty(\mathbb{R})$ for $i=1,2$.

To prove $u_1(x+,t)\leq u_2(x+,t)$ on $\mathbb{R}\times \mathbb{R}^+$,
denote $u^+_1:=u_1(x+,t)$.
Then $\eqref{c2.4}$ implies that
$E_1(u^+_1;x,t)-E_1(u;x,t)>0$ for any $u<u^+_1$.
Since $\varphi_1(x)\leq\varphi_2(x)\ { a.e.}\ {\rm on}\ \mathbb{R}$,
then, if $u<u^+_1$,
\begin{align*}
E_2(u^+_1;x,t)-E_2(u;x,t)
&=E_1(u^+_1;x,t)-E_1(u;x,t)\\
&\quad \, +t\int_u^{u_1^+} f''(s)\big(\varphi_2 (x-tf'(s))-\varphi_1 (x-tf'(s))\big)\,{\rm d}s\\[1mm]
&>t\int_u^{u_1^+} f''(s)\big(\varphi_2 (x-tf'(s))-\varphi_1 (x-tf'(s))\big)\,{\rm d}s
\geq 0.
\end{align*}
Since $E_2(\cdot\,;x,t)$ attains its maximum at $u_2(x+,t)$, then $u_2(x+,t)\geq u_1(x+,t)$.

\smallskip
To prove $u_1(x-,t)\leq u_2(x-,t)$ on $\mathbb{R}\times \mathbb{R}^+$,
denote $u^-_1:=u_1(x-,t)$. Then $\eqref{c2.4}$ implies that
$E_1(u^-_1; x,t)-E_1(u;x,t)\geq 0$ for any $u\in \mathbb{R}$.
Since $\varphi_1(x)\leq\varphi_2(x)\ { a.e.}\ {\rm on}\ \mathbb{R}$,
then, if $u<u^-_1$,
\begin{align*}
E_2(u^-_1;x,t)-E_2(u;x,t)
&=E_1(u^-_1;x,t)-E_1(u;x,t)\\
&\quad \, +t\int_u^{u_1^-} f''(s)\big(\varphi_2 (x-tf'(s))-\varphi_1 (x-tf'(s))\big)\,{\rm d}s\\[1mm]
&\geq t\int^{u^-_1}_{u}f''(s)\big(\varphi_2(x-tf'(s))-\varphi_1(x-tf'(s))\big)\,{\rm d}s
\geq 0.
\end{align*}
Since $u^-_2(x,t)$ is the supremum of the set of points at which
$E_2(\cdot\,;x,t)$ attains its maximum, then $u_2(x-,t)\geq u_1(x-,t)$.

\medskip
\noindent
{\bf 3.} We now prove $\eqref{c7.30}$ and $\eqref{c7.31}$.

\smallskip
{\bf (a).} To prove $\eqref{c7.30}$,
suppose that $\varphi_n,\varphi \in L^\infty$,
and $\{\varphi_n\}_n$ is an increasing sequence such that
$\lim_{n\rightarrow \infty}\|\varphi_n-\varphi\|_{L^\infty}=0$.
From $\eqref{c2.50b}$,
$$
u_1(x+,t)\leq u_2(x+,t)\leq \cdots\leq u(x+,t).
$$

Denote $\hat{u}(x,t):=\lim_{n\rightarrow \infty}u_n(x+,t)$.
Then $\hat{u}(x,t)\leq u(x+,t)$.
Thus, it suffices to show that $\hat{u}(x,t)\geq u(x+,t)$.
Let $E_n(u;x,t)$ and $E(u;x,t)$ be given by $\eqref{c2.1}$ with $\varphi_n(x)$ and $\varphi(x)$, respectively.
Then, for any fixed $(x,t)\in \mathbb{R}\times \mathbb{R}^+$,
\begin{equation}\label{c7.32}
E(u;x,t)=\lim_{n\rightarrow \infty}E_n(u;x,t)
\qquad{ \rm for\ any}\ u\in \mathbb{R}.
\end{equation}

Since $\{u_n(x+,t)\}$ and $\{\varphi_n\}$ are both uniformly bounded sequences,
it follows from $\eqref{c7.32}$ that,
for any fixed $(x,t)\in \mathbb{R}\times \mathbb{R}^+$,
\begin{align}\label{c7.33}
E(\hat{u};x,t)-E_n(u_n^+;x,t)
&=E(\hat{u};x,t)-E_n(\hat{u};x,t)\nonumber\\
&\quad\,
+t\int^{\hat{u}}_{u^+_n}f''(s)\big(\varphi_n(x-tf'(s))-s\big)\,{\rm d}s
\rightarrow 0 \qquad\, \mbox{as $n\to \infty$},
\end{align}
where $\hat{u}:=\hat{u}(x,t)$ and $u_n^+:=u_n(x+,t)$.

\vspace{2pt}
On the other hand, from $\eqref{c2.4}$, for every $n$,
$E_n(u;x,t)\leq E_n(u^+_n;x,t)$ for any $u\in \mathbb{R}$.
Then $\eqref{c7.32}$--$\eqref{c7.33}$ imply that,
for any $u\in \mathbb{R}$,
\begin{equation*}
E(\hat{u};x,t)= \lim_{n\rightarrow \infty}E_n(u_n^+;x,t)
\geq \lim_{n\rightarrow \infty}E_n(u;x,t)= E(u;x,t),
\end{equation*}
which means that $E(\cdot\,;x,t)$ attains its maximum at $\hat{u}=\hat{u}(x,t)$.
Since $u(x+,t)$ is the infimum of the set of points at which $E(\cdot\,;x,t)$ attains its maximum,
then $\hat{u}(x,t)\geq u(x+,t)$.

\smallskip
{\bf (b).} To prove $\eqref{c7.31}$,
suppose that $\varphi_n,\varphi \in L^\infty$,
and $\{\varphi_n\}_n$ is a decreasing sequence such that
$\lim_{n\rightarrow \infty}\|\varphi_n-\varphi\|_{L^\infty}=0$.
From $\eqref{c2.50b}$,
$$
u_1(x-,t)\geq u_2(x-,t)\geq \cdots\geq u(x-,t).
$$

\vspace{2pt}
Denote $\check{u}(x,t):=\lim_{n\rightarrow \infty}u_n(x-,t)$.
Then $\check{u}(x,t)\geq u(x-,t)$.
Thus, it suffices to show that $\check{u}(x,t)\leq u(x-,t)$.
Let $E_n(u;x,t)$ and $E(u;x,t)$ be given by $\eqref{c2.1}$ with $\varphi_n(x)$ and $\varphi(x)$, respectively.
Then, for any fixed $(x,t)\in \mathbb{R}\times \mathbb{R}^+$,
\begin{equation}\label{c7.35}
E(u;x,t)=\lim_{n\rightarrow \infty}E_n(u;x,t)\qquad\,\,{ \rm for\ any}\ u\in \mathbb{R}.
\end{equation}

Since $\{u_n(x-,t)\}$ and $\{\varphi_n\}$ are both uniformly bounded sequences,
it follows from $\eqref{c7.35}$ that,
for any fixed $(x,t)\in \mathbb{R}\times \mathbb{R}^+$,
\begin{align}\label{c7.36}
E(\check{u};x,t)-E_n(u_n^-;x,t)
&=E(\check{u};x,t)-E_n(\check{u};x,t)
\nonumber\\
&\quad \,
+t\int^{\check{u}}_{u^-_n}f''(s)\big(\varphi_n(x-tf'(s))-s\big)\,{\rm d}s\, \rightarrow 0
\qquad\, \mbox{as $n\to \infty$},
\end{align}
where $\check{u}:=\check{u}(x,t)$ and $u_n^-:=u_n(x-,t)$.

\vspace{2pt}
On the other hand, from $\eqref{c2.4}$, for every $n$,
$E_n(u;x,t)\leq E_n(u^-_n;x,t)$ holds for any $u\in \mathbb{R}$.
Then $\eqref{c7.35}$--$\eqref{c7.36}$ imply that,
for any $u\in \mathbb{R}$,
$$
E(\check{u};x,t)=\lim\limits_{n\rightarrow \infty}E_n(u^-_n;x,t)\geq \lim\limits_{n\rightarrow \infty}E_n(u;x,t)= E(u;x,t),
$$
which means that $E(\cdot\,;x,t)$ attains its maximum at $\check{u}(x,t)$.
Since $u(x-,t)$ is the supremum of the set of points at which $E(\cdot\,;x,t)$ attains its maximum,
then $\check{u}(x,t)\leq u(x{-},t)$.

\smallskip
To sum up, we have completed the proof of Corollary $\ref{cor:c2.1}$.

\subsubsection{Proof of {\rm Corollary} $\ref{cor:c2.2}$}
The proof is divided into three steps.

\smallskip
\noindent
{\bf 1.} We first show that, for any $(x,t)\in \mathbb{R}\times \mathbb{R}^+$ with $t>\tau$,
$$
c\in \mathcal{U}(x,t;\tau)
\qquad {\rm if\ and\ only\ if}
\qquad c\in \mathcal{U}(x,t).
$$

\vspace{1pt}
{\bf (a).} Let $c\in \mathcal{U}(x,t)$.
For any $w\in \mathbb{R}$, denote $y(w):=x-(t-\tau)f'(w)$
and
$$
u^-(w):=u(y(w)-,\tau), \qquad x(w):= y(w)-\tau f'(u^-(w)).
$$
Then, choosing $v$ so that $x=x(w)+tf'(v)$,
we have
\begin{align}\label{c2.63}
\begin{cases}
\displaystyle x-tf'(c)=y(c)-\tau f'(c),\\[1mm]
\displaystyle x-tf'(v)=y(w)-\tau f'(u^-(w)),\\[1mm]
\displaystyle tf'(v)=(t-\tau)f'(w)+\tau f'(u^-(w));
\end{cases}
\end{align}
see also Fig. $\ref{figTransport}$.
Taking $x_1=y(c)$, $x_2=y(w)$, and $t=\tau$ into $\eqref{c2.13}$ leads to
\begin{equation}\label{c2.64}
\int_{y(c)}^{y(w)}u(\xi,\tau)\,{\rm d}\xi=\hat{E}(y(c),\tau)-\hat{E}(y(w),\tau).
\end{equation}
Since $u(y(w)-,\tau)\in \mathcal{U}(y(w),\tau)$ and $c\in \mathcal{U}(x,t)$,
from $\eqref{c2.63}$--$\eqref{c2.64}$,
for any $w\in \mathbb{R}$,
\begin{align}\label{c2.65}
&\ E(c;x,t;\tau)-E(w;x,t;\tau) \nonumber\\[1mm]
&=(t-\tau)\int_w^c f''(s)\big(u(x-(t-\tau)f'(s),\tau)-s\big)\,{\rm d}s \nonumber\\[1mm]
&=\int_{y(c)}^{y(w)}u(\xi,\tau)\,{\rm d}\xi-(t-\tau)\int_w^c sf''(s)\,{\rm d}s \nonumber\\[1mm]
&=\hat{E}(y(c),\tau)-\hat{E}(y(w),\tau)+(t-\tau)\int_c^w sf''(s)\,{\rm d}s
\nonumber\\[1mm]
&\geq E(c;y(c),\tau)-E(u^-(w);y(w),\tau)+(t-\tau)\int_c^w sf''(s)\,{\rm d}s
+\int_{y(c)-\tau f'(0)}^{y(w)-\tau f'(0)}\varphi(\xi)\,{\rm d}\xi \nonumber\\[1mm]
&=\int_{y(c)-\tau f'(c)}^{y(w)-\tau f'(u^-(w))}\varphi(\xi)\,{\rm d}\xi
+(t-\tau)\int_c^w sf''(s)\,{\rm d}s +\tau\int_c^{u^-(w)}sf''(s)\,{\rm d}s \nonumber\\[1mm]
&=\int_{x-tf'(c)}^{x-t f'(v)}\varphi(\xi)\,{\rm d}\xi
+(t-\tau)\int_c^w sf''(s)\,{\rm d}s +\tau\int_c^{u^-(w)}sf''(s)\,{\rm d}s \nonumber\\[1mm]
& =E(c;x,t)-E(v;x,t)+(t-\tau)\int_v^w sf''(s)\,{\rm d}s+\tau\int_v^{u^-(w)} sf''(s)\,{\rm d}s \nonumber\\[1mm]
& \geq (t-\tau)\int_v^w sf''(s)\,{\rm d}s+\tau\int_v^{u^-(w)} sf''(s)\,{\rm d}s
\geq 0,
\end{align}
where the last step follows from the convexity of $f(u)$:
\begin{align*}
&\ (t-\tau)\int_v^w sf''(s)\,{\rm d}s+\tau\int_v^{u^-(w)} sf''(s)\,{\rm d}s\\[1mm]
&=(t-\tau)\big(sf'(s)\big|_v^w+ f(v)-f(w)\big)+\tau\big(sf'(s)\big|_v^{u^-(w)}+ f(v)-f(u^-(w))\big)\\[1mm]
&\geq (t-\tau)\big(sf'(s)\big|_v^w+f'(w)(v-w)\big)
+\tau\big(sf'(s)\big|_v^{u^-(w)}+f'(u^-(w))(v-u^-(w))\big)\\[1mm]
&=v\big({-}tf'(v)+(t-\tau)f'(w)+\tau f'(u^-(w))\big)
=0.
\end{align*}
This implies that $c\in \mathcal{U}(x,t;\tau)$.

\begin{figure}[H]
	\begin{center}
		{\includegraphics[width=0.7\columnwidth]{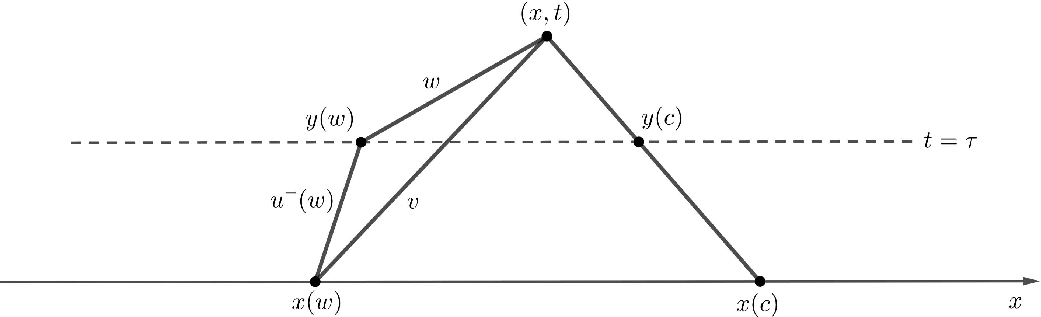}}
 \caption{The positional relation of $y(w)$ and $x(w)$ as in $\eqref{c2.63}$.}\label{figTransport}
	\end{center}
\end{figure}

\smallskip
{\bf (b).}
Let $c\in \mathcal{U}(x,t;\tau)$.
Then, for any $(x,t)$ with $t>\tau$,
$$
E(c;x,t;\tau)-E(w;x,t;\tau)\geq 0
\qquad\,\, {\rm for\ any}\ w\in \mathbb{R},
$$
which, by $y(w)=x-(t-\tau)f'(w)$, implies
\begin{equation}\label{c2.66}
\int_{y(c)}^{y(w)}(u(\xi,\tau)-c)\,{\rm d}\xi\geq-(t-\tau)\int_c^w (s-c)f''(s)\,{\rm d}s
\qquad {\rm for\ any}\ w\in \mathbb{R}.
\end{equation}
Since $y(w)-y(c)=(t-\tau)\big(f'(c)-f'(w)\big)$, by the strictly increasing property of $f'(u)$,
$$
y(w)-y(c)<0\quad \mbox{if}\ w>c, \qquad\,\,\, y(w)-y(c)>0 \quad \mbox{if}\ w<c.
$$
Notice that solution $u(x,t)$ has both the left- and right-traces.
It follows from $\eqref{c2.66}$ that
\begin{align*}
u(y(c){-},\tau)-c
&=\lim_{w\rightarrow c{+}}\dfrac{1}{y(w)-y(c)}\int_{y(c)}^{y(w)}(u(\xi,\tau)-c)\,{\rm d}\xi \\[1mm]
&\leq \lim_{w\rightarrow c{+}}\frac{1}{f'(w)-f'(c)}\int_c^w (s-c)f''(s)\,{\rm d}s\\[1mm]
&=\lim_{w\rightarrow c{+}}\frac{(w-c)f'(w)-\big(f(w)-f(c)\big)}{f'(w)-f'(c)}\\[1mm]
&=\lim_{w\rightarrow c{+}}\frac{f'(w)-f'(\zeta)}{f'(w)-f'(c)}(w-c) =0,
\end{align*}
where $\zeta$ lies between $w$ and $c$.

\vspace{2pt}
Similarly, we see that $u(y(c){+},\tau)-c\geq 0$ so that
$
u(y(c){-},\tau)\leq c\leq u(y(c){+},\tau),
$
which, by $u(y(c){-},\tau)=u^-(y(c),\tau)\geq u^+(y(c),\tau)= u(y(c){+},\tau)$,
implies that $u(y(c){\pm},\tau)=c$. This means that
$\mathcal{U}(y(c),\tau)=\{c\}$.
Then it follows from $\eqref{c2.65}$ that,
for any $w\in \mathbb{R}$,
\begin{align*}
0&\leq E(c;x,t;\tau)-E(w;x,t;\tau) \\[1mm]
&=\hat{E}(y(c),\tau)-\hat{E}(y(w),\tau)+(t-\tau)\int_c^w sf''(s)\,{\rm d}s
\\[1mm]
&\leq E(c;y(c),\tau)-E(w;y(w),\tau)+(t-\tau)\int_c^w sf''(s)\,{\rm d}s+\int_{y(c)-\tau f'(0)}^{y(w)-\tau f'(0)}\varphi(\xi)\,{\rm d}\xi\\[1mm]
&=-\int_{x-tf'(w)}^{x-tf'(c)}\varphi(\xi)\,{\rm d}\xi-t\int_w^c sf''(s)\,{\rm d}s
=E(c;x,t)-E(w;x,t),
\end{align*}
which implies that $c\in \mathcal{U}(x,t)$.

\smallskip
\noindent
{\bf 2.} For any fixed $(x,t)\in \mathbb{R}\times \mathbb{R}^+$ with $t>\tau$,
choosing $c\in \mathcal{U}(x,t)=\mathcal{U}(x,t;\tau)$, then, from $\eqref{c2.58}$,
$$
\hat{E}(x,t;\tau)=E(c;x,t;\tau)+\hat{E}(x-(t-\tau)f'(0),\tau),
$$
which, by using $\eqref{c2.64}$ with $w=0$, implies
\begin{align*}
\hat{E}(x,t;\tau)
&=-\int_{y(0)}^{y(c)}u(\xi,\tau)\,{\rm d}\xi-(t-\tau)\int_0^c sf''(s)\,{\rm d}s+\hat{E}(x-(t-\tau)f'(0),\tau)\\[1mm]
&=\hat{E}(y(c),\tau)-\hat{E}(y(0),\tau)-(t-\tau)\int_0^c sf''(s)\,{\rm d}s+\hat{E}(y(0),\tau)\\[1mm]
&=\hat{E}(y(c),\tau)-(t-\tau)\int_0^c sf''(s)\,{\rm d}s\\[1mm]
&=-\int_0^{x-tf'(c)}\varphi(\xi)\,{\rm d}\xi-t\int_0^c sf''(s)\,{\rm d}s
=\hat{E}(x,t).
\end{align*}

In particular, by the definitions of $u(x,t)$ and $v(x,t)$ in $\eqref{c2.4}$ and $\eqref{c2.58}$,
respectively, it follows from $\mathcal{U}(x,t;\tau)=\mathcal{U}(x,t)$ in $\eqref{c2.61c}$
that $\eqref{c2.61}$ holds.

\smallskip
\noindent
{\bf 3.}
According to $\eqref{c2.51}$,
to prove that the map $S_t$ determines a contractive semigroup in the $L^1_{\rm loc}$--norm,
it suffices to verify that
$S_t\varphi=S_{t-\tau}S_{\tau}\varphi$ for any $\tau\in [0,t]$.
In fact, this is directly implied by $\eqref{c2.61}$ via
$$S_t\varphi=u(\cdot,t)=v(\cdot,t)=S_{t-\tau}u(\cdot,\tau)=S_{t-\tau}S_{\tau}\varphi.$$

\smallskip
Up to now, we have completed the proof of Corollary $\ref{cor:c2.2}$.

\begin{Rem}\label{rem:c1.2}
The function, $E(u;x,t)$, defined by $\eqref{c2.1}$ can be understood formally as follows{\rm :}
It is well known that the local smooth solution of the Cauchy problem $\eqref{c1.1}$--$\eqref{ID}$,
if it exists, is determined by the following implicit function of $u\in \mathbb{R}$,
\begin{equation}\label{smooth}
\varphi(x-tf'(u))-u=0.
\end{equation}
If the local smooth solution blows up in a finite time,
$\eqref{smooth}$ can not hold exactly as time goes on, because
the characteristics will interact with each other.
{\it Even in this case,
the integral of $\varphi(x-tf'(u))-u$ with respect to $u$ as in $\eqref{c2.1}$
can exactly describe the interaction of characteristics in some sense}.
This insight indicates that, as shown in {\rm \S 8},
it is possible to identify the formula for entropy solutions for hyperbolic conservation laws
if such an intrinsic formula for the smooth solutions is known,
which may reveal a deeper connection
between smooth solutions and entropy solutions for hyperbolic conservation laws.
\end{Rem}

\begin{Rem}
As shown in Theorem $\ref{the:mt}$, for any point $(x,t)$ with $t>0$,
$E(u;x,t)$ in $\eqref{c2.1}$ as a function of $u$ attains 
its maximum at $u=u(x,t)$ in $\eqref{c2.50}$,
which is the unique entropy solution of the Cauchy problem $\eqref{c1.1}$--$\eqref{ID}$.
For the subsequent development, based on the analysis of function $E(u;x,t)$
and function $\Phi(x)=\int_0^x\varphi(\xi)\,{\rm d}\xi$
with its convex hull $\bar{\Phi}(x)$ over $({-}\infty,\infty)$
defined by $\eqref{c5.4}$--$\eqref{c5.5}$,
we can obtain the exact descriptions of the set $\mathcal{C}(x_0)$
in Definition $\ref{def:c4.1}$
as the set of the speeds of characteristic lines emitting from $x_0$ on the $x$--axis,
the lifespan expressions $\eqref{plc2}$ for $t_p(x_0,c)$ and $\eqref{elc1}$ for $t_*(x_0,c)$
of any characteristic line $L_{c}(x_0)$ emitting from $x_0$ on the $x$--axis
with speed $c\in \mathcal{C}(x_0)$,
the set $\mathcal{U}(x,t)$ in $\eqref{c2.3}$
as the set of the speeds of the shock-free backward characteristics emitting from point $(x,t)$,
the set $\mathcal{D}(x_0)$ in Definition $\ref{def:c5.1}$
as the set of the speeds of the divides emitting from $x_0$ on the $x$--axis,
and the point set $\mathcal{K}_0$ in $\eqref{c5.6}$ as the set of
the divide generation points.
Using the exact descriptions of sets $\mathcal{C}(x_0)$, $\mathcal{U}(x,t)$,
$\mathcal{D}(x_0)$, and $\mathcal{K}_0$,
and lifespans $t_p(x_0,c)$ and $t_*(x_0,c)$,
we can obtain various fine properties of entropy solutions, especially including
the generation of initial waves for the Cauchy problem,
the structures of entropy solutions inside the backward characteristic triangles,
the formation and development of shocks,
the necessary and sufficient conditions respectively for the divides as well as their locations and speeds,
the global structures of entropy solutions,
and the asymptotic profiles and decay rates respectively in the $L^\infty$--norm and the $L^p_{{\rm loc}}$--norm.
\end{Rem}

\smallskip
\section{Characteristics and Initial Waves for the Cauchy Problem}
In this section, we focus on the dynamic behaviors of characteristics
and the generation of initial waves of the Cauchy problem.
In \S 3.1, we establish the criteria for all the cases of 
the sets of characteristic generation values in Theorem $\ref{the:c4.0}$
with Propositions $\ref{pro:c4.2}$--$\ref{pro:c4.4}$;
in \S 3.2, the criteria for all the six types of initial waves for the Cauchy problem 
are established by providing and proving the necessary and sufficient conditions
for all the six types of initial waves for the Cauchy problem in Theorem $\ref{the:c4.1}$.
In \S 3.3--3.4, we introduce and prove 
an upper bound of lifespan and the lifespan of characteristics in Theorems $\ref{lem:plc1}$--$\ref{pro:elc1}$,
and provide a classification of characteristics based on the ways
of their terminations in Theorem $\ref{the:elc0}$.

\subsection{Generation of characteristics }
For any point $x_0\in \mathbb{R}$ and constant $c\in \mathbb{R}$,
define the line:
\begin{equation}\label{c4.1c}
L_c(x_0):=\big\{(x,t)\,:\, x=x_0+tf'(c) \,\,\,\mbox{for} \,\,\, t>0\big\}.
\end{equation}

\begin{Def}\label{def:c4.1}
$L_c(x_0)$ in $\eqref{c4.1c}$ is called a characteristic
emitting from $x_0$
if there exists $(x,t)\in L_c(x_0)$ such that $c\in \mathcal{U}(x,t)$,
for which $c$ is called a characteristic generation value of $x_0$.
The point set $\mathcal{C}(x_0)$ of characteristic generation values of $x_0$ is defined by
\begin{equation}\label{c4.0}
\mathcal{C}(x_0)
:=\big\{c \in \mathbb{R}\,:\, c \in \mathcal{U}(x,t)\,\,\, {\rm for\ some\ } (x,t)\in L_c(x_0) \big\}.
\end{equation}
\end{Def}

From $\eqref{c2.7}$, it is direct to see that
$\mathcal{C}(x_0)\subset [-\|\varphi\|_{L^\infty},\|\varphi\|_{L^\infty}]$.
As shown later in this section, $\mathcal{C}(x_0)$ can be determined by
$\overline{{\rm D}}_- \Phi(x_0)$ and $\underline{{\rm D}}_+ \Phi(x_0)$ in $\eqref{c4.1}$.
Let
\begin{equation}\label{c4.1a}
\Phi(x)=\int^x_{0} \varphi(\xi)\,{\rm d}\xi.
\end{equation}
Then the Dini derivatives of $\Phi(x)$ at a point $x_0\in \mathbb{R}$ are given by
\begin{equation}\label{c4.1}
\begin{cases}
\overline{{\rm D}}_\pm\Phi(x_0)
:=\displaystyle\limsup\limits_{y\rightarrow x_{0}{\pm }}\frac{1}{y-x_0}\int^y_{x_0} \varphi(\xi)\,{\rm d}\xi
=\limsup\limits_{l\rightarrow 0\pm}\frac{1}{l}\int^{x_0+l}_{x_0} \varphi(\xi)\,{\rm d}\xi, \\[4mm]
\underline{{\rm D}}_\pm\Phi(x_0)
:=\displaystyle\liminf\limits_{y\rightarrow x_{0}{\pm }}\frac{1}{y-x_0}\int^y_{x_0} \varphi(\xi)\,{\rm d}\xi
=\liminf\limits_{l\rightarrow 0\pm}\frac{1}{l}\int^{x_0+l}_{x_0} \varphi(\xi)\,{\rm d}\xi.
\end{cases}
\end{equation}

For any point $(x,t)\in L_c(x_0)$,
from $\eqref{c2.1}$--$\eqref{c2.4}$,
$c\in \mathcal{U}(x,t)$ if and only if,
for any $u\in \mathbb{R}$,
\begin{align}\label{c4.2}
0\leq E(c;x,t)-E(u;x,t)
=\int^{x_0+l}_{x_0}(\varphi(\xi)-c)\,{\rm d}\xi+t\int^u_c(s-c)f''(s)\,{\rm d}s,
\end{align}
where $l:=t(f'(c)-f'(u))$.
Equivalently, $c\in \mathcal{U}(x,t)$ if and only if,
for any $l=
l(u;t,c)$,
\begin{align}\label{c4.2b}
\Phi(l;x_0,c)
:=\int^{x_0+l}_{x_0}(\varphi(\xi)-c)\,{\rm d}\xi
\geq -t\int^{u(l;t,c)}_c(s-c)f''(s)\,{\rm d}s=:F(l;t,c),
\end{align}
where $u(l;t,c):=(f')^{-1}(f'(c)-\frac{l}{t})$ is the inverse function of
$l=l(u;t,c):=t(f'(c)-f'(u))$.

\begin{Lem}\label{lem:c4.1}
For any $x_0\in\mathbb{R}$ and $c\in\mathbb{R}$,
\begin{itemize}
\item [(i)] $c\in\mathcal{C}(x_0)$ if and only if there exists $t>0$ such that
 \begin{align}\label{c4.2c}
\Phi(l;x_0,c)>F(l;t,c)\qquad {\rm for\ small}\ l\neq 0.
\end{align}

\item [(ii)] $c\notin \mathcal{C}(x_0)$ if and only if, for any $t>0$,
there exists $l_n\neq 0$ with $\lim_{n\rightarrow\infty} l_n=0$ such that
\begin{align}\label{gc1}
\Phi(l_n;x_0,c)<F(l_n;t,c).
\end{align}
\end{itemize}
\end{Lem}

\begin{proof} We divide the proof into two steps accordingly.

\vspace{1pt}
\noindent
{\bf 1}. If $c\in \mathcal{C}(x_0)$, then, by $\eqref{c4.0}$,
there exists $(x_1,t_1)\in L_c(x_0)$ such that $c\in \mathcal{U}(x_1,t_1)$,
which, by $\eqref{uiff}$, implies that,
for any $(x,t)\in L_c(x_0)$ with $t\in (0,t_1)$, $\mathcal{U}(x,t)=\{c\}$, $i.e.$,
$c$ is the unique maximum point of $E(\cdot\,;x,t)$.
Then it follows from $\eqref{c4.2b}$ that $\eqref{c4.2c}$ holds.

On the other hand, suppose that there exist $t_1>0$ and small $l_0>0$ such that
\begin{align*}
\Phi(l;x_0,c)>F(l;t_1,c)\qquad {\rm for\ any}\ l\in ({-}l_0,0)\cup(0,l_0).
\end{align*}
For any fixed $l$ and $c$, it follows from $l=t(f'(c)-f'(u))$ that
$tf''(u)\,{\rm d}u=(f'(c)-f'(u))\,{\rm d}t$ such that
\begin{align}\label{c4.2e}
\frac{{\rm d}F}{{\rm d}t}
&=-\int^{u}_c(s-c)f''(s)\,{\rm d}s-t(u-c)f''(u)\, \frac{{\rm d}u}{{\rm d}t}\nonumber\\[1mm]
&=-\int^{u}_c(s-c)f''(s)\,{\rm d}s+(u-c)\big(f'(u)-f'(c)\big)
=\int^{u}_c(u-s)f''(s)\,{\rm d}s>0.
\end{align}
Thus, for any $t<t_1$,
\begin{align}\label{c4.2c2}
\Phi(l;x_0,c)>F(l;t_1,c)>F(l;t,c)\qquad {\rm for\ any}\ l\in ({-}l_0,0)\cup(0,l_0).
\end{align}
Since, for any $(x,t)\in \mathbb{R}\times\mathbb{R}^+$,
$E(\cdot\,;x,t)$ attains its maximum on $u\in [-\|\varphi\|_{L^\infty},\|\varphi\|_{L^\infty}]$,
then we can choose $t<t_1$ sufficiently small such that
$l=t(f'(c)-f'(u))\in ({-}l_0,l_0)$
for any $u\in [-\|\varphi\|_{L^\infty},\|\varphi\|_{L^\infty}].$
This, by $\eqref{c4.2c2}$, implies that,
for sufficiently small $t<t_1$,
$c$ is the unique maximum point of $E(\cdot\,;x,t)$ for $(x,t)\in L_c(x_0)$.
This shows that $c\in \mathcal{C}(x_0)$.

\smallskip
\noindent
{\bf 2}. It follows from $\eqref{c4.2c}$ that $c\notin \mathcal{C}(x_0)$ if and only if,
for any $t'>0$,
there exists a sequence $l_n\neq 0$ with $\lim_{n\rightarrow\infty} l_n=0$ such that
$\Phi(l_n;x_0,c)\leq F(l_n;t',c)$.
From $\eqref{c4.2e}$, for any $t>t'$, $F(l_n;t',c)<F(l_n;t,c)$. Then
$
\Phi(l_n;x_0,c)\leq F(l_n;t',c)<F(l_n;t,c),
$
which, by the arbitrariness of $t'>0$, implies $\eqref{gc1}$.
\end{proof}

\begin{Pro}\label{pro:c4.1}
For any $x_0\in\mathbb{R}$, if there exists $c\in \mathcal{C}(x_0)$, then
\begin{equation}\label{c4.3}
\overline{{\rm D}}_- \Phi(x_0)\leq c\leq \underline{{\rm D}}_+ \Phi(x_0).
\end{equation}
\end{Pro}

\begin{proof}
Suppose that $c\in \mathcal{C}(x_0)$. Since $f'(u)$ is strictly increasing,
$l=t(f'(c)-f'(u))$ implies that $l(u-c)<0$ for $u\neq c$
and $u=
u(l;t,c)\rightarrow c\mp$ as $l\rightarrow 0\pm$.
Then, from $\eqref{c4.1}$ and $\eqref{c4.2b}$,
\begin{align*}
\overline{{\rm D}}_-\Phi(x_0)-c
&=\mathop{\overline{\lim}}\limits_{l\rightarrow 0{-}}
\dfrac{1}{l}\int^{x_0+l}_{x_0}\varphi (\xi)\,{\rm d}\xi-c
\leq\lim_{u\rightarrow c{+}}\frac{1}{f'(u)-f'(c)}\int^u_c(s-c)f''(s)\,{\rm d}s\nonumber\\[1mm]
&=\lim_{u\rightarrow c{+}}\frac{f'(u)-f'(\zeta)}{f'(u)-f'(c)}(u-c)=0,
\end{align*}
where $\zeta$ lies between $u$ and $c$.

Similarly, it is direct to check that
$\underline{{\rm D}}_+\Phi(x_0)-c\geq 0.$
This implies $\eqref{c4.3}$.
\end{proof}

\begin{Pro}\label{pro:c4.2}
For any $x_0\in\mathbb{R}$, if $\overline{{\rm D}}_- \Phi(x_0)>\underline{{\rm D}}_+ \Phi(x_0)$, then $\mathcal{C}(x_0)= \varnothing$.
\end{Pro}

\begin{proof}
It follows immediately from Proposition $\ref{pro:c4.1}$.
\end{proof}

\begin{Pro}\label{pro:c4.2b}
For any $x_0\in\mathbb{R}$, if $\overline{{\rm D}}_- \Phi(x_0)<\underline{{\rm D}}_+ \Phi(x_0)$, then
\begin{equation}\label{c4.3a}
(\overline{{\rm D}}_- \Phi(x_0),\underline{{\rm D}}_+ \Phi(x_0))\subset\mathcal{C}(x_0)
\subset [\overline{{\rm D}}_- \Phi(x_0),\underline{{\rm D}}_+ \Phi(x_0)].
\end{equation}
\end{Pro}

\begin{proof}
Denote $a:=\overline{{\rm D}}_- \Phi(x_0)$ and $b:=\underline{{\rm D}}_+ \Phi(x_0)$. Then $a<b$.

If $c\in (a,b)$, it follows from $\eqref{c4.1}$ that,
for any given $\varepsilon$ with $0<\varepsilon <\min\{b-c,c-a\}$,
\begin{align*}
\mathop{\underline{\lim}}\limits_{l\rightarrow 0{+}}
\frac{1}{l}\int^{x_0+l}_{x_0}(\varphi(\xi)-c)\,{\rm d}\xi>\varepsilon,\qquad
\mathop{\overline{\lim}}\limits_{l\rightarrow 0{-}}
\frac{1}{l}\int^{x_0+l}_{x_0}(\varphi(\xi)-c)\,{\rm d}\xi<-\varepsilon,
\end{align*}
so that, for any fixed $t>0$,
\begin{align*}
\Phi(l;x_0,c)=\int^{x_0+l}_{x_0}(\varphi(\xi)-c)\,{\rm d}\xi>0>F(l;t,c)
\qquad {\rm for\ small}\ l\neq 0,
\end{align*}
which, by $\eqref{c4.2c}$, yields $\eqref{c4.3a}$.
\end{proof}

By Propositions $\ref{pro:c4.2}$--$\ref{pro:c4.2b}$,
the remaining task to determine $\mathcal{C}(x_0)$ completely is to judge
whether $\overline{{\rm D}}_- \Phi(x_0), \underline{{\rm D}}_+ \Phi(x_0)\in \mathcal{C}(x_0)$
for the case that $\overline{{\rm D}}_- \Phi(x_0)<\underline{{\rm D}}_+ \Phi(x_0)$,
and $\overline{{\rm D}}_- \Phi(x_0)\in\mathcal{C}(x_0)$
for the case that $\overline{{\rm D}}_- \Phi(x_0)=\underline{{\rm D}}_+ \Phi(x_0)$.
These are done by Propositions $\ref{pro:c4.3}$--$\ref{pro:c4.4}$ below.

\begin{Pro}\label{pro:c4.3}
Let $\overline{{\rm D}}_- \Phi(x_0)<\underline{{\rm D}}_+ \Phi(x_0)$.
Denote $a:=\overline{{\rm D}}_- \Phi(x_0)$ and $b:=\underline{{\rm D}}_+ \Phi(x_0)$.
Then
\begin{itemize}
\item [(i)] $a\in \mathcal{C}(x_0)$ if and only if there exists $t_1>0$
such that
\begin{equation}\label{c4.6}
\Phi(l;x_0,a)>F(l;t_1,a)\qquad {\rm for\ small }\ l< 0.
\end{equation}
Furthermore, if there exist some constants $\alpha_+\geq 0$ and $N_+>0$ such that
\begin{equation}\label{c4.7a}
f''(u)=(N_++o(1))|u-a|^{\alpha_+}\qquad\,\, {\rm as}\ u\rightarrow a{+},
\end{equation}
then $a\in \mathcal{C}(x_0)$ if and only if
\begin{equation}\label{c4.7}
\mathop{\underline{\lim}}\limits_{l\rightarrow 0{-}}{|l|^{-\frac{2+\alpha_+}{1+\alpha_+}}}
\int^{x_0+l}_{x_0}(\varphi (\xi)-a)\,{\rm d}\xi>-\infty;
\end{equation}
or equivalently, $a\notin \mathcal{C}(x_0)$ if and only if
\begin{equation*}
\mathop{\underline{\lim}}\limits_{l\rightarrow 0{-}}{|l|^{-\frac{2+\alpha_+}{1+\alpha_+}}}
\int^{x_0+l}_{x_0}(\varphi (\xi)-a)\,{\rm d}\xi=-\infty.
\end{equation*}
If, in addition, there exist some constants $\gamma_->0$ and $C_{\gamma_-}\neq 0$ such that
\begin{equation}\label{c4.7b}
\varphi(x_0+l)-a=(C_{\gamma_-}+o(1)){\rm sgn}(l) |l|^{\gamma_-}\qquad\,\, {\rm as}\ l\rightarrow 0{-},
\end{equation}
then $a\in \mathcal{C}(x_0)$ if and only if
\begin{equation}\label{c4.7c}
C_{\gamma_-}>0 \qquad {\rm or}\qquad \gamma_- (1+\alpha_+)\geq 1;
\end{equation}
or equivalently, $a\notin \mathcal{C}(x_0)$ if and only if
$C_{\gamma_-}<0$ and $\gamma_- (1+\alpha_+)< 1.$
In this case, if $a\notin \mathcal{C}(x_0)$, then, for any $t>0$,
\begin{equation}\label{c4.7d}
\Phi(l;x_0,a)<F(l;t,a)\qquad\,\, {\rm for\ small }\ l< 0.
\end{equation}

\item [(ii)] $b\in \mathcal{C}(x_0)$ if and only if
there exists $t_2>0$ such that
\begin{equation}\label{c4.8}
\Phi(l;x_0,b)>F(l;t_2,b)\qquad\,\, {\rm for\ small }\ l> 0.
\end{equation}
Furthermore, if there exist some constants $\alpha_-\geq 0$ and $N_->0$
such that
\begin{equation}\label{c4.9a}
f''(u)=(N_-+o(1))|u-b|^{\alpha_-}\qquad\,\, {\rm as}\ u\rightarrow b{-},
\end{equation}
then $b\in \mathcal{C}(x_0)$ if and only if
\begin{equation}\label{c4.9}
\mathop{\underline{\lim}}\limits_{l\rightarrow 0{+}}{|l|^{-\frac{2+\alpha_-}{1+\alpha_-}}}
\int^{x_0+l}_{x_0}(\varphi(\xi)-b)\,{\rm d}\xi>-\infty;
\end{equation}
or equivalently, $b\notin \mathcal{C}(x_0)$ if and only if
\begin{equation*}
\mathop{\underline{\lim}}\limits_{l\rightarrow 0{+}}{|l|^{-\frac{2+\alpha_-}{1+\alpha_-}}}
\int^{x_0+l}_{x_0}(\varphi(\xi)-b)\,{\rm d}\xi=-\infty.
\end{equation*}
If, in addition, there exist some constants $\gamma_+>0$ and $C_{\gamma_+}\neq 0$ such that
\begin{equation}\label{c4.9b}
\varphi(x_0+l)-b=(C_{\gamma_+}+o(1)){\rm sgn}(l)|l|^{\gamma_+}
\qquad\,\, {\rm as}\ l\rightarrow 0{+},
\end{equation}
then $b\in \mathcal{C}(x_0)$ if and only if
\begin{equation}\label{c4.9ca}
C_{\gamma_+}>0 \qquad {\rm or}\qquad \gamma_+ (1+\alpha_-)\geq 1;
\end{equation}
or equivalently, $b\notin \mathcal{C}(x_0)$ if and only if
$C_{\gamma_+}<0$ and $\gamma_+ (1+\alpha_-)< 1.$
In this case, if $b\notin \mathcal{C}(x_0)$, then, for any $t>0$,
\begin{equation}\label{c4.9d}
\Phi(l;x_0,b)<F(l;t,b)\qquad\,\, {\rm for\ small }\ l> 0.
\end{equation}
\end{itemize}
\end{Pro}

\begin{proof} We divide the proof into two steps accordingly.

\smallskip
\noindent
{\bf 1}. From Lemma $\ref{lem:c4.1}$, $a\in \mathcal{C}(x_0)$ if and only if $\eqref{c4.2c}$ holds.

\smallskip
{\bf (a)}.
To prove $\eqref{c4.6}$, it suffices to show that,
for any fixed $t>0$, $\eqref{c4.6}$ always holds for small $l>0$.
In fact, since $a<b$, it follows from $\eqref{c4.1}$ that,
for any given $\varepsilon\in(0,b-a)$,
\begin{align*}
\mathop{\underline{\lim}}\limits_{l\rightarrow 0{+}}
\frac{1}{l}\int^{x_0+l}_{x_0}(\varphi(\xi)-a)\,{\rm d}\xi=b-a>\varepsilon,
\end{align*}
which implies that, for any fixed $t>0$,
\begin{align}\label{c4.12a}
\Phi(l;x_0,a)=\int^{x_0+l}_{x_0}(\varphi(\xi)-a)\,{\rm d}\xi>0>F(l;t,a)
\qquad {\rm for\ small}\ l>0.
\end{align}

\smallskip
{\bf (b)}. Assume that $\eqref{c4.7a}$ holds.
For any fixed $t>0$, it follows from $l=t(f'(a)-f'(u))$ that
$u=u(l;t,a)\rightarrow a{+}$ as $l\rightarrow 0{-}$.
From $\eqref{c4.2b}$ and $\eqref{aa2}$--$\eqref{aa3}$,
for small $l<0$,
\begin{align}\label{c4.13}
F(l;t,a)&=-t\int^{u(l;t,a)}_a(s-a)f''(s)\,{\rm d}s
=-t\frac{N_++o(1)}{2+\alpha_+}|u-a|^{2+\alpha_+}\nonumber\\[1mm]
&=-t\frac{N_++o(1)}{2+\alpha_+}\Big(\frac{1+\alpha_+}{N_++o(1)}
 \big|f'(u)-f'(a)\big|\Big)^{\frac{2+\alpha_+}{1+\alpha_+}}\nonumber\\[1mm]
&=-t\frac{N_++o(1)}{2+\alpha_+}
\Big(\frac{1+\alpha_+}{N_++o(1)}\frac{|l|}{t}\Big)^{\frac{2+\alpha_+}{1+\alpha_+}}\nonumber\\[1mm]
&=-(1+o(1))\frac{1+\alpha_+}{2+\alpha_+}\Big(\frac{1+\alpha_+}{N_+}\Big)^{\frac{1}{1+\alpha_+}}\,
t^{-\frac{1}{1+\alpha_+}}\,|l|^{\frac{2+\alpha_+}{1+\alpha_+}}.
\end{align}

If $a\in \mathcal{C}(x_0)$, then it follows from $\eqref{c4.6}$ and $\eqref{c4.13}$ that
\begin{align*}
\mathop{\underline{\lim}}\limits_{l\rightarrow 0{-}}
|l|^{-\frac{2+\alpha_+}{1+\alpha_+}}
\int^{x_0+l}_{x_0}(\varphi (\xi)-a)\,{\rm d}\xi
&=\mathop{\underline{\lim}}\limits_{l\rightarrow 0{-}}
|l|^{-\frac{2+\alpha_+}{1+\alpha_+}}\,\Phi(l;x_0,a)\nonumber\\[1mm]
&\geq -\frac{1+\alpha_+}{2+\alpha_+}\Big(\frac{1+\alpha_+}{N_+}\Big)^{\frac{1}{1+\alpha_+}}\,
t_1^{-\frac{1}{1+\alpha_+}}>-\infty.
\end{align*}
On the other hand, if $\eqref{c4.7}$ holds, there exists $M>0$ such that
\begin{align*}
\mathop{\underline{\lim}}\limits_{l\rightarrow 0{-}}|l|^{-\frac{2+\alpha_+}{1+\alpha_+}}
\int^{x_0+l}_{x_0}(\varphi (\xi)-a)\,{\rm d}\xi\geq -\frac{M}{4},
\end{align*}
which, by combining with $\eqref{c4.13}$, implies that, for small $l<0$,
\begin{align*}
\Phi(l;x_0,a)=\int^{x_0+l}_{x_0}(\varphi (\xi)-a)\,{\rm d}\xi
>-\frac{M}{2} \,|l|^{\frac{2+\alpha_+}{1+\alpha_+}}>F(l;\tilde{t}_1,a)
\end{align*}
as desired, where
$\tilde{t}_1$ is given by
$$
\tilde{t}_1:=\frac{1+\alpha_+}{N_+}\big(\frac{1+\alpha_+}{2+\alpha_+}\frac{1}{M}\big)^{1+\alpha_+}.
$$

\smallskip
{\bf (c)}.
If, in addition, $\eqref{c4.7b}$ holds, then, by a simple calculation,
\begin{align}\label{c4.13d1}
|l|^{-\frac{2+\alpha_+}{1+\alpha_+}}\int^{x_0+l}_{x_0}(\varphi (\xi)-a)\,{\rm d}\xi
=\frac{C_{\gamma_-}+o(1)}{1+\gamma_-}|l|^{\frac{\gamma_-(1+\alpha_+)-1}{1+\alpha_+}},
\end{align}
which implies that $\eqref{c4.7}$ holds if and only if
$C_{\gamma_-}>0$ or $\gamma_-(1+\alpha_+)\geq 1$.
Then $\eqref{c4.7c}$ holds.
It is direct to check from $\eqref{c4.13}$--$\eqref{c4.13d1}$ that $\eqref{c4.7d}$ holds.

\medskip
\noindent
{\bf 2}. By the same arguments as in Step 1 above, it can be checked that $\eqref{c4.8}$--$\eqref{c4.9d}$ hold.
\end{proof}

\begin{Pro}\label{pro:c4.4}
Let $\overline{{\rm D}}_- \Phi(x_0)=\underline{{\rm D}}_+ \Phi(x_0)=:a$.
Then
\begin{equation}\label{c4.16}
\mathcal{C}(x_0)=\{a\} \qquad {\rm or}\qquad \mathcal{C}(x_0)=\varnothing.
\end{equation}
Moreover, the following properties are valid{\rm :}
\begin{itemize}
\item[(i)] $\mathcal{C}(x_0)=\{a\}$ if and only if there exists $t_1>0$ such that
\begin{equation}\label{c4.17}
\Phi(l;x_0,a)>F(l;t_1,a)\qquad\,\, {\rm for\ small}\ l\neq 0.
\end{equation}
\item[(ii)] If there exist some constants $\alpha_\pm\geq 0$ and $N_\pm>0$ such that
\begin{equation}\label{c4.18a}
f''(u)=(N_\pm+o(1))|u-a|^{\alpha_\pm}\qquad\,\, {\rm as}\ u\rightarrow a{\pm},
\end{equation}
then $\mathcal{C}(x_0)=\varnothing$ if and only if
at least one of the following two limits holds{\rm :}
\begin{equation}\label{c4.18}
\mathop{\underline{\lim}}\limits_{l\rightarrow 0{\pm}}
|l|^{-\frac{2+\alpha_\mp}{1+\alpha_\mp}}\int^{x_0+l}_{x_0}(\varphi (\xi)-a)\,{\rm d}\xi=-\infty.
\end{equation}
\item[(iii)] If, in addition, there exist some constants $\gamma_\pm>0$ and $C_{\gamma_\pm}\neq 0$ such that
\begin{equation}\label{c4.18c}
\varphi(x_0+l)-a=\big(C_{\gamma_\pm}+o(1)\big){\rm sgn}(l)|l|^{\gamma_\pm}
\qquad\,\, {\rm as}\ l\rightarrow 0{\pm},
\end{equation}
then $\mathcal{C}(x_0)=\varnothing$ if and only if
at least one of the following two holds{\rm :}
\begin{equation}\label{c4.18d}
C_{\gamma_\pm}<0 \qquad {\rm and}\qquad \gamma_\pm (1+\alpha_\mp)< 1.
\end{equation}
In this case, if $\mathcal{C}(x_0)=\varnothing$, then, for any $t>0$,
at least one of the following two holds{\rm :}
\begin{equation*}
\Phi(l;x_0,a)<F(l;t,a)\qquad\,\,{\rm for\ small}\ \pm l> 0.
\end{equation*}
\end{itemize}
\end{Pro}

\begin{proof}
It follows immediately by letting $b=a$ in Proposition $\ref{pro:c4.3}$.
\end{proof}

To sum up, for any $x_0\in \mathbb{R}$,
the point set $\mathcal{C}(x_0)$ of characteristic generation values of $x_0$
can be completely determined by Theorem $\ref{the:c4.0}$ as follows:

\begin{The}[Criteria for all the cases of $\mathcal{C}(x_0)$]\label{the:c4.0}
Let the point set $\mathcal{C}(x_0)$ of characteristic generation values of $x_0$,
and the Dini derivatives $\overline{{\rm D}}_- \Phi(x_0)$ and $\underline{{\rm D}}_+ \Phi(x_0)$ at $x_0$
be given by $\eqref{c4.0}$ and $\eqref{c4.1a}$--$\eqref{c4.1}$, respectively.

\begin{itemize}
\item [(i)] If $\overline{{\rm D}}_- \Phi(x_0)>\underline{{\rm D}}_+ \Phi(x_0)$,
then $\mathcal{C}(x_0)=\varnothing$.

\item [(ii)] If $\overline{{\rm D}}_- \Phi(x_0)<\underline{{\rm D}}_+ \Phi(x_0)$,
then $\eqref{c4.3a}$ holds, $i.e.$,
$$
(\overline{{\rm D}}_- \Phi(x_0),\underline{{\rm D}}_+ \Phi(x_0))\subset\mathcal{C}(x_0)
\subset [\overline{{\rm D}}_- \Phi(x_0),\underline{{\rm D}}_+ \Phi(x_0)].
$$
Moreover, the following properties are valid{\rm :}
\begin{itemize}
\item[(a)] $\overline{{\rm D}}_- \Phi(x_0) \in \mathcal{C}(x_0)\
(resp.,\ \underline{{\rm D}}_+ \Phi(x_0)\in \mathcal{C}(x_0))$
if and only if $\eqref{c4.6}\ (resp.,\ \eqref{c4.8})$ holds{\rm ;}
\item[(b)] For $\overline{{\rm D}}_- \Phi(x_0)$,
if $\eqref{c4.7a}$ holds, then $\overline{{\rm D}}_- \Phi(x_0)\in \mathcal{C}(x_0)$
if and only if $\eqref{c4.7}$ holds
and, if in addition $\eqref{c4.7b}$ holds,
then $\overline{{\rm D}}_- \Phi(x_0)\in \mathcal{C}(x_0)$
if and only if $\eqref{c4.7c}$ holds{\rm ;}
\item[(c)]
For $\underline{{\rm D}}_+ \Phi(x_0)$,
if $\eqref{c4.9a}$ holds,
then $\underline{{\rm D}}_+ \Phi(x_0)\in \mathcal{C}(x_0)$
if and only if $\eqref{c4.9}$ holds and,
if in addition $\eqref{c4.9b}$ holds,
then $\underline{{\rm D}}_+ \Phi(x_0)\in \mathcal{C}(x_0)$
if and only if $\eqref{c4.9ca}$ holds.
\end{itemize}

\item [(iii)] If $\overline{{\rm D}}_- \Phi(x_0)=\underline{{\rm D}}_+ \Phi(x_0)$,
then $\eqref{c4.16}$ holds, $i.e.$,
$$
\mathcal{C}(x_0)=\big\{\underline{{\rm D}}_+ \Phi(x_0)\big\}
\qquad {\rm or}\qquad \mathcal{C}(x_0)=\varnothing.$$
Moreover, the following properties are valid{\rm :}
\begin{itemize}
\item[(a)]
$\underline{{\rm D}}_+ \Phi(x_0)\in \mathcal{C}(x_0)$
if and only if $\eqref{c4.17}$ holds{\rm ;}
\item[(b)] If $\eqref{c4.18a}$ holds,
then $\mathcal{C}(x_0)=\varnothing$ if and only if
at least one case of $\eqref{c4.18}$ holds{\rm ;}
\item[(c)]
If $\eqref{c4.18a}$ and $\eqref{c4.18c}$ hold,
then $\mathcal{C}(x_0)=\varnothing$ if and only if
at least one case of $\eqref{c4.18d}$ holds.
\end{itemize}
\end{itemize}
\end{The}

As a summative theorem, it follows directly from Propositions $\ref{pro:c4.2}$--$\ref{pro:c4.4}$.

\subsection{Generation of all the six types of initial waves for the Cauchy problem}
Combining Theorem $\ref{the:c4.0}$ with Propositions $\ref{pro:c4.2}$--$\ref{pro:c4.4}$ together,
we obtain the criteria for all the six types of initial waves for the Cauchy problem 
emitting from any point $x_0$ on the $x$--axis as follows:

\begin{The}[Classification of the initial waves for the Cauchy problem]\label{the:c4.1}
For any point $x_0$ on the $x$--axis,
the initial waves for the Cauchy problem emitting from $x_0$ must be one of the following six cases{\rm :}
\begin{itemize}
\item [(i)] If $\overline{{\rm D}}_- \Phi(x_0)>\underline{{\rm D}}_+ \Phi(x_0)$,
or $\overline{{\rm D}}_- \Phi(x_0)=\underline{{\rm D}}_+ \Phi(x_0)$
with $\mathcal{C}(x_0)=\varnothing$,
then the initial wave for the Cauchy problem emitting from $x_0$ is a shock $\mathcal{S}${\rm ;}

\item [(ii)] If $\overline{{\rm D}}_- \Phi(x_0)=\underline{{\rm D}}_+ \Phi(x_0)$
with $\mathcal{C}(x_0)\not=\varnothing$,
then the initial wave for the Cauchy problem emitting from $x_0$ is a characteristic{\rm;}

\item [(iii)] If $\overline{{\rm D}}_- \Phi(x_0)<\underline{{\rm D}}_+ \Phi(x_0)$
and $\mathcal{C}(x_0)=[\overline{{\rm D}}_- \Phi(x_0),\underline{{\rm D}}_+ \Phi(x_0)]$,
then the initial wave for the Cauchy problem emitting from $x_0$ is a rarefaction wave $\mathcal{R}${\rm ;}

\item [(iv)] If $\overline{{\rm D}}_- \Phi(x_0)<\underline{{\rm D}}_+ \Phi(x_0)$
and $\mathcal{C}(x_0)=(\overline{{\rm D}}_- \Phi(x_0),\underline{{\rm D}}_+ \Phi(x_0)]$,
then the initial wave for the Cauchy problem emitting from $x_0$ is $\mathcal{S}+\mathcal{R}${\rm ;}

\item [(v)] If $\overline{{\rm D}}_- \Phi(x_0)<\underline{{\rm D}}_+ \Phi(x_0)$
and $\mathcal{C}(x_0)=[\overline{{\rm D}}_- \Phi(x_0),\underline{{\rm D}}_+ \Phi(x_0))$,
then the initial wave for the Cauchy problem
emitting from $x_0$ is $\mathcal{R}+\mathcal{S}${\rm ;}

\item [(vi)] If $\overline{{\rm D}}_- \Phi(x_0)<\underline{{\rm D}}_+ \Phi(x_0)$
and $\mathcal{C}(x_0)=(\overline{{\rm D}}_- \Phi(x_0),\underline{{\rm D}}_+ \Phi(x_0))$,
then the initial wave for the Cauchy problem emitting from $x_0$ is $\mathcal{S}+\mathcal{R}+\mathcal{S}$.
\end{itemize}
\noindent
In the above, $\mathcal{C}(x_0)$, $\overline{{\rm D}}_- \Phi(x_0)$, and $\underline{{\rm D}}_+ \Phi(x_0)$
are given by $\eqref{c4.0}$--$\eqref{c4.1}$.
See {\rm Fig.} $\ref{figPhiFab}$ and {\rm Fig.} $\ref{figPhiFabab}$ for the details.
\end{The}

\begin{figure}[H]
	\begin{center}
		{\includegraphics[width=0.8\columnwidth]{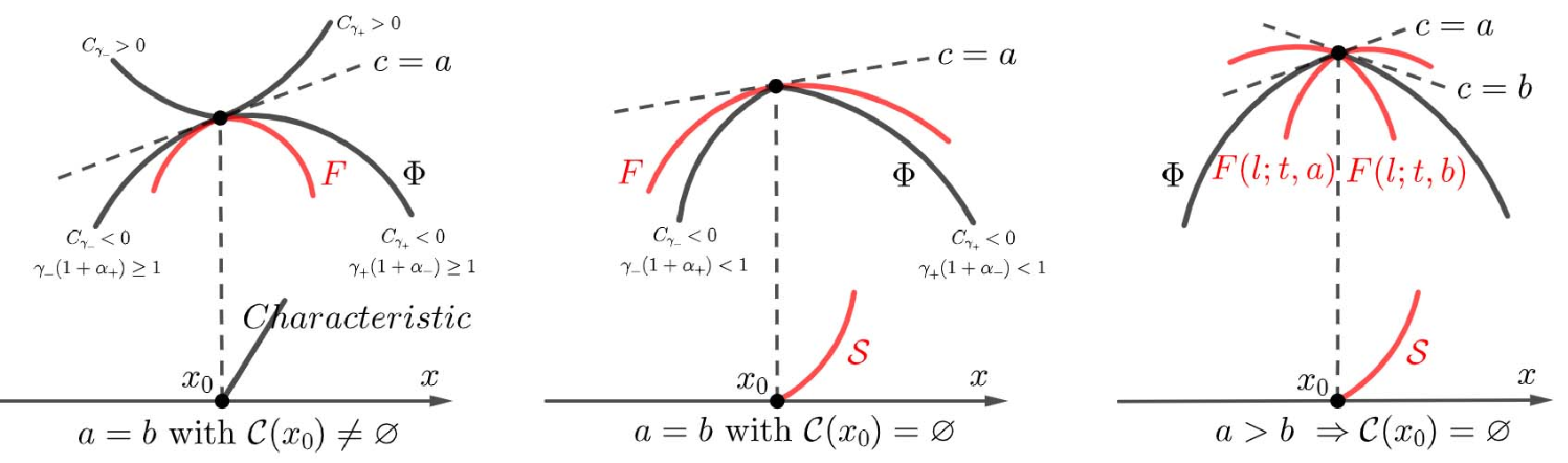}}
 \caption{$a=\overline{{\rm D}}_- \Phi(x_0)\geq\underline{{\rm D}}_+ \Phi(x_0)=b$.
		}\label{figPhiFabab}
	\end{center}
\end{figure}

\begin{proof} We divide the proof into three steps.

\smallskip
\noindent
{\bf 1}. From $\eqref{c4.3}$, $\overline{{\rm D}}_- \Phi(x_0)>\underline{{\rm D}}_+ \Phi(x_0)$
implies that $\mathcal{C}(x_0)= \varnothing$.
Now suppose $\mathcal{C}(x_0)= \varnothing$.
For any $t=t_1>0$, we let
\begin{align*}
\begin{cases}
\displaystyle a(t_1)=\sup\big\{x\in \mathbb{R}\,:\, x-t_1f'(u(x{-},t_1))<x_0\big\},\\[1mm]
\displaystyle b(t_1)\,=\,\inf\big\{x\in \mathbb{R}\,:\, x-t_1f'(u(x{-},t_1))>x_0\big\}.
\end{cases}
\end{align*}
Then $a(t_1)\leq b(t_1)$.
If $a(t_1)<b(t_1)$, for any $x_1\in(a(t_1),b(t_1))$,
$x_1-t_1f'(u(x_1{-},t_1))=x_0$ so that $u(x_1{-},t_1)\in \mathcal{C}(x_0)$,
which contradicts with $\mathcal{C}(x_0)=\varnothing$.
Therefore, $a(t_1)=b(t_1)$.

\smallskip
{\bf (a)}. Set $x=x(t): =a(t)$ for $t>0$. Then $x(t)$ is a single-valued curve with $x(0)=x_0$.
From $\eqref{c2.40}$, for any $t_1>0$,
\begin{align*}
\begin{cases}
\displaystyle x(t_1)-t_1f'(u(x(t_1){-},t_1))
=\lim\limits_{x\rightarrow x(t_1)-}\big(x-t_1f'(u(x{-},t_1))\big)\leq x_0,\\
\displaystyle x(t_1)-t_1f'(u(x(t_1){+},t_1))
=\lim\limits_{x\rightarrow x(t_1)+}\big(x-t_1f'(u(x{-},t_1))\big)\geq x_0.
\end{cases}
\end{align*}
Then, from $\mathcal{C}(x_0)= \varnothing$,
\begin{equation}\label{c4.5}
x(t_1)-t_1f'(u(x(t_1){-},t_1))<x_0<x(t_1)-t_1f'(u(x(t_1){+},t_1)),
\end{equation}
so that, by the strictly increasing property of $f'(u)$, $u(x(t_1){+},t_1)<u(x(t_1){-},t_1)$,
which implies that the entropy solution $u(x,t)$ is discontinuous along the single-valued curve $x=x(t)$.

\smallskip
{\bf (b)}. We now prove the continuity of $x=x(t)$.
Denote $y^\pm (x(t),t):=x(t)-tf'(u(x(t){\pm},t))$.
From $\eqref{uiff}$ and $\eqref{c4.5}$, for any $t_2>t_1>0$,
\begin{equation}\label{c43.10}
y^-(x(t_2),t_2)\leq y^-(x(t_1),t_1)\leq y^+(x(t_1),t_1) \leq y^+(x(t_2),t_2).
\end{equation}
Thus, $y^-(x(t),t)$ is nonincreasing and $y^+(x(t),t)$ is nondecreasing.
Letting $t_2 \rightarrow t_{1}{+}$ in $\eqref{c43.10}$,
\begin{equation*}
y^-(\overline{x(t_{1}+)},t_1)\leq y^-(x(t_1),t_1)\leq y^+(x(t_1),t_1)\leq y^+(\underline{ x(t_{1}+)},t_1),
\end{equation*}
where $\overline{x(t_{1}+)}:=\mathop{\overline{\lim}}_{t_2 \rightarrow t_{1}{+}}x(t_2)$
and $\underline{ x(t_{1}+)}:=\mathop{\underline{\lim}}_{t_2 \rightarrow t_{1}{+}}x(t_2)$.

By Lemma $\ref{lem:c2.4}$,
Since $y^{\pm}(\cdot,t_1)$ are nondecreasing,
then $\overline{x(t_{1}+)}\leq x(t_1)\leq\underline{ x(t_{1}+)},$
which implies that
$\lim_{t\rightarrow t_{1}{+}}x(t)=x(t_1)$ for any $t_1>0$.
Thus, $x(t)$ is right-continuous on $(0,\infty)$.

Similarly, letting $t_1 \rightarrow t_{2}{-}$ in $\eqref{c43.10}$,
$$
y^- (x(t_2),t_2)\leq y^- (\underline{x(t_{2}-)},t_2)
\leq y^+ (\overline{x(t_{2}-)},t_2)\leq y^+ (x(t_{2}),t_2),
$$
where $\overline{x(t_{2}-)}:=\mathop{\overline{\lim}}_{t_1 \rightarrow t_{2}{-}}x(t_1)$
and $\underline{ x(t_{2}-)}:=\mathop{\underline{\lim}}_{t_1 \rightarrow t_{2}{-}}x(t_1)$.
This, together with Lemma $\ref{lem:c2.4}$, implies that
$\overline{x(t_{2}-)}\leq x(t_2)\leq \underline{x(t_{2}-)}$ so that
$x(t)$ is left-continuous on $(0,\infty)$.

Therefore, $x=x(t)$ is continuous on $(0,\infty)$.
Meanwhile, letting $t_1\rightarrow 0{+}$ in $\eqref{c4.5}$,
we see that $\lim_{t\rightarrow 0{+}}x(t)=x_0$.

\smallskip
\noindent
{\bf 2}. Denote $a:=\overline{{\rm D}}_- \Phi(x_0)$.
Since $\mathcal{C}(x_0)\not=\varnothing$, by Proposition $\ref{pro:c4.1}$,
there exists only one characteristic emitting from $x_0$,
which is given by $x=x_0+tf'(a)$ with $t\in [0,t_0)$ for some $t_0>0$.
This means that there exists no rarefaction wave emitting from $x_0$.

We now prove that there does not exist a shock emitting from $x_0$.
Otherwise, there exists a shock $x=x(t)$ emitting from $x_0$.
By $\eqref{uiff}$, shock $x=x(t)$ cannot intersect
the characteristic: $x=x_0+tf'(a)$ on $(0,t_0)$.
Without loss of generality, assume that $x(t)<x_0+tf'(a)$ on $(0,t_0)$.
Since $x(0)=x_0$, then $x(t)-tf'(u(x(t){+},t))\geq x_0$. Since, for any $t_1\in (0,t_0)$,
the shock-free characteristic $x=x(t_1)+(t-t_1)f'(u(x(t_1){+},t_1))$
cannot intersect the shock-free characteristic $x=x_0+tf'(a)$ on $(0,t_1)$, then
$$x(t_1)-t_1f'(u(x(t_1){+},t_1))=x_0,$$
which implies that $u(x(t_1){+},t_1)\in \mathcal{C}(x_0)$.
Since $x(t)<x_0+tf'(a)$ on $(0,t_0)$,
then
$$u(x(t_1){+},t_1)<a=\overline{{\rm D}}_- \Phi(x_0),$$
which contradicts to Proposition $\ref{pro:c4.1}$.

\smallskip
\noindent
{\bf 3}.
For the remaining four cases, by Propositions $\ref{pro:c4.2}$--$\ref{pro:c4.4}$,
it suffices to prove that, when $a=\overline{{\rm D}}_- \Phi(x_0)<\underline{{\rm D}}_+ \Phi(x_0)=b$,
$a \notin \mathcal{C}(x_0)$$(${\it resp.},
$b \notin \mathcal{C}(x_0)$$)$ if and only if
there exists a shock emitting from $x_0$
and lying on the left$(${\it resp.}, right$)$ of the rarefaction wave emitting from $x_0$.
In fact, if there exists a shock emitting from $x_0$
and lying on the left of the rarefaction wave emitting from $x_0$,
then $a \notin \mathcal{C}(x_0)$.
If $a \in \mathcal{C}(x_0)$, by the same arguments as (ii),
there exists no shock emitting from $x_0$
and lying on the left of the characteristic: $x=x_0+tf'(a)$.
\end{proof}

Applying the solution formula $\eqref{c2.50}$ for the Cauchy problem $\eqref{c2.59}$
with the initial data function $u(x,\tau)$
as shown in $\eqref{c2.57}$--$\eqref{c2.58}$.
We define $\mathcal{C}_\tau(x)$ similar to $\mathcal{C}(x_0)$
as in $\eqref{c4.0}$,
and $\Phi_\tau(x):=\int_0^xu(\xi,\tau)\,{\rm d}\xi$
similar to $\Phi(x)=\int_0^x\varphi(\xi)\,{\rm d}\xi$
on the real line $t=\tau$ for any $\tau>0$.

Since $u(x{-},\tau)$ and $u(x{+},\tau)$ always exist, then
\begin{equation*}
\overline{{\rm D}}_- \Phi_{\tau}(x)=u(x{-},\tau)
\geq u(x{+},\tau)=\underline{{\rm D}}_+ \Phi_{\tau}(x).
\end{equation*}

Combining Theorem $\ref{the:c4.1}$ with Corollary $\ref{cor:c2.2}$, 
the initial waves for the Cauchy problem emitting the line $t=\tau$ for any $\tau>0$ are as follows:

\begin{Cor}\label{cor:c4.1}
The initial wave for the Cauchy problem emitting from any point $(x,t)$ with $t>0$
must be one of the following two waves{\rm :}
\begin{itemize}
\item [(i)] If $u(x{-},t)>u(x{+},t)$, or $u(x{-},t)= u(x{+},t)$ with $\mathcal{C}_t(x)=\varnothing$,
then the initial wave for the Cauchy problem emitting from point $(x,t)$ is a shock $\mathcal{S}${\rm ;}

\item [(ii)] If $u(x{-},t)= u(x{+},t)$ with $\mathcal{C}_t(x)=\{u(x{-},t)\}$,
then the initial wave for the Cauchy problem emitting from point $(x,t)$ is a characteristic. 
\end{itemize}
\end{Cor}

\smallskip
\subsection{Lifespan of characteristics}
In this subsection, we introduce a upper bound of lifespan and the lifespan of characteristics.

For any $x_0\in \mathbb{R}$ and $c\in \mathcal{C}(x_0)$,
we define a upper bound of lifespan
$t_p:=t_p(x_0,c)$ of characteristic $L_c(x_0)$ defined by $\eqref{c4.1c}$ as follows:
\begin{align}\label{plc1}
t_p=t_p(x_0,c):=\min\big\{t_p^\pm(x_0,c)\big\},
\end{align}
where $t_p^\pm:=t_p^\pm(x_0,c)$ are defined by
\begin{align}\label{plc2}
t_p^\pm=t_p^\pm(x_0,c):=\sup\big\{t\in \mathbb{R}^+\,:\,\Phi(l;x_0,c)>F(l;t,c)
\quad\,\,\, {\rm for\ small}\ \pm l>0\big\}
\end{align}
with  $\Phi(l;x_0,c)$ and $F(l;t,c)$ given by $\eqref{c4.2b}$, and $\pm l$ as the symbol representing $+l$ and $-l$ together 
from now on.

\smallskip
Since $c\in \mathcal{C}(x_0)$,
combining $\eqref{plc1}$--$\eqref{plc2}$ with Lemma $\ref{lem:c4.1}$ yields that
$t_p^\pm>0$, $i.e.$, $t_p^\pm \in (0,\infty]$,
which implies that $t_p \in (0,\infty]$.

\smallskip
From $\eqref{c4.2b}$, $F(l;t,c)<0$ for any $l\neq 0$ and $t>0$.
Therefore, if $\Phi(l;x_0,c)\geq 0$ for small $\pm l>0$,
then $\Phi(l;x_0,c)>F(l;t,c)$ for small $\pm l>0$ and $t>0$,
which means that $t_p^\pm=\infty$.
Otherwise, $t_p^\pm$ belong to $(0,\infty)$ or equal to $\infty$.

\begin{Lem}\label{lem:plc0}
For characteristic $L_c(x_0)$
defined by $\eqref{c4.1c}$
with $c\in \mathcal{C}(x_0)$,
if $t_p^\pm\in(0,\infty)$, then
\begin{itemize}
\item [(i)] For $t<t_p^\pm$,
\begin{align}\label{plc2c}
\Phi(l;x_0,c)>F(l;t,c) \qquad {\rm for\ small}\ \pm l>0.
\end{align}

\item [(ii)] For $t=t_p^\pm$,
\begin{align}\label{plc2a}
\Phi(l;x_0,c)\geq F(l;t_p^\pm,c)\qquad {\rm for\ small}\ \pm l>0.
\end{align}

\item [(iii)] For $t>t_p^+\, (\mbox{or}\,\,\, t_p^-)$,
there exists 
a sequence 
$l_n^+>0\, (\mbox{or}\,\,\,l_n^-<0)$ with
$\lim_{n\rightarrow\infty} l_n^\pm=0$ such that
\begin{align}\label{plc2b}
\Phi(l_n^\pm;x_0,c)<F(l_n^\pm;t,c).
\end{align}
\end{itemize}

\noindent
In the above, $\Phi(l;x_0,c)$ and $F(l;t,c)$ are given by $\eqref{c4.2b}$.
\end{Lem}

\begin{proof} We divide the proof into three steps correspondingly.

\smallskip
\noindent
{\bf 1}. For $t<t_p^\pm$, $\eqref{plc2c}$ is directly implied by the definition of $t_p^\pm$ in $\eqref{plc2}$.

\smallskip
\noindent
{\bf 2}. For $t=t_p^\pm$, from the continuity of $F(l;\cdot,c)$ inferred by $\eqref{c4.2e}$,
letting $t\rightarrow (t_p^\pm)-$ in $\eqref{plc2c}$, we obtain $\eqref{plc2a}$.

\smallskip
\noindent
{\bf 3}. For $t>t_p^\pm$, we prove $\eqref{plc2b}$ by contradiction.
In fact, if $\eqref{plc2b}$ does not hold, there exists $l_0>0$ such that
$\Phi(l;x_0,c)\geq F(l;t,c)$ for $\pm l\in (0,l_0).$
$$\Phi(l;x_0,c)\geq F(l;t,c)\qquad {\rm for}\ \pm l\in (0,l_0).$$
Since $F(l;\cdot,c)$ is strictly increasing which is inferred by $\eqref{c4.2e}$,
then, for any $t'\in(t_p^\pm,t)$,
\begin{align*}
\Phi(l;x_0,c)\geq F(l;t,c)>F(l;t',c)\qquad {\rm for }\ \pm l\in (0,l_0),
\end{align*}
which contradicts to the definition of $t_p^\pm$ in $\eqref{plc2}$.
\end{proof}

In order to determine the exact values of $t_p^\pm$, we assume that
\begin{align}\label{plc3}
\begin{cases}
f''(u)=(N_\pm+o(1))|u-c|^{\alpha_\pm}\qquad &{\rm as}\ u\rightarrow c{\pm},\\[2mm]
\varphi(x_0+l)-c=(C_{\gamma_\pm}+o(1)){\rm sgn}(l) |l|^{\gamma_\pm}\qquad &{\rm as}\ l\rightarrow 0{\pm},
\end{cases}
\end{align}
for some constants $\alpha_\pm\geq 0, \gamma_\pm>0$, $N_\pm>0$, and $C_{\gamma_\pm}\neq 0$.

\medskip
For characteristic $L_c(x_0)$ with $c\in \mathcal{C}(x_0)$,
the exact values of $t_p^\pm$ can be determined as follows:

\begin{The}[A upper bound of lifespan of characteristics]\label{lem:plc1}
Suppose that $\mathcal{C}(x_0)\neq\varnothing$.
For any characteristic $L_c(x_0)$ defined by $\eqref{c4.1c}$ with $c\in \mathcal{C}(x_0)$,
$t_p^\pm=t_p^\pm(x_0,c)>0$.  There are two cases{\rm :}
\begin{itemize}
\item [(i)] Case{\rm :} $\overline{{\rm D}}_- \Phi(x_0)<\underline{{\rm D}}_+ \Phi(x_0)$.
\begin{itemize}
\item [(a)] If $c\in (\overline{{\rm D}}_- \Phi(x_0),\,\underline{{\rm D}}_+ \Phi(x_0))$, then $t_p=t_p^\pm=\infty$.

\item [(b)] If $c=\overline{{\rm D}}_- \Phi(x_0)\in \mathcal{C}(x_0)$, then $t_p^+=\infty$.
Assume that $\eqref{plc3}$ holds for some constants $\alpha_+\geq 0, \gamma_->0$,
$N_+>0$, and $C_{\gamma_-}\neq 0$,
then $t^-_p\in(0,\infty)$ if and only if
\begin{align}\label{plc4a}
C_{\gamma_-}<0\qquad {\rm and}\quad \gamma_-(1+\alpha_+)=1.
\end{align}
Furthermore, if $t^-_p\in(0,\infty)$, then
\begin{align}\label{plc4b}
t^-_p=\big(\gamma_- N_+|C_{\gamma_-}|^{1+\alpha_+}\big)^{-1},
\end{align}
and, for any $t>t_p^-$, $\Phi(l;x_0,c)<F(l;t,c)$ for small $l<0$.

\item [(c)] If $c=\underline{{\rm D}}_+ \Phi(x_0)\in \mathcal{C}(x_0)$, then $t_p^-=\infty$.
Assume that $\eqref{plc3}$ holds for some constants $\alpha_-\geq 0, \gamma_+>0$,
$N_->0$, and $C_{\gamma_+}\neq 0$,
then $t^+_p\in(0,\infty)$ if and only if
\begin{align}\label{plc5a}
C_{\gamma_+}<0\qquad {\rm and}\quad \gamma_+(1+\alpha_-)=1.
\end{align}
Furthermore, if $t^+_p\in(0,\infty)$, then
\begin{align}\label{plc5b}
t^+_p=\big(\gamma_+ N_-|C_{\gamma_+}|^{1+\alpha_-}\big)^{-1},
\end{align}
and, for any $t>t_p^+$, $\Phi(l;x_0,c)<F(l;t,c)$ for small $l>0$.
\end{itemize}

\item [(ii)] Case{\rm :} $\overline{{\rm D}}_- \Phi(x_0)=\underline{{\rm D}}_+ \Phi(x_0)$.
Assume that $\eqref{plc3}$ holds,
then $t^\pm_p\in(0,\infty)$ if and only if
\begin{align}\label{plc6a}
C_{\gamma_\pm}<0\qquad {\rm and}\quad \gamma_\pm(1+\alpha_\mp)=1.
\end{align}
Furthermore, if $t^\pm_p\in(0,\infty)$, then
\begin{align}\label{plc6b}
t^\pm_p=\big(\gamma_\pm N_\mp|C_{\gamma_\pm}|^{1+\alpha_\mp}\big)^{-1},
\end{align}
and for any $t>t_p^\pm$, $\Phi(l;x_0,c)<F(l;t,c)$ for small $\pm l>0$.
\end{itemize}
\noindent
In the above,
$\mathcal{C}(x_0)$, $\overline{{\rm D}}_- \Phi(x_0)$,
and $\underline{{\rm D}}_+ \Phi(x_0)$
are given by $\eqref{c4.0}$--$\eqref{c4.1}${\rm ,}
$\Phi(l;x_0,c)$ and $F(l;t,c)$ are given by $\eqref{c4.2b}${\rm ,}
and $t_p$ with $t_p^{\pm}$ is given by $\eqref{plc1}$--$\eqref{plc2}$.
See also {\rm Fig. $\ref{figPls}$}.
\end{The}

\begin{figure}[H]
	\begin{center}
		{\includegraphics[width=0.65\columnwidth]{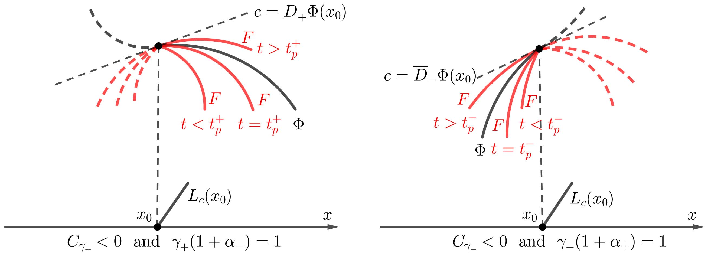}}
 \caption{The positional relations of $\Phi(\cdot\,;x_0,c)$
 and $F(\cdot\,;t,c)$ for $t_p^\pm\in (0,\infty)$.}\label{figPls}
	\end{center}
\end{figure}

\begin{proof}
Suppose that $\mathcal{C}(x_0)\neq\varnothing$.
For any characteristic $L_c(x_0)$ given by $\eqref{c4.1c}$ with $c\in \mathcal{C}(x_0)$,
it is direct to see from $\eqref{plc2}$ and Lemma $\ref{lem:c4.1}$ that $t_p^\pm=t_p^\pm(x_0,c)>0$.

\smallskip
\noindent
{\bf 1}. For the case that $\overline{{\rm D}}_- \Phi(x_0)<\underline{{\rm D}}_+ \Phi(x_0)$,
denote $a:=\overline{{\rm D}}_- \Phi(x_0)$ and $b:=\underline{{\rm D}}_+ \Phi(x_0)$.

\smallskip
{\bf (a).} If $c\in (\overline{{\rm D}}_- \Phi(x_0),\,\underline{{\rm D}}_+ \Phi(x_0))$,
since $a<b$, from $\eqref{c4.1}$,
for
$0<\varepsilon <\min\{b-c,c-a\}$,
$\mathop{\underline{\lim}}_{l\rightarrow 0{+}}
\frac{1}{l}\int^{x_0+l}_{x_0}(\varphi(\xi)-c)\,{\rm d}\xi>\varepsilon$
and
$\mathop{\overline{\lim}}_{l\rightarrow 0{-}}
\frac{1}{l}\int^{x_0+l}_{x_0}(\varphi(\xi)-c)\,{\rm d}\xi<{-}\varepsilon$.
Thus, for any $t>0$,
\begin{align*}
\Phi(l;x_0,c)=\int^{x_0+l}_{x_0}(\varphi(\xi)-c)\,{\rm d}\xi>0>F(l;t,c)
\qquad {\rm for\ small}\ \pm l>0,
\end{align*}
which implies that $t_p=t_p^\pm=\infty$.

\smallskip
{\bf (b).} If $c=\overline{{\rm D}}_- \Phi(x_0)\in \mathcal{C}(x_0)$,
it follows from $\eqref{c4.12a}$ that $t_p^+=\infty$.
For the case of $t^-_p$,
similar to $\eqref{c4.13}$--$\eqref{c4.13d1}$,
it is direct to check that, for small $l<0$,
\begin{align}\label{plc9}
\begin{cases}
\Phi(l;x_0,c)=(1+o(1))\frac{C_{\gamma_-}}{1+\gamma_-}|l|^{1+\gamma_-},\\[2mm]
F(l;t,c)=-(1+o(1))\frac{1+\alpha_+}{2+\alpha_+}
\Big(\frac{1+\alpha_+}{N_+}\Big)^{\frac{1}{1+\alpha_+}}\,t^{-\frac{1}{1+\alpha_+}}
\,|l|^{1+\frac{1}{1+\alpha_+}}.
\end{cases}
\end{align}

From $\eqref{plc9}$, when $C_{\gamma_-}>0$,
then $\Phi(l;x_0,c)>0$ for small $l<0$ so that, for any $t>0$,
\begin{align}\label{plc9a}
\Phi(l;x_0,c)>F(l;t,c) \qquad {\rm for\ small}\ l<0,
\end{align}
which implies that $t_p^-=\infty$.
When $C_{\gamma_-}<0$, from $\eqref{c4.7c}$,
$c\in \mathcal{C}(x_0)$ implies that $\gamma_-(1+\alpha_+)\geq 1$.
If $\gamma_-(1+\alpha_+)> 1$, by $\eqref{plc9}$,
$\eqref{plc9a}$ holds for any $t>0$ so that $t_p^-=\infty$;
and if $\gamma_-(1+\alpha_+)= 1$, $\eqref{plc9a}$ is equivalent to that,
for small $l<0$,
\begin{align*}
t^{\frac{1}{1+\alpha_+}}
<(1+o(1))\frac{1+\gamma_-}{|C_{\gamma_-}|}\frac{1+\alpha_+}{2+\alpha_+}
\Big(\frac{1+\alpha_+}{N_+}\Big)^{\frac{1}{1+\alpha_+}}
=(1+o(1))\big(\gamma_- N_+|C_{\gamma_-}|^{1+\alpha_+}\big)^{-\frac{1}{1+\alpha_+}},
\end{align*}
which, by letting $l \rightarrow 0{-}$, implies $\eqref{plc4b}$.

\smallskip
Therefore, $t^-_p\in(0,\infty)$ if and only if $\eqref{plc4a}$ holds.
It follows from $\eqref{plc9}$ that,
for any $t>t_p^-$, $\Phi(l;x_0,c)<F(l;t,c)$ for small $l<0.$

\smallskip
{\bf (c).} By the same argument as {\bf (b)}, it can be checked that $\eqref{plc5a}$--$\eqref{plc5b}$ hold.

\smallskip
\noindent
{\bf 2}. $\eqref{plc6a}$--$\eqref{plc6b}$ follow immediately
by letting $\overline{{\rm D}}_- \Phi(x_0)=\underline{{\rm D}}_+ \Phi(x_0)$ in {\bf (b)} and {\bf (c)}.
\end{proof}

From $\eqref{c4.2}$--$\eqref{c4.2b}$, for any $x_0\in \mathbb{R}$ and $c\in \mathcal{C}(x_0)$,
the lifespan $t_*:=t_*(x_0,c)$ of characteristic $L_c(x_0)$ defined by $\eqref{c4.1c}$ is given by
\begin{align}\label{elc1}
t_*=t_*(x_0,c)&:=\sup\big\{t\in \mathbb{R}^+\,:\,c\in \mathcal{U}(x,t)
\quad\,\,\, {\rm for}\ (x,t)\in L_c(x_0)\big\}\nonumber\\[1mm]
&\;=\sup\big\{t\in \mathbb{R}^+\,:\,\Phi(l;x_0,c)\geq F(l;t,c)
\quad\,\,\, {\rm for\ all}\,\, l=l(u;t,c) \big\},
\end{align}
where $l(u;t,c)=t(f'(c)-f'(u))$ for $u\in \mathbb{R}$,
and the $\Phi(l;x_0,c)$ and $F(l;t,c)$ are given by $\eqref{c4.2b}$.

For any characteristic $L_c(x_0)$, since $c\in \mathcal{C}(x_0)$,
it follows from the definitions of $\mathcal{C}(x_0)$ in $\eqref{c4.0}$
and $t_*$ in $\eqref{elc1}$ that $t_*>0$.

\begin{The}[Lifespan of characteristics]\label{pro:elc1}
Suppose that $\mathcal{C}(x_0)\neq \varnothing$. If $c\in \mathcal{C}(x_0)$,
then $t_*=t_*(x_0,c)$ given by $\eqref{elc1}$
is the lifespan of characteristic $L_c(x_0)$
given by
$\eqref{c4.1c}$
in the following sense{\rm :} 
\begin{itemize}
\item [(i)] If $t_*<\infty$,
then, for any $(x,t)\in L_c(x_0)$,
\begin{align}\label{elc2a}
\{c\}= \mathcal{U}(x,t) \,\,\,{\rm if}\ t< t_*, \quad\,
\{c\}\subset\mathcal{U}(x_*,t_*) \,\,\,{\rm if}\ t=t_*,
\quad\, c\notin \mathcal{U}(x,t) \,\,\,{\rm if}\ t> t_*,
\end{align}
where $x_*:=x_0+t_*f'(c)$.

\smallskip
\item [(ii)] If $t_*=\infty$, then
\begin{align}\label{elc2b}
\{c\}= \mathcal{U}(x,t) \qquad \mbox{{\rm for\ any} $(x,t)\in L_c(x_0)$ {\rm with}\ $t>0$}.
\end{align}
\end{itemize}
\end{The}

\begin{proof}
From $\eqref{uiff}$, if $c\in \mathcal{U}(x_1,t_1)$ for some $t_1>0$,
then $\{c\}=\mathcal{U}(x,t)$ for any $(x,t)\in L_c(x_0)$ with $t\in (0,t_1)$ so that,
the set
\begin{align*}
T(x_0,c):=\big\{t\in \mathbb{R}^+\,:\,c\in \mathcal{U}(x,t)\quad {\rm for}\ (x,t)\in L_c(x_0)\big\}
\end{align*}
is an interval with $0=\inf T(x_0,c)$.

\smallskip
\noindent
{\bf 1}. For $t_*<\infty$, since $T(x_0,c)$ is an interval with $0=\inf T(x_0,c)$ and $t_*=\sup T(x_0,c)$,
then $(0,t_*)\subset T(x_0,c)\subset (0,t_*]$,
which, by $\eqref{c2.39}$, implies that $c=\lim_{t\rightarrow t_*{-}}u(x,t)\in \mathcal{U}(x_*,t_*)$,
where the limit is taken along line $L_c(x_0)$.
Then, by $\eqref{uiff}$, $\{c\}=\mathcal{U}(x,t)$ for any $(x,t)\in L_c(x_0)$ with $t\in (0,t_*)$.
Thus, $\eqref{elc2a}$ holds for $t\leq t_*$.
For $t>t_*$, $c\notin \mathcal{U}(x,t)$ follows immediately from the definition of $t_*$ in $\eqref{elc1}$.

\smallskip
\noindent
{\bf 2}. For $t_*=\infty$, since $T(x_0,c)$ is an interval with $0=\inf T(x_0,c)$ and $t_*=\sup T(x_0,c)$,
then $T(x_0,c)=(0,\infty)$,
which means that $c\in \mathcal{U}(x,t)$ for any $(x,t)\in L_c(x_0)$ with $t>0$.
We prove $\eqref{elc2b}$ by contradiction.
Otherwise, if there exists $(x_1,t_1)\in L_c(x_0)$ such that $\{c\}\subsetneqq \mathcal{U}(x_1,t_1)$, $i.e.$,
there exists $u_1\in \mathcal{U}(x_1,t_1)$ with $u_1\neq c$ satisfying
\begin{align}\label{elc2d}
\Phi(l_1;x_0,c)=F(l_1;t_1,c)\qquad {\rm with}\ l_1:=l(u_1;t_1,c).
\end{align}
Since $l_1=t_1(f'(c)-f'(u_1))$ and $f'(u)$ is strictly increasing,
then, for any $t>t_1$, there exists $\hat{u}_1$ lying between $c$ and $u_1$
such that $l_1=t(f'(c)-f'(\hat{u}_1))$. From $\eqref{c4.2e}$,
$F(l_1;\cdot,c)$ is strictly increasing, then, by $\eqref{elc2d}$,
$\Phi(l_1;x_0,c)=F(l_1;t_1,c)<F(l_1;t,c)$ for any $t>t_1,$
which implies that $c\notin \mathcal{U}(x,t)$ for any $(x,t)\in L_c(x_0)$ with $t>t_1$.
This contradicts to $c\in \mathcal{U}(x,t)$ for any $(x,t)\in L_c(x_0)$ with $t>0$.
\end{proof}

\subsection{Termination of characteristics}
In this subsection, using the upper bound of lifespan and the lifespan,
we provide a classification of characteristics by the way of their terminations.

Combining $\eqref{elc1}$ with $\eqref{plc1}$--$\eqref{plc2}$,
it is direct to see that
\begin{align}\label{elc3}
t_*=t_*(x_0,c)\leq t_p(x_0,c)=\min\big\{t^\pm_p(x_0,c)\big\}=\min\big\{t_p^\pm\big\}.
\end{align}
Thus, if $t_p=\infty$, then $t_*<t_p=\infty$ or $t_*=t_p=\infty$;
and if $t_p\in (0,\infty)$,
then $(x_p,t_p)\in L_c(x_0)$ with $x_p:=x_0+t_pf'(c)$ and
\begin{itemize}
\item [(a)] If $c\notin \mathcal{U}(x_p,t_p)$, $t_*<t_p$ by $\eqref{elc2a}$.

\vspace{1pt}
\item [(b)] If $c\in \mathcal{U}(x_p,t_p)$,  $t_*=t_p$ by $\eqref{elc2a}$ and $\eqref{elc3}$.
There exist two cases:
\begin{align*}
\{c\}=\mathcal{U}(x_p,t_p)\qquad {\rm or} \qquad \{c\}\subsetneqq\mathcal{U}(x_p,t_p).
\end{align*}
\end{itemize}

Finally, we present a classification of characteristics by the way of their terminations.

\begin{The}\label{the:elc0}
Any characteristic $L_c(x_0)$ defined by $\eqref{c4.1c}$
emitting from point $x_0$ on the $x$--axis
must be one of the following four cases{\rm :}
\begin{itemize}
\item [(i)] For $t_*=t_p\in (0,\infty)$ with $\{c\}=\mathcal{U}(x_p,t_p)$,
characteristic $L_c(x_0)$ terminates at $(x_p,t_p)$
only due to the compression of local characteristics near $x_0$,
and $(x_p,t_p)$ is a continuous shock generation point.

\item [(ii)] For $t_*=t_p\in (0,\infty)$ with $\{c\}\subsetneqq\mathcal{U}(x_p,t_p)$,
characteristic $L_c(x_0)$ terminates at $(x_p,t_p)$
due to the combination of the compression of local characteristics near $x_0$ for parts
with some characteristics away from $x_0$,
and $(x_p,t_p)$ is a discontinuous shock generation point or a point of shocks.

\item [(iii)] For $t_*<t_p\in (0,\infty]$,
characteristic $L_c(x_0)$ terminates at $(x_*,t_*)$
due to the collision of $L_c(x_0)$ with some characteristics
away from $x_0$, and $(x_*,t_*)$ is a point of a shock.

\item [(iv)] For $t_*=t_p=\infty$, characteristic $L_c(x_0)$ exists for all time.
\end{itemize}
\noindent
In the above, $t_p$ and $t_*$ are defined by
$\eqref{plc1}$ and $\eqref{elc1}$, respectively.
See {\rm Fig. $\ref{figCCP}$} for the details.
\end{The}

\begin{figure}[H]
	\begin{center}
		{\includegraphics[width=0.65\columnwidth]{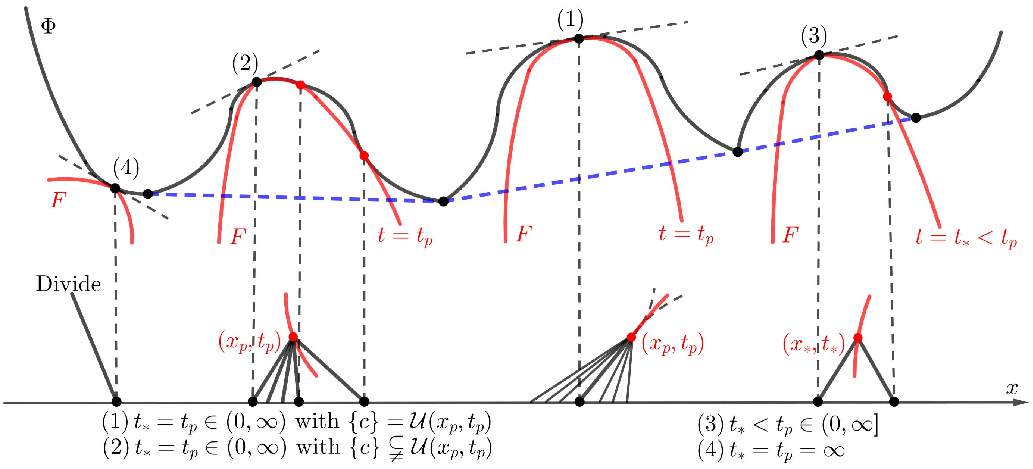}}
 \caption{Four types of characteristics classified by $t_*$ and $t_p$.}\label{figCCP}
	\end{center}
\end{figure}

\begin{proof} We divide the proof into four steps correspondingly.

\smallskip
\noindent
{\bf 1}. For $t_*=t_p\in (0,\infty)$ with $\{c\}=\mathcal{U}(x_p,t_p)$,
it follows from $\eqref{elc2a}$ that $L_c(x_0)$ is shock-free for $t\leq t_p$.
$\{c\}=\mathcal{U}(x_p,t_p)$ implies that solution $u=u(x,t)$ is continuous at $(x_p,t_p)$.
By Corollary $\ref{cor:c4.1}$(i), to prove that $(x_p,t_p)$ is a continuous shock generation point,
it suffices to show that $\mathcal{C}_{t_p}(x_p)=\varnothing$
or equivalently $c\notin \mathcal{C}_{t_p}(x_p)$.
Otherwise, $c\in \mathcal{C}_{t_p}(x_p)$ so that
there exists $(x,t)\in L_c(x_0)$ with $t>t_p$ such that $c\in \mathcal{U}(x,t;t_p)$,
which, by Corollary $\ref{cor:c2.2}$, implies that $c\in \mathcal{U}(x,t)$.
This contradicts to $\eqref{elc2a}$ that $c\notin \mathcal{U}(x,t)$ if $t>t_*=t_p$.

\smallskip
\noindent
{\bf 2}. For $t_*=t_p\in (0,\infty)$ with $\{c\}\subsetneqq\mathcal{U}(x_p,t_p)$,
it follows from $\eqref{elc2a}$ that $L_c(x_0)$ is shock-free for $t< t_p$.
$\{c\}\subsetneqq\mathcal{U}(x_p,t_p)$ implies that solution $u=u(x,t)$ is discontinuous at $(x_p,t_p)$.
Furthermore, by Lemma $\ref{lem:c2.5}$(iii),
if $\mathcal{U}(x_p,t_p)=[u(x_p+,t_p),u(x_p-,t_p)]$,
then point $(x_p,t_p)$ is a discontinuous shock generation point;
and if $\mathcal{U}(x_p,t_p)\subsetneqq[u(x_p+,t_p),u(x_p-,t_p)]$,
then point $(x_p,t_p)$ is a point of shocks.

\smallskip
\noindent
{\bf 3}. For $t_*<t_p\in (0,\infty]$,
it follows from $\eqref{elc2a}$ that $L_c(x_0)$ is shock-free for $t< t_*$.
It suffices to show that $\{c\}\subsetneqq\mathcal{U}(x_*,t_*)$,
which is proved by contradiction.
Otherwise, assume that $\{c\}=\mathcal{U}(x_*,t_*)$.
From $\eqref{elc2a}$, $c\notin \mathcal{U}(x,t)$ if $t>t_*$,
then there exists $u(t)\neq c$ such that
$u(t)\in \mathcal{U}(x,t)$ for any $(x,t)\in L_c(x_0)$ with $t>t_*$.
From $\eqref{c2.39}$, it follows from $\{c\}=\mathcal{U}(x_*,t_*)$ that
$\lim_{t\rightarrow t_*{+}}u(t)=c$, then
\begin{align}\label{elc3e}
\lim_{t\rightarrow t_*{+}}l(t):=\lim_{t\rightarrow t_*{+}}t\big(f'(c)-f'(u(t))\big)=0.
\end{align}

On the other hand, from $\eqref{plc2c}$, for any $t_1\in [t_*,t_p)$,
there exists small $l_1>0$ such that
$\Phi(l;x_0,c)>F(l;t_1,c)$ for $l\in ({-}l_1,0)\cup (0,l_1)$.
This, by $\eqref{c4.2e}$, implies that, for any $t\in [t_*,t_1)$,
\begin{align}\label{elc3d}
\Phi(l;x_0,c)>F(l;t_1,c)>F(l;t,c) \qquad {\rm for}\ l\in ({-}l_1,0)\cup (0,l_1).
\end{align}
Since $c\notin \mathcal{U}(x,t)$ and $u(t)\in \mathcal{U}(x,t)$,
then $E(c;x,t)-E(u(t);x,t)<0$,
which, by $\eqref{c4.2b}$ and $\eqref{elc3d}$, implies that
$l(t)=t(f'(c)-f'(u(t)))\notin ({-}l_1,l_1).$
This contradicts to $\eqref{elc3e}$.

\smallskip
\noindent
{\bf 4}. It follows from $\eqref{elc2b}$ immediately.
\end{proof}

\begin{Rem}
From {\rm Proposition} $\ref{pro:c4.3}$--$\ref{pro:c4.4}$,
the necessary and sufficient conditions for the rays emitting from points on the $x$--axis to be characteristics,
at least local-in-time, hold for all the points,
{\it even when the point on the $x$--axis is not a Lebesgue point of the initial data function.
Owing to the completeness of such results,
without {\rm a priori} assumptions on the continuity, or the traces,
or the property of the approximating jump of the initial data function,
we can obtain the criteria for all the six types of initial waves for the Cauchy problem
with the flux functions satisfying $\eqref{c1.2}$ and the initial data only in $L^\infty$}.
\end{Rem}

\section{New Features of Formation and Development of Shocks}
In this section, we study the formation and development of shocks.
In \S 4.1, we discover all the five types of continuous shock generation points
and prove the necessary and sufficient conditions of them in Theorem $\ref{the:c4.3}$;
in \S 4.2, we prove the optimal regularities of shock curves and entropy solutions
near all the five types of continuous shock generation points in Theorems $\ref{the:dsw1}$,
in which the continuous shock generation points are not necessarily assumed to be isolated.

\subsection{Formation of shocks
for all the five types of continuous shock generation points}
By Theorem $\ref{the:elc0}$,
a continuous shock generation point $(x,t)$
must be a point $(x_p,t_p)$ of some characteristic $L_c(x_0)$, $i.e.$,
\begin{align}\label{fsw1}
x=x_p=x_0+t_pf'(c),\qquad t=t_p\in(0,\infty).
\end{align}
Moreover,
point $(x_p,t_p)$ is a continuous shock generation point if and only if
$t_*=t_p\in (0,\infty)$ and $\{c\}=\mathcal{U}(x_p,t_p),$
which, by $t_*\leq t_p$ in $\eqref{elc3}$
and the fact that $\{c\}=\mathcal{U}(x_p,t_p)$ implying $t_*\geq t_p$,
is equivalent to
\begin{align}\label{fsw2a}
t_p\in (0,\infty) \qquad {\rm and}\qquad \{c\}=\mathcal{U}(x_p,t_p).
\end{align}
Furthermore, from Theorem $\ref{lem:plc1}$,
the necessary condition of $t_p\in (0,\infty)$ of characteristic $L_c(x_0)$ is
$c\in \partial \mathcal{C}(x_0)$,
where $\partial \mathcal{C}(x_0)$ is the boundary of $\mathcal{C}(x_0)$.

As shown in Theorem $\ref{lem:plc1}$, in order to determine the exact values of $t_p^\pm$,
we need the information of $f''(u)$ near $c$ and $\varphi(x)$ near $x_0$. Assume that
\begin{align}\label{fsw3}
\begin{cases}
f''(u)=(N_\pm+o(1))|u-c|^{\alpha_\pm}\qquad &{\rm as}\ u\rightarrow c{\pm},\\[1mm]
\varphi(x_0+l)-c=(C_{\gamma_\pm}+o(1)){\rm sgn}(l) |l|^{\gamma_\pm}\qquad &{\rm as}\ l\rightarrow 0{\pm},
\end{cases}
\end{align}
for some constants $\alpha_\pm\geq 0, \gamma_\pm>0$, $N_\pm>0$, and $C_{\gamma_\pm}\neq 0$.

\begin{The}[Formation of the shocks with continuous shock generation point]\label{the:c4.3}
A continuous shock generation point $(x,t)\in \mathbb{R}\times \mathbb{R}^+$
must be a point $(x_p,t_p)$ as in $\eqref{fsw1}$ of some characteristic $L_c(x_0)$
defined by $\eqref{c4.1c}$. 
Moreover, point $(x_p,t_p)$
is a continuous shock generation point if and only if
$t_p\in(0,\infty)$ and
\begin{align}\label{fsw4}
\Phi(x_0+l)-\Phi(x_0)-c\,l>\big(\rho((f')^{-1}(f'(c)-l/t_p),c)-c\big)\,l
\qquad {\rm for\ all}\ l\neq 0,
\end{align}
where $l=t_p(f'(c)-f'(u))$ with $u\in \mathbb{R}$,
$\Phi(x)=\int^x_0\varphi(\xi)\,{\rm d}\xi$,
and $\rho(u,c)$ is given by $\eqref{aa5a}$.
Furthermore, if
$\eqref{fsw3}$ is assumed,
then the following statements hold{\rm :}
\vspace{3pt}
\begin{itemize}
\item [(i)] For $\overline{{\rm D}}_- \Phi(x_0)<\underline{{\rm D}}_+ \Phi(x_0)$,
there exist two cases{\rm :}
\vspace{3pt}
\begin{itemize}
\item [(a)] If $c=\overline{{\rm D}}_- \Phi(x_0)\in \mathcal{C}(x_0)$,
then $t_p\in(0,\infty)$ if and only if
\begin{align}\label{fsw4a}
C_{\gamma_-}<0\qquad {\rm and}\qquad \gamma_-(1+\alpha_+)=1,
\end{align}
and, if $t_p\in(0,\infty)$, then
$t_p=t^-_p=(\gamma_- N_+|C_{\gamma_-}|^{1+\alpha_+})^{-1}.$

\vspace{3pt}
\item [(b)] If $c=\underline{{\rm D}}_+ \Phi(x_0)\in \mathcal{C}(x_0)$,
then $t_p\in(0,\infty)$ if and only if
\begin{align}\label{fsw5a}
C_{\gamma_+}<0\qquad {\rm and}\qquad \gamma_+(1+\alpha_-)=1,
\end{align}
and, if $t_p\in(0,\infty)$, then
$t_p=t^+_p=(\gamma_+ N_-|C_{\gamma_+}|^{1+\alpha_-})^{-1}.$
\end{itemize}

\vspace{3pt}
\item [(ii)] For $\overline{{\rm D}}_- \Phi(x_0)=\underline{{\rm D}}_+ \Phi(x_0)$,
then $t^\pm_p\in(0,\infty)$ if and only if
\begin{align}\label{fsw6a}
C_{\gamma_\pm}<0\qquad {\rm and}\quad \gamma_\pm(1+\alpha_\mp)=1,
\end{align}
and, if $t^\pm_p\in(0,\infty)$, then
$t^\pm_p=(\gamma_\pm N_\mp|C_{\gamma_\pm}|^{1+\alpha_\mp})^{-1}.$
Furthermore, if $t_p=\min\big\{t^\pm_p\big\}\in (0,\infty)$, then
\begin{align}\label{fsw6c}
t_p=\begin{cases}
\big(\gamma_- N_+|C_{\gamma_-}|^{1+\alpha_+}\big)^{-1}\quad
&{\rm if}\ t^-_p\in (0,\infty)\ {\rm and} t^+_p=\infty,\\[1mm]
\big(\gamma_+ N_-|C_{\gamma_+}|^{1+\alpha_-}\big)^{-1}\quad
&{\rm if}\ t^+_p\in (0,\infty)\ {\rm and}\ t^-_p=\infty,\\[1mm]
\min\big\{\big(\gamma_\pm N_\mp|C_{\gamma_\pm}|^{1+\alpha_\mp}\big)^{-1}\big\} \quad
&{\rm if}\ t^+_p,t^-_p\in (0,\infty).
\end{cases}
\end{align}
\end{itemize}
\noindent
In the above,
$\mathcal{C}(x_0)$, $\overline{{\rm D}}_- \Phi(x_0)$,
and $\underline{{\rm D}}_+ \Phi(x_0)$
are given by $\eqref{c4.0}$--$\eqref{c4.1}$.
See also {\rm Fig.} $\ref{figFCSGP}$.
\end{The}

\begin{figure}[H]
	\begin{center}
		{\includegraphics[width=0.68\columnwidth]{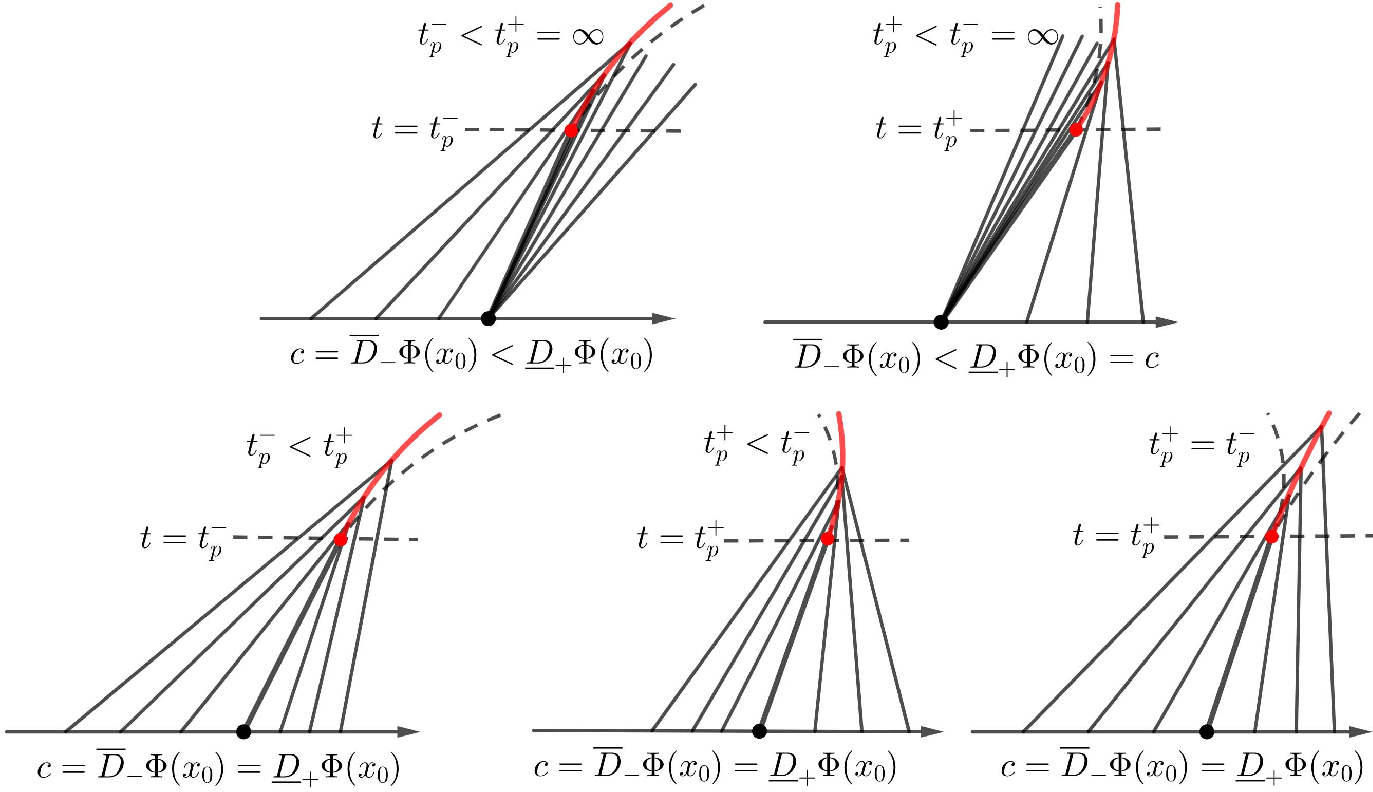}}
 \caption{The five types of continuous shock generation points.}\label{figFCSGP}
	\end{center}
\end{figure}

\begin{proof}
By the arguments in $\eqref{fsw1}$--$\eqref{fsw2a}$,
point $(x_p,t_p)$
is a continuous shock generation point if and only if $\eqref{fsw2a}$ holds.
From the definition of $\mathcal{U}(x,t)$ in $\eqref{c2.3}$,
$\{c\}=\mathcal{U}(x_p,t_p)$ is equivalent to that
$c$ is the unique maximum point of $E(\cdot\,;x_p,t_p)$,
which, by $\eqref{c4.2}$--$\eqref{c4.2b}$, implies that,
for all $l=t_p(f'(c)-f'(u))\neq 0$ with $u\in \mathbb{R}$,
\begin{align*}
\Phi(x_0+l)-\Phi(x_0)-c\,l&
>-t_p\int^u_{c}(s-c)f''(s)\,{\rm d}s
=\frac{\int^u_{c}(s-c)f''(s)\,{\rm d}s}{\int^u_{c}f''(s)\,{\rm d}s}\,l
\nonumber\\[1mm]
&
=\big(\rho(u,c)-c\big)\,l=\big(\rho((f')^{-1}(f'(c)-l/t_p),c)-c\big)\,l.
\end{align*}
From Theorem $\ref{lem:plc1}$,
$t_p\in (0,\infty)$ implies that $c\in \partial \mathcal{C}(x_0)$.
There exist two cases:
\begin{enumerate}
\item[(i)] For $\overline{{\rm D}}_- \Phi(x_0)<\underline{{\rm D}}_+ \Phi(x_0)$,
there exist two subcases:

\vspace{1pt}
\noindent
If $c=\overline{{\rm D}}_- \Phi(x_0)\in \mathcal{C}(x_0)$,
by Theorem $\ref{lem:plc1}$, $t^+_p=\infty$
and $\eqref{fsw4a}$ follows from $\eqref{plc4a}$--$\eqref{plc4b}$.

\vspace{1pt}
\noindent
If $c=\underline{{\rm D}}_+ \Phi(x_0)\in \mathcal{C}(x_0)$,
by Theorem $\ref{lem:plc1}$, $t^-_p=\infty$
and $\eqref{fsw5a}$ follows from $\eqref{plc5a}$--$\eqref{plc5b}$.

\smallskip
\item[(ii)] For $\overline{{\rm D}}_- \Phi(x_0)=\underline{{\rm D}}_+ \Phi(x_0)$,
$\eqref{fsw6a}$--$\eqref{fsw6c}$ follow from $\eqref{plc6a}$--$\eqref{plc6b}$ immediately.
\end{enumerate}

\smallskip
This completes the proof of Theorem $\ref{the:c4.3}$.
\end{proof}

\begin{Cor}[Formation of shocks for the flux functions of uniform convexity]
Suppose that $f''(u)>0$ and $\varphi(x)\in C^1(\mathbb{R})$.
A point $(x,t)\in \mathbb{R}\times\mathbb{R}^+$ is a continuous shock generation point
if and only if there exists $x_0$ on the $x$--axis with $\varphi'(x_0)<0$ such that
\begin{align}\label{fsw7}
\begin{cases}
x=x_p=x_0+t_pf'(\varphi(x_0)),\\[2mm]
t=t_p=\big(f''(\varphi(x_0))|\varphi'(x_0)|\big)^{-1},
\end{cases}
\end{align}
and, for all $l=t_p(f'(\varphi(x_0))-f'(u))\neq 0$ with $u\in \mathbb{R}$,
\begin{align}\label{fsw7a}
\Phi(x_0+l)-\Phi(x_0)-\varphi(x_0)\,l
>-
|\varphi'(x_0)|\,k((f')^{-1}(f'(\varphi(x_0))-l/t_p),\varphi(x_0))\, l^2,
\end{align}
where
$k(u,c)$ is given by
\begin{align*}
k(u,c):=\dfrac{f''(c)\int^1_0f''(c+\theta(u-c))\theta\,{\rm d}\theta}
{\big(\int^1_0f''(c+\theta(u-c))\,{\rm d}\theta\big)^2},
\qquad\,\,\, \lim_{u\rightarrow c}k(u,c)=\frac{1}{2}.
\end{align*}
In particular, if $f(u)=\frac{1}{2}u^2$, then $\eqref{fsw7a}$ is equivalent to
\begin{align*}
\Phi(x_0+l)-\Phi(x_0)-\varphi(x_0)\,l>-\frac{1}{2}|\varphi'(x_0)|\, l^2
\qquad {\rm for\ any}\ l\neq 0.
\end{align*}
\end{Cor}

\begin{proof}
Since $f''(u)>0$ and $\varphi(x)\in C^1(\mathbb{R})$,
for any $x_0$ with $\varphi'(x_0)\neq 0$,
then $\alpha_\pm\equiv 0$, $\gamma_\pm\equiv 1$, $N_{\pm}=f''(\varphi(x_0))>0$,
and $C_{\gamma_\pm}=\varphi'(x_0)\neq 0$.
By $\eqref{fsw6a}$, $t_p\in(0,\infty)$ if and only if $\varphi'(x_0)<0$,
and it follows from $\eqref{fsw1}$ and $\eqref{fsw6c}$ that $\eqref{fsw7}$ holds.
Finally, from $\eqref{fsw4}$ and $\eqref{aa5a}$,
\begin{align*}
\rho((f')^{-1}(f'(c)-l/t_p),c)-c
&=-\frac{l}{t_p}\,\frac{\int^1_0f''(c+\theta(u-c))\theta\,{\rm d}\theta}
{\big(\int^1_0f''(c+\theta(u-c))\,{\rm d}\theta\big)^2}
=-|\varphi'(x_0)|k(u,c)\, l,
\end{align*}
which, by taking $c=\varphi(x_0)$, implies $\eqref{fsw7a}$.

\smallskip
In particular, if $f(u)=\frac{1}{2}u^2$,
then
$f''(u)\equiv 1$ so that
$k(u,c)\equiv \frac{1}{2}$.
\end{proof}

For characteristic $L_c(x_0)$ with $c\in \mathcal{C}(x_0)$,
if $\eqref{fsw3}$ holds,
by Theorem $\ref{the:c4.3}$, $t_p^\pm\in (0,\infty)$ if and only if
$C_{\gamma_\pm}<0$ and $\gamma_\pm(1+\alpha_\mp)=1$.
Thus, from $\eqref{fsw3}$ and $\eqref{aa2}$,
for small $\pm l>0$,
\begin{align*}
f'(\varphi(x_0+l))-f'(c)
&=-{\rm sgn}(l) \frac{N_\mp+o(1)}{1+\alpha_\mp}|\varphi(x_0+l)-c|^{1+\alpha_\mp}\nonumber\\[1mm]
&=-{\rm sgn}(l)(1+o(1))\gamma_\pm N_\mp |C_{\gamma_\pm}|^{1+\alpha_\mp}\, |l|
=-(1+o(1))\overline{C}_{\gamma_\pm}\,l,
\end{align*}
where $\overline{C}_{\gamma_\pm}$ are given by
\begin{align}\label{fsw9a}
\overline{C}_{\gamma_\pm}:=\gamma_\pm N_\mp |C_{\gamma_\pm}|^{1+\alpha_\mp}=(t_p^\pm)^{-1}.
\end{align}

\medskip
We now present a necessary condition of continuous shock generation points,
which relays the further information of $f''(u)$ near $c$ and $\varphi(x)$ near $x_0$ as in $\eqref{fsw3}$.

\begin{Lem}\label{lem:fsw2}
Assume that $\varphi(x_0+l)$ is continuous on $\pm l\in (0,l_0)$ for some $l_0>0$,
and $\eqref{fsw3}$ holds for $c=\underline{{\rm D}}_+ \Phi(x_0)$
or $c=\overline{{\rm D}}_-\Phi(x_0)$, respectively.
Suppose that $(x_p,t_p)$ is a continuous shock generation point with $t_p=t_p^\pm\in (0,\infty)$.
If, for small $\pm l>0$,
\begin{align}\label{fsw10a}
f'(\varphi(x_0+l))-f'(c)=-\overline{C}_{\gamma_\pm}l+\omega(l)l
\end{align}
for $\overline{C}_{\gamma_\pm}$ in $\eqref{fsw9a}$
and a continuous function $\omega(l)$ possessing a unique zero point $\omega(0)=0$, then
\begin{align}\label{fsw10b}
\omega(l)>0 \qquad {\rm for\ small}\ \pm l>0.
\end{align}
\end{Lem}

\begin{proof}
Suppose that point $(x_p,t_p)$ of characteristic $L_c(x_0)$ is a continuous shock generation point
determined by $c=\underline{{\rm D}}_+ \Phi(x_0)$ with $t_p=t_p^+$,
or $c=\overline{{\rm D}}_- \Phi(x_0)$ with $t_p=t_p^-$, respectively.

\smallskip
For $(x,t_p)$ with $\pm l_p:=\pm(x-x_p)>0$ small,
by the Oleinik-type one-sided inequality $\eqref{c1.4}$,
\begin{equation}\label{fsw11}
\pm l:=\pm \big(l_p-t_p(f'(u(x_p+l_p,t_p))-f'(c))\big)\geq 0.
\end{equation}
We claim that
\begin{equation}\label{fsw11a}
\pm l>0 \,\,\,\, \mbox{{\rm for} $\pm l_p>0$},
\qquad\lim_{\pm l_p\rightarrow 0{+}}l=0,
\qquad u(x_p+l_p,t_p)=\varphi(x_0+l).
\end{equation}
In fact, by continuity of solution $u=u(x,t)$ at point $(x_p,t_p)$,
$\lim_{\pm l_p\rightarrow 0{+}}u(x_p+l_p,t_p)=c$ so that,
by $\eqref{fsw11}$, $\lim_{\pm l_p\rightarrow 0{+}}l=0$.
From $\eqref{fsw11}$,
\begin{equation*}
x_0+l=x_0+l_p+t_p\big(f'(c)-f'(u(x_p+l_p,t_p))\big)=x-t_pf'(u(x_p+l_p,t_p)),
\end{equation*}
so that $(x_p+l_p,t_p)\in L_{u(x_p+l_p,t_p)}(x_0+l)$,
which implies that
$u(x_p+l_p,t_p)\in \mathcal{C}(x_0+l)\neq \varnothing.$

\smallskip
If there exists $\pm l_p>0$ such that $l=0$, then, from $\eqref{fsw11}$,
$$\pm\big(f'(u(x_p+l_p,t_p))-f'(c)\big)=\pm l_p/t_p>0.$$
Since $f'(u)$ is strictly increasing,
then, for the case that $c=\underline{{\rm D}}_+ \Phi(x_0)$ with $t_p=t_p^+$,
$\underline{{\rm D}}_+ \Phi(x_0)=c<u(x_p+l_p,t_p)\in \mathcal{C}(x_0),$
which contradicts to $\eqref{c4.3a}$;
for the case that $c=\overline{{\rm D}}_- \Phi(x_0)$ with $t_p=t_p^-$,
$\overline{{\rm D}}_- \Phi(x_0)=c>u(x_p+l_p,t_p)\in \mathcal{C}(x_0),$
which contradicts to $\eqref{c4.3a}$.
Thus, we conclude that $\pm l>0$ for $\pm l_p>0$.

Since $\varphi(x_0+l)$ is continuous on $\pm l\in (0,l_0)$,
$u(x_p+l_p,t_p)\in \mathcal{C}(x_0+l)\neq \varnothing$ implies that
$\mathcal{C}(x_0+l)=\{\varphi(x_0+l)\}$,
so that $u(x_p+l_p,t_p)=\varphi(x_0+l)$.

\smallskip
By $\eqref{fsw11a}$, it follows from $\eqref{fsw10a}$ and $\eqref{fsw11}$ that
\begin{align*}
0<\pm l_p&=\pm \big(l+t_p\big(f'(u(x_p+l_p,t_p))-f'(c)\big)\big)\nonumber\\[1mm]
&=\pm \big(l+t_p\big(f'(\varphi(x_0+l))-f'(c)\big)\big)\nonumber\\[1mm]
&=\pm \big(l-t^\pm_p\overline{C}_{\gamma_\pm}l+t^\pm_p\omega(l)l\big)
=t^\pm_p\omega(l)(\pm l),
\end{align*}
which implies that $\omega(l)>0$ for small $\pm l>0$.
\end{proof}

\subsection{Development of shocks
near the five types of continuous shock generation points}
From Theorem $\ref{the:elc0}$,
by the arguments as for $\eqref{fsw1}$--$\eqref{fsw2a}$,
for any characteristic $L_c(x_0):\, x=x_0+tf'(c)$,
point $(x_p,t_p)\in L_c(x_0)$ with $x_p=x_0+t_pf'(c)$
and $t_p$ given by $\eqref{plc1}$--$\eqref{plc2}$
is a continuous shock generation point if and only if
$t_p\in (0,\infty)$ and $\{c\}=\mathcal{U}(x_p,t_p)$.
There are three cases:
\begin{align*}
t_p^+=t_p^-\in (0,\infty),\qquad 0<t_p^+<t_p^-\in (0,\infty],
\qquad 0<t_p^-<t_p^+\in (0,\infty].
\end{align*}
The basic assumptions on the local behaviors of $f''(u)$
and $\varphi(x_0+l)$ near $c$ and $x_0$ respectively are presented as follows:
For simplicity, assume that, for some constants $\alpha\geq 0 $ and $N>0$,
\begin{align}\label{dsw2}
f''(u)=(N+o(1))|u-c|^{\alpha}\qquad {\rm as}\ u\rightarrow c.
\end{align}

\noindent
{\bf Case I.} $t_p^+=t_p^-\in (0,\infty)$.
Then $c=\overline{D}_-\Phi(x_0)=\underline{D}_+\Phi(x_0)$,
and it follows from $\eqref{fsw6a}$ that
$\gamma_+=\gamma_-$ and $C_{\gamma_+}=C_{\gamma_-}$.
For simplicity, assume that there exists $l_0>0$ such that
$\varphi(x_0+l)$ is continuous for $l\in ({-}l_0,0)\cup(0,l_0)$ and
\begin{align}\label{dsw3a}
\varphi(x_0+l)-c=(C_{\gamma}+o(1)){\rm sgn}(l) |l|^{\gamma} \qquad\mbox{as $l\rightarrow 0$},
\end{align}
where the constants satisfy that $C_{\gamma}<0$ and $\gamma(1+\alpha)=1.$
Furthermore, according to $\eqref{fsw10a}$--$\eqref{fsw10b}$,
assume that
\begin{align}\label{dsw3b}
f'(\varphi(x_0+l))-f'(c)
=-\overline{C}_{\gamma}\, l+\big(\overline{C}_{\sigma_\pm}+o(1)\big)\,{\rm sgn}(l)|l|^{1+\sigma}\qquad \mbox{as $l\rightarrow 0\pm$},
\end{align}
where the constants satisfy that
\begin{align*}
\overline{C}_{\gamma}:=\frac{1}{1+\alpha} N |C_{\gamma}|^{1+\alpha}=(t_p)^{-1}>0,
\qquad \overline{C}_{\sigma_\pm}>0, \qquad \sigma>0.
\end{align*}
Moreover, if $\overline{C}_{\sigma_+}=\overline{C}_{\sigma_-}=:\overline{C}_{\sigma}$,
assume that
\begin{align}\label{dsw3c}
f'(\varphi(x_0+l))-f'(c)
=-\overline{C}_{\gamma}\,l+\overline{C}_{\sigma}\,{\rm sgn}(l)|l|^{1+\sigma}
+(\overline{C}_{\rho}+o(1))\, |l|^{1+\sigma+\rho}\qquad \mbox{ as $l\rightarrow 0$},
\end{align}
where the constants satisfy that
\begin{align*}
\overline{C}_{\gamma}=\frac{1}{1+\alpha} N |C_{\gamma}|^{1+\alpha}=(t_p)^{-1}>0, \qquad \overline{C}_{\sigma}>0, \qquad \overline{C}_{\rho}\neq 0, \qquad \sigma,\rho>0.
\end{align*}

\noindent
{\bf Case II.} $0<t_p^+<t_p^-\in (0,\infty]$.
Then $c=\underline{D}_+\Phi(x_0)\geq \overline{D}_-\Phi(x_0)$.
For simplicity, assume that there exists $l_0>0$ such that
$\varphi(x_0+l)$ is continuous for $l\in (0,l_0)$ and
\begin{align}\label{dsw4a}
\varphi(x_0+l)-c=(C_{\gamma_+}+o(1)){\rm sgn}(l) |l|^{\gamma_+}\qquad \mbox{ as $l\rightarrow 0{+}$},
\end{align}
where the constants satisfy that $C_{\gamma_+}<0$ and $\gamma_+(1+\alpha)=1.$
Furthermore, according to $\eqref{fsw10a}$--$\eqref{fsw10b}$,
assume that
\begin{align}\label{dsw4b}
f'(\varphi(x_0+l))-f'(c)
=-\overline{C}_{\gamma_+}\, l+\big(\overline{C}_{\sigma_+}+o(1)\big)\,{\rm sgn}(l)|l|^{1+\sigma_+}\qquad \mbox{as $l\rightarrow 0{+}$},
\end{align}
where the constants satisfy that
\begin{align*}
\overline{C}_{\gamma_+}=\frac{1}{1+\alpha} N |C_{\gamma_+}|^{1+\alpha}=(t_p)^{-1}>0, \qquad \overline{C}_{\sigma_+}>0, \qquad \sigma_+>0.
\end{align*}
Moreover, if $\underline{D}_+\Phi(x_0)=\overline{D}_-\Phi(x_0)$,
assume that $\varphi(x_0+l)$ is continuous for $-l\in (0,l_0)$ and
\begin{align}\label{dsw4c}
\varphi(x_0+l)-c=(C_{\gamma_-}+o(1)){\rm sgn}(l) |l|^{\gamma_-}\qquad {\rm as}\ l\rightarrow 0{-},
\end{align}
where, by $t_p^+<t_p^-$, the constants $\gamma_->0$ and $C_{\gamma_-}\neq 0$ satisfy that
\begin{align}\label{dsw4c1}
\begin{cases}
C_{\gamma_-}>0 \quad &{\rm if}\ \gamma_-(1+\alpha)<1,\\ 
C_{\gamma_-}>C_{\gamma_+} \quad &{\rm if}\ \gamma_-(1+\alpha)=1,\\ 
C_{\gamma_-}\neq 0 \quad &{\rm if}\ \gamma_-(1+\alpha)>1.
\end{cases}
\end{align}

\noindent
{\bf Case III.} $0<t_p^-<t_p^+\in (0,\infty].$
Then $c=\overline{D}_-\Phi(x_0)\leq \underline{D}_+\Phi(x_0)$.
For simplicity, assume that there exists $l_0>0$ such that
$\varphi(x_0+l)$ is continuous for $-l\in (0,l_0)$ and
\begin{align}\label{dsw5a}
\varphi(x_0+l)-c=(C_{\gamma_-}+o(1)){\rm sgn}(l) |l|^{\gamma_-} \qquad \mbox{ as $l\rightarrow 0{-}$},
\end{align}
where the constants satisfy that $C_{\gamma_-}<0$ and $\gamma_-(1+\alpha)=1$.
Furthermore, according to $\eqref{fsw10a}$--$\eqref{fsw10b}$,
assume that
\begin{align}\label{dsw5b}
f'(\varphi(x_0+l))-f'(c)
=-\overline{C}_{\gamma_-}\, l+\big(\overline{C}_{\sigma_-}+o(1)\big)\,{\rm sgn}(l)|l|^{1+\sigma_-}  \qquad \mbox{as $l\rightarrow 0{-}$},
\end{align}
where the constants satisfy
\begin{align*}
\overline{C}_{\gamma_-}=\frac{1}{1+\alpha} N |C_{\gamma_-}|^{1+\alpha}=(t_p)^{-1}>0, \qquad \overline{C}_{\sigma_-}>0, \qquad \sigma_->0.
\end{align*}
Moreover, if $\underline{D}_+\Phi(x_0)=\overline{D}_-\Phi(x_0)$,
assume that $\varphi(x_0+l)$ is continuous for $l\in (0,l_0)$ and
\begin{align}\label{dsw5c}
\varphi(x_0+l)-c=(C_{\gamma_+}+o(1)){\rm sgn}(l) |l|^{\gamma_+}
\qquad {\rm as}\ l\rightarrow 0{+},
\end{align}
where, by $t_p^-<t_p^+$, the constants $\gamma_+>0$ and $C_{\gamma_+}\neq 0$ satisfy that
\begin{align}\label{dsw5c1}
\begin{cases}
C_{\gamma_+}>0 \quad &{\rm if}\ \gamma_+(1+\alpha)<1,\\ 
C_{\gamma_+}>C_{\gamma_-} \quad &{\rm if}\ \gamma_+(1+\alpha)=1,\\ 
C_{\gamma_+}\neq 0 \quad &{\rm if}\ \gamma_+(1+\alpha)>1.
\end{cases}
\end{align}

Let $(x_p,t_p)$ be a continuous shock generation point,
and let $x=x(t)$ for $t\geq t_p$ be the unique shock emitting from $(x_p,t_p)$.
Denote $u^\pm(t):=u(x(t)\pm,t)$. Since $\mathcal{U}(x_p,t_p)=\{c\}$,
it follows from $\eqref{c2.39}$ that $\lim_{t\rightarrow t_p{+}}u^\pm(t)=c$.
From $\lim_{t\rightarrow t_p{+}}x(t)=x_p=x_0+t_pf'(c)$,
\begin{equation*}
y^\pm(x(t),t):=x(t)-tf'(u^\pm(t))\longrightarrow x_0
\qquad {\rm as }\ t\rightarrow t_p{+},
\end{equation*}
so that there exists $\varepsilon_0>0$ such that
\begin{equation*}
x_0-l_0<y^-(x(t_p+\varepsilon_0),t_p+\varepsilon_0)<x_0
<y^+(x(t_p+\varepsilon_0),t_p+\varepsilon_0)<x_0+l_0.
\end{equation*}
Denote $(x_{\varepsilon_0},t_{\varepsilon_0}):=(x(t_p+\varepsilon_0),t_p+\varepsilon_0)$
and $x_0(t):=x_0+tf'(c)$.
Let
\begin{equation}\label{dswt2}
\Omega_\pm:=\big\{(x,t)\in \Delta_I(x_{\varepsilon_0},t_{\varepsilon_0})\,:\,
\pm (x-x_0(t))\geq 0\ {\rm if}\ t\leq t_p,
\,\,\, \pm (x-x(t))\geq 0\ {\rm if}\ t\geq t_p\big\},
\end{equation}
where the $\Delta_I(x_{\varepsilon_0},t_{\varepsilon_0})$ is the backward characteristic triangle of
point $(x_{\varepsilon_0},t_{\varepsilon_0})$, {\it i.e.,}
$$\Delta_I(x_{\varepsilon_0},t_{\varepsilon_0}):=\big\{(x,t)\,:\,
x_{\varepsilon_0}+(t-t_{\varepsilon_0})f'(u(x_{\varepsilon_0}{-},t_{\varepsilon_0}))
\leq x\leq 
x_{\varepsilon_0}+(t-t_{\varepsilon_0})f'(u(x_{\varepsilon_0}{+},t_{\varepsilon_0}))\big\}.$$
See also  {\rm Fig.} $\ref{figOmega}$ for the details.

\begin{figure}[H]
	\begin{center}
		{\includegraphics[width=0.7\columnwidth]{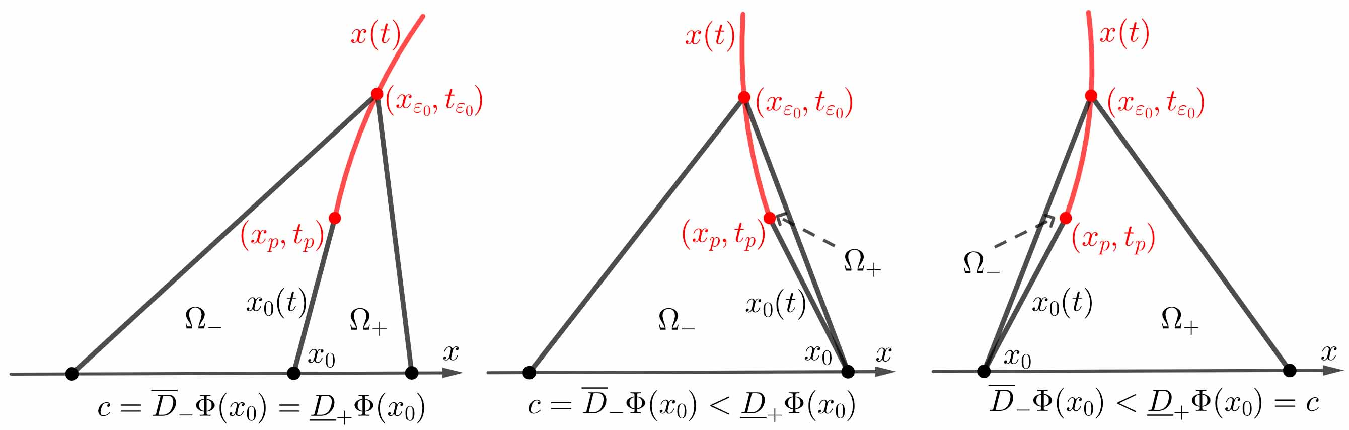}}
 \caption{The three types of domains $\Omega_\pm$ defined by $\eqref{dswt2}$.}\label{figOmega}
	\end{center}
\end{figure}

\begin{The}[Development of shocks near continuous shock generation points]\label{the:dsw1}
Suppose that $(x_p,t_p)$ determined by $\eqref{plc1}$ and $\eqref{fsw1}$, 
as a point on the characteristic line $L_c(x_0)$
defined by $\eqref{c4.1c}$ with $c\in \mathcal{C}(x_0)$,
is a continuous shock generation point and $\eqref{dsw2}$ holds.
Let $x=x(t)$ for $t>t_p$ be the unique shock emitting from $(x_p,t_p)$.
Then the regularities of 
both
the shock curve $x=x(t)$ and the entropy solution $u=u(x,t)$
near point $(x_p,t_p)$ are as follows{\rm :}
\begin{itemize}
\item [Case I.]  $t_p^+=t_p^-\in (0,\infty)$. 
Suppose that $\eqref{dsw3a}$--$\eqref{dsw3c}$ hold
for $c=\overline{D}_-\Phi(x_0)=\underline{D}_+\Phi(x_0)$.

\vspace{2pt}
\begin{itemize}
\item [(i)] If $\overline{C}_{\sigma_+}\neq\overline{C}_{\sigma_-}$,
then
\begin{align}\label{dsw6a1}
\quad\qquad\qquad x(t)-x_p-(t-t_p)f'(c)\cong
{\rm sgn}(\overline{C}_{\sigma_+}-\overline{C}_{\sigma_-})O_1(1)(t-t_p)^{\frac{1+\sigma}{\sigma}}
\quad \mbox{{\rm as} $t\rightarrow t_p{+}$};
\end{align}
and, for $(x,t)\in \Omega_\pm$ near $(x_p,t_p)$,
\begin{align}\label{dsw6a}
\qquad\qquad |u(x,t)-c|
\lessapprox |C_\gamma| \bigg\{&(1-\delta)^{-\frac{\gamma}{\sigma}}
\bigg(\frac{(\overline{C}_{\gamma})^2}{\overline{C}_{\sigma_\pm}}\bigg)^{\frac{\gamma}{\sigma}}
\big|t-t_p\big|^{\frac{\gamma}{\sigma}}\nonumber\\[1mm]
&\,+\delta^{-\frac{\gamma}{1+\sigma}}\bigg(\frac{\overline{C}_{\gamma}}{\overline{C}_{\sigma_\pm}}\bigg)^{\frac{\gamma}{1+\sigma}}
\big|x-x_p-(t-t_p)f'(c)\big|^{\frac{\gamma}{1+\sigma}}\bigg\}
\end{align}
hold for $\delta\in(0,1)$ and $\delta>\max\{1{-}(Q_\pm)^{-1}\}$.

\vspace{1pt}
\item [(ii)] If $\overline{C}_{\sigma_+}=\overline{C}_{\sigma_-}$,
then
\begin{align}\label{dsw6a2}
x(t)-x_p-(t-t_p)f'(c)\cong
{\rm sgn}(\overline{C}_{\rho})O_2(1)(t-t_p)^{\frac{1+\sigma+\rho}{\sigma}}\quad \mbox{{\rm as} $t\rightarrow t_p{+}$};
\end{align}
and $\eqref{dsw6a}$ with $\overline{C}_{\sigma_\pm}=\overline{C}_{\sigma}$ holds
for $(x,t)\in \Omega_-\cup \Omega_+$ near $(x_p,t_p)$ with $\delta\in(0,1)$.
\end{itemize}

\smallskip
\item [Case II.] $0<t_p^+<t_p^-\in(0,\infty]$. 
Suppose that $\eqref{dsw4a}$--$\eqref{dsw4c1}$ hold
for $c=\underline{D}_+\Phi(x_0)$.
Then
\begin{align}\label{dsw7}
x(t)-x_p-(t-t_p)f'(c)\cong-O^+_3(1)(t-t_p)^{\frac{1+\sigma_+}{\sigma_+}}
\qquad {\rm as}\ t\rightarrow t_p{+}.
\end{align}
For $(x,t)\in \Omega_+$ near $(x_p,t_p)$,
$\eqref{dsw6a}$ with $\overline{C}_{\sigma_+}$
holds for $\delta>1-(Q_3^+)^{-1}${\rm ;}

\noindent
For  $(x,t)\in \Omega_-$ near $(x_p,t_p)$,
\begin{align}\label{dsw7a}
\qquad u(x,t)-c\cong\begin{cases}
-C_{\gamma_-}(O^-_4(1))^{-\varrho}\,\big|x-x_p-(t-t_p)f'(c)\big|^{\varrho} & {\rm for}\ a=c,\\[2mm]
C_{\gamma_+}\big|x-x_p-(t-t_p)f'(c)\big|^{\gamma_+} & {\rm for}\ a<c,
\end{cases}
\end{align}
where $\varrho=\max\{\gamma_-,\gamma_+\}$ and $a:=\overline{D}_-\Phi(x_0)$.

\smallskip
\item [Case III.]  $0<t_p^-<t_p^+\in(0,\infty]$. 
Suppose that $\eqref{dsw5a}$--$\eqref{dsw5c1}$ hold
for $c=\overline{D}_-\Phi(x_0)$.
Then
\begin{align}\label{dsw8}
x(t)-x_p-(t-t_p)f'(c)\cong O^-_3(1)(t-t_p)^{\frac{1+\sigma_-}{\sigma_-}}
\qquad {\rm as}\ t\rightarrow t_p{+}.
\end{align}
For $(x,t)\in \Omega_-$ near $(x_p,t_p)$,
$\eqref{dsw6a}$ with $\overline{C}_{\sigma_-}$
holds for $\delta>1-(Q_3^-)^{-1}${\rm ;}

\noindent
For $(x,t)\in \Omega_+$ near $(x_p,t_p)$,
\begin{align}\label{dsw8a}
\qquad u(x,t)-c\cong\begin{cases}
C_{\gamma_+}(O^+_4(1))^{-\varrho}\,\big|x-x_p-(t-t_p)f'(c)\big|^{\varrho} & {\rm for}\ c=b,\\[2mm]
-C_{\gamma_-}\big|x-x_p-(t-t_p)f'(c)\big|^{\gamma_-} & {\rm for}\ c<b,
\end{cases}
\end{align}
where $\varrho=\max\{\gamma_-,\gamma_+\}$ and $b:=\underline{D}_+\Phi(x_0)$.
\end{itemize}
\noindent
In the above,
$\mathcal{C}(x_0)$, $\overline{{\rm D}}_- \Phi(x_0)$,
and $\underline{{\rm D}}_+ \Phi(x_0)$
are given by $\eqref{c4.0}$--$\eqref{c4.1}${\rm ;}
constants $O_1(1)$ and $O_2(1)$ are given by
\begin{align}\label{dsw9}
\begin{cases}
O_1(1):=
\overline{C}_{\gamma}\Big(\frac{(\overline{C}_{\gamma})^2Q_+}{\overline{C}_{\sigma_+}}\Big)^{\frac{1}{\sigma}}|Q_+-1|
=\overline{C}_{\gamma}\Big(\frac{(\overline{C}_{\gamma})^2Q_-}{\overline{C}_{\sigma_-}}\Big)^{\frac{1}{\sigma}}|Q_--1|,\\[3mm]
O_2(1):=
\frac{(1+\gamma)(1-\gamma)+\sigma}{(1+\gamma)(1-\gamma)+\sigma(2+\sigma)}
\frac{|\overline{C}_{\rho}|}{\overline{C}_{\gamma}}
\Big(\frac{(\overline{C}_{\gamma})^2}{\overline{C}_{\sigma}}\Big)^{\frac{1+\sigma+\rho}{\sigma}},
\end{cases}
\end{align}
where $Q_\pm$ are given by
\begin{align}\label{dsw9a}
\begin{cases}
Q_+=\frac{\gamma\sigma}{(1+\gamma)(1+\sigma)}\frac{(\lambda_1)^\gamma(1+\lambda_1)}{1+(\lambda_1)^\gamma}
+\frac{1+\gamma+\sigma}{(1+\gamma)(1+\sigma)},\\[3mm]
Q_-=\frac{\gamma\sigma}{(1+\gamma)(1+\sigma)}\frac{1+\lambda_1}{\lambda_1(1+(\lambda_1)^\gamma)}
+\frac{1+\gamma+\sigma}{(1+\gamma)(1+\sigma)}
\end{cases}
\end{align}
for  the unique real root $\lambda_1$ of the following equation{\rm :}
\begin{align*}
\gamma\sigma(1+\lambda)(\lambda_0-\lambda^{1+\gamma+\sigma})
=(1+\gamma+\sigma)\lambda(1+\lambda^\gamma)(\lambda^\sigma-\lambda_0)
\qquad {\rm with}\ \lambda_0:=\textstyle\frac{\overline{C}_{\sigma_+}}{\overline{C}_{\sigma_-}};
\end{align*}
and constants $O^\pm_3(1)$
with $Q^{\pm}_3$
and $O^\pm_4(1)$ are given by
\begin{align}
&O^\pm_3(1):=\overline{C}_{\gamma_\pm}
(1-Q^{\pm}_3)(Q^{\pm}_3)^{\frac{1}{\sigma_\pm}}
\bigg(\frac{(\overline{C}_{\gamma_\pm})^2}{\overline{C}_{\sigma_\pm}}\bigg)^{\frac{1}{\sigma_\pm}},\qquad \,\,\,
Q^{\pm}_3=\frac{1+\gamma_\pm+\sigma_\pm}{(1+\gamma_\pm)(1+\sigma_\pm)},
\label{dsw9d}\\[2mm]
&O^\pm_4(1):=\begin{cases}
\frac{\overline{C}_{\gamma_\pm}}{\overline{C}_{\gamma_\mp}} \quad
&{\rm if}\ \gamma_\pm(1+\alpha)<1,\\[1mm]
1+{\rm sgn}(C_{\gamma_\pm})\frac{\overline{C}_{\gamma_\pm}}{\overline{C}_{\gamma_\mp}} \quad
&{\rm if}\ \gamma_\pm(1+\alpha)=1,\\[1mm]
1 \quad
&{\rm if}\ \gamma_\pm(1+\alpha)>1.
\end{cases}\label{dsw9c}
\end{align}
See also
{\rm Fig.} $\ref{figUCSGP}$ and {\rm Fig.} $\ref{figUCSGPab}$.
\end{The}

\begin{figure}[H]
	\begin{center}
		{\includegraphics[width=0.6\columnwidth]{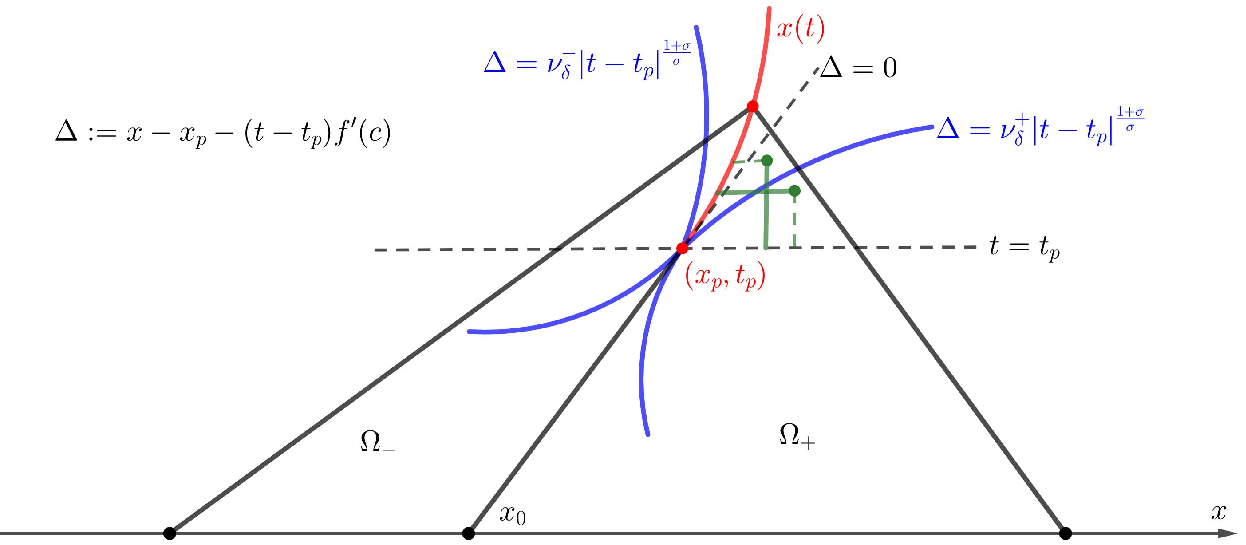}}
 \caption{The shock curve $x(t)$
 for the case that $t_p^+=t_p^-\in (0,\infty)$.}\label{figUCSGP}
	\end{center}
\end{figure}

\begin{proof}
We complete the proof case by case accordingly.

\smallskip
\noindent
{\bf Case I.} In this case, $t_p^+=t_p^-\in (0,\infty)$. By Theorem $\ref{the:c4.3}$,
 $c=\overline{D}_-\Phi(x_0)=\underline{D}_+\Phi(x_0)$.
Denote $u^\pm(t):=u(x(t)\pm,t)$ for $t>t_p$.
By Corollary $\ref{cor:c2.2}$, $u^\pm(t)\in \mathcal{U}(x(t),t;t_p)$ so that,
by $\eqref{c2.57}$--$\eqref{c2.61}$,
$$
E(u^+(t);\,x(t),t;\,t_p)=E(u^-(t);\,x(t),t;\,t_p),
$$
which, similar to $\eqref{c4.2}$, is equivalent to
\begin{equation}\label{dsw10}
\frac{1}{t-t_p}\int_{x_p+l_p^-}^{x_p+l_p^+}(u(\xi,t_p)-c)\,{\rm d}\xi
=\int_{u^+(t)}^{u^-(t)}(s-c)f''(s)\,{\rm d}s,
\end{equation}
where $l_p^\pm:=l_p^\pm(t)$ are given by
\begin{equation}\label{dsw10a}
l_p^\pm=l_p^\pm(t):=x(t)-x_p-(t-t_p)f'(u^\pm(t))\qquad {\rm for}\ t>t_p.
\end{equation}

Denote $U_\pm:=U_\pm(t,t_p)$ by
\begin{equation*}
U_\pm=U_\pm(t,t_p):=\frac{1}{l_p^\pm}\int_0^{l_p^\pm} (u(x_p+\varsigma,t_p)-c)\,{\rm d}\varsigma.
\end{equation*}
Then we have
\begin{align*}
\int_{x_p+l_p^-}^{x_p+l_p^+}(u(\xi,t_p)-c)\,{\rm d}\xi
=\Big(\int_{0}^{l_p^+}+\int_{l_p^-}^{0}\Big) (u(x_p+\varsigma,t_p)-c)\,{\rm d}\varsigma
=U_+\,l_p^+-U_-\,l_p^-,
\end{align*}
which, by combining with $\eqref{dsw10}$--$\eqref{dsw10a}$, implies
\begin{align}
&\frac{x(t)-x_p}{t-t_p}-f'(u^\pm(t))=\frac{{-}\int_{u^+(t)}^{u^-(t)}(s-c)f''(s)\,{\rm d}s
+U_\mp(f'(u^-(t))-f'(u^+(t)))}{U_--U_+},\label{dsw10d}\\[2mm]
&\frac{x(t)-x_p}{t-t_p}-f'(c)
=\frac{{-}\int_{u^+(t)}^{u^-(t)}(s-c)f''(s)\,{\rm d}s
+U_-(f'(u^-(t))-f'(c))-U_+(f'(u^+(t))-f'(c))}{U_--U_+}.\label{dsw10e}
\end{align}

The remaining proof of {\bf Case I} is divided into four steps:

\smallskip
\noindent
{\bf 1.} We first prove that
\begin{align}\label{dsw11a}
\lambda_1:=\lim_{t\rightarrow t_p{+}}\lambda(t)
\begin{cases}
>1 \quad & {\rm if}\ \overline{C}_{\sigma_+}>\overline{C}_{\sigma_-},\\[1mm]
=1 \quad & {\rm if}\ \overline{C}_{\sigma_+}=\overline{C}_{\sigma_-},\\[1mm]
<1 \quad & {\rm if}\ \overline{C}_{\sigma_+}<\overline{C}_{\sigma_-},
\end{cases}
\end{align}
where $A_\pm:=A_\pm(t)$ and $\lambda:=\lambda(t)$ are given by
\begin{align}\label{dsw11}
A_\pm=A_\pm(t):=|u^\pm(t)-c|,\qquad
\lambda=\lambda(t):=\Big(\frac{A_-(t)}{A_+(t)}\Big)^{1+\alpha}.
\end{align}

{\bf (a).} Denote $l^\pm:=l^\pm(t)$ by
\begin{align}\label{dsw12}
l^\pm=l^\pm(t):=x(t)-x_0-tf'(u^\pm(t)).
\end{align}
By $\eqref{dsw3b}$, it follows from $\eqref{dsw10a}$ and $\eqref{dsw12}$ that
\begin{align}\label{dsw12a}
l_p^\pm&=l^\pm+t_p(f'(u^\pm(t))-f'(c))
=l^\pm+t_p(f'(\varphi(x_0+l^\pm))-f'(c))\nonumber\\[1mm]
&=(1-t_p\overline{C}_\gamma)l^\pm+t_p(\overline{C}_{\sigma_\pm}+o(1)){\rm sgn}(l^\pm)|l^\pm|^{1+\sigma}\nonumber\\
&=(1+o(1))\frac{\overline{C}_{\sigma_\pm}}{\overline{C}_{\gamma}}{\rm sgn}(l^\pm)|l^\pm|^{1+\sigma}.
\end{align}
Then, by $\eqref{dsw3a}$, we obtain
\begin{align}\label{dsw12c}
u^\pm(t)-c&=u(x_p+l_p^\pm,t_p)-c
=\varphi(x_0+l^\pm)-c
=(C_\gamma+o(1)){\rm sgn}(l^\pm)|l^\pm|^{\gamma}
\nonumber\\[1mm]
&=-(1+o(1))|C_\gamma|
\bigg(\frac{\overline{C}_{\gamma}}{\overline{C}_{\sigma_\pm}}\bigg)^{\frac{\gamma}{1+\sigma}}
{\rm sgn}(l_p^\pm)|l_p^\pm|^{\frac{\gamma}{1+\sigma}}\nonumber\\[1mm]
&=:-(1+o(1))\overline{C}_\pm{\rm sgn}(l_p^\pm)|l_p^\pm|^{\frac{\gamma}{1+\sigma}}.
\end{align}
Thus, by $\eqref{dsw12c}$, it follows from $\eqref{dsw10a}$ and $\eqref{dsw11}$ that
\begin{align}\label{dsw12e}
\frac{x(t)-x_p}{t-t_p}-f'(u^\pm(t))=\frac{l_p^\pm}{t-t_p}
={\rm sgn}(l_p^\pm)\frac{1+o(1)}{t-t_p}\bigg(\frac{A_\pm}{\overline{C}_\pm}\bigg)^{\frac{1+\sigma}{\gamma}}.
\end{align}

On the other hand, from $\eqref{dsw12c}$,
\begin{align}\label{dsw13}
U_\pm &=-\frac{1}{l_p^\pm}\int_0^{l_p^\pm}
 (1+o(1))\overline{C}_\pm{\rm sgn}(l_p^\pm)|\varsigma|^{\frac{\gamma}{1+\sigma}}
\,{\rm d}\varsigma\nonumber\\[1mm]
&=-(1+o(1))\frac{1+\sigma}{1+\gamma+\sigma}
\overline{C}_\pm{\rm sgn}(l_p^\pm)\,|l_p^\pm|^{\frac{\gamma}{1+\sigma}}\nonumber\\[1mm]
&=-(1+o(1))\frac{1+\sigma}{1+\gamma+\sigma}{\rm sgn}(l_p^\pm)\,A_\pm.
\end{align}
Furthermore, by $\eqref{dsw11}$, it follows from $\eqref{aa2}$--$\eqref{aa3}$ that
\begin{align}\label{dsw13a}
f'(u^\pm(t))-f'(c)
=\mp \frac{N+o(1)}{1+\alpha}|u^\pm(t)-c|^{1+\alpha}
=\mp \frac{N+o(1)}{1+\alpha}(A_\pm)^{1+\alpha},
\end{align}
\begin{align}\label{dsw13b}
-\int_{u^+(t)}^{u^-(t)}(s-c)f''(s)\,{\rm d}s
=-\frac{N+o(1)}{2+\alpha}(A_-)^{2+\alpha}+ \frac{N+o(1)}{2+\alpha}(A_+)^{2+\alpha}.
\end{align}
Taking $\eqref{dsw12e}$--$\eqref{dsw13b}$ into $\eqref{dsw10d}$, we obtain
\begin{align}\label{dsw14}
\pm \frac{1}{t-t_p}\Big(\frac{A_\pm}{\overline{C}_\pm}\Big)^{\frac{1+\sigma}{\gamma}}
=\frac{1+o(1)}{\frac{1+\sigma}{1+\gamma+\sigma}(A_-+A_+)}
\bigg(& \pm \frac{1+\sigma}{1+\gamma+\sigma}\frac{N+o(1)}{1+\alpha}A_\mp
\big((A_-)^{1+\alpha}+(A_+)^{1+\alpha}\big)\nonumber\\[2mm]
&\, +\frac{N+o(1)}{2+\alpha}\big((A_+)^{2+\alpha}-(A_-)^{2+\alpha}\big)
\bigg).
\end{align}
Then, as $t\rightarrow t_p{+}$,
\begin{align*}
\frac{\overline{C}_{\sigma_+}}{\overline{C}_{\sigma_-}}\Big(\frac{A_+}{A_-}\Big)^{\frac{1+\sigma}{\gamma}}
\cong
\frac{\gamma\sigma(A_-)^{2+\alpha}+(1+\gamma)(1+\sigma)A_-(A_+)^{1+\alpha}
+(1+\gamma+\sigma)(A_+)^{2+\alpha}}
{\gamma\sigma(A_+)^{2+\alpha}+(1+\gamma)(1+\sigma)A_+(A_-)^{1+\alpha}
+(1+\gamma+\sigma)(A_-)^{2+\alpha}},
\end{align*}
which, by $A_-=A_+\lambda^{\frac{1}{1+\alpha}}$ inferred by $\eqref{dsw11}$ and $\gamma(1+\alpha)=1$ in $\eqref{dsw3a}$, implies
\begin{align*}
\lambda_0=\frac{\overline{C}_{\sigma_+}}{\overline{C}_{\sigma_-}}
\cong
\frac{\gamma\sigma\lambda^{2+\gamma+\sigma}+(1+\gamma)(1+\sigma)\lambda^{1+\gamma+\sigma}
+(1+\gamma+\sigma)\lambda^{1+\sigma}}
{\gamma\sigma+(1+\gamma)(1+\sigma)\lambda+(1+\gamma+\sigma)\lambda^{1+\gamma}}.
\end{align*}
By a simple calculation, we obtain
\begin{align}\label{dsw14c}
\gamma\sigma(1+\lambda)(\lambda_0-\lambda^{1+\gamma+\sigma})
\cong(1+\gamma+\sigma)\lambda(1+\lambda^\gamma)(\lambda^\sigma-\lambda_0).
\end{align}

{\bf (b).} We now solve equation $\eqref{dsw14c}$. First, it is direct to see that
\begin{align*}
0<\mathop{\underline{\lim}}\limits_{t\rightarrow t_p{+}}\lambda(t)<\mathop{\overline{\lim}}\limits_{t\rightarrow t_p{+}}\lambda(t)<\infty.
\end{align*}
Due to the magnitude of $\overline{C}_{\sigma_+}$ and $\overline{C}_{\sigma_-}$,
there are three cases:

{(i)} $\,$ If $\overline{C}_{\sigma_+}=\overline{C}_{\sigma_-}$, then $\lambda_0=1$.
By letting $t\rightarrow t_p{+}$ in $\eqref{dsw14c}$,
any converged subsequence of $\{\lambda(t)\}$ must converge to $1$,
which implies that $\lambda_1=\lim_{t\rightarrow t_p{+}}\lambda(t)=1.$

{(ii)} $\,$ If $\overline{C}_{\sigma_+}<\overline{C}_{\sigma_-}$, then $\lambda_0<1$. To prove
\begin{align*}
\gamma\sigma(1+\lambda)(\lambda_0-\lambda^{1+\gamma+\sigma})=(1+\gamma+\sigma)\lambda(1+\lambda^\gamma)(\lambda^\sigma-\lambda_0)
\end{align*}
has a unique real root,
it suffices to show that
\begin{align}\label{dsw15c}
h(\lambda):=\frac{1+\gamma+\sigma}{\gamma\sigma}\frac{\lambda(1+\lambda^\gamma)}{1+\lambda}+\frac{\lambda^{1+\gamma+\sigma}-\lambda_0}{\lambda^\sigma-\lambda_0}=0
\end{align}
has a unique root for $\lambda\in(0,\infty)$.

Denote $\lambda_0^+=(\lambda_0)^{\frac{1}{1+\gamma+\sigma}}$
and $\lambda_0^-=(\lambda_0)^{\frac{1}{\sigma}}.$
If $\lambda>\lambda_0^+$ or $\lambda<\lambda_0^-$,
we see from $\eqref{dsw15c}$ that $h(\lambda)>0$.
Meanwhile, it is direct to check from $\eqref{dsw15c}$ that
\begin{align*}
h(\lambda_0^+)=\frac{1+\gamma+\sigma}{\gamma\sigma}\,
\frac{\lambda_0^+(1+(\lambda_0^+)^\gamma)}{1+\lambda_0^+}>0,\qquad
\lim_{\lambda\rightarrow \lambda_0^-{+}}h(\lambda)=-\infty,
\end{align*}
and, for $\lambda\in(\lambda_0^-,\lambda_0^+)$,
\begin{align*}
h'(\lambda)
=\frac{1+\gamma+\sigma}{\gamma\sigma}\frac{1+(1+\gamma)\lambda^\gamma+\gamma\lambda^{1+\gamma}}{(1+\lambda)^2}
+\frac{
(1+\gamma+\sigma)\lambda^{1+\gamma}(\lambda^{\sigma}{-}\lambda_0)
+\sigma(\lambda_0{-}\lambda^{1+\gamma+\sigma})
}
{\lambda^{1-\sigma}(\lambda^\sigma{-}\lambda_0)^2}>0,
\end{align*}
which implies that there exists a unique $\lambda=\lambda_1\in (\lambda_0^-,\lambda_0^+)\subset (0,1)$
such that $h(\lambda_1)=0$.
Therefore, we conclude
\begin{align*}
\lim_{t\rightarrow t_p{+}}\lambda(t)=\lambda_1<1 \qquad {\rm if}\ \overline{C}_{\sigma_+}<\overline{C}_{\sigma_-}.
\end{align*}

{(iii)} $\,$ If $\overline{C}_{\sigma_+}>\overline{C}_{\sigma_-}$,
then $\lambda_0>1$ so that $(\lambda_0)^{-1}<1$.
By taking $\bar{\lambda}:=\lambda^{-1}$ into $\eqref{dsw14c}$,
\begin{align}\label{dsw17}
\gamma\sigma(1+\bar{\lambda})\big((\lambda_0)^{-1}-\bar{\lambda}^{1+\gamma+\sigma}\big)
\cong(1+\gamma+\sigma)\bar{\lambda}(1+\bar{\lambda}^\gamma)\big(\bar{\lambda}^\sigma-(\lambda_0)^{-1}\big),
\end{align}
which is the same as $\eqref{dsw14c}$
via replacing $\bar{\lambda}$ and $(\lambda_0)^{-1}$ by $\lambda$ and $\lambda_0$, respectively.

Since $(\lambda_0)^{-1}<1$, by the same arguments
as ${\rm (ii)}$,
equation $\eqref{dsw17}$ has a unique real root $\bar{\lambda}_1\in ((\lambda_0)^{-\frac{1}{\sigma}},(\lambda_0)^{-\frac{1}{1+\gamma+\sigma}})\subset (0,1)$ so that
\begin{align*}
\lambda_1=\lim_{t\rightarrow t_p{+}}\lambda(t)=(\bar{\lambda}_1)^{-1}>1 \qquad {\rm if}\ \overline{C}_{\sigma_+}>\overline{C}_{\sigma_-}.
\end{align*}

\smallskip
\noindent
{\bf 2.} We now prove $\eqref{dsw6a1}$ with $O_1(1)$ in $\eqref{dsw9}$.
Since $\overline{C}_{\sigma_+}\neq\overline{C}_{\sigma_-}$,
it follows from $\eqref{dsw11a}$ that $\lambda_1\neq 1$.

\smallskip
{\bf (a).} By taking $A_-=(1+o(1))(\lambda_1)^\gamma A_+$ into $\eqref{dsw14}$,
it is direct to check that, as $t\rightarrow t_p{+}$,
\begin{align}\label{dsw18}
A_+&
=(1+o(1))|C_\gamma|
\bigg(\frac{(\overline{C}_{\gamma})^2}{\overline{C}_{\sigma_+}}\bigg)^{\frac{\gamma}{\sigma}}
\bigg(\frac{\gamma\sigma}{(1+\gamma)(1+\sigma)}\frac{(\lambda_1)^\gamma(1+\lambda_1)}{1+(\lambda_1)^\gamma}
+\frac{1+\gamma+\sigma}{(1+\gamma)(1+\sigma)}\bigg)^{\frac{\gamma}{\sigma}}
\,(t-t_p)^{\frac{\gamma}{\sigma}}\nonumber\\[1mm]
&=(1+o(1))|C_\gamma|
\bigg(\frac{\big(\overline{C}_{\gamma}\big)^2}{\overline{C}_{\sigma_+}}\bigg)^{\frac{\gamma}{\sigma}}
(Q_+)^{\frac{\gamma}{\sigma}}\,(t-t_p)^{\frac{\gamma}{\sigma}},
\end{align}
\begin{align}\label{dsw18a}
A_-&
=(1+o(1))|C_\gamma|
\bigg(\frac{(\overline{C}_{\gamma})^2}{\overline{C}_{\sigma_-}}\bigg)^{\frac{\gamma}{\sigma}}
\bigg(\frac{\gamma\sigma}{(1+\gamma)(1+\sigma)}\frac{1+\lambda_1}{\lambda_1(1+(\lambda_1)^\gamma)}
+\frac{1+\gamma+\sigma}{(1+\gamma)(1+\sigma)}\bigg)^{\frac{\gamma}{\sigma}}
\,(t-t_p)^{\frac{\gamma}{\sigma}}\nonumber\\[1mm]
&=(1+o(1))|C_\gamma|
\bigg(\frac{\big(\overline{C}_{\gamma}\big)^2}{\overline{C}_{\sigma_-}}\bigg)^{\frac{\gamma}{\sigma}}
(Q_-)^{\frac{\gamma}{\sigma}}\,(t-t_p)^{\frac{\gamma}{\sigma}},
\end{align}
where $Q_\pm$, given by $\eqref{dsw9a}$, can be seen from $\eqref{dsw11a}$ that
\begin{align}\label{dsw18b}
Q_-<1<Q_+ \,\,\, {\rm if}\ \overline{C}_{\sigma_+}>\overline{C}_{\sigma_-},
\qquad\,\, Q_+<1<Q_- \,\,\, {\rm if}\ \overline{C}_{\sigma_+}<\overline{C}_{\sigma_-}.
\end{align}
Then, by $\eqref{dsw12e}$,
it is direct to check from $\eqref{dsw18}$--$\eqref{dsw18a}$ that
\begin{align}\label{dsw18c}
l_p^\pm=\pm (1+o(1))\overline{C}_\gamma
\bigg(\frac{(\overline{C}_{\gamma})^2}{\overline{C}_{\sigma_\pm}}\bigg)^{\frac{1}{\sigma}}
(Q_\pm)^{\frac{1+\sigma}{\sigma}}\,(t-t_p)^{\frac{1+\sigma}{\sigma}}\qquad \mbox{ as $t\rightarrow t_p{+}$},
\end{align}
which, by $\eqref{dsw12a}$, implies that
\begin{align}\label{dsw18d}
l^\pm=\pm (1+o(1))
\bigg(\frac{(\overline{C}_{\gamma})^2}{\overline{C}_{\sigma_\pm}}\bigg)^{\frac{1}{\sigma}}
(Q_\pm)^{\frac{1}{\sigma}}\,(t-t_p)^{\frac{1}{\sigma}} \qquad \mbox{ as $t\rightarrow t_p{+}$}.
\end{align}

{\bf (b).} By $\eqref{dsw3b}$, it follows from $\eqref{dsw10a}$ and $\eqref{dsw12a}$ that,
as $t\rightarrow t_p{+}$,
\begin{align}\label{dsw19}
x(t)-x_p-(t-t_p)f'(c)
&=x(t)-x_p-(t-t_p)f'(u^\pm(t))+(t-t_p)(f'(u^\pm(t))-f'(c))\nonumber\\[1mm]
&=l_p^\pm+(t-t_p)(f'(u^\pm(t))-f'(c))\nonumber\\[1mm]
&=l_p^\pm+(t-t_p)\big({-}\overline{C}_{\gamma}\, l^\pm
+\big(\overline{C}_{\sigma_\pm}+o(1)\big)\,{\rm sgn}(l^\pm)|l^\pm|^{1+\sigma}\big)\nonumber\\[1mm]
&=l_p^\pm-\overline{C}_{\gamma}(t-t_p)\big(l^\pm-(1+o(1))l_p^\pm\big)\nonumber\\[1mm]
&=(1+o(1))l_p^\pm-\overline{C}_{\gamma}(t-t_p)l^\pm,
\end{align}
which, by combining with $\eqref{dsw18b}$--$\eqref{dsw18d}$, implies that,
as $t\rightarrow t_p{+}$,
\begin{align*}
x(t)-x_p-(t-t_p)f'(c)
&=(1+o(1))l_p^\pm-\overline{C}_{\gamma}(t-t_p)l^\pm \nonumber\\[1mm]
&=\pm (1+o(1))\overline{C}_\gamma\bigg(\frac{(\overline{C}_{\gamma})^2}{\overline{C}_{\sigma_\pm}}\bigg)^{\frac{1}{\sigma}}
(Q_\pm)^{\frac{1}{\sigma}}(Q_\pm-1)\,(t-t_p)^{\frac{1+\sigma}{\sigma}} \nonumber\\
&=(1+o(1)){\rm sgn}\big(\overline{C}_{\sigma_+}-\overline{C}_{\sigma_-}\big)
\overline{C}_\gamma\bigg(\frac{(\overline{C}_{\gamma})^2Q_\pm}
{\overline{C}_{\sigma_\pm}}\bigg)^{\frac{1}{\sigma}}\,
|Q_\pm-1|\,(t-t_p)^{\frac{1+\sigma}{\sigma}}\nonumber\\[1mm]
&\cong {\rm sgn}\big(\overline{C}_{\sigma_+}-\overline{C}_{\sigma_-}\big)O_1(1)
\,(t-t_p)^{\frac{1+\sigma}{\sigma}}.
\end{align*}
It remains to check $O_1(1)$ in $\eqref{dsw9}$, $i.e.$,
$$
\bigg(\frac{(\overline{C}_{\gamma})^2Q_+}{\overline{C}_{\sigma_+}}\bigg)^{\frac{1}{\sigma}}|Q_+-1|
=\bigg(\frac{(\overline{C}_{\gamma})^2Q_-}{\overline{C}_{\sigma_-}}\bigg)^{\frac{1}{\sigma}}|Q_--1|,
$$
which is implied by the following two facts:
It is direct to check from $\eqref{dsw9a}$ that
\begin{align*}
|Q_+-1|
=\frac{\gamma\sigma}{(1+\gamma)(1+\sigma)}\,
\frac{\big|(\lambda_1)^{1+\gamma}-1\big|}{1+(\lambda_1)^\gamma}
=\lambda_1|Q_--1|,
\end{align*}
and, by $A_-=(1+o(1))(\lambda_1)^\gamma A_+$,
it follows from $\eqref{dsw18}$--$\eqref{dsw18a}$ that
$$\lambda_1=\lim_{t\rightarrow t_p{+}}\bigg(\frac{A_-(t)}{A_+(t)}\bigg)^{\frac{1}{\gamma}}
=\bigg(\frac{Q_-}{Q_+}\lambda_0\bigg)^{\frac{1}{\sigma}}.$$

\smallskip
\noindent
{\bf 3.} We now prove $\eqref{dsw6a2}$ with $O_2(1)$ in $\eqref{dsw9}$.
Since $\overline{C}_{\sigma_+}=\overline{C}_{\sigma_-}
=\overline{C}_{\sigma}$,
then $\lambda_1=1$ so that $Q_+=Q_-$.
By $\eqref{dsw18c}$--$\eqref{dsw18d}$,
it follows from $\eqref{dsw19}$ that
\begin{align}\label{dsw20}
x(t)-x_p-(t-t_p)f'(c)=o(1)(t-t_p)^{\frac{1+\sigma}{\sigma}}.
\end{align}
Thus, to estimate shock $x=x(t)$,
we need assumption $\eqref{dsw3c}$.

\smallskip
{\bf (a).} Since $\overline{C}_{\sigma_\pm}=\overline{C}_{\sigma}$,
it follows from $\eqref{dsw11a}$ that
$$\lambda(t)=1+o(1)\qquad {\rm as}\ t\rightarrow t_p{+},$$
which implies
$$
\frac{\lambda(t)^{1+\gamma}-1}{\lambda(t)^\gamma+1}
=\frac{\big(1+(\lambda(t)-1)\big)^{1+\gamma}-1}{2+o(1)}
=\frac{1+o(1)}{2}(1+\gamma)(\lambda(t)-1).
$$
Then, by taking $\eqref{dsw12c}$--$\eqref{dsw13b}$ into $\eqref{dsw10e}$,
it can be checked that
\begin{align}\label{dsw20a}
&\ \frac{1}{t-t_p}\big(x(t)-x_p-(t-t_p)f'(c)\big)\nonumber\\
&=(1+o(1))\frac{N}{1+\alpha}\frac{(1+\gamma)(1-\gamma)+\sigma}{(1+\gamma)(1+\sigma)}
\frac{(A_-(t))^{2+\alpha}-(A_+(t))^{2+\alpha}}{A_-(t)+A_+(t)}\nonumber\\[1mm]
&=(1+o(1))\frac{N}{1+\alpha}\frac{(1+\gamma)(1-\gamma)+\sigma}{(1+\gamma)(1+\sigma)}
(A_+(t))^{1+\alpha}\,
\frac{\lambda(t)^{1+\gamma}-1}{\lambda(t)^\gamma+1}\nonumber\\[1mm]
&=(1+o(1))\frac{N}{1+\alpha}\frac{(1+\gamma)(1-\gamma)+\sigma}{2(1+\gamma)(1+\sigma)}
(A_+(t))^{1+\alpha} (1+\gamma)(\lambda(t)-1)\nonumber\\[1mm]
&=(1+o(1))\frac{(1+\gamma)(1-\gamma)+\sigma}{2(1+\sigma)}\frac{N}{1+\alpha}
\big((A_-(t))^{1+\alpha}-(A_+(t))^{1+\alpha}\big)\nonumber\\[1mm]
&=(1+o(1))\frac{(1+\gamma)(1-\gamma)+\sigma}{2(1+\sigma)}
\big(|f'(u^-(t))-f'(c)|-|f'(u^+(t))-f'(c)|\big)\nonumber\\[1mm]
&=(1+o(1))\frac{(1+\gamma)(1-\gamma)+\sigma}{2(1+\sigma)}
\big((f'(u^-(t))-f'(c))+(f'(u^+(t))-f'(c))\big)\nonumber\\[1mm]
&=-(1+o(1))\frac{(1+\gamma)(1-\gamma)+\sigma}{2(1+\sigma)}\,
\overline{C}_\gamma\,(l^++l^-).
\end{align}

{\bf (b).} On the other hand, from $\eqref{dsw18d}$ and $\eqref{dsw20}$,
for $t>t_p$, we set
\begin{align}\label{dsw21}
\begin{cases}
\mu^\pm=\mu^\pm(t):=\frac{l^\pm(t)}{s}=\frac{l^\pm}{s},\qquad
s=(t-t_p)^{\frac{1}{\sigma}},\\[2mm]
\nu=\nu(t):=\big(x(t)-x_p-(t-t_p)f'(c)\big)s^{-1-\sigma}.
\end{cases}
\end{align}
Then, by $\overline{C}_{\sigma_\pm}=\overline{C}_{\sigma}$ and $Q_\pm=1$, we obtain
\begin{align}\label{dsw21a}
\mu_0^\pm:=\lim_{t\rightarrow t_p{+}}\mu^\pm(t)=\lim_{t\rightarrow t_p{+}}\frac{l^\pm(t)}{s}
=\pm\bigg(\frac{(\overline{C}_{\gamma})^2}{\overline{C}_{\sigma}}\bigg)^{\frac{1}{\sigma}},
\qquad \lim_{t\rightarrow t_p{+}}\nu(t)=0.
\end{align}
Noticing that $x_p=x_0+t_pf'(c)$, it follows from $\eqref{dsw12}$ and $\eqref{dsw21}$ that,
for small $t>t_p$,
\begin{align}\label{dsw21b}
\nu
&=\big(l^\pm+t(f'(u^\pm(t))-f'(c))\big)s^{-1-\sigma}
\nonumber\\[1mm]
&
=\big(l^\pm+t(f'(\varphi(x_0+l^\pm))-f'(c))\big)s^{-1-\sigma}
\nonumber\\[1mm]
&=\Big\{
l^\pm+(t_p+s^\sigma)\big({-}\overline{C}_{\gamma}\,l^\pm
+\overline{C}_{\sigma}\,{\rm sgn}(l^\pm)|l^\pm|^{1+\sigma}
+(\overline{C}_{\rho}+o(1))\, |l^\pm|^{1+\sigma+\rho}\big)
\Big\}s^{-1-\sigma}\nonumber\\[1mm]
&=-\overline{C}_{\gamma}\,\mu^\pm
+(t_p+s^\sigma)\overline{C}_{\sigma}\,|\mu^\pm|^{\sigma}\mu^\pm
+(1+o(1))t_p\overline{C}_{\rho}\, |\mu^\pm|^{1+\sigma+\rho}s^{\rho}.
\end{align}
From $\eqref{dsw21a}$, letting $s\rightarrow 0{+}$ in $\eqref{dsw21b}$ leads to
\begin{align*}
0=-\overline{C}_{\gamma}\,\mu_0^\pm +t_p \overline{C}_{\sigma} \,|\mu_0^\pm|^{\sigma}\mu_0^\pm,
\end{align*}
so that, by subtracting into $\eqref{dsw21b}$,
\begin{align*}
\nu&=-\overline{C}_{\gamma}\,(\mu^\pm-\mu_0^\pm)+t_p\overline{C}_{\sigma}\,
\big(|\mu^\pm|^{\sigma}\mu^\pm-|\mu_0^\pm|^{\sigma}\mu_0^\pm\big)
\nonumber\\[1mm]
&\quad \,
+\overline{C}_{\sigma}\,|\mu_0^\pm|^{\sigma}\mu_0^\pm s^{\sigma}+(1+o(1))t_p\overline{C}_{\rho}\,|\mu_0^\pm|^{1+\sigma+\rho}s^{\rho}\nonumber\\[1mm]
&\quad \, +\overline{C}_{\sigma}\,\big(|\mu^\pm|^{\sigma}\mu^\pm-|\mu_0^\pm|^{\sigma}\mu_0^\pm\big) s^{\sigma}
+(1+o(1))t_p\overline{C}_{\rho}\,\big(|\mu^\pm|^{1+\sigma+\rho}-|\mu_0^\pm|^{1+\sigma+\rho}\big)s^{\rho}
\nonumber\\[1mm]
&=\big({-}\overline{C}_{\gamma}
+(1+\sigma+o(1))t_p\overline{C}_{\sigma}|\mu_0^\pm|^{\sigma}+o(1)\big)\,(\mu^\pm-\mu_0^\pm)
\nonumber\\[1mm]
&\quad \,
+\overline{C}_{\sigma}\,|\mu_0^\pm|^{\sigma}\mu_0^\pm s^{\sigma}+t_p(\overline{C}_{\rho}+o(1))\,|\mu_0^\pm|^{1+\sigma+\rho}s^{\rho}
\nonumber\\[1mm]
&=(1+o(1))\sigma\overline{C}_{\gamma}(\mu^\pm-\mu_0^\pm)
+\overline{C}_{\sigma}\,|\mu_0^\pm|^{\sigma}\mu_0^\pm s^{\sigma}
+t_p(\overline{C}_{\rho}+o(1))\,|\mu_0^\pm|^{1+\sigma+\rho}s^{\rho},
\end{align*}
which, by $\eqref{dsw21}$--$\eqref{dsw21a}$, implies
\begin{align}\label{dsw21e}
2\nu&=(1+o(1))\sigma\overline{C}_{\gamma}(\mu^++\mu^-)+t_p(\overline{C}_{\rho}+o(1))\,
\big(|\mu_0^+|^{1+\sigma+\rho}+|\mu_0^-|^{1+\sigma+\rho}\big)s^{\rho}\nonumber\\[1mm]
&=(1+o(1))\sigma\overline{C}_{\gamma}\frac{l^++l^-}{s}
+2t_p(\overline{C}_{\rho}+o(1))\,|\mu_0^+|^{1+\sigma+\rho}s^{\rho}.
\end{align}
Thus, combining $\eqref{dsw20a}$--$\eqref{dsw21}$
with $\eqref{dsw21e}$,
we obtain
\begin{align*}
\nu s&=\big(x(t)-x_p-(t-t_p)f'(c)\big)s^{-\sigma}\nonumber\\[1mm]
&=-(1+o(1))\frac{(1+\gamma)(1-\gamma)+\sigma}{2(1+\sigma)}\,
\overline{C}_\gamma\,(l^++l^-)\nonumber\\[1mm]
&=-(1+o(1))\frac{(1+\gamma)(1-\gamma)+\sigma}{\sigma(1+\sigma)}
\big(\nu s-t_p(\overline{C}_{\rho}+o(1))\,|\mu_0^+|^{1+\sigma+\rho}s^{1+\rho}\big).
\end{align*}
This, by a simple calculation, implies
\begin{align*}
\nu&=(1+o(1))\frac{(1+\gamma)(1-\gamma)+\sigma}{(1+\gamma)(1-\gamma)+\sigma(2+\sigma)}
t_p(\overline{C}_{\rho}+o(1))\,|\mu_0^+|^{1+\sigma+\rho}\,s^{\rho}\nonumber\\[1mm]
&=(1+o(1)){\rm sgn}(\overline{C}_{\rho})
\frac{(1+\gamma)(1-\gamma)+\sigma}{(1+\gamma)(1-\gamma)+\sigma(2+\sigma)}
\frac{|\overline{C}_{\rho}|}{\overline{C}_{\gamma}}
\bigg(\frac{(\overline{C}_{\gamma})^2}{\overline{C}_{\sigma}}\bigg)^{\frac{1+\sigma+\rho}{\sigma}}
s^{\rho}
\nonumber\\[1mm]
&
\cong {\rm sgn}(\overline{C}_{\rho})O_2(1)\, s^{\rho}.
\end{align*}
Therefore, it follows from $\eqref{dsw21}$ that
\begin{align*}
x(t)-x_p-(t-t_p)f'(c)=\nu s^{1+\sigma}\cong
{\rm sgn}(\overline{C}_{\rho})O_2(1)\, (t-t_p)^{\frac{1+\sigma+\rho}{\sigma}}.
\end{align*}

\smallskip
\noindent
{\bf 4.} We now estimate the behavior of solution $u=u(x,t)$ near point $(x_p,t_p)$.

\smallskip
{\bf (a).} For any $\nu \in \mathbb{R}$,
we define the curve:
\begin{align}\label{dswr1}
\chi(t,\nu):\, x-x_p-(t-t_p)f'(c)=\nu |t-t_p|^\frac{1+\sigma}{\sigma}
\qquad {\rm for\ small\ } t-t_p.
\end{align}
For $(x,t)\in \chi(t,\nu)$ with $\pm(x-x(t))>0$,
denote $\tilde{l}^\pm:=\tilde{l}^\pm(t,\nu)$ by
\begin{align*}
\tilde{l}^\pm=\tilde{l}^\pm(t,\nu):=y(x,t,u(x,t))-x_0=x-tf'(u(x,t))-x_0.
\end{align*}
Since $\varphi(x_0+l)$ is continuous for $\pm l\in (0,l_0)$,
it follows from $\eqref{c2.31}$ and $\eqref{uiff}$ that
$\pm\tilde{l}^\pm>0$ satisfy that, for small $t-t_p$,
\begin{align}\label{dswr2a}
\tilde{l}^\pm=\tilde{l}^\pm(t,\nu)
&=x-tf'(u(x,t))-x_0
=x-tf'(\varphi(x_0+\tilde{l}^\pm))-x_0\nonumber\\[1mm]
&=x-x_p-(t-t_p)f'(c)-t\big(f'(\varphi(x_0+\tilde{l}^\pm))-f'(c)\big)\nonumber\\[1mm]
&=\nu |t-t_p|^\frac{1+\sigma}{\sigma}-t\big(f'(\varphi(x_0+\tilde{l}^\pm))-f'(c)\big).
\end{align}
Denote $\tilde{\mu}^\pm:=\tilde{\mu}^\pm(t,\nu)$ and $s$ by
\begin{align}\label{dswr3}
\tilde{\mu}^\pm=\tilde{\mu}^\pm(t,\nu):=\frac{\tilde{l}^\pm(t,\nu)}{s}=\frac{\tilde{l}^\pm}{s},
\quad s=|t-t_p|^{\frac{1}{\sigma}}
\qquad {\rm for\ small}\ t-t_p\neq 0.
\end{align}
From $\eqref{dswr2a}$--$\eqref{dswr3}$,
it follows from $\pm\tilde{l}^\pm>0$ that $\pm\tilde{\mu}^\pm>0$
and, for small $s>0$,
\begin{align}\label{dswr4}
\nu &=\big(\tilde{l}^\pm+t(f'(\varphi(x_0+\tilde{l}^\pm))-f'(c))\big)s^{-1-\sigma}\nonumber\\[1mm]
&=\Big\{\tilde{l}^\pm
+(t_p+{\rm sgn}(t-t_p)s^\sigma)\big({-}\overline{C}_{\gamma}\,\tilde{l}^\pm
+\big(\overline{C}_{\sigma_\pm}+o(1)\big)\,{\rm sgn}(\tilde{l}^\pm)|\tilde{l}^\pm|^{1+\sigma}\big)
\Big\}s^{-1-\sigma}\nonumber\\[1mm]
&=-{\rm sgn}(t-t_p)\overline{C}_{\gamma}\,\tilde{\mu}^\pm
+(1+o(1))t_p\overline{C}_{\sigma_\pm}\,|\tilde{\mu}^\pm|^{\sigma}\tilde{\mu}^\pm.
\end{align}

{\bf (b).} We now estimate $\tilde{l}^\pm$. There exist two cases:

\smallskip
{\bf (b1)}. For $t<t_p$, from $\eqref{dswr4}$,
$\nu =\big(\,\overline{C}_{\gamma}+(1+o(1))t_p\overline{C}_{\sigma_\pm}\,|\tilde{\mu}^\pm|^{\sigma}\big)\tilde{\mu}^\pm$
so that
\begin{align*}
|\nu|=\big(\,\overline{C}_{\gamma}+(1+o(1))t_p\overline{C}_{\sigma_\pm}\,|\tilde{\mu}^\pm|^{\sigma}\big)|\tilde{\mu}^\pm|
\geq (1+o(1))t_p\overline{C}_{\sigma_\pm}\,|\tilde{\mu}^\pm|^{1+\sigma},
\end{align*}
which, by $\eqref{dswr3}$, implies
\begin{align}\label{dswr7b}
\big|\tilde{l}^\pm\big|=\big|\tilde{\mu}^\pm\big| s
&\leq(1+o(1))\big(t_p \overline{C}_{\sigma_\pm}\big)^{-\frac{1}{1+\sigma}}
|\nu|^{\frac{1}{1+\sigma}} s\nonumber\\[1mm]
&=(1+o(1))\bigg(\frac{\overline{C}_{\gamma}}{\overline{C}_{\sigma_\pm}}\bigg)^{\frac{1}{1+\sigma}}
\big|x-x_p-(t-t_p)f'(c)\big|^{\frac{1}{1+\sigma}}.
\end{align}

\smallskip
{\bf (b2)}. For $t>t_p$, $\eqref{dswr4}$ becomes
\begin{align}\label{dswr6a}
\nu =-\overline{C}_{\gamma}\,\tilde{\mu}^\pm+(1+o(1))t_p\overline{C}_{\sigma_\pm}\,|\tilde{\mu}^\pm|^{\sigma}\tilde{\mu}^\pm.
\end{align}
Denote $\tilde{\mu}_0^\pm:=\pm\Big(\frac{(\overline{C}_{\gamma})^2}
{\overline{C}_{\sigma_\pm}}\Big)^{\frac{1}{\sigma}}$,
or equivalently, $\tilde{\mu}_0^\pm$ with $\pm\tilde{\mu}_0^\pm>0$ satisfy that
\begin{align*}
0=-\overline{C}_{\gamma} +t_p \overline{C}_{\sigma_\pm} \,|\tilde{\mu}_0^\pm|^{\sigma}.
\end{align*}
For
$g_\pm(\mu):=-\overline{C}_{\gamma}\,\mu+t_p\overline{C}_{\sigma_\pm}\,|\mu|^{\sigma}\mu$
for $\pm\mu>0$,
it is direct to see that
\begin{align}\label{dswr9a}
g_\pm(\mu) \qquad \mbox{ is\ strictly\ increasing\ \, for $\pm(\mu-\bar{\mu}^\pm_0)>0$},
\end{align}
where $\bar{\mu}^\pm_0$ with $\pm\bar{\mu}^\pm_0>0$ is the unique real root of $g'_\pm(\mu)=0$, $i.e.$,
\begin{align*}
0=-\overline{C}_{\gamma}+(1+\sigma)t_p\overline{C}_{\sigma_\pm}\,|\bar{\mu}^\pm_0|^{\sigma},
\qquad {\rm or \ equivalently }\quad \bar{\mu}^\pm_0=(1+\sigma)^{-\frac{1}{\sigma}}\tilde{\mu}^\pm_0.
\end{align*}

Denote $\nu_0$ by
\begin{align*}
\nu_0:
={\rm sgn}\big(\overline{C}_{\sigma_+}-\overline{C}_{\sigma_-}\big)O_1(1)
={\rm sgn}\big(\overline{C}_{\sigma_+}-\overline{C}_{\sigma_-}\big)
\overline{C}_\gamma(Q_\pm)^{\frac{1}{\sigma}} |\tilde{\mu}^\pm_0| |Q_\pm-1|.
\end{align*}
From $\eqref{dsw6a1}$,
for $\nu$ with $\pm(\nu-\nu_0)>0$,
if $(x,t)\in \chi(t,\nu)$ with small $s=t-t_p>0$, then
\begin{align}\label{dswr10aa}
\pm(x-x(t))
&=\pm \big((x-x_p-(t-t_p)f'(c)\big)-\big(x(t)-x_p-(t-t_p)f'(c))\big)\nonumber\\[1mm]
&\cong\pm(\nu s^{1+\sigma}-\nu_0 s^{1+\sigma})=\pm(\nu-\nu_0) s^{1+\sigma}>0,
\end{align}
which, by $\eqref{c2.31}$, $\eqref{dsw12}$, and $\eqref{dswr2a}$,
implies that $\pm\tilde{l}^\pm\geq \pm l^\pm$.
Hence, from $\eqref{dsw18d}$ and $\eqref{dswr3}$,
\begin{align}\label{dswr4c2}
\pm \tilde{\mu}^\pm=\pm\frac{\tilde{l}^\pm}{s}\geq\pm \frac{l^\pm}{s}
=(1+o(1))|\tilde{\mu}^\pm_0|(Q_\pm)^{\frac{1}{\sigma}}
>(1+\sigma)^{-\frac{1}{\sigma}}|\tilde{\mu}^\pm_0|=|\bar{\mu}^\pm_0|,
\end{align}
where we have used that
$Q_\pm>\frac{1}{1+\sigma}\frac{1+\gamma+\sigma}{1+\gamma}
=\frac{1}{1+\sigma}\big(1+\frac{\sigma}{1+\gamma}\big)>\frac{1}{1+\sigma}$ by $\eqref{dsw9a}$.

Furthermore, for any fixed $\nu$ with $\pm(\nu-\nu_0)>0$,
by letting $s\rightarrow 0{+}$ in $\eqref{dswr6a}$,
it is direct to see that $\tilde{\mu}^\pm$ is bounded for small $s>0$.
Thus, by $\eqref{dswr4c2}$,
any limit point of the converged subsequence of $\tilde{\mu}^\pm$
solves $\nu=g_\pm(\mu)$ on $\pm(\mu-\bar{\mu}^\pm_0)>0$ uniquely,
which means that
\begin{align}\label{dswr6b}
\lim_{s\rightarrow 0{+}}\frac{\tilde{l}^\pm}{s}=\lim_{s\rightarrow 0{+}}\tilde{\mu}^\pm=:\tilde{\mu}_0^\pm(\nu)
\end{align}
with $\tilde{\mu}_0^\pm(\nu)$ solving $\nu=g_\pm(\mu)$ on $\pm(\mu-\bar{\mu}^\pm_0)>0$ uniquely, $i.e.$,
for $\nu$ with $\pm(\nu-\nu_0)>0$,
\begin{align}\label{dswr6a1}
\nu=g_\pm(\tilde{\mu}_0^\pm(\nu))
=-\overline{C}_{\gamma}\tilde{\mu}_0^\pm(\nu)
+t_p \overline{C}_{\sigma_\pm} \,|\tilde{\mu}_0^\pm(\nu)|^{\sigma}\tilde{\mu}_0^\pm(\nu)
\qquad\, {\rm with}\ \pm(\tilde{\mu}_0^\pm(\nu)-\bar{\mu}^\pm_0)>0.
\end{align}

On the other hand, for any $\delta\in(0,1)$ with $\delta>\max\{1-(Q_\pm)^{-1}\}$,
let $\tilde{\mu}_\delta^\pm$ be the unique real root of the following equation:
\begin{align*}
g_\pm(\mu)=-\overline{C}_{\gamma}\,\mu
+t_p\overline{C}_{\sigma_\pm}\,|\mu|^{\sigma}\mu
=\delta t_p\overline{C}_{\sigma_\pm}\,|\mu|^{\sigma}\mu
=:g^{\delta}_\pm(\mu)
\qquad\, {\rm for}\ \pm\mu>0.
\end{align*}
Then it is direct to check that
$\tilde{\mu}_\delta^\pm=(1-\delta)^{-\frac{1}{\sigma}}\tilde{\mu}_0^\pm.$
Since $\delta>\max\{1-(Q_\pm)^{-1}\}$, then
\begin{align}\label{dswr11d}
\pm\tilde{\mu}_\delta^\pm
=\pm(1-\delta)^{-\frac{1}{\sigma}}\tilde{\mu}_0^\pm
=(1-\delta)^{-\frac{1}{\sigma}}|\tilde{\mu}_0^\pm|
>(Q_\pm)^{\frac{1}{\sigma}}|\tilde{\mu}_0^\pm|
>(1+\sigma)^{-\frac{1}{\sigma}}|\tilde{\mu}^\pm_0|
=|\bar{\mu}^\pm_0|.
\end{align}
See Fig. $\ref{figDSWS}$ for the positional relations of $\bar{\mu}^\pm_0$, $\tilde{\mu}^\pm_0$,
$\hat{\mu}_0^\pm:=(Q_\pm)^{\frac{1}{\sigma}}\tilde{\mu}_0^\pm$, and $\tilde{\mu}_\delta^\pm$.

\begin{figure}[H]
	\begin{center}
		{\includegraphics[width=0.7\columnwidth]{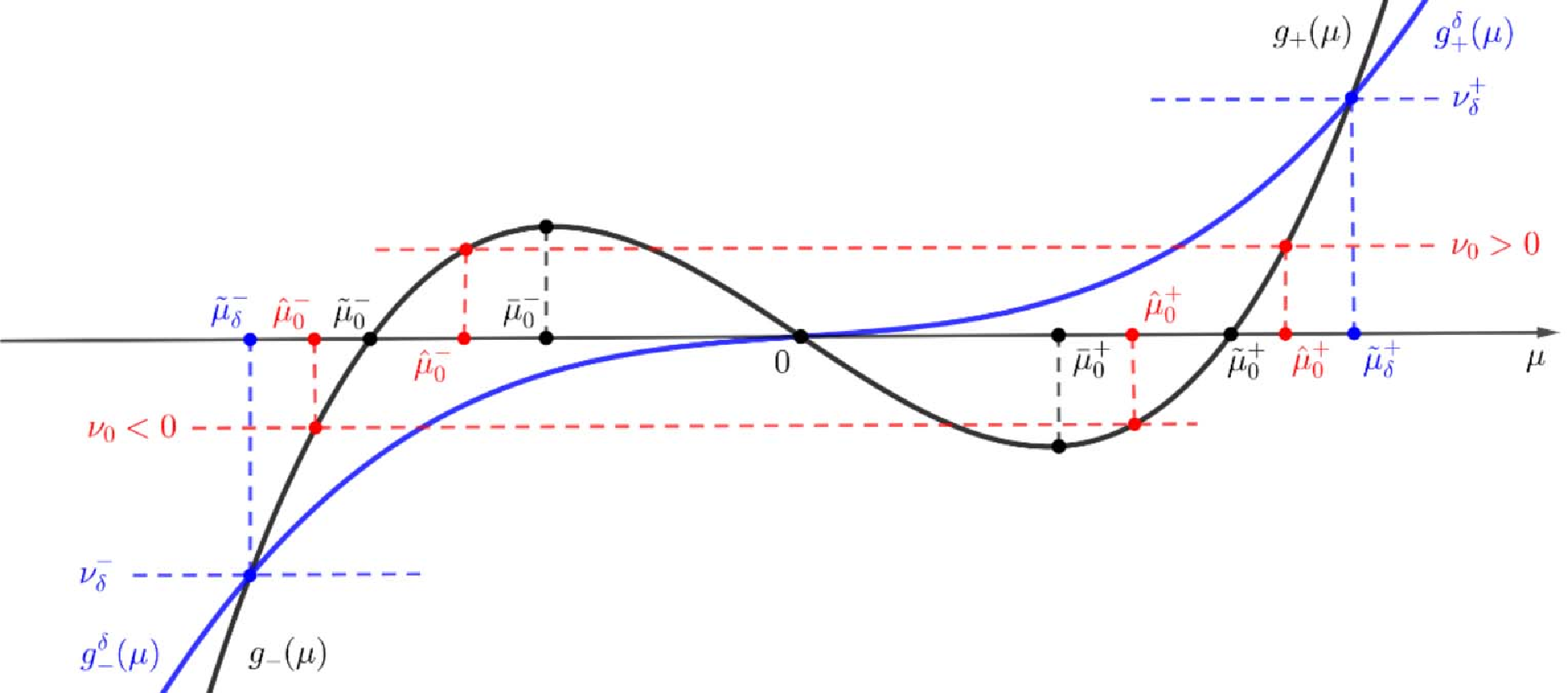}}
 \caption{The positional relations of
 $\bar{\mu}^\pm_0$, $\tilde{\mu}^\pm_0$,
$\hat{\mu}_0^\pm$, and $\tilde{\mu}_\delta^\pm$.
 }\label{figDSWS}
	\end{center}
\end{figure}

For $\delta\in(0,1)$ with $\delta>\max\{1-(Q_\pm)^{-1}\}$, denote $\nu_\delta^\pm$ by
\begin{align}\label{dswr11b}
\nu_\delta^\pm:=g_\pm(\tilde{\mu}_\delta^\pm)=\delta t_p\overline{C}_{\sigma_\pm}\,|\tilde{\mu}_\delta^\pm|^{\sigma}\tilde{\mu}_\delta^\pm
=g_\pm^\delta(\tilde{\mu}_\delta^\pm).
\end{align}
Then it follows from
$\eqref{dswr11d}$--$\eqref{dswr11b}$ that
\begin{align}\label{dswr11e}
\pm\nu_\delta^\pm
&=\pm\delta t_p\overline{C}_{\sigma_\pm}\,|\tilde{\mu}_\delta^\pm|^{\sigma}\tilde{\mu}_\delta^\pm
=\delta t_p\overline{C}_{\sigma_\pm}\,|\tilde{\mu}_\delta^\pm|^{1+\sigma}
>\delta t_p\overline{C}_{\sigma_\pm}Q_\pm(Q_\pm)^{\frac{1}{\sigma}}\,|\tilde{\mu}_0^\pm|^{1+\sigma}\nonumber\\[1mm]
&=(\delta Q_\pm)\big(t_p\overline{C}_{\sigma_\pm}|\tilde{\mu}_0^\pm|^\sigma\big)(Q_\pm)^{\frac{1}{\sigma}}\,|\tilde{\mu}_0^\pm|
>(Q_\pm-1)\overline{C}_{\gamma}(Q_\pm)^{\frac{1}{\sigma}}\,|\tilde{\mu}_0^\pm|\nonumber\\[1mm]
&=\pm{\rm sgn}\big(\overline{C}_{\sigma_+}-\overline{C}_{\sigma_-}\big)
\overline{C}_\gamma(Q_\pm)^{\frac{1}{\sigma}} |\tilde{\mu}^\pm_0| |Q_\pm-1|
=\pm\nu_0,
\end{align}
where we have used that
$Q_\pm-1=\pm{\rm sgn}\big(\overline{C}_{\sigma_+}-\overline{C}_{\sigma_-}\big) |Q_\pm-1|$ by $\eqref{dsw18b}$.

Thus, for any $\nu$ with $\pm (\nu-\nu_\delta^\pm)\geq 0$,
it follows from $\eqref{dswr11e}$ that
$\pm(\nu-\nu_0)>0$, which, by $\eqref{dswr10aa}$, implies that
$\pm(x-x(t))>0$ for $(x,t)\in \chi(t,\nu)$ with small $t-t_p>0$.
Hence, for any $\nu$ with $\pm (\nu-\nu_\delta^\pm)\geq 0$,
it follows from $\eqref{dswr6a1}$ and $\eqref{dswr11b}$ that
$$\pm \big(g_\pm(\tilde{\mu}_0^\pm(\nu))-g_\pm(\tilde{\mu}_\delta^\pm)\big)
=\pm (\nu-\nu_\delta^\pm)\geq 0,$$
which, by combining $\eqref{dswr6a1}$--$\eqref{dswr11d}$ with $\eqref{dswr9a}$, implies
\begin{align*}
\pm\big(\tilde{\mu}_0^\pm(\nu)-\tilde{\mu}_\delta^\pm\big)\geq 0
\qquad\, {\rm for}\ \pm (\nu-\nu_\delta^\pm)\geq 0.
\end{align*}
Meanwhile, it follows from $\eqref{dswr11d}$ that
\begin{align*}
\pm g_\pm(\mu)\geq\pm \delta t_p\overline{C}_{\sigma_\pm}\,|\mu|^{\sigma}\mu
 \qquad {\rm for\ } \pm (\mu-\tilde{\mu}_\delta^\pm)\geq 0.
\end{align*}
Therefore, for any $\nu$ with $\pm (\nu-\nu_\delta^\pm)\geq 0$,
\begin{align*}
|\nu|=\pm \nu=\pm g_\pm(\tilde{\mu}_0^\pm(\nu))
\geq\pm \delta t_p\overline{C}_{\sigma_\pm}\,|\tilde{\mu}_0^\pm(\nu)|^{\sigma}\tilde{\mu}_0^\pm(\nu)
=\delta t_p\overline{C}_{\sigma_\pm}\,|\tilde{\mu}_0^\pm(\nu)|^{1+\sigma},
\end{align*}
which implies
\begin{align}\label{dswr12c}
|\tilde{\mu}_0^\pm(\nu)|\leq \delta^{-\frac{1}{1+\sigma}}
(t_p\overline{C}_{\sigma_\pm})^{-\frac{1}{1+\sigma}}|\nu|^{\frac{1}{1+\sigma}}
 \qquad {\rm for\ } \pm (\nu-\nu_\delta^\pm)\geq 0.
\end{align}

Combining all the steps above, we conclude that, for the case that $t>t_p$,
\begin{itemize}
\item [(i)] If $\pm (\nu-\nu_\delta^\pm)\geq 0$,
then $\pm(x-x(t))>0$ for $(x,t)\in \chi(t,\nu)$ with small $t-t_p>0$.
Then, by $\eqref{dswr1}$ and $\eqref{dswr3}$,
it follows from $\eqref{dswr6b}$ and $\eqref{dswr12c}$ that,
as $t\rightarrow t_p{+}$,
\begin{align}\label{dswr13}
|\tilde{l}^\pm|=|\tilde{l}^\pm(t,\nu)|
&=(1+o(1))|\tilde{\mu}_0^\pm(\nu)|s
\leq(1+o(1))\delta^{-\frac{1}{1+\sigma}}
(t_p\overline{C}_{\sigma_\pm})^{-\frac{1}{1+\sigma}}
|\nu|^{\frac{1}{1+\sigma}} s\nonumber\\[1mm]
&=(1+o(1))\delta^{-\frac{1}{1+\sigma}}
(t_p\overline{C}_{\sigma_\pm})^{-\frac{1}{1+\sigma}}
|x-x_p-(t-t_p)f'(c)|^{\frac{1}{1+\sigma}}\nonumber\\[1mm]
&=(1+o(1))\delta^{-\frac{1}{1+\sigma}}
\bigg(\frac{\overline{C}_{\gamma}}{\overline{C}_{\sigma_\pm}}\bigg)^{\frac{1}{1+\sigma}}
|x-x_p-(t-t_p)f'(c)|^{\frac{1}{1+\sigma}}.
\end{align}
For point $(x,t)$ on $t=t_p$ with small $\pm\tilde{l}_p^\pm=\pm(x-x_p)>0$,
similar to $\eqref{dsw12a}$,
\begin{align*}
|\tilde{l}^\pm|
=(1+o(1))\bigg(\frac{\overline{C}_{\gamma}}{\overline{C}_{\sigma_\pm}}\bigg)^{\frac{1}{1+\sigma}}
|\tilde{l}_p^\pm|^{\frac{1}{1+\sigma}}
=(1+o(1))\bigg(\frac{\overline{C}_{\gamma}}{\overline{C}_{\sigma_\pm}}\bigg)^{\frac{1}{1+\sigma}}
|x-x_p|^{\frac{1}{1+\sigma}}.
\end{align*}

\item [(ii)] If $\nu\in [\nu_\delta^-,\nu_\delta^+]$,
by $\eqref{c2.31}$,
it follows from $\eqref{dswr6b}$--$\eqref{dswr11b}$ that,
for $(x,t)\in \chi(t,\nu)$ with $\pm(x-x(t))>0$,
as $t\rightarrow t_p{+}$,
\begin{align}\label{dswr13a}
|\tilde{l}^\pm|=|\tilde{l}^\pm(t,\nu)|
&\leq |\tilde{l}^\pm(t,\nu_\delta^\pm)|
=(1+o(1))|\tilde{\mu}_0^\pm(\nu_\delta^\pm)|s
=(1+o(1))|\tilde{\mu}_\delta^\pm|s\nonumber\\[1mm]
&=(1+o(1))(1-\delta)^{-\frac{1}{\sigma}}
\bigg(\frac{(\overline{C}_{\gamma})^2}{\overline{C}_{\sigma_\pm}}\bigg)^{\frac{1}{\sigma}}
|t-t_p|^{\frac{1}{\sigma}}.
\end{align}
For point $(x,t)$ on $x=x(t)$,
since $\delta>\max\{1-(Q_\pm)^{-1}\}$, then $Q_\pm\leq (1-\delta)^{-1}$, and
it follows from $\eqref{dsw18d}$ that $\eqref{dswr13a}$ holds.
\end{itemize}

{\bf (c).} We now prove $\eqref{dsw6a}$.
By $\eqref{uiff}$, it follows from $\eqref{dsw3a}$ and $\eqref{dswr2a}$ that,
as $t\rightarrow t_p$,
\begin{align*}
|u(x,t)-c|=|\varphi(x_0+\tilde{l}^\pm)-c|
=|C_{\gamma}+o(1)|\, |\tilde{l}^\pm|^{\gamma}.
\end{align*}
Then, for any $\delta\in(0,1)$ with $\delta>\max\{1-(Q_\pm)^{-1}\}$,
$\eqref{dsw6a}$ can be established as follows:
\begin{itemize}
\item [(i)] For $t\geq t_p$, if $(x,t)\in \Omega_\pm$ near $(x_p,t_p)$ with
\begin{align*}
\pm \big( x-x_p-(t-t_p)f'(c)-\nu_{\delta}^\pm |t-t_p|^{\frac{1+\sigma}{\sigma}}\big)\geq 0,
\end{align*}
then it follows from $\eqref{dswr13}$ that
\begin{align}\label{dswr8a2}
\,\,\,\,\,\,|u(x,t)-c|
\leq (1+o(1))\delta^{-\frac{\gamma}{1+\sigma}}|C_{\gamma}|
\bigg(\frac{\overline{C}_{\gamma}}{\overline{C}_{\sigma_\pm}}\bigg)^{\frac{\gamma}{1+\sigma}}
|x-x_p-(t-t_p)f'(c)|^{\frac{\gamma}{1+\sigma}};
\end{align}
and, if $(x,t)\in \Omega_\pm$ near $(x_p,t_p)$ with
\begin{align*}
\nu_{\delta}^-|t-t_p|^{\frac{1+\sigma}{\sigma}}
\leq x-x_p-(t-t_p)f'(c)
\leq \nu_{\delta}^+|t-t_p|^{\frac{1+\sigma}{\sigma}},
\end{align*}
then it follows from $\eqref{dswr13a}$ that
\begin{align*}
|u(x,t)-c|
\leq (1+o(1))(1-\delta)^{-\frac{\gamma}{\sigma}}|C_{\gamma}|
\bigg(\frac{(\overline{C}_{\gamma})^2}{\overline{C}_{\sigma_\pm}}\bigg)^{\frac{\gamma}{\sigma}}
|t-t_p|^{\frac{\gamma}{\sigma}}.
\end{align*}

\item [(ii)] For $t<t_p$, if $(x,t)\in \Omega_\pm$ near $(x_p,t_p)$, by $\delta\in(0,1)$,
$\eqref{dswr8a2}$ follows from $\eqref{dswr7b}$.
\end{itemize}

For the case that $\overline{C}_{\sigma_\pm}=\overline{C}_{\sigma}$,
it follows from $\eqref{dsw9a}$ and $\eqref{dsw11a}$ that $Q_\pm=1$.
Then, for any $\delta\in(0,1)$,
$\eqref{dsw6a}$ follows immediately
via replacing $\overline{C}_{\sigma_\pm}$ by $\overline{C}_{\sigma}$.

\smallskip
\noindent
{\bf Case II.} Suppose that $0<t_p^+<t_p^-\in(0,\infty]$.
By Theorem $\ref{the:c4.3}$, $c=\underline{D}_+\Phi(x_0)$.

Denote $u^\pm(t):=u(x(t)\pm,t)$ for $t>t_p$.
From Corollary $\ref{cor:c2.2}$, $u^\pm(t)\in \mathcal{U}(x(t),t;t_p)$ so that,
by $\eqref{c2.57}$--$\eqref{c2.61}$, $E(u^+(t);x(t),t;t_p)=E(u^-(t);x(t),t;t_p)$.
By the same arguments as for {\bf Case I},
it is direct to check that $\eqref{dsw10}$--$\eqref{dsw10e}$ hold.

\smallskip
\noindent
{\bf 1.} We first consider the case that
$c=\underline{D}_+\Phi(x_0)=\overline{D}_-\Phi(x_0)=a$.
Similar to $\eqref{dsw12}$--$\eqref{dsw12c}$,
from $\eqref{dsw4a}$--$\eqref{dsw4b}$,
it can be checked from $\eqref{dsw10a}$ and $\eqref{dsw12}$ that,
\begin{align}\label{dswn12a}
l_p^+=(1+o(1))\frac{\overline{C}_{\sigma_+}}{\overline{C}_{\gamma_+}}|l^+|^{1+\sigma_+}
\qquad {\rm as}\ t\rightarrow t_p{+},
\end{align}
and, for $l_p^-$, by $\eqref{aa2}$,
it follows from $\eqref{dsw2}$ and $\eqref{dsw4c}$--$\eqref{dsw4c1}$ that
\begin{align*}
f'(\varphi(x_0+l^-))-f'(c)
&=-{\rm sgn} (C_{\gamma_-})(1+o(1))\frac{N}{1+\alpha}|\varphi(x_0+l^-)-c|^{1+\alpha}\nonumber\\[1mm]
&=-{\rm sgn} (C_{\gamma_-})(1+o(1))\frac{N}{1+\alpha}|C_{\gamma_-}|^{1+\alpha}
|l^-|^{\gamma_-(1+\alpha)}\nonumber\\[1mm]
&=-{\rm sgn} (C_{\gamma_-})(1+o(1))\overline{C}_{\gamma_-}|l^-|^{\gamma_-(1+\alpha)},
\end{align*}
which, by $\eqref{dsw10a}$ and $\eqref{dsw12}$, implies
\begin{align}\label{dswn12a2}
l_p^-&=l^-+t_p\big(f'(u^-(t))-f'(c)\big)
=l^-+t_p\big(f'(\varphi(x_0+l^-))-f'(c)\big)\nonumber\\[1mm]
&=l^--{\rm sgn} (C_{\gamma_-})(1+o(1))t_p\overline{C}_{\gamma_-}|l^-|^{\gamma_-(1+\alpha)}
=-(1+o(1))O^-_4(1)|l^-|^{\varrho_0},
\end{align}
where $\varrho_0:=\min\{\gamma_-(1+\alpha),1\}$, and
$O_4^-(1)$ is given by $\eqref{dsw9c}$.

Then, similar to $\eqref{dsw12c}$, it follows from $\eqref{dswn12a}$--$\eqref{dswn12a2}$ that,
as $t\rightarrow t_p{+}$,
\begin{align}\label{dswn12c}
\begin{cases}
u^+(t)-c
\cong-|C_{\gamma_+}|
\Big(\frac{\overline{C}_{\gamma_+}}{\overline{C}_{\sigma_+}}\Big)^{\frac{\gamma_+}{1+\sigma_+}} |l_p^+|^{\frac{\gamma_+}{1+\sigma_+}}
=:-O_+(1) |l_p^+|^{\frac{\gamma_+}{1+\sigma_+}},\\[2mm]
u^-(t)-c
\cong -{\rm sgn}(C_{\gamma_-})|C_\gamma|
\big((O_4^-(1))^{-1}|l_p^-|\big)^{\frac{\gamma_-}{\varrho_0}}
=:-{\rm sgn}(C_{\gamma_-})O_-(1) |l_p^-|^{\varrho},
\end{cases}
\end{align}
where $\varrho=\max\{\gamma_-,\gamma_+\}=\frac{\gamma_-}{\varrho_0}$.

\smallskip
\noindent
{\bf 2.} We now prove that
\begin{align}\label{dswn11a}
\tilde{\lambda}_1:=\lim_{t\rightarrow t_p{+}}\tilde{\lambda}(t)
=\frac{1+\gamma_++\sigma_+}{\gamma_+\sigma_+}\,
\frac{|C_{\gamma_+}|^{\frac{1+\sigma_+}{\gamma_+}}}{|C_{\gamma_-}|^{\frac{1}{\varrho}}}\,
\frac{\overline{C}_{\gamma_+}}{\overline{C}_{\sigma_+}}\, O_4^-(1),
\end{align}
where $A_\pm:=A_\pm(t)$ and $\tilde{\lambda}:=\tilde{\lambda}(t)$ are given by
\begin{align}\label{dswn11}
A_\pm=A_\pm(t):=|u^\pm(t)-c|,\qquad
\tilde{\lambda}=\tilde{\lambda}(t)
:=\big(A_+(t)\big)^{\frac{1+\sigma_+}{\gamma_+}}
\big(A_-(t)\big)^{-\frac{1}{\varrho}}.
\end{align}

\smallskip
{\bf (a).} Similar to $\eqref{dsw12e}$,
it follows from $\eqref{dswn12c}$ and $\eqref{dswn11}$ that
\begin{align}\label{dswn12e}
\begin{cases}
\displaystyle\frac{x(t)-t_p}{t-t_p}-f'(u^+(t))=\frac{l_p^+}{t-t_p}
=\frac{1+o(1)}{t-t_p}\Big(\frac{A_+}{O_+(1)}\Big)^{\frac{1+\sigma_+}{\gamma_+}},\\[4mm]
\displaystyle\frac{x(t)-t_p}{t-t_p}-f'(u^-(t))=\frac{l_p^-}{t-t_p}
=-\frac{1+o(1)}{t-t_p}\Big(\frac{A_-}{O_-(1)}\Big)^{\frac{1}{\varrho}}.
\end{cases}
\end{align}
Similar to $\eqref{dsw13}$, we have
\begin{align*}
U_+=-(1+o(1))\frac{1+\sigma_+}{1+\gamma_++\sigma_+}A_+,\qquad
U_-=-{\rm sgn}(C_{\gamma_-})\frac{1+o(1)}{1+\varrho}A_-,
\end{align*}
which, by plugging $\eqref{dsw13a}$--$\eqref{dsw13b}$ into $\eqref{dsw10d}$, implies
\begin{align}\label{dswn14}
&\frac{1+o(1)}{t-t_p}\Big(\frac{A_+}{O_+(1)}\Big)^{\frac{1+\sigma_+}{\gamma_+}}\nonumber\\[1mm]
&\,\,\,\,=\frac{-\frac{N+o(1)}{2+\alpha}\big((A_-)^{2+\alpha}-(A_+)^{2+\alpha}\big)
-{\rm sgn}(C_{\gamma_-})\frac{1+o(1)}{1+\varrho}\frac{N+o(1)}{1+\alpha}
A_-\big((A_-)^{1+\alpha}+(A_+)^{1+\alpha}\big)}
{-{\rm sgn}(C_{\gamma_-})\frac{1+o(1)}{1+\varrho}A_-
+(1+o(1))\frac{1+\sigma_+}{1+\gamma_++\sigma_+}A_+},\\[2mm]
&-\frac{1+o(1)}{t-t_p}\Big(\frac{A_-}{O_-(1)}\Big)^{\frac{1}{\varrho}}\nonumber\\[1mm]
&\,\,\,\,=\frac{-\frac{N+o(1)}{2+\alpha}\big((A_-)^{2+\alpha}-(A_+)^{2+\alpha}\big)
-\frac{1+\sigma_+}{1+\gamma_++\sigma_+}\frac{N+o(1)}{1+\alpha}
A_+\big((A_-)^{1+\alpha}+(A_+)^{1+\alpha}\big)}
{-{\rm sgn}(C_{\gamma_-})\frac{1+o(1)}{1+\varrho}A_-
+(1+o(1))\frac{1+\sigma_+}{1+\gamma_++\sigma_+}A_+}.\label{dswn14aa}
\end{align}
Then, dividing $\eqref{dswn14}$ by $\eqref{dswn14aa}$, we obtain
\begin{align}\label{dswn14a}
&\ \frac{1+\varrho}{1+\gamma_++\sigma_+}
\frac{(O_-(1))^{\frac{1}{\varrho}}}{(O_+(1))^{\frac{1+\sigma_+}{\gamma_+}}}\,
\tilde{\lambda}\nonumber\\[1mm]
&\cong
\frac{-[(1+\varrho)+{\rm sgn}(C_{\gamma_-})(1+\gamma_+)](A_-)^{2+\alpha}
-{\rm sgn}(C_{\gamma_-})(1+\gamma_+)A_-(A_+)^{1+\alpha}+(1+\varrho)(A_+)^{2+\alpha}}
{(1+\gamma_++\sigma_+)(A_-)^{2+\alpha}+(1+\gamma_+)(1+\sigma_+)A_+(A_-)^{1+\alpha}
+\gamma_+\sigma_+(A_+)^{2+\alpha}}.
\end{align}

\smallskip
{\bf (b).} We now prove $A_-=o(1)A_+$ by contradiction.
Otherwise, there are two cases:
\begin{itemize}
\item [(i)] If $A_-=O(1)A_+$, it follows from $\eqref{dswn14aa}$ that,
as $t\rightarrow t_p{+}$,
\begin{align*}
(A_-)^{\frac{1}{\varrho}}=O(1)(A_+)^{1+\alpha}(t-t_p)=O(1)(A_-)^{1+\alpha}(t-t_p),
\end{align*}
which, by a simple calculation, leads to the following contradiction,
\begin{align*}
1=O(1)(t-t_p)^\varrho(A_-)^{\varrho(1+\alpha)-1}=o(1)
\qquad {\rm as}\ t\rightarrow t_p{+},
\end{align*}
where the last step is due to the fact that $\varrho(1+\alpha)-1\geq 0$,
inferred by $\gamma_+(1+\alpha)=1$ and $\varrho=\max\{\gamma_-,\gamma_+\}$.

\item [(ii)] If $A_+=o(1)A_-$, from $\eqref{dsw4c1}$, there exist two subcases:
For the case that $C_{\gamma_-}>0$, or $C_{\gamma_-}<0$ with $\gamma_-(1+\alpha)>1$,
we see that
$$
(1+\varrho)+{\rm sgn}(C_{\gamma_-})(1+\gamma_+)>0
$$
so that the left side of $\eqref{dswn14a}$ is positive
and the right side of $\eqref{dswn14a}$ is negative, which is a contradiction.
For the case that $C_{\gamma_-}<0$ with $\gamma_-(1+\alpha)=1$,
from $\eqref{dswn14a}$,
\begin{align*}
\frac{1+\varrho}{1+\gamma_++\sigma_+}
\frac{(O_-(1))^{\frac{1}{\varrho}}}{(O_+(1))^{\frac{1+\sigma_+}{\gamma_+}}}\,
\frac{(A_+)^{\frac{1+\sigma_+}{\gamma_+}}}{(A_-)^{\frac{1}{\varrho}}}
\cong \frac{(1+\gamma_+)(A_+)^{1+\alpha}}{(1+\gamma_++\sigma_+)(A_-)^{1+\alpha}},
\end{align*}
which, by a simple calculation, implies that $A_+=O(1)$.
This contradicts to the fact that $\lim_{t\rightarrow t_p{+}}u^+(t)=c$.
\end{itemize}
Therefore, $A_-=o(1)A_+$.
Taking $A_-=o(1)A_+$ into $\eqref{dswn14a}$ and using $\eqref{dswn12c}$ lead to
\begin{align*}
\tilde{\lambda}=(1+o(1))\frac{1+\gamma_++\sigma_+}{\gamma_+\sigma_+}
\frac{(O_+(1))^{\frac{1+\sigma_+}{\gamma_+}}}{(O_-(1))^{\frac{1}{\varrho}}}
=(1+o(1))\frac{1+\gamma_++\sigma_+}{\gamma_+\sigma_+}
\frac{|C_{\gamma_+}|^{\frac{1+\sigma_+}{\gamma_+}}}{|C_{\gamma_-}|^{\frac{1}{\varrho}}}\,
\frac{\overline{C}_{\gamma_+}}{\overline{C}_{\sigma_+}}\,
O_4^-(1),
\end{align*}
which, by letting $t\rightarrow t_p{+}$, implies $\eqref{dswn11a}$, as desired.

\smallskip
\noindent
{\bf 3.} We now prove $\eqref{dsw7}$ with $O^+_3(1)$ in $\eqref{dsw9d}$.

\smallskip
{\bf (a).} By taking $A_-=o(1) A_+$ into $\eqref{dswn14}$, it is direct to check that,
as $t\rightarrow t_p{+}$,
\begin{align*}
A_+=(1+o(1))|C_{\gamma_+}|
\bigg(\frac{1+\gamma_++\sigma_+}{(1+\gamma_+)(1+\sigma_+)}\bigg)^{\frac{\gamma_+}{\sigma_+}}
\bigg(\frac{\big(\overline{C}_{\gamma_+}\big)^2}{\overline{C}_{\sigma_+}}\bigg)^{\frac{\gamma_+}{\sigma_+}}
\,(t-t_p)^{\frac{\gamma_+}{\sigma_+}},
\end{align*}
which, by $\eqref{dswn11a}$--$\eqref{dswn11}$, implies
\begin{align*}
A_-&=(1+o(1))|C_{\gamma_-}|
\bigg(\frac{\gamma_+\sigma_+}{(1+\gamma_+)(1+\sigma_+)}\bigg)^{\varrho}
\bigg(\frac{1+\gamma_++\sigma_+}{(1+\gamma_+)(1+\sigma_+)}\bigg)^{\frac{\varrho}{\sigma_+}}
\nonumber\\[2mm]
&\quad \, \times
\bigg(\frac{\overline{C}_{\gamma_+}}{O_4^-(1)}\bigg)^{\varrho}
\bigg(\frac{\big(\overline{C}_{\gamma_+}\big)^2}{\overline{C}_{\sigma_+}}\bigg)^{\frac{\varrho}{\sigma_+}}
\,(t-t_p)^{\frac{\varrho(1+\sigma_+)}{\sigma_+}}.
\end{align*}
From $\eqref{dswn12e}$,
as $t\rightarrow t_p{+}$,
\begin{align*}
|l_p^-|\cong\bigg(\frac{A_-}{O_-(1)}\bigg)^{\frac{1}{\varrho}}
\cong\frac{\gamma_+\sigma_+}{1+\gamma_++\sigma_+}
\bigg(\frac{A_+}{O_+(1)}\bigg)^{\frac{1+\sigma_+}{\gamma_+}}
\cong\frac{\gamma_+\sigma_+}{1+\gamma_++\sigma_+}|l_p^+|.
\end{align*}
Then, by using $\eqref{dswn12c}$, it is direct to check that,
as $t\rightarrow t_p{+}$,
\begin{align}
&l_p^+=(1+o(1))\overline{C}_{\gamma_+}
\bigg(\frac{1+\gamma_++\sigma_+}{(1+\gamma_+)(1+\sigma_+)}\bigg)^{\frac{1+\sigma_+}{\sigma_+}}
\bigg(\frac{(\overline{C}_{\gamma_+})^2}{\overline{C}_{\sigma_+}}\bigg)^{\frac{1}{\sigma_+}}
\,(t-t_p)^{\frac{1+\sigma_+}{\sigma_+}},\label{dswn3c}\\[2mm]
&l_p^-=(1+o(1))\overline{C}_{\gamma_+}
\frac{\gamma_+\sigma_+}{(1+\gamma_+)(1+\sigma_+)}
\bigg(\frac{1+\gamma_++\sigma_+}{(1+\gamma_+)(1+\sigma_+)}\bigg)^{\frac{1}{\sigma_+}}
\bigg(\frac{(\overline{C}_{\gamma_+})^2}{\overline{C}_{\sigma_+}}\bigg)^{\frac{1}{\sigma_+}}
\,(t-t_p)^{\frac{1+\sigma_+}{\sigma_+}}, \label{dswn3c1}
\end{align}
which, by $\eqref{dswn12a}$--$\eqref{dswn12a2}$, implies that,
as $t\rightarrow t_p{+}$,
\begin{align}
l^+&=(1+o(1))
\bigg(\frac{1+\gamma_++\sigma_+}{(1+\gamma_+)(1+\sigma_+)}\bigg)^{\frac{1}{\sigma_+}}
\bigg(\frac{(\overline{C}_{\gamma_+})^2}{\overline{C}_{\sigma_+}}\bigg)^{\frac{1}{\sigma_+}}
\,(t-t_p)^{\frac{1}{\sigma_+}},\label{dswn3d}\\[2mm]
l^-&=-(1+o(1))
\bigg(\frac{\gamma_+\sigma_+}{(1+\gamma_+)(1+\sigma_+)}\bigg)^{\frac{1}{\varrho_0}}
\bigg(\frac{1+\gamma_++\sigma_+}{(1+\gamma_+)(1+\sigma_+)}\bigg)^{\frac{1}{\varrho_0\sigma_+}}\nonumber\\[2mm]
&\quad \,\,\,\,\,\,\times \bigg(\frac{\overline{C}_{\gamma_+}}{O_4^-(1)}\bigg)^{\frac{1}{\varrho_0}}
\bigg(\frac{(\overline{C}_{\gamma_+})^2}{\overline{C}_{\sigma_+}}\bigg)^{\frac{1}{\varrho_0\sigma_+}}
\,(t-t_p)^{\frac{1+\sigma_+}{\varrho_0\sigma_+}}.\label{dswn3d1}
\end{align}

{\bf (b).} By $\eqref{dswn3c}$ and $\eqref{dswn3d}$, it follows from $\eqref{dsw19}$ that,
as $t\rightarrow t_p{+}$,
\begin{align*}
x(t)-x_p-(t-t_p)f'(c)&=(1+o(1))l_p^+-\overline{C}_{\gamma_+}(t-t_p)l^+\nonumber\\[1mm]
&=(1+o(1))\bigg(\frac{1+\gamma_++\sigma_+}{(1+\gamma_+)(1+\sigma_+)}-1\bigg)
\overline{C}_{\gamma_+}(t-t_p)l^+\nonumber\\[1mm]
&=-(1+o(1))O_3^+(1)\,(t-t_p)^{\frac{1+\sigma_+}{\sigma_+}}.
\end{align*}
Equivalently, this can also be seen as follows:
Since $\varrho_0=\min\{\gamma_-(1+\alpha),1\}\leq 1$,
it follows from $\eqref{dswn3c1}$ and $\eqref{dswn3d1}$ that
\begin{align*}
l^-=O(1)l_p^- \quad {\rm if}\ \gamma_-(1+\alpha)\geq 1, 
\qquad\,\, l^-=o(1)l_p^- \quad {\rm if}\ \gamma_-(1+\alpha)< 1,
\end{align*}
which, by $\eqref{dsw19}$, implies that, as $t\rightarrow t_p{+}$,
\begin{align*}
x(t)-x_p-(t-t_p)f'(c)
&=l^-_p+(t-t_p)(f'(u^-(t))-f'(c))\nonumber\\[1mm]
&=l_p^--\overline{C}_{\gamma_+}(t-t_p)(l^--(1+o(1))l_p^-)\nonumber\\
&=(1+o(1))l_p^-
=-(1+o(1))O_3^+(1)\,(t-t_p)^{\frac{1+\sigma_+}{\sigma_+}}.
\end{align*}

\smallskip
\noindent
{\bf 4.} For $(x,t)\in \Omega_+$ near $(x_p,t_p)$,
since the proof of $\eqref{dsw6a}$ in {\bf Case I} is established for $\Omega_+$ and $\Omega_-$ separately,
then it is direct to see that both $\eqref{dsw6a}$
with $\overline{C}_{\sigma_+}$
and $\delta>1-(Q^+_3)^{-1}$
hold.
For $(x,t)\in \Omega_-$,
by the same arguments as those for $\eqref{dswn12a}$--$\eqref{dswn12c}$,
it follows from $\eqref{dswr2a}$ that,
as $(x,t)\rightarrow (x_p,t_p)$
in $\Omega_-$,
\begin{align*}
x-x_p-(t-t_p)f'(c)&=\tilde{l}^-+t\big(f'(\varphi(x_0+\tilde{l}^-))-f'(c)\big)\nonumber\\[1mm]
&=\tilde{l}^--{\rm sgn} (C_{\gamma_-})(1+o(1))t_p\overline{C}_{\gamma_-}|\tilde{l}^-|^{\gamma_-(1+\alpha)}
=-(1+o(1))O^-_4(1)|\tilde{l}^-|^{\varrho_0},
\end{align*}
which, by $\varrho=\max\{\gamma_-,\gamma_+\}=\frac{\gamma_-}{\varrho_0}$, implies
\begin{align}\label{dswn5a}
u(x,t)-c&=u(x_p+l_p^-,t_p)-c=\varphi(x_0+\tilde{l}^-)-c
=-(C_{\gamma_-}+o(1))|\tilde{l}^-|^{\gamma_-}\nonumber\\
&=-(1+o(1))C_{\gamma_-}\big((O_4^-(1))^{-1}
|x-x_p-(t-t_p)f'(c)|\big)^{\frac{\gamma_-}{\varrho_0}}\nonumber\\[1mm]
&=-(1+o(1))C_{\gamma_-}(O_4^-(1))^{-\varrho}
|x-x_p-(t-t_p)f'(c)|^{\varrho}.
\end{align}

\smallskip
\noindent
{\bf 5.} For the case that $c=\underline{D}_+\Phi(x_0)>\overline{D}_-\Phi(x_0)=a$,
the behavior of solution $u=u(x,t)$ near $(x_p,t_p)$ is quite similar
to the case that $C_{\gamma_-}>0$ and $\gamma_-(1+\alpha)=1$.
To prove $\eqref{dsw7}$--$\eqref{dsw7a}$ for the case that $c>a$,
it suffices to establish $\eqref{dswn12c}$ for $u^-(t)$ and $\eqref{dswn5a}$.

In fact, by the continuity of solution $u=u(x,t)$ at point $(x_p,t_p)$,
the forward generalized characteristic of the point lying on characteristic $L_a(x_0)$
for the case that $a\in \mathcal{C}(x_0)$,
or the shock curve emitting from $x_0$ for the case that $a\notin \mathcal{C}(x_0)$,
must lie on the left of point $(x_p,t_p)$.
This means that, for $(x,t)\in \Omega_-$ near $(x_p,t_p)$,
$x_0=x-tf'(u(x,t))$ so that
\begin{align*}
f'(u(x,t))-f'(c)&=\frac{x-x_0}{t}-f'(c)
=\frac{1}{t}\big(x-x_p-(t-t_p)f'(c)\big)\nonumber\\
&=(1+o(1))\overline{C}_{\gamma_+}\big(x-x_p-(t-t_p)f'(c)\big)\nonumber\\[1mm]
&=-(1+o(1))\overline{C}_{\gamma_+}|x-x_p-(t-t_p)f'(c)|.
\end{align*}
Then, instead of $\eqref{dswn12c}$,
it follows from $\eqref{aa2}$ and $\gamma_+(1+\alpha)=1$ that
\begin{align*}
u^-(t)-c&=u(x_p+l_p^-,t_p)-c
\nonumber\\
&
=-\Big(\frac{1+\alpha}{N+o(1)}|f'(u(x_p+l_p^-,t_p))-f'(c)|\Big)^{\frac{1}{1+\alpha}}\nonumber\\
&=-\Big(\frac{1+\alpha}{N+o(1)}\overline{C}_{\gamma_+}|l_p^-|\Big)^{\frac{1}{1+\alpha}}
=-(1+o(1))|C_{\gamma_+}|\,|l_p^-|^{\gamma_+},
\end{align*}
and, instead of $\eqref{dswn5a}$, we have
\begin{align*}
u(x,t)-c
&=-\Big(\frac{1+\alpha}{N+o(1)}|f'(u(x,t))-f'(c)|\Big)^{\frac{1}{1+\alpha}}\nonumber\\
&=-(1+o(1))|C_{\gamma_+}|\,|x-x_p-(t-t_p)f'(c)|^{\frac{1}{1+\alpha}}\nonumber\\[1mm]
&=(1+o(1))C_{\gamma_+}|x-x_p-(t-t_p)f'(c)|^{\gamma_+}.
\end{align*}

\smallskip
\noindent
{\bf Case III.} Suppose that $\eqref{dsw5a}$--$\eqref{dsw5b}$ hold,
and in addition $\eqref{dsw5c}$--$\eqref{dsw5c1}$ hold
for the case that $\underline{D}_+\Phi(x_0)=\overline{D}_-\Phi(x_0)$.
By the same arguments as {\bf Case II},
it can be checked that $\eqref{dsw8}$ with $O^-_3(1)$ in $\eqref{dsw9d}$
and $\eqref{dsw8a}$ with $O^+_4(1)$ in $\eqref{dsw9c}$ hold.

\smallskip
Up to now, we have completed the proof of Theorem $\ref{the:dsw1}$.
\end{proof}

\begin{Rem}
For {\rm Theorems} $\ref{the:c4.3}$ and $\ref{the:dsw1}$, it is worth to point out that,
{\it the continuous shock generation points are not necessarily to be
the first blowup point of entropy solutions,
which even are not assumed to be isolated}.
Furthermore, we present the following three observations{\rm:}
\begin{itemize}
\item[{\rm (i)}] For the case that $0<t_p^+<t_p^-\in(0,\infty]$
{\rm (}{\it resp.,} $0<t_p^-<t_p^+\in(0,\infty]${\rm)},
the developed shock curve is caused by the compression of local characteristics
emitting from the right {\rm (}{\it resp.,} left{\rm)} side of $x_0$,
so is the convexity of such a shock curve.

\vspace{1pt}
\item[{\rm (ii)}] For the case that $t_p^+=t_p^-\in(0,\infty)$,
the developed shock curve is caused by the compression of local characteristics
emitting from both sides of $x_0$. Moreover, if $\overline{C}_{\sigma_+}\neq\overline{C}_{\sigma_-}$,
the strengths of the compression from the two sides of $x_0$ are different
so that the convexity of such a shock curve is determined by the magnitude of
$\overline{C}_{\sigma_+}$ and $\overline{C}_{\sigma_-}${\rm ;}
and if $\overline{C}_{\sigma_+}=\overline{C}_{\sigma_-}$,
such strengths are of the same magnitude
so that the convexity of the developed shock curve is determined by the sign of
the coefficient $\overline{C}_{\rho}$ of
the higher term $|l|^{1+\sigma+\rho}$ as in $\eqref{dsw3c}$.

\vspace{1pt}
\item[{\rm (iii)}] In addition,
$O^\pm_3(1)$ with $Q^{\pm}_3$ in $\eqref{dsw9d}$ can be viewed as the limits of
$O_1(1)$ with $Q_{\pm}$ in $\eqref{dsw9}$--$\eqref{dsw9a}$ as $(\lambda_1)^{\pm 1} \rightarrow 0{+}$, respectively.
Combining this with the continuity of $O_1(1)$ and $Q_{\pm}$ with respect to $\lambda_1$
and the continuity of $\lambda_1$ with respect to $\lambda_0$,
may indicate some exceptional stability of the developed shock curves and entropy solutions
under the perturbation of the initial data function $\varphi(x)$ near $x_0$.
\end{itemize}
\end{Rem}

\section{Fine Structures of Entropy Solutions}
In this section, we focus on the fine structures of entropy solutions.
In \S 5.1, we provide and prove
the structures of entropy solutions inside the backward characteristic triangles in Theorem $\ref{pro:c3.3}$;
in \S 5.2--5.3, we prove 
the directional limits of entropy solutions at the discontinuous points in Theorem $\ref{pro:c3.4}$
and the general structures of shocks in Theorem $\ref{pro:c3.6}$;
and in \S 5.4, we establish 
the fine structures of entropy solutions on the shock sets in Theorem $\ref{the:c4.2}$.

\subsection{New structures of entropy solutions inside the backward characteristic triangles}
By the definition of ${\mathcal U}(x,t)$ in $\eqref{c2.3}$,
for any $(x_0,t_0)\in \mathbb{R} \times \mathbb{R}^+$,
${\mathcal U}(x_0,t_0)$ is a bounded closed set such that
\begin{equation*}
{\mathcal U}(x_0,t_0)\subset [u(x_0{+},t_0),u(x_0{-},t_0)]=:I(x_0,t_0)\subset \mathbb{R}.
\end{equation*}
Let $\mathring{{\mathcal U}}(x_0,t_0)$ and $\partial {\mathcal U}(x_0,t_0)$
be the interior and the boundary of ${\mathcal U}(x_0,t_0)$ respectively,
and let ${\mathcal U}^c(x_0,t_0):=I(x_0,t_0)-{\mathcal U}(x_0,t_0)$
be the complement of ${\mathcal U}(x_0,t_0)$ under $I(x_0,t_0)$.
Then
\begin{equation}\label{c3.02}
I(x_0,t_0)=\partial {\mathcal U}(x_0,t_0)\cup \mathring{{\mathcal U}}(x_0,t_0) \cup {\mathcal U}^c(x_0,t_0).
\end{equation}
Since ${\mathcal U}(x_0,t_0)$ is a closed set,
then $\mathring{{\mathcal U}}(x_0,t_0)$ and ${\mathcal U}^c(x_0,t_0)$ are both open sets
so that they can be expressed by disjoint unions of at most countable open intervals respectively:
\begin{equation}\label{c3.1}
\mathring{{\mathcal U}}(x_0,t_0)=\cup_m I_m:=\cup_m (a_m,b_m),
\qquad {\mathcal U}^c(x_0,t_0)=\cup_n J_n:=\cup_n (c_n,d_n)
\end{equation}
with $a_m, b_m, c_n, d_n \in \partial {\mathcal U}(x_0,t_0)$.

\smallskip
From $\eqref{c3.02}$--$\eqref{c3.1}$, $I:=I(x_0,t_0)$ can be expressed by
the following countable disjoint union:
\begin{align}\label{c3.2}
I=I(x_0,t_0)
&=\partial{\mathcal U}(x_0,t_0)\cup\big(\cup_m I_m\big)\cup\big(\cup_n J_n\big)\nonumber\\[1mm]
&=\partial{\mathcal U}(x_0,t_0)\cup\big(\cup_m (a_m,b_m)\big)\cup\big(\cup_n (c_n,d_n)\big).
\end{align}
For any point set $F\subset \mathbb{R}$, let
\begin{equation}\label{c3.3}
\Delta_F (x_0,t_0):=\big\{(x,t)\in \mathbb{R} \times \mathbb{R}^+ \,:\,
(x,t) \in l_u(x_0,t_0) \quad{\rm for}\ u\in F\big\},
\end{equation}
where $l_u(x_0,t_0)$ is defined by $\eqref{c2.43}$, $i.e.$,
$$
l_u(x_0,t_0)=\big\{(x,t)\,:\, x=x_0+(t-t_0)f'(u)\quad {\rm for}\ t\in (0,t_0)\big\}.
$$
The triangle region $\Delta_I(x_0,t_0)$,
defined by applying $\eqref{c3.3}$ to $I=I(x_0,t_0)$,
is called the backward characteristic triangle from point $(x_0,t_0)$,
which is a triangle when ${\mathcal U}(x_0,t_0)$ contains more than one point,
and is a single characteristic line when ${\mathcal U}(x_0,t_0)$ contains only one point.

\smallskip
Denote $u^{\pm}_0:=u(x_0{\pm},t_0)$.
Then the backward characteristic triangle $\Delta_I(x_0,t_0)$ and $\mathbb{R}\times (0,t_0)$
can be expressed by the following countable disjoint unions:
\begin{equation}\label{c3.4}
\begin{cases}
\displaystyle\Delta_I(x_0,t_0)\,
=\Delta_{\partial{\mathcal U}}(x_0,t_0)\cup\big(\cup_m \Delta_{I_m}(x_0,t_0)\big)
\cup\big(\cup_n \Delta_{J_n}(x_0,t_0)\big),\\[2mm]
\displaystyle\mathbb{R}\times (0,t_0)
=\Delta_I(x_0,t_0)\cup\Delta_{({-}\infty,u^+_0)}(x_0,t_0)\cup\Delta_{(u^-_0,\infty)}(x_0,t_0).
\end{cases}
\end{equation}
See Fig. $\ref{figUT}$ for the details.

\begin{figure}[H]
	\begin{center}
		{\includegraphics[width=0.65\columnwidth]{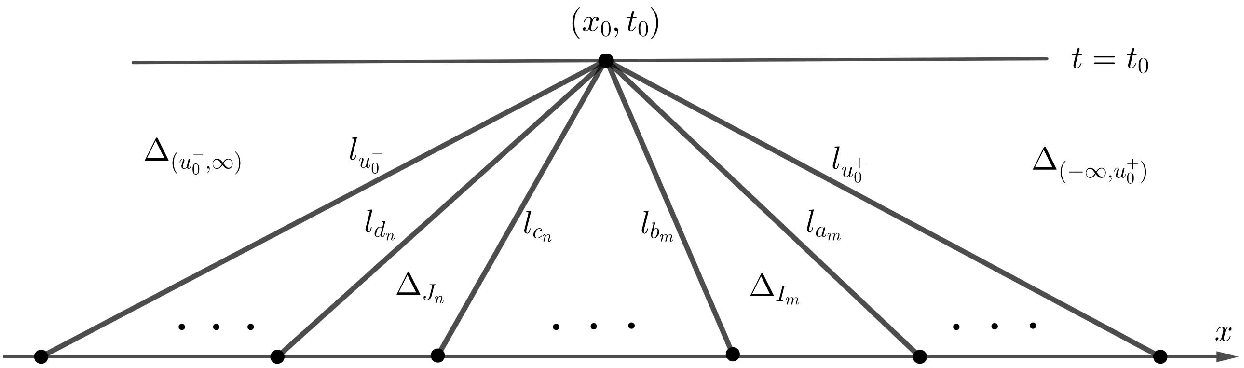}}
 \caption{The partitions of $\Delta_I(x_0,t_0)$ and $\mathbb{R}\times (0,t_0)$ as in $\eqref{c3.4}$.
		}\label{figUT}
	\end{center}
\end{figure}

We first have the following properties of the backward characteristic triangles.

\begin{Pro}\label{pro:c3.1}
 For any $(x_1,t_1), (x_2,t_2)\in \mathbb{R}\times \mathbb{R}^+$,
 $\Delta_I(x_1,t_1)$ and $\Delta_I(x_2,t_2)$ must possess one of the following two relations{\rm :}
\begin{itemize}
\item [(i)] $\Delta_I(x_1,t_1)\cap \Delta_I(x_2,t_2)=\varnothing${\rm ;}

\smallskip
\item [(ii)] $\Delta_I(x_1,t_1)\supset \Delta_I(x_2,t_2)\,$
or $\,\Delta_I(x_1,t_1)\subset \Delta_I(x_2,t_2)$.
\end{itemize}
Furthermore, if $\Delta_I(x_1,t_1)\subset \Delta_I(x_2,t_2)$, then
one of the following two cases holds{\rm :}
\begin{itemize}
\item [(i)] $\Delta_I(x_1,t_1)\subset l_u(x_2,t_2)$ for some $u\in \mathcal{U}(x_2,t_2)${\rm ;}

\smallskip
\item [(ii)] $\Delta_I(x_1,t_1)\subset \Delta_{J_n}(x_2,t_2)$ for some $J_n \subset \mathcal{U}^c(x_2,t_2)$.
\end{itemize}
\end{Pro}

\begin{proof}
From $\eqref{uiff}$, for any $(x_0,t_0)\in \mathbb{R} \times \mathbb{R}^+$ and $c\in \mathcal{U}(x_0,t_0)$,
the entropy solution $u(x,t)$ is continuous on $l_{c}(x_0,t_0)$ so that,
for any $u_i\in \mathcal{U}(x_i,t_i)$,
$l_{u_1}(x_1,t_1)$ and $l_{u_2}(x_2,t_2)$ are disjoint, or one containing another.
Then it is direct to check that,
if $\Delta_I(x_1,t_1)\cap \Delta_I(x_2,t_2)\neq\varnothing$,
then $\Delta_I(x_1,t_1)\supset \Delta_I(x_2,t_2)$ or $\Delta_I(x_1,t_1)\subset \Delta_I(x_2,t_2)$.

Furthermore, if $\Delta_I(x_1,t_1)\subset \Delta_I(x_2,t_2)$,
then $(x_1,t_1)\in \Delta_I(x_2,t_2)$.
From $\eqref{c3.4}$,
if $(x_1,t_1)\in \Delta_{\partial{\mathcal U}}(x_2,t_2)$ or $(x_1,t_1)\in \Delta_{I_m}(x_2,t_2)$,
then $(x_1,t_1)\in l_{u}(x_2,t_2)$ for some $u\in \mathcal{U} (x_2,t_2)$ so that $u(x_1{\pm},t_1)=u$,
which implies that $\Delta_I(x_1,t_1) \subset l_{u}(x_2,t_2)$.
If $(x_1,t_1)\in \Delta_{J_n}(x_2,t_2)$ for some $J_n=(c_n,d_n)\subset {\mathcal U}^c(x_2,t_2)$,
it follows from $\eqref{uiff}$ and $c_n, d_n\in \partial {\mathcal U}(x_2,t_2)$ that
$$
x_2-t_2f'(d_n)\leq x_1-t_1f'(u(x_1{-},t_1))\leq x_1-t_1f'(u(x_1{+},t_1))\leq x_2-t_2f'(c_n),
$$
which implies that $\Delta_I(x_1,t_1) \subset \Delta_{J_n}(x_2,t_2)$.
\end{proof}

\begin{Lem}[Forward generalized characteristics]\label{pro:c3.2}
For any $(x_0,t_0)\in \mathbb{R}\times \mathbb{R}^+$,
there exists a unique forward generalized characteristic{\rm :}
\begin{equation}\label{c3.5a}
X(x_0,t_0)=\big\{x=x(t)\,:\, x(t_0)=x_0,\,\, t_0\le t<\infty\big\}
\end{equation}
such that
\begin{equation}\label{c3.5}
\Delta_I(x_0,t_0)\subset \Delta_I(x(t_1),t_1)\subset \Delta_I(x(t_2),t_2)
\qquad{\rm for \ any}\ t_2>t_1>t_0;
\end{equation}
see also {\rm Fig.} $\ref{figForwardcurve}$.
Furthermore, the following properties hold{\rm :}
\begin{itemize}
\item [(i)] For any fixed $t_1$ with $t_1>t_0$,
$X(x(t_1),t_1)=X(x_0,t_0)$ on $[t_1,\infty)$.

\vspace{1pt}
\item [(ii)] If $u(x_0{+},t_0)<u(x_0{-},t_0)$, $X(x_0,t_0)$ is a discontinuity of the entropy solution $u(x,t)$.
For any $t_1> t_0$, there exists $J_n \subset\mathcal{U}^c(x_1,t_1)$
such that $(x(t),t)\in \Delta_{J_n}(x(t_1),t_1)$ on $[t_0,t_1)$.
\end{itemize}
\noindent
In the above, $\Delta_{I}(x,t)$ with $I$ and $\Delta_{J_n}(x,t)$ with $J_n$ are given by $\eqref{c3.02}$--$\eqref{c3.3}$.
\end{Lem}

\begin{figure}[H]
	\begin{center}
		{\includegraphics[width=0.65\columnwidth]{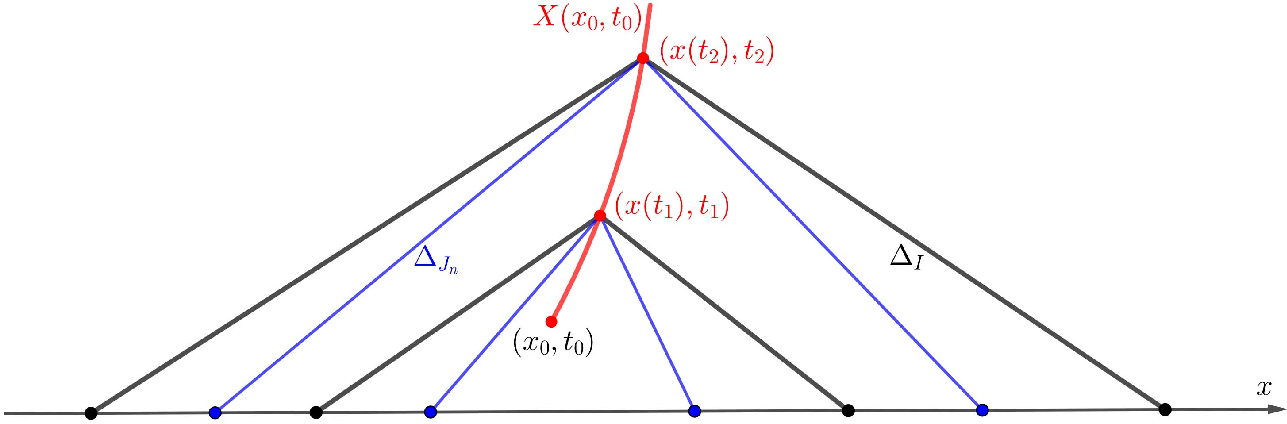}}
 \caption{The forward generalized characteristic $x(t)$ (red curve) with $\Delta_I(x(t),t)$ (black triangle) 
 and $\Delta_{J_n}(x(t),t)$ (blue triangle)
in Lemma $\ref{pro:c3.2}$.
 }
 \label{figForwardcurve}
	\end{center}
\end{figure}

\begin{proof}
We divide the proof into four steps.

\smallskip
\noindent
{\bf 1.} For any $t_1\geq t_0$, let
\begin{equation}\label{c3.7}
A_\pm(x_0,t_0,t_1):=\big\{x\in \mathbb{R}\,:\, \pm\big(x+(t_0-t_1)f'(u(x{+},t_1))-x_0\big)>0\big\}.
\end{equation}
Then
$$
a(t_1):=\sup A_-(x_0,t_0,t_1)\leq \inf A_+(x_0,t_0,t_1)=: b(t_1).
$$

We first claim: {\it $\, a(t_1)=b(t_1)$ {\rm for any} $t_1\geq t_0$.}
This can be seen as follows:
For $t_1=t_0$, it is direct to check that $a(t_0)=b(t_0)=x_0$.
For the case that $t_1>t_0$,
if $a(t_1)<b(t_1)$,
then, for any $x,x'\in (a(t_1),b(t_1))$ with $x<x'$,
\begin{equation}\label{c3.8}
x+(t_0-t_1)f'(u(x{+},t_1))=x_0=x'+(t_0-t_1)f'(u(x'{+},t_1)),
\end{equation}
so that
$(x_0,t_0)\in l_{u(x{+},t_1)}(x,t_1)\cap l_{u(x'{+},t_1)}(x',t_1),$
which, by $\eqref{uiff}$, implies
\begin{equation}\label{c3.9}
u(x{+},t_1)=u(x_0{\pm},t_0)=u(x'{+},t_1).
\end{equation}
On the other hand, since $f'(u)$ is strictly increasing, $\eqref{c3.8}$ implies
$$
f'(u(x'{+},t_1))-f'(u(x{+},t_1))=\frac{x'-x}{t_1-t_0}>0,
$$
so that $u(x'{+},t_1)>u(x{+},t_1)$,
which contradicts to $\eqref{c3.9}$.

\vspace{1pt}
Denote $X(x_0,t_0): x=x(t):=a(t)$ for $ t\geq t_0$.
Then $X(x_0,t_0)$ with $x(t_0)=x_0$ is a single-valued function on $[t_0,\infty)$.

\smallskip
\noindent
{\bf 2.} We now show that, for any $t_1>t_0$,
$(x(t_1),t_1)$ is the unique point on line $t=t_1$ such that
$\Delta_I(x_0,t_0)\subset \Delta_I(x(t_1),t_1)$.
If $x<x(t_1)$, then $x\in A_- (x_0,t_0,t_1)$,
which, by $\eqref{c3.7}$, implies that
$x+(t_0-t_1)f'(u(x{+},t_1))<x_0$.
Thus, $(x_0,t_0)$ lies on the right side of $\Delta_I(x,t_1)$,
which, by Proposition $\ref{pro:c3.1}$, implies that $\Delta_I(x_0,t_0)\cap \Delta_I(x,t_1)=\varnothing$;
and if $x>x(t_1)$, from $\eqref{c2.31}$ and $\eqref{c3.7}$,
for any $x'\in (x(t_1),x)$,
$$
x_0<x'+(t_0-t_1)f'(u(x'{+},t_1))\leq x+(t_0-t_1)f'(u(x{-},t_1)),
$$
so that $(x_0,t_0)$ lies on the left side of $\Delta_I(x,t_1)$,
which implies that $\Delta_I(x_0,t_0)\cap \Delta_I(x,t_1)=\varnothing$.

\vspace{2pt}
For $\Delta_I(x(t_1),t_1)$,
by Lemma $\ref{lem:c2.5}$,
$\lim_{x\rightarrow x(t_1){\pm}}u(x{+},t_1)=u(x(t_1){\pm},t_1)$ so that, by $\eqref{c3.7}$,
\begin{align*}
\begin{cases}
\displaystyle x(t_1)+(t_0-t_1)f'(u(x(t_1){+},t_1))
=\lim\limits_{x\rightarrow x(t_1){+}}\big(x+(t_0-t_1)f'(u(x{+},t_1))\big)\geq x_0, \\[1mm]
\displaystyle x(t_1)+(t_0-t_1)f'(u(x(t_1){-},t_1))
=\lim\limits_{x\rightarrow x(t_1){-}}\big(x+(t_0-t_1)f'(u(x{+},t_1))\big)\leq x_0,
\end{cases}
\end{align*}
which implies that $(x_0,t_0)\in \Delta_I(x(t_1),t_1)$
so that $\Delta_I(x_0,t_0)\subset \Delta_I(x(t_1),t_1)$.

\vspace{2pt}
Thus, for any $ t_2 >t_1 >t_0$, $\Delta_I(x_0,t_0)\subset \Delta_I(x(t_i),t_i)$
so that $\Delta_I(x(t_1),t_1)\cap \Delta_I(x(t_2),t_2)\neq \varnothing$,
which, by Proposition $\ref{pro:c3.1}$ and $(x(t_2),t_2)\notin \Delta_I(x(t_1),t_1)$, implies $\eqref{c3.5}$.

\medskip
\noindent
{\bf 3.} We now prove that $X(x_0,t_0)$
is continuous on $[t_0,\infty)$ with $\lim_{t\rightarrow t_{0}{+}}x(t)=x_0$.

\vspace{2pt}
Denote $y^\pm (x(t),t):=x(t)-tf'(u(x(t){\pm},t))$.
From $\eqref{c3.5}$, for any $t_2>t_1\geq t_0$,
\begin{equation}\label{c3.10}
y^-(x(t_2),t_2)\leq y^-(x(t_1),t_1)\leq y^+(x(t_1),t_1) \leq y^+(x(t_2),t_2),
\end{equation}
so that $y^-(x(t),t)$ is nonincreasing and $ y^+(x(t),t)$ is nondecreasing.
Letting $t_2 \rightarrow t_{1}{+}$ in $\eqref{c3.10}$,
\begin{equation*}
y^-(\overline{x(t_{1}+)},t_1)\leq y^-(x(t_1),t_1)
\leq y^+(x(t_1),t_1) \leq y^+(\underline{ x(t_{1}+)},t_1),
\end{equation*}
where $\overline{x(t_{1}+)}:=\mathop{\overline{\lim}}_{t_2 \rightarrow t_{1}{+}}x(t_2)$
and $\underline{ x(t_{1}+)}:=\mathop{\underline{\lim}}_{t_2 \rightarrow t_{1}{+}}x(t_2)$.

\smallskip
Since $y^{\pm}(\cdot,t_1)$ are nondecreasing,
then $\overline{x(t_{1}+)}\leq x(t_1)\leq\underline{ x(t_{1}+)}$,
which implies
\begin{equation*}
\lim\limits_{t\rightarrow t_{1}{+}}x(t)=x(t_1)\qquad {\rm for\ any\ } t_1\geq t_0.
\end{equation*}
Similarly, letting $t_1 \rightarrow t_{2}{-}$ in $\eqref{c3.10}$,
$$
y^- (x(t_2),t_2)\leq y^- (\underline{x(t_{2}-)},t_2)
\leq y^+ (\overline{x(t_{2}-)},t_2)\leq y^+ (x(t_{2}),t_2),
$$
so that, by Lemma $\ref{lem:c2.4}$,
$\overline{x(t_{2}-)}\leq x(t_2)\leq \underline{x(t_{2}-)}$,
which implies
\begin{equation*}
\lim_{t\rightarrow t_{2}{-}}x(t)=x(t_2)\qquad {\rm for\ any\ } t_2>t_0,
\end{equation*}
where $\overline{x(t_{2}-)}:=\mathop{\overline{\lim}}_{t_1 \rightarrow t_{2}{-}}x(t_1)$
and $\underline{ x(t_{2}-)}:=\mathop{\underline{\lim}}_{t_1 \rightarrow t_{2}{-}}x(t_1)$.

\smallskip
Therefore,
$X(x_0,t_0)$ is continuous on $[t_0,\infty)$ with $\lim_{t\rightarrow t_{0}{+}}x(t)=x_0$.

\smallskip
\noindent
{\bf 4.} For any fixed $t_1>t_0$,
the forward generalized characteristic of point $(x(t_1),t_1)$ is given by
$$
X(x(t_1),t_1):=\big\{x=x_1(t)\,:\, x_1(t_1)=x(t_1),\, t_1\le t<\infty\big\}.
$$
Then, by the definition of forward generalized characteristics in $\eqref{c3.5a}$--$\eqref{c3.5}$,
for any $t_2>t_1$,
$$
\varnothing\neq\Delta_I(x_0,t_0)\subset\Delta_I(x(t_1),t_1)
=\Delta_I(x_1(t_1),t_1)\subset \Delta_I(x(t_2),t_2)\cap \Delta_I(x_1(t_2),t_2),
$$
so that, by Proposition $\ref{pro:c3.1}$,
$$\Delta_I(x(t_2),t_2)\supset \Delta_I(x_1(t_2),t_2)\qquad
{\rm or}\ \Delta_I(x(t_2),t_2)\subset \Delta_I(x_1(t_2),t_2),$$
which implies that
$x(t_2)=x_1(t_2)$.
By the arbitrariness of $t_2$ with $t_2>t_1$ and $x_1(t_1)=x(t_1)$,
we conclude that $X(x(t_1),t_1)=X(x_0,t_0)$ on $[t_1,\infty)$.

Finally, if $u(x_0{+},t_0)<u(x_0{-},t_0)$, then $\Delta_I(x_0,t_0)$ is a triangle so that,
by $\eqref{c3.5}$, $\Delta_I(x(t),t)$ is also a triangle for any $t>t_0$,
which implies that $u(x(t){+},t)<u(x(t){-},t)$ for any $t>t_0$.
This means that $X(x_0,t_0)$ is a discontinuity of entropy solution $u(x,t)$.

Since $X(x_0,t_0)$
is the forward generalized characteristic of $(x_0,t_0)$,
then, for any $t\in (t_0,t_1)$,
$$
\Delta_I(x_0,t_0)\subset \Delta_I(x(t),t)\subset \Delta_I(x(t_1),t_1),
$$
which, by Proposition $\ref{pro:c3.1}$, implies that
$(x(t),t)\in \Delta_{J_n}(x(t_1),t_1)$ for some $J_n\subset \mathcal{U}^c(x_1,t_1)$.
By the continuity of $x=x(t)$ and
$\Delta_{J_{n_1}}(x(t_1),t_1)\cap\Delta_{J_{n_2}}(x(t_1),t_1)=\varnothing$ for any $n_1\not= n_2$,
the forward generalized characteristic $x=x(t)$ on $[t_0,t_1)$ contains in the same $\Delta_{J_{n}}(x(t_1),t_1)$.
\end{proof}

\begin{The}[Structures of entropy solutions inside the backward characteristic triangles]\label{pro:c3.3}
For any $(x_0,t_0)\in \mathbb{R}\times \mathbb{R}^+$,
the entropy solution $u(x,t)$ on $\mathbb{R}\times (0,t_0)$ as in $\eqref{c3.4}$
satisfies the following properties{\rm :}
\begin{itemize}
\item [(i)] If $(x,t)\in \Delta_{\partial \mathcal{U}}(x_0,t_0)$,
then $u(x{-},t)=u(x{+},t)$,
and $u(x{\pm},t)\in \partial \mathcal{U}(x_0,t_0)$.

\vspace{2pt}
\item [(ii)] If there exists $I_m=(a_m,b_m)$,
then solution $u(x,t)$ in $\Delta_{I_m}(x_0,t_0)$ can be concretely expressed by
\begin{equation}\label{c3.13}
u(x,t)=(f')^{-1}(\frac{x-x_0}{t-t_0})\in I_m,
\end{equation}
and the initial data function $\varphi(x)$ satisfies
\begin{equation}\label{c3.14}
\varphi(x)=(f')^{-1}(\frac{x_0-x}{t_0}) \qquad\,\,\, {a.e.} \ x\in [x_0-t_0f'(b_m),\,x_0-t_0f'(a_m)].
\end{equation}

\item [(iii)] If there exists $J_n=(c_n,d_n)$,
then, for any $(x_1,t_1)\in \Delta_{J_n}(x_0,t_0)$,
the forward generalized characteristic $X(x_1,t_1): x=x(t)$ for $t\geq t_1$
passes through point $(x_0,t_0)$,
and there exists $\underline{t}<t_0$ such that
\begin{equation}\label{c3.15}
u(x(t){+},t)<u(x(t){-},t)\qquad\, {\rm on} \ t\in (\underline{t},t_0].
\end{equation}
Furthermore, let $X(x_i,t_i):\, x=x_i(t)$ for $t\geq t_i$
be the forward generalized characteristics with respect to two different points $(x_i,t_i)\in \Delta_{J_n}(x_0,t_0)$ for $i=1,2$.
Then there exists
$\bar{t}$ with $\max\{t_1,t_2\}<\bar{t}<t_0$ such that
\begin{equation}\label{c3.16}
x_1(t)=x_2(t)\qquad {\rm on}\ [\bar{t},t_0].
\end{equation}
Moreover, the set of discontinuous points of the entropy solution $u(x,t)$ in $\Delta_{J_n}(x_0,t_0):$
$$
\Gamma (\Delta_{J_n}(x_0,t_0)):= \big\{(x,t)\in \Delta_{J_n}(x_0,t_0)\,:\, u(x{+},t)<u(x{-},t)\big\}
$$
is path-connected.

\item [(iv)]
For any point $(x_1,t_1)\in \Delta_{({-}\infty,u^+_0)}(x_0,t_0)\cup\Delta_{(u^-_0,\infty)}(x_0,t_0)$,
the forward generalized characteristic of point $(x_1,t_1)$ do not pass point $(x_0,t_0)$.
\end{itemize}
\noindent
In the above, 
$\partial \mathcal{U}(x_0,t_0)$, 
$\Delta_{\partial \mathcal{U}}(x_0,t_0)$, $\Delta_{I_m}(x_0,t_0)$ with $I_m$,
and $\Delta_{J_n}(x_0,t_0)$ with $J_n$
are given by $\eqref{c3.02}$--$\eqref{c3.3}$.
See also
{\rm Fig.} $\ref{figStriangle}$.
\end{The}

\begin{proof} We divide the proof into four steps accordingly.

\smallskip
\noindent
{\bf 1.} If $(x,t)\in \Delta_{\partial \mathcal{U}}(x_0,t_0)$,
there exists $c\in \partial \mathcal{U}(x_0,t_0)$ such that $(x,t)\in l_{c}(x_0,t_0)$
which, by $\eqref{uiff}$, implies that $u(x{+},t)=u(x{-},t)=c$.

\smallskip
\noindent
{\bf 2.} Suppose that $I_m=(a_m,b_m)\neq \varnothing$.
Since $I_m \subset \mathring {\mathcal{U}}(x_0,t_0)\subset \mathcal{U}(x_0,t_0)$,
then, for any $(x,t)\in \Delta_{I_m}(x_0,t_0)$,
there exists $c\in I_m \subset \mathcal{U}(x_0,t_0)$ such that $(x,t)\in l_c(x_0,t_0)$
which, by $\eqref{uiff}$, implies that $u(x{\pm},t)=c$. Thus,
$x=x_0+(t-t_0)f'(c)=x_0+(t-t_0)f'(u(x{\pm},t))$
which, by a simple calculation, $\eqref{c3.13}$ follows.

To prove $\eqref{c3.14}$, let $\bar{\Delta}_{I_m}(x_0,t_0)$ be the closure of $\Delta_{I_m}(x_0,t_0)$.
Then $(f')^{-1}(\frac{x-x_0}{t-t_0})$ is continuous on the trapezoidal area
$\bar{\Delta}_{I_m}(x_0,t_0)\cap(\mathbb{R}\times [0,\frac{t_0}{2}])$
so that it is uniformly continuous.
By $\eqref{c1.3c}$ and the Lebesgue differential theorem,
for $x\in (x_0-t_0f'(b_m), \, x_0-t_0f'(a_m))$,
\begin{align*}
\varphi(x)
&\mathop{=}\limits^{a.e.}\lim\limits_{x' \rightarrow x}\frac{1}{x'-x}\int_{x}^{x'} \varphi(\xi)\,{\rm d}\xi
=\lim\limits_{x' \rightarrow x}
\bigg(\frac{1}{x'-x}\,\lim\limits_{t \rightarrow 0{+}}\int_{x}^{x'} u(\xi,t)\,{\rm d}\xi\bigg)\\[1mm]
& =\lim\limits_{x' \rightarrow x}\lim\limits_{t \rightarrow 0{+}}
\bigg(\frac{1}{x'-x}\int_{x}^{x'}(f')^{-1}(\frac{\xi-x_0}{t-t_0})\,{\rm d}\xi\bigg)\\[1mm]
& =\lim\limits_{t \rightarrow 0{+}}
\bigg(\lim\limits_{x' \rightarrow x}\frac{1}{x'-x}\int_{x}^{x'}(f')^{-1}(\frac{\xi-x_0}{t-t_0})\,{\rm d}\xi\bigg)\\[1mm]
&=\lim\limits_{t \rightarrow 0{+}}(f')^{-1}(\frac{x-x_0}{t-t_0})
=(f')^{-1}(\frac{x_0-x}{t_0}).
\end{align*}

\smallskip
\noindent
{\bf 3.} Suppose that $J_n=(c_n,d_n)\neq \varnothing$.
From Proposition $\ref{pro:c3.1}$, for any $(x_1,t_1)\in \Delta_{J_n}(x_0,t_0)$,
$
\Delta_I(x_1,t_1)\subset \Delta_{J_n}(x_0,t_0)\subset \Delta_I(x_0,t_0)
$
so that, by Lemma $\ref{pro:c3.2}$, $(x_0,t_0)$ lies on $X(x_1,t_1)$.

\smallskip
{\bf (a).} We first prove $\eqref{c3.15}$. Let
$$
\underline{t}:=\inf \big\{t\in [t_1,\infty)\,:\,
u(x(t){+},t)<u(x(t){-},t)\,\,\, {\rm for}\ (x(t),t)\in X(x_1,t_1)\big\}.
$$
Since $J_n\neq \varnothing$, then $u(x_0{+},t_0)<u(x_0{-},t_0)$ so that $\underline{t}\leq t_0$.
We claim that $\underline{t}< t_0$.
Otherwise, if $\underline{t}= t_0$, then $u(x(t){+},t)=u(x(t){-},t)$ for any $t\in (t_1,t_0)$
so that, by $\eqref{c3.5}$,
$$
\Delta_I(x_1,t_1)\subset \Delta_I(x(t),t)=l_{u(x(t){-},t)}(x(t),t).
$$
Thus, $(x_1,t_1)\in l_{u(x(t){-},t)}(x(t),t)$ so that,
by $\eqref{uiff}$, $u(x(t){-},t)=u(x_1,t_1)$ for $t\in (t_1,t_0)$.
This means that $X(x_1,t_1)$ satisfies
\begin{equation}\label{c3.17}
x(t)=x_1-(t_1-t)f'(u(x_1,t_1)) \qquad {\rm for \ any}\ t\in (t_1,t_0).
\end{equation}
Since $X(x_1,t_1)$ is continuous, then $\lim_{t \rightarrow t_{0}{-}}x(t)=x_0$
so that, by $\eqref{c2.40}$,
\begin{equation}\label{c3.18}
u(x_1,t_1)=\lim\limits_{t \rightarrow t_{0}{-}}u(x(t){-},t)\in \mathcal{U}(x_0,t_0).
\end{equation}
Noting that $(x_1,t_1)\in \Delta_{J_n}(x_0,t_0)$,
and letting $t \rightarrow t_{0}{-}$ in $\eqref{c3.17}$,
we obtain
\begin{equation*}
u(x_1,t_1)=(f')^{-1}(\frac{x_1-x_0}{t_1-t_0})\in (c_n,d_n)\subset\mathcal{U}^c(x_0,t_0),
\end{equation*}
which contradicts to $\eqref{c3.18}$.
Thus, $\underline{t}< t_0$ and $\eqref{c3.15}$ follows.

\smallskip
{\bf (b).} We now prove $\eqref{c3.16}$.
By Lemma $\ref{pro:c3.2}$,
$(x_i,t_i)\in \Delta_{J_n}(x_0,t_0)\subset \Delta_{I}(x_0,t_0)$ for $i=1,2$, which
imply that $X(x_i,t_i)$ passes $(x_0,t_0)$.
Since $x_1(t_0)=x_0=x_2(t_0)$, we let
$$
\bar{t}:=\inf \big\{t\in (0,\infty)\,:\, x_1(t)=x_2(t)\big\}.
$$
Then
$\bar{t}\leq t_0$.
To prove $\eqref{c3.16}$, it suffices to show that $\bar{t}< t_0$.

\vspace{2pt}
Otherwise, if $\bar{t}= t_0$, then $x_1(t)\neq x_2(t)$ for $t<t_0$.
Without loss of generality, assume that $x_1(t)<x_2(t)$.
From $\eqref{c3.15}$, there exists $\tilde{t}<t_0$ such that
$$
u(x_1(t){+},t)<u(x_1(t){-},t),\quad u(x_2(t){+},t)<u(x_2(t){-},t)
\,\qquad {\rm for \ any}\ t\in [\tilde{t},t_0],
$$
so that, for any $t\in [\tilde{t},t_0]$,
\begin{align}\label{c3.20}
x_0-t_0f'(d_n)
&\leq y^-(x_1(\tilde{t}),\tilde{t})<y^+(x_1(\tilde{t}),\tilde{t})\leq y^+(x_1(t),t)\nonumber\\[1mm]
& \leq y^-(x_2(t),t)\leq y^-(x_2(\tilde{t}),\tilde{t})<y^+(x_2(\tilde{t}),\tilde{t})
\leq x_0-t_0f'(c_n),
\end{align}
where $y^{\pm}(x_i(t),t):=x_i(t)-tf'(u(x_i(t){\pm},t))$ for $i=1,2$.

\smallskip
Since $y^+(x_1(t),t)$ is nondecreasing in $t$,
letting $t\rightarrow t_{0}{-}$ in $\eqref{c3.20}$,
$$
x_0-t_0f'(d_n)<y^+(x_1(\tilde{t}),\tilde{t})
\leq y_0:=\lim\limits_{t\rightarrow t_{0}{-}} y^+(x_1(t),t)
\leq y^-(x_2(\tilde{t}),\tilde{t})<x_0-t_0f'(c_n),
$$
which implies
\begin{equation}\label{c3.21}
(f')^{-1}(\frac{x_0-y_0}{t_0})\in (c_n,d_n) \subset \mathcal{U}^c(x_0,t_0).
\end{equation}
On the other hand, since $X(x_1,t_1) : x=x_1(t)$ for $t\geq t_1$
is continuous and $x_1(t_0)=x_0$, then, by $\eqref{c2.39}$,
it follows from $y^+(x_1(t),t)$ $=x_1(t)-tf'(u(x_1(t){+},t))$ that
\begin{equation*}
(f')^{-1}(\frac{x_0-y_0}{t_0})=\lim\limits_{t\rightarrow t_{0}{-}} u(x_1(t){+},t)\in \mathcal{U}(x_0,t_0),
\end{equation*}
which contradicts to $\eqref{c3.21}$.
Thus, $\bar{t}< t_0$. By the definition of $\bar{t}$, $\eqref{c3.16}$ holds.

\vspace{1pt}
{\bf (c).} Moreover, for any two points $(x_i,t_i)\in \Gamma (\Delta_{J_n}(x_0,t_0))$ for $i=1,2$,
by Lemma $\ref{pro:c3.2}$(ii), the forward generalized characteristics $X(x_1,t_1)$ and $X(x_2,t_2)$
are both discontinuities of the entropy solution $u(x,t)$ so that
$\eqref{c3.16}$ directly implies that $\Gamma (\Delta_{J_n}(x_0,t_0))$ is path-connected.

\smallskip
\noindent
{\bf 4}. For any point $(x_1,t_1)\in \Delta_{({-}\infty,u^+_0)}(x_0,t_0)\cup\Delta_{(u^-_0,\infty)}(x_0,t_0)$,
then $(x_1,t_1)\not\in \Delta_I(x_0,t_0)$.
From $t_1<t_0$, $(x_0,t_0)\not\in \Delta_I(x_1,t_1)$.
Then, by Proposition $\ref{pro:c3.1}$, $\Delta_I(x_1,t_1)\cap\Delta_I(x_0,t_0)=\varnothing$.
On the other hand, it follows from $\eqref{c3.5}$ that,
if $(x_0,t_0)\in X(x_1,t_1)$, then $\Delta_I(x_1,t_1)\subset \Delta_I(x_0,t_0)$.
This is a contradiction.
Thus, $(x_0,t_0)\not\in X(x_1,t_1)$,
which means that the forward generalized characteristic $X(x_1,t_1)$ of point $(x_1,t_1)$ do not pass point $(x_0,t_0)$.

\smallskip
Up to now, we have completed the proof of Theorem $\ref{pro:c3.3}$.
\end{proof}

\begin{Rem}
From {\rm Theorem} $\ref{pro:c3.3}$,  
every shock curve passing through $(x_0,t_0)$ must contain
in the backward characteristic triangle $\Delta_I(x_0,t_0)$ when $t<t_0$.
Furthermore,
since $\Delta_{J_n}(x_0,t_0)\cap \Delta_{J_{n'}}(x_0,t_0)=\varnothing$ for any $n\neq n'$,
two shock curves passing through $(x_0,t_0)$ collide with each other at a time $t<t_0$
if and only if
they are contained in the same triangle $\Delta_{J_n}(x_0,t_0)$,
and two shock curves passing through $(x_0,t_0)$ collide with each other at $t=t_0$
if and only if
they are contained in two different triangles $\Delta_{J_n}(x_0,t_0)$.
This means that {\it there is a one-to-one correspondence
between the path-connected branches of
the set of discontinuous points in $\Delta_I(x_0,t_0)$
and the disjoint intervals in $\mathcal{U}^c(x_0,t_0)=\cup_n J_n$}.
\end{Rem}

\subsection{Directional limits of entropy solutions at discontinuous points}
For any fixed $(x_0,t_0)\in \mathbb{R}\times \mathbb{R}^+$, if $u(x_0{+},t_0)<u(x_0{-},t_0)$,
the forward generalized characteristic $X(x_0,t_0):x=x(t)$ for $t\geq t_0$ is a continuous single-valued curve,
along which the entropy solution $u(x,t)$ is discontinuous.
Denote
\begin{equation*}
l^{\pm}(t):=x_0+(t-t_0)f'(u(x_0{\pm},t_0))\qquad {\rm for}\ t\in (0,t_0),
\end{equation*}
and let
\begin{equation}\label{c3.24}
\begin{cases}
\displaystyle \Delta^-(x_0,t_0):=\big\{(x,t)\,:\, x<x(t)\ {\rm if}\ t\geq t_0,\,\,\,
x<l^-(t)\ {\rm if}\ t\in (0,t_0)\big\}, \\[2mm]
\displaystyle \Delta^+(x_0,t_0):=\big\{(x,t)\,:\, x>x(t)\ {\rm if}\ t\geq t_0,\,\,\,
x>l^+(t)\ {\rm if}\ t\in (0,t_0)\big\}.
\end{cases}
\end{equation}
If $\mathcal{U}^c(x_0,t_0)=\cup_nJ_n\neq \varnothing$, there exists $J_n\neq \varnothing$.
Let $X(x_n,t_n): x=x_n(t)$ for $t\geq t_n$
be the forward generalized characteristic of point $(x_n,t_n)\in \Delta_{J_n}(x_0,t_0)$.
Then, by Theorem $\ref{pro:c3.3}$,
there exists $\hat{t}_n\in (t_n,t_0)$ such that
$u(x_n(t){+},t)<u(x_n(t){-},t)$ for any $ t\in [\hat{t}_n,t_0]$.
We set
\begin{equation}\label{c3.25}
\begin{cases}
\displaystyle \Delta^-_{J_n}(x_0,t_0):=
\big\{(x,t)\in\Delta_{J_n}(x_0,t_0)\,:\, x<x_n(t) \,\,\, {\rm for}\ t\in (\hat{t}_n,t_0)\big\}, \\[2mm]
\displaystyle \Delta^+_{J_n}(x_0,t_0):=
\big\{(x,t)\in\Delta_{J_n}(x_0,t_0)\,:\, x>x_n(t) \,\,\, {\rm for}\ t\in (\hat{t}_n,t_0)\big\}.
\end{cases}
\end{equation}

\smallskip
The directional limits of entropy solutions at the discontinuous points are as follows:

\begin{The}[Directional limits of entropy solutions]\label{pro:c3.4}
 For any $(x_0,t_0)\in \mathbb{R}\times \mathbb{R}^+$ with $u(x_0{+},t_0)<u(x_0{-},t_0)$,
\begin{itemize}
\item [(i)] If $x=x(t)$ for $t\geq t_0$ is the forward generalized characteristic $X(x_0,t_0)$ of point $(x_0,t_0)$, then
\begin{equation}\label{c3.26}
\lim_{\genfrac{}{}{0pt}{3}{(x,t)\rightarrow (x_0,t_0)}{(x,t)\in \Delta^-(x_0,t_0)}}
u(x{\pm},t)=u(x_0{-},t_0),\qquad
\lim_{\genfrac{}{}{0pt}{3}{(x,t)\rightarrow (x_0,t_0)}{(x,t)\in \Delta^+(x_0,t_0)}}
u(x{\pm},t)=u(x_0{+},t_0),
\end{equation}
\begin{equation}\label{c3.27}
\lim_{t\rightarrow t_{0}{+}}u(x(t){-},t)=u(x_0{-},t_0),\qquad
\lim_{t\rightarrow t_{0}{+}}u(x(t){+},t)=u(x_0{+},t_0).
\end{equation}

\item [(ii)] If there exists $J_n \subset \mathcal{U}^c(x_0,t_0)$ such that $J_n=(c_n,d_n)\neq \varnothing$, then
\begin{equation}\label{c3.28}
\lim_{\genfrac{}{}{0pt}{3}{(x,t)\rightarrow (x_0,t_0)}{(x,t)\in \Delta^-_{J_n}(x_0,t_0)}}u(x{\pm},t)=d_n,\qquad
\lim_{\genfrac{}{}{0pt}{3}{(x,t)\rightarrow (x_0,t_0)}{(x,t)\in \Delta^+_{J_n}(x_0,t_0)}}u(x{\pm},t)=c_n,
\end{equation}
\begin{equation}\label{c3.29}
\lim_{t\rightarrow t_{0}{-}}u(x_n(t){-},t)=d_n,\qquad
\lim_{t\rightarrow t_{0}{-}}u(x_n(t){+},t)=c_n.
\end{equation}
\end{itemize}
\noindent
In the above, $\Delta^{\pm}(x_0,t_0)$ and $\Delta^{\pm}_{J_n}(x_0,t_0)$
are given by $\eqref{c3.24}$ and $\eqref{c3.25}$, respectively{\rm ;}
and the curve $x=x_n(t)$ for $t\geq t_n$
is the forward generalized characteristic $X(x_n,t_n)$ of any point $(x_n,t_n)\in \Delta_{J_n}(x_0,t_0)$. 
See also {\rm Fig.} $\ref{figDlimit}$.
\end{The}

\begin{figure}[H]
	\begin{center}
		{\includegraphics[width=0.7\columnwidth]{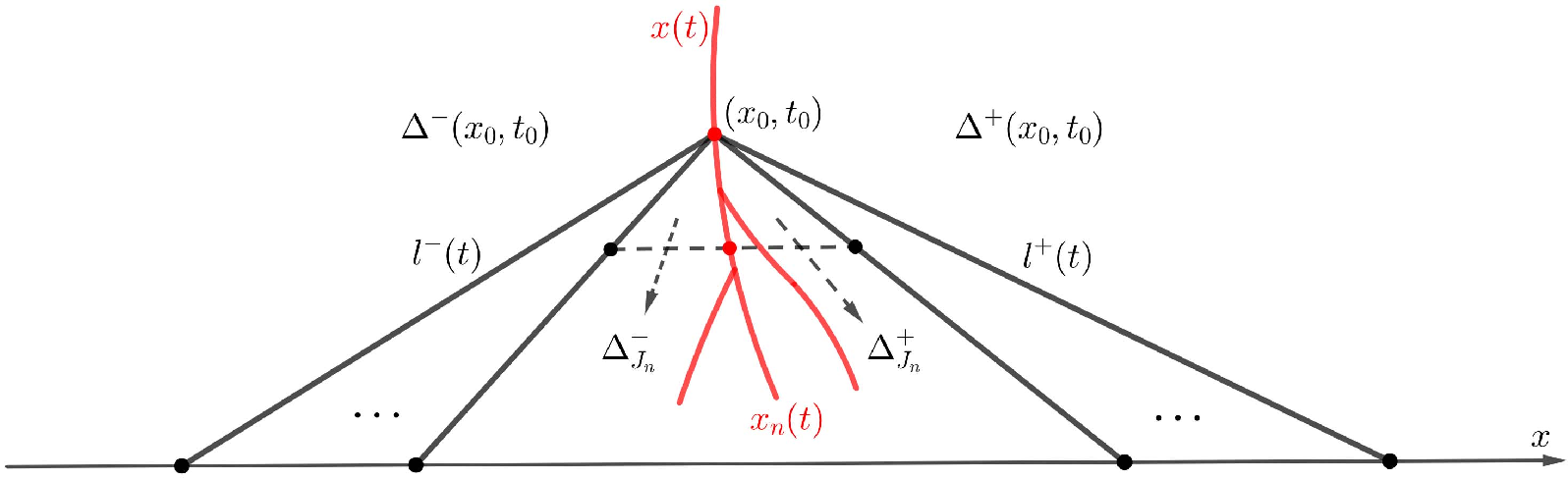}}
 \caption{The directional areas $\Delta^\pm(x_0,t_0)$ and $\Delta_{J_n}^\pm(x_0,t_0)$.
		}\label{figDlimit}
	\end{center}
\end{figure}

\begin{proof} We divide the proof into two steps accordingly.

\smallskip
\noindent
{\bf 1}. From $\eqref{c2.39}$, for any $(x_0,t_0)\in \mathbb{R}\times \mathbb{R}^+$,
\begin{equation}\label{c3.30}
\limsup_{\genfrac{}{}{0pt}{3}{(x,t)\rightarrow (x_0,t_0)}{(x,t)\in \Delta^-(x_0,t_0)}}
u(x{\pm},t)\leq u(x_0{-},t_0), \qquad
\liminf_{\genfrac{}{}{0pt}{3}{(x,t)\rightarrow (x_0,t_0)}{(x,t)\in \Delta^+(x_0,t_0)}}
u(x{\pm},t)\geq u(x_0{+},t_0).
\end{equation}

Denote $y^{\pm}(x,t):=x-tf'(u(x{\pm},t))$.
Then, from Lemma $\ref{lem:c2.4}$,
\begin{align*}
\begin{cases}
\displaystyle y^\pm(x,t)\leq y^-(x(t),t)\leq y^-(x_0,t_0)
\qquad {\rm for\ any}\ (x,t)\in \Delta^-(x_0,t_0), \\[2mm]
\displaystyle y^\pm(x,t)\geq y^+(x(t),t)\geq y^+(x_0,t_0)
\qquad {\rm for\ any}\ (x,t)\in \Delta^+(x_0,t_0),
\end{cases}
\end{align*}
so that
\begin{align*}
\begin{cases}
\displaystyle x\leq y^-(x_0,t_0)+tf'(u(x\pm,t))
\qquad {\rm for\ any}\ (x,t)\in \Delta^-(x_0,t_0), \\[2mm]
\displaystyle x\geq y^+(x_0,t_0)+tf'(u(x\pm,t))
\qquad {\rm for\ any}\ (x,t)\in \Delta^+(x_0,t_0),
\end{cases}
\end{align*}
which, by the strictly increasing property of $f'(u)$, implies
\begin{equation}\label{c3.31}
\liminf_{\genfrac{}{}{0pt}{3}{(x,t)\rightarrow (x_0,t_0)}{(x,t)\in \Delta^-(x_0,t_0)}}
u(x\pm,t)\geq u(x_0{-},t_0),\qquad
\limsup_{\genfrac{}{}{0pt}{3}{(x,t)\rightarrow (x_0,t_0)}{(x,t)\in \Delta^+(x_0,t_0)}}
u(x\pm,t)\leq u(x_0{+},t_0).
\end{equation}
Combining $\eqref{c3.30}$ with $\eqref{c3.31}$ leads to $\eqref{c3.26}$.

\vspace{2pt}
From $\eqref{c3.5}$ in Lemma $\ref{pro:c3.2}$,
$$
y^-(x(t),t)\leq y^-(x_0,t_0)\leq y^+(x_0,t_0)\leq y^+(x(t),t),
$$
so that, from $\lim_{t\rightarrow t_{0}{+}}x(t)=x_0$,
by letting $t\rightarrow t_{0}{+}$, we obtain
\begin{equation*}
\mathop{\underline{\lim}}\limits_{t\rightarrow t_{0}{+}}u(x(t){-},t)\geq u(x_0{-},t_0),\qquad
\mathop{\overline{\lim}}\limits_{t\rightarrow t_{0}{+}}u(x(t){+},t)\leq u(x_0{+},t_0),
\end{equation*}
which, by $\eqref{c2.39}$, implies $\eqref{c3.27}$.

\smallskip
\noindent
{\bf 2}. From Proposition $\ref{pro:c3.1}$, for any $(x,t)\in\Delta_{J_n}(x_0,t_0)$,
we see that $\Delta_I(x,t)\subset \Delta_{J_n}(x_0,t_0)$ so that,
for any $(x,t)\in\Delta^+_{J_n}(x_0,t_0)$,
$$
x_0-t_0f'(d_n)\leq y^-(x_n(\hat{t}_n),\hat{t}_n)
<y^+(x_n(\hat{t}_n),\hat{t}_n)
\leq y^+(x_n(t),t)\leq y^{\pm}(x,t)\leq x_0-t_0f'(c_n),
$$
and, for any $(x,t)\in \Delta^-_{J_n}(x_0,t_0)$,
$$
x_0-t_0f'(d_n)\leq y^{\pm} (x,t)\leq y^-(x_n(t),t)
\leq y^-(x_n(\hat{t}_n),\hat{t}_n)
<y^+(x_n(\hat{t}_n),\hat{t}_n)\leq x_0-tf'(c_n).
$$
Since $f'(u)$ is strictly increasing,
letting $(x,t)\in \Delta_{J_n}^\pm(x_0,t_0)\rightarrow (x_0,t_0)$ respectively,
we obtain
\begin{equation}\label{c3.33}
\begin{cases}
\displaystyle c_n\leq
\liminf\limits_{\genfrac{}{}{0pt}{3}{(x,t)\rightarrow (x_0,t_0)}{(x,t)\in \Delta^+_{J_n}(x_0,t_0)}}u(x{\pm},t)
\leq\limsup\limits_{\genfrac{}{}{0pt}{3}{(x,t)\rightarrow (x_0,t_0)}{(x,t)\in \Delta^+_{J_n}(x_0,t_0)}}
u(x{\pm},t)<d_n, \\[5mm]
\displaystyle c_n<
\liminf\limits_{\genfrac{}{}{0pt}{3}{(x,t)\rightarrow (x_0,t_0)}{(x,t)\in \Delta^-_{J_n}(x_0,t_0)}}u(x{\pm},t)
\leq\limsup\limits_{\genfrac{}{}{0pt}{3}{(x,t)\rightarrow (x_0,t_0)}{(x,t)\in \Delta^-_{J_n}(x_0,t_0)}}u(x{\pm},t)
\leq d_n.
\end{cases}
\end{equation}
From $\eqref{c2.39}$, the upper- and lower-limits in $\eqref{c3.33}$ are all in $\mathcal{U}(x_0,t_0)$,
which, by $[c_n,d_n]\cap \mathcal{U}(x_0,t_0)=\{c_n,d_n\}$, yields $\eqref{c3.28}$.

From $\Delta_I(x_n(t),t)\subset \Delta_{J_n}(x_0,t_0)$
and $u(x_n(\hat{t}_n){+},\hat{t}_n)< u(x_n(\hat{t}_n){-},\hat{t}_n)$,
$$
x_0-t_0f'(d_n)\leq y^-(x_n(t),t)\leq y^-(x_n(\hat{t}),\hat{t})
< y^+(x_n(\hat{t}),\hat{t})\leq y^+(x_n(t),t)\leq x_0-t_0f'(c_n),
$$
so that, from $\lim_{t\rightarrow t_{0}{-}}x_n(t)=x_0$,
by letting $t\rightarrow t_{0}{-}$, we obtain
\begin{equation*}
\begin{cases}
\displaystyle c_n<\mathop{\underline{\lim}}\limits_{t\rightarrow t_{0}{-}}
u(x_n(t){-},t)\leq \mathop{\overline{\lim}}\limits_{t\rightarrow t_{0}{-}}u(x_n(t){-},t)\leq d_n, \\[4mm]
\displaystyle c_n\leq \mathop{\underline{\lim}}\limits_{t\rightarrow t_{0}{-}}
u(x_n(t){+},t)\leq \mathop{\overline{\lim}}\limits_{t\rightarrow t_{0}{-}}u(x_n(t){+},t)<d_n,
\end{cases}
\end{equation*}
which, by $[c_n,d_n]\cap \mathcal{U}(x_0,t_0)=\{c_n,d_n\}$, yields $\eqref{c3.29}$.
\end{proof}

\begin{Rem}
From {\rm Lemma} $\ref{lem:c2.5}$, the limit of the entropy solution $u(x,t)$
under any sequence $(x_n,t_n)\rightarrow (x_0,t_0)$ contains in $\mathcal{U}(x_0,t_0)$.
Furthermore, from {\rm Theorem} $\ref{pro:c3.4}$,
the directional limit should be
what a point in $\mathcal{U}(x_0,t_0)$ in the respective case is described exactly.
We point out
that {\it the directional limits in any triangle $\Delta_{J_n}(x_0,t_0)$ are independent of
the choice of the forward generalized characteristics}.
\end{Rem}

\subsection{Left- and right-derivatives of shock curves and the piecewise smoothness}
From Lemma $\ref{pro:c3.2}$, if $u(x_0{+},t_0)<u(x_0{-},t_0)$,
the forward generalized characteristic $X(x_0,t_0): x=x(t)$ for $t\geq t_0$
is a single-valued continuous curve along which the entropy solution $u(x,t)$ is discontinuous, $i.e.$,
$u(x(t){+},t)<u(x(t){-},t)$ for any $t> t_0$.
The opposite is also true,
which can be seen in the following theorem about the general structures of shocks.

\begin{The}[General structures of shocks]\label{pro:c3.6}
Let $x=x(t)$ for $t\geq t_0$ be a single-valued continuous curve,
along which the entropy solution $u(x,t)$ is discontinuous, $i.e.$,
$u(x(t){+},t)<u(x(t){-},t)$ for any $t> t_0$.
Then $x=x(t)$ is the forward generalized characteristic of $(x_0,t_0)$,
which can be described by the triangle sequence $\{\Delta_I(x(t),t)\}$
and a unique triangle sequence $\{\Delta_{J_n}(x(t),t)\}$
with $J_n\subset \mathcal{U}^c(x(t),t)$ as
\begin{equation}\label{c3.45}
\begin{cases}
\displaystyle \Delta_I(x_0,t_0)\subset \Delta_I(x(t_1),t_1)\subset \Delta_I(x(t_2),t_2)\quad
&{\rm for\ any}\ t_2>t_1>t_0, \\[1mm]
\displaystyle \Delta_{J_n}(x_0,t_0)\subset \Delta_{J_n}(x(t_1),t_1)\subset \Delta_{J_n}(x(t_2),t_2)\quad
&{\rm for\ any}\ t_2>t_1>t_0.
\end{cases}
\end{equation}
Furthermore, the
curve{\rm :} $x=x(t)$ for $t\geq t_0$ is a Lipschitz continuous shock such that
\begin{itemize}
\item [(i)] For any $t> t_0$, $x=x(t)$ is differentiable from the right, and
\begin{equation}\label{c3.36}
\frac{\,{\rm d}^+x(t)}{\,{\rm d}t}:=\lim_{t'\rightarrow t{+}}\frac{x(t')-x(t)}{t'-t}
=\frac{f(u(x(t){-},t))-f(u(x(t){+},t))}{u(x(t){-},t)-u(x(t){+},t)}=:\frac{[f(u)]}{[u]}\Big|_I.
\end{equation}

\item [(ii)] For any $t> t_0$, $x=x(t)$ is differentiable from the left, and
\begin{equation}\label{c3.37}
\frac{\,{\rm d}^-x(t)}{\,{\rm d}t}:=\lim_{t'\rightarrow t{-}}\frac{x(t')-x(t)}{t'-t}
=\frac{f(d_n)-f(c_n)}{d_n-c_n}=:\frac{[f(u)]}{[u]}\Big|_{J_n}
\end{equation}
for the $J_n=(c_n,d_n)\subset \mathcal{U}^c(x(t),t)$
such that $(x(\tau),\tau)\in \Delta_{J_n}(x(t),t)$ on $[t_0,t)$.

\item [(iii)] Except at most countable points of $t$ with $t>t_0$,
$\frac{\,{\rm d}^+x(t)}{\,{\rm d}t}=\frac{\,{\rm d}^-x(t)}{\,{\rm d}t}$.
\end{itemize}
\noindent
In the above, $\Delta_{I}(x,t)$ with $I$
and $\Delta_{J_n}(x,t)$ with $J_n$
are given by $\eqref{c3.02}$--$\eqref{c3.3}$.
See also {\rm Fig.} $\ref{figSwave}$.
\end{The}

\begin{figure}[H]
	\begin{center}
		{\includegraphics[width=0.7\columnwidth]{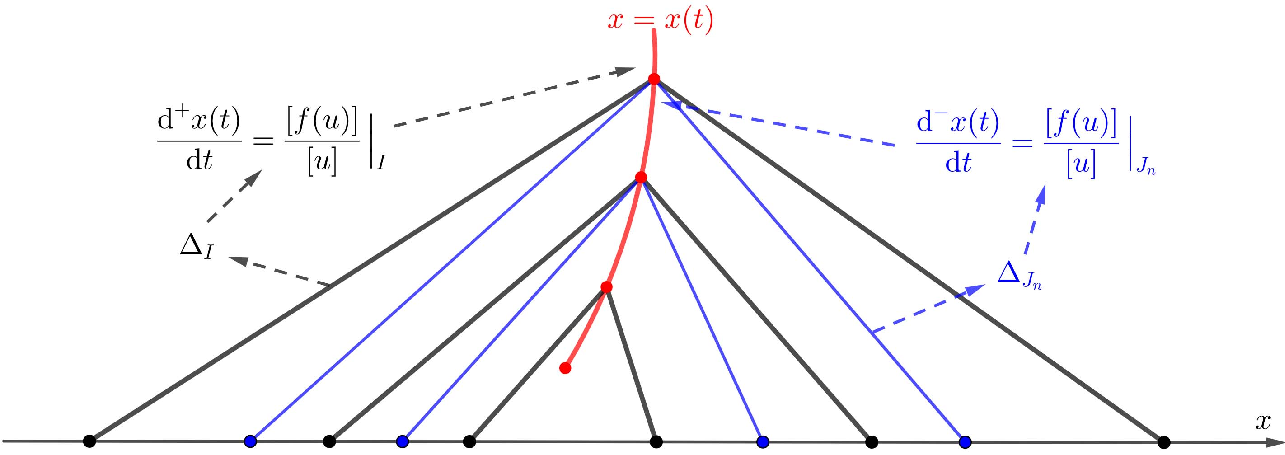}}
  \caption{The general structures of shocks
 as in Theorem $\ref{pro:c3.6}$.
		}\label{figSwave}
	\end{center}
\end{figure}

\begin{proof}
From Lemma $\ref{pro:c3.2}$, for any $t>t_0$,
there exists a unique point $(x,t)$ such that $(x_0,t_0)\in \Delta_I(x,t)$.
To prove $\eqref{c3.45}$ for the triangle sequence $\{\Delta_I(x(t),t)\}$,
it suffices to show that $(x_0,t_0)\in \Delta_I(x(t),t)$ for any $t>t_0$.
This follows by applying  Step 2 in the proof of Lemma $\ref{pro:c3.2}$ to curve $x=x(t)$.
Moreover, by Lemma $\ref{pro:c3.2}$, for any $t>t_0$,
there exists a unique $J_n\subset \mathcal{U}^c(x(t),t)$ such that
$(x(\tau),\tau)\in \Delta_{J_n}(x(t),t)$ for any $\tau\in [t_0,t)$.
This means that $\eqref{c3.45}$ holds for the triangle sequence $\{\Delta_{J_n}(x(t),t)\}$.

\smallskip
\noindent
{\bf 1}. For any $t'>t$, denote $u^\pm(t'):=u(x(t'){\pm},t')$.
From $\eqref{c2.61c}$,
$u^\pm(t')\in \mathcal{U}(x(t'),t';\,t)$ so that,
by $\eqref{c2.57}$--$\eqref{c2.61}$,
$$
E(u^-(t');\,x(t'),t';\,t)=E(u^+(t');\,x(t'),t';\,t),
$$
which is equivalent to
\begin{equation}\label{c3.38}
(t'-t)\int^{u^-(t')}_{u^+(t')}f''(s)\big(u(x(t')-(t'-t)f'(s),t)-s\big)\,{\rm d}s=0.
\end{equation}
Let $x_\pm(t'):=x(t')-(t'-t)f'(u^\pm(t'))$.
Taking $\xi=x(t')-(t'-t)f'(s)$ into $\eqref{c3.38}$,
\begin{equation}\label{c3.39}
\int^{x_-(t')}_{x(t)}u(\xi,t)\,{\rm d}\xi+\int^{x(t)}_{x_+(t')}u(\xi,t)\,{\rm d}\xi
=-(t'-t)\int^{u^-(t')}_{u^+(t')}sf''(s)\,{\rm d}s.
\end{equation}
Denote $U_\pm:=U_\pm(t',t)$ by
$$
U_\pm=U_\pm(t',t):=\frac{1}{x_\pm(t')-x(t)}\int^{x_\pm(t')}_{x(t)}u(\xi,t)\,{\rm d}\xi.
$$
From $\eqref{c2.40}$ and $\eqref{c3.27}$,
$\lim_{t'\rightarrow t{+}}u^\pm(t')=u(x(t){\pm},t)$ and
$\lim_{x\rightarrow x(t){\pm}}u(x,t)=u(x(t){\pm},t)$.
By $x_-(t')<x(t)<x_+(t')$ and $\lim_{t'\rightarrow t{+}}x_{\pm} (t')=x(t)$, then
$\lim_{t'\rightarrow t{+}}U_\pm(t',t)=u(x(t){\pm},t)$
which, by $u(x(t){-},t)>u(x(t){+},t)$, implies that
$U_-(t',t)-U_+(t',t)\neq 0$
for sufficiently small $t'-t>0$.
Thus, by taking $U_\pm=U_\pm(t',t)$ into $\eqref{c3.39}$,
$$
U_-\,\big(x_-(t')-x(t)\big)-U_+\,\big(x_+(t')-x(t)\big)
=-(t'-t)\int^{u^-(t')}_{u^+(t')}sf''(s)\,{\rm d}s,
$$
so that, by a simple calculation, we obtain
$$
\frac{x(t')-x(t)}{t'-t}
=\frac{f(u^-(t'))-f(u^+(t'))+\big(U_--u^-(t')\big)f'(u^-(t'))-\big(U_+-u^+(t')\big)f'(u^+(t'))}{U_--U_+},
$$
which, by letting $t'\rightarrow t{+}$, yields $\eqref{c3.36}$.

\smallskip
\noindent
{\bf 2}. Since $x=x(t)$ for $t\geq t_0$ is the forward generalized characteristic of $(x_0,t_0)$,
then, by Lemma $\ref{pro:c3.2}$, for any $t>t_0$,
there exists $J_n=(c_n,d_n) \subset \mathcal{U}^c(x(t),t)$ such that
$(x(\tau),\tau)\in \Delta_{J_n}(x(t),t)$ for $\tau \in [t_0,t)$.
From $\eqref{c2.61c}$,
for any $t'<t$, $\,c_n,d_n\in \mathcal{U}(x(t),t)=\mathcal{U}(x(t),t;t')$ so that,
by $\eqref{c2.57}$--$\eqref{c2.61}$,
$E(d_n;x(t),t;t')=E(c_n;x(t),t;t')$, $i.e.$,
\begin{equation*}
(t-t')\int^{d_n}_{c_n}f''(s)\big(u(x(t)-(t-t')f'(s),t')-s\big)\,{\rm d}s=0,
\end{equation*}
which, by letting $y_-(t'):=x(t)-(t-t')f'(d_n)$ and $y_+(t'):=x(t)-(t-t')f'(c_n)$, implies
\begin{equation}\label{c3.41}
\int^{y_-(t')}_{y_+(t')}u(\xi,t')\,{\rm d}\xi=-(t-t')\int^{d_n}_{c_n}sf''(s)\,{\rm d}s.
\end{equation}
Denote $V_\pm:=V_\pm(t',t)$ by
$$
V_\pm=V_\pm(t',t):=\frac{1}{y_\pm(t')-x(t')}\int^{y_\pm(t')}_{x(t')}u(\xi,t')\,{\rm d}\xi.
$$
Since $y_-(t')<x(t')<y_+(t')$ and $\lim_{t'\rightarrow t{-}}y_\pm (t')=x(t)$,
by Theorem $\ref{pro:c3.4}$(ii), 
$$\lim_{t'\rightarrow t{-}}V_-(t',t)=d_n,\qquad \lim_{t'\rightarrow t{-}}V_+(t',t)=c_n,$$
which, by $d_n>c_n$, implies that
$V_-(t',t)-V_+(t',t)\neq 0$ for sufficiently small $t'-t<0$.
Thus, by taking $V_\pm=V_\pm(t',t)$ into $\eqref{c3.41}$,
$$
V_-\,\big(y_-(t')-x(t')\big)-V_+\,\big(y_+(t')-x(t')\big)
=-(t-t')\int^{d_n}_{c_n}sf''(s)\,{\rm d}s,
$$
so that, by a simple calculation, we obtain
$$
\frac{x(t')-x(t)}{t'-t}
=\frac{f(d_n)-f(c_n)+\big(V_--d_n\big)f'(d_n)-\big(V_+-c_n\big)f'(c_n)}{V_--V_+},
$$
which, by letting $t'\rightarrow t{-}$, yields $\eqref{c3.37}$.

Since $u(x{\pm},t)\in L^\infty(\mathbb{R}\times\mathbb{R}^+)$,
$\eqref{c3.36}$--$\eqref{c3.37}$ imply that
$x=x(t)$ for $t\geq t_0$ is a Lipschitz continuous shock curve.

\smallskip
\noindent
{\bf 3}. For any $t>t_0$, denote $u^\pm(t):=u(x(t){\pm},t)$.
If $\mathcal{U}(x(t),t)=\big\{u^+(t),u^-(t)\big\}$,
then $d_n=u^-(t)$ and $c_n=u^+(t)$ so that
$\frac{{\rm d}^+x(t)}{{\rm d}t}=\frac{{\rm d}^-x(t)}{{\rm d}t}$.

\vspace{1pt}
On the other hand, if there exists $t_1>t_0$ such that
\begin{equation}\label{c3.42}
\frac{{\rm d}^+x(t)}{{\rm d}t}\Big|_{t=t_1}\neq\frac{{\rm d}^-x(t)}{{\rm d}t}\Big|_{t=t_1},
\end{equation}
then $d_n\neq u^-(t_1)$ or $c_n \neq u^+(t_1)$ so that 
$$J_n(t_1):=J_n=(c_n,d_n)\subsetneqq [u^+(t_1),u^-(t_1)].$$
Then, from $\eqref{c3.2}$, there exists $J_{n'}$ or $I_m\subset [u^+(t_1),u^-(t_1)]$ such that
$J_{n'}\cap J_{n}=\varnothing$ or $I_m\cap J_{n}=\varnothing$, respectively.
Denote the open interval $J_{n'}$ or $I_m$ by $J'_n(t_1)$.
Then $J'_n(t_1)\cap J_{n}(t_1)=\varnothing$.

We set
\begin{equation}\label{c3.43}
\begin{cases}
\displaystyle Y_n(t_1):=
\big\{y\in\mathbb{R}\,:\, y=x(t_1)-t_1f'(s)\quad {\rm for}\ s\in J_n(t_1)\big\}, \\[2mm]
\displaystyle Y'_n(t_1):=
\big\{y\in\mathbb{R}\,:\, y=x(t_1)-t_1f'(s)\quad {\rm for}\ s\in J'_n(t_1)\big\}.
\end{cases}
\end{equation}
Then $Y_n(t_1)$ and $Y'_n(t_1)$ are open intervals on the $x$--axis such that
$Y_n(t_1)\cap Y'_n(t_1)=\varnothing.$

\vspace{2pt}
For any $t_2\in (t_0,t_1)$,
by Lemma $\ref{pro:c3.2}$, $(x(\tau),\tau)\in \Delta_{J_n(t_1)}(x(t_1),t_1)$ on $ \tau\in [t_0,t_1)$
so that $(x(t_2),t_2)\in \Delta_{J_n(t_1)}(x(t_1),t_1)$.
If $\eqref{c3.42}$ holds for $t=t_2$,
by Proposition $\ref{pro:c3.1}$(ii),
$$
\Delta_{J_n(t_2)}(x(t_2),t_2)\cup\Delta_{J'_n(t_2)}(x(t_2),t_2)
\subset\Delta_I(x(t_2),t_2)\subset\Delta_{J_n(t_1)}(x(t_1),t_1),
$$
which, by $J'_n(t_1)\cap J_{n}(t_1)=\varnothing$, implies
that
$\Delta_{J'_n(t_1)}(x(t_1),t_1)\cap \Delta_{J'_n(t_2)}(x(t_2),t_2)=\varnothing.$
This, by $\eqref{c3.43}$, infers that $Y'_n(t_1)\cap Y'_n(t_2)=\varnothing$.

Therefore, for any $t_1>t_0$ such that $\eqref{c3.42}$ holds,
there exists an open interval $Y'_n(t_1)$ on the $x$--axis such that the map:
$t_1 \mapsto Y'_n(t_1)$ is one-to-one with $Y'_n(t_1)\cap Y'_n(t_2)=\varnothing$ if $t_1\neq t_2$.
Since the disjoint open intervals on the $x$--axis are at most countable,
then $x=x(t), t\geq t_0$, has the same left- and right-derivatives
except at most countable points of $t\geq t_0$.

\smallskip
This completes the proof of Theorem $\ref{pro:c3.6}$.
\end{proof}

\begin{Rem}
 Along any shock curve $x(t)$ for $t\geq t_0$, the following equation holds{\rm :}
\begin{equation}\label{c3.47}
E(u(x(t){-},t);\,x(t),t)=E(u(x(t){+},t);\,x(t),t) \qquad {\rm for\ any}\ t>t_0.
\end{equation}
{\it In general, we can not directly integrate the ordinary differential equation $\eqref{c3.36}$.
However, it is interesting to see that, in general,
the expression of shock curves can be obtained by using $\eqref{c3.47}$,
if the primitive of the initial data function $\varphi(x)$ is known}.
\end{Rem}

\begin{Rem}
Since $f'(u)$ is strictly increasing, $\eqref{c3.36}$--$\eqref{c3.37}$
imply that, along any shock of the entropy solution $u(x,t)$,
$$
 \frac{{\rm d}^\pm x(t)}{{\rm d}t}\in (f'(u(x(t){+},t)),f'(u(x(t){-},t))).
$$
This means that {\it the entropy solution
given by the solution formula $\eqref{c2.50}$
is the unique Lax's entropy solution in} \cite{LPD57,LPD73}.
Furthermore, since $\frac{{\rm d} x(t)}{{\rm d}t}=f'(u(x(t),t))$
along any shock-free characteristic, the forward generalized characteristics are exactly
the generalized characteristics introduced by Dafermos \cite{DCM2,DCM3,DCM5}.
\end{Rem}

\subsection{Fine structures of entropy solutions on the shock sets}
In this subsection, we show the fine structures of entropy solutions on the shock sets.
In fact, by combining the results in \S 3--4 and \S 5.1--5.4, 
we can obtain a comprehensive understanding of the fine structures of entropy solutions on the shock set 
in the upper half-plane,  
which, in detail, are presented in the following theorem.

\begin{The}[Fine structures of entropy solutions on the shock sets]\label{the:c4.2}
Any point $(x,t)$ with $t>0$ must be one of the following five cases{\rm :}
\begin{itemize}
\item [(i)] If $u(x{-},t)=u(x{+},t)$ with $\mathcal{C}_t(x)\neq\varnothing$,
then $(x,t)$
is a point that
lies in the inner of a characteristic segment.

\item [(ii)] If $u(x{-},t)=u(x{+},t)$ with $\mathcal{C}_t(x)=\varnothing$,
then $(x,t)$ is a continuous shock generation point.
The formation and development of shocks,
especially including the optimal regularities of the entropy solution and the shock curve near the continuous shock generation point,
are given by {\rm Theorems $\ref{the:c4.3}$ and $\ref{the:dsw1}$}.

\item [(iii)] If $u(x{-},t)>u(x{+},t)$ and the closed set $\mathcal{U}(x,t)=[u(x{+},t),u(x{-},t)]$,
then $(x,t)$ is a discontinuous shock generation point.
The entropy solution inside the backward characteristic triangle from such points is given by {\rm Theorem $\ref{pro:c3.3}${\rm(ii)}}.

\item [(iv)] If $u(x{-},t)>u(x{+},t)$
and the open set $\mathcal{U}^c(x,t)\subset [u(x{+},t),u(x{-},t)]$ is an open interval,
then $(x,t)$ is a point of shocks,
and there is only one shock passing through point $(x,t)$ 
that any two shocks passing point $(x,t)$ must coincide with each other after some $\bar{t}<t$.

Moreover, if $\mathcal{U}^c(x,t)=(u(x{+},t),u(x{-},t))$,
point $(x,t)$ is a regular point of shocks{\rm ;}
and if $\mathcal{U}^c(x,t)\subsetneqq (u(x{+},t),u(x{-},t))$,
point $(x,t)$ is an irregular point of shocks,
at which there is a collision of a shock with one or two centered compression waves.
The directional limits of the entropy solution at such points are
given by {\rm Theorem $\ref{pro:c3.4}$}.
The left- and right-derivatives of the shock curve at such points are given by $\eqref{c3.36}$--$\eqref{c3.37}$.

\item [(v)] If $u(x{-},t)>u(x{+},t)$ and the open set $\mathcal{U}^c(x,t)$
is a disjoint union of two or more open intervals,
then there are at least two shocks collided at point $(x,t)$
and $(x,t)$ is an irregular point of shocks,
at which there is a collision of 
more than one shock with or without centered compression waves.

Furthermore, there is a one-to-one correspondence to the shocks collided at point $(x,t)$
and the open intervals in the disjoint union of $\mathcal{U}^c(x,t)$.
The directional limits of the entropy solution
at such points are given by {\rm Theorem $\ref{pro:c3.4}$}.
The right-derivative of the shock curve at such points is given by $\eqref{c3.36}$,
and the left-derivatives of the shock branches at such points are given by $\eqref{c3.37}$ correspondingly.
\end{itemize}
\noindent
In the above, $\mathcal{U}^c(x,t)$ is the complement of $\mathcal{U}(x,t)$
as in $\eqref{c3.02}$.
See {\rm Fig.} $\ref{figCpoints}$ for the details.
\end{The}

\begin{figure}[H]
\begin{center}
{\includegraphics[width=0.7\columnwidth]{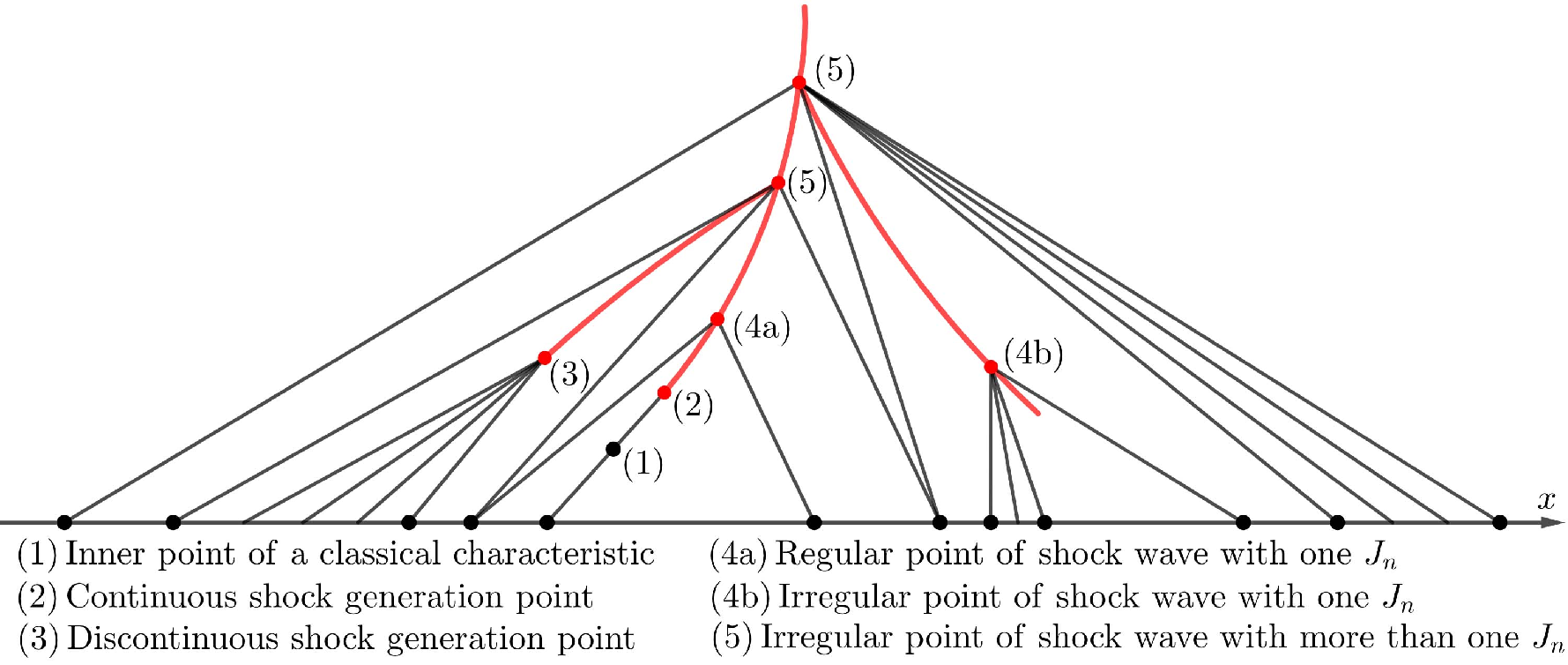}}\caption{Different types of points on the shock sets.}\label{figCpoints}
\end{center}
\end{figure}

\begin{proof} We divide the proof into five steps accordingly.

\vspace{2pt}
\noindent
{\bf 1}. According to Lemma $\ref{lem:c2.5}$ and Corollary $\ref{cor:c4.1}$(ii), (i) is true.

\smallskip
\noindent
{\bf 2}.
By Lemma $\ref{lem:c2.5}$,
the solution $u(x,t)$ is continuous at point $(x,t)$.
From Corollary $\ref{cor:c4.1}$(i),
point $(x,t)$ is a continuous shock generation point.

\smallskip
\noindent
{\bf 3}.
Since $u(x{-},t)>u(x{+},t)$,
then the solution $u(x,t)$ is discontinuous at point $(x,t)$.
From $\mathcal{U}(x,t)=[u(x{+},t),u(x{-},t)]$,
by Theorem $\ref{pro:c3.3}$(ii) and Corollary $\ref{cor:c4.1}$(i),
point $(x,t)$ is a discontinuous shock generation point.

\smallskip
\noindent
{\bf 4}. Since $\mathcal{U}^c(x,t)$ is an open interval,
there is only one $J_n$ at point $(x,t)$.
By Theorem $\ref{pro:c3.3}$(iii),
there exists only one shock passing point $(x,t)$ in the sense that
any two shocks passing point $(x,t)$ must coincide with each other after some time $\bar{t}<t$.

Moreover, if $\mathcal{U}^c(x,t)=J_n=(u(x{+},t),u(x{-},t))$,
then $\mathring{{\mathcal U}}(x,t)=\varnothing$.
Thus, there exist no centered compression waves centered at point $(x,t)$,
so that $(x,t)$ is a regular point of shock curves.
If $\mathcal{U}^c(x,t)=J_n\subsetneqq (u(x{+},t),u(x{-},t))$,
then $\mathring{{\mathcal U}}(x,t)\not=\varnothing$.
Thus, there exists one or two $I_m\subset \mathring{{\mathcal U}}(x,t)$ at point $(x,t)$,
which means that there exist one or two centered compression waves centered at point $(x,t)$.
Thus, point $(x,t)$ is an irregular point of shocks,
which is generated by the collision of a shock with one or two centered compression waves.

\smallskip
\noindent
{\bf 5}. If $\mathcal{U}^c(x,t)$ can be expressed by a disjoint union of two or more open intervals,
then there exist at least two $J_n\subset \mathcal{U}^c(x,t)$ at point $(x,t)$.
By Theorem $\ref{pro:c3.3}$(iii),
there exist at least two shocks collided at point $(x,t)$,
and there is a one-to-one correspondence to the shocks collided at point $(x,t)$
and the open intervals in the disjoint union of $\mathcal{U}^c(x,t)$.
\end{proof}

\section{Invariants, Divides, and Global Structures of Entropy Solutions}
In this section, we analyze the global structures of entropy solutions.
In \S 6.1, we present and prove four invariants of entropy solutions;
in \S 6.2, we provide and prove the necessary and sufficient conditions respectively for the divides as well as their locations and speeds of entropy solutions in Theorem $\ref{the:c5.1}$, 
and obtain all the divides of entropy solutions in Theorem $\ref{the:c5.2}$;
in \S 6.3, we establish the global structures of entropy solutions in Theorem $\ref{pro:c6.2}$.

\subsection{Four invariants of entropy solutions}
For the initial data function $\varphi(x)\in L^{\infty}(\mathbb{R})$,
the following limits are all finite numbers:
\begin{equation}\label{c5.7}
\begin{cases}
\overline{u}_l:
=\displaystyle\mathop{\overline{\lim}}\limits_{x\rightarrow -\infty}
\frac{1}{x}\int^x_0\varphi (\xi)\,{\rm d}\xi,\quad &
\underline{u}_l:=\displaystyle\mathop{\underline{\lim}}\limits_{x\rightarrow -\infty}
\frac{1}{x}\int^x_0\varphi (\xi)\,{\rm d}\xi,\\[4mm]
\overline{u}_r:=\displaystyle\mathop{\overline{\lim}}\limits_{x\rightarrow \infty}
\frac{1}{x}\int^x_0\varphi (\xi)\,{\rm d}\xi,\quad &
\underline{u}_r:
=\displaystyle\mathop{\underline{\lim}}\limits_{x\rightarrow \infty}
\frac{1}{x}\int^x_0\varphi (\xi)\,{\rm d}\xi.
\end{cases}
\end{equation}

In order to show that
$\overline{u}_l$, $\underline{u}_l$, $\overline{u}_r$, and $\underline{u}_r$, given by $\eqref{c5.7}$,
are all the invariants of entropy solutions, we need the following lemma.

\begin{Lem}\label{lem:c6.1}
Let $c\in\mathcal{U}(x_1,t_0)$ and $d\in\mathcal{U}(x_2,t_0)$ with $t_0>0$,
and let the line segments $x_1(t)$ and $x_2(t)$ be given by
$$
x_1(t)=x_1+(t-t_0)f'(c),\quad x_2(t)=x_2+(t-t_0)f'(d)\qquad\,\, {\rm for}\ t\in [0,t_0].
$$
Then, for any $t\in [0,t_0]$,
\begin{equation}\label{c6.11}
tc\big(f'(d)-f'(c)\big)
\leq\int_{x_1(t)}^{x_2(t)}u(x,t)\,{\rm d}x-\int_{x_1(0)}^{x_2(0)}\varphi(x)\,{\rm d}x
\leq td\big(f'(d)-f'(c)\big).
\end{equation}
In particular, if $c=d$, then
\begin{equation}\label{c6.12}
\int_{x_1(t)}^{x_2(t)}u(x,t)\,{\rm d}x=\int_{x_1-t_0f'(c)}^{x_2-t_0f'(c)}\varphi(x)\,{\rm d}x
\qquad\,\, {\rm for\ any}\ t\in [0,t_0].
\end{equation}
\end{Lem}

\begin{proof}
Letting $\varphi_2=\varphi$ and $\varphi_1\equiv c$ in $\eqref{c2.52}$,
and replacing $u_{2,1}$ and $u_{2,2}$ by $c\in\mathcal{U}_2(x_1,t_0)$
and $d\in\mathcal{U}_2(x_2,t_0)$ into $\eqref{c2.52}$ respectively,
we obtain that, for any $t\in [0,t_0]$,
\begin{equation*}
tc\big(f'(d)-f'(c)\big)
\leq\int_{x_1(t)}^{x_2(t)}u(x,t)\,{\rm d}x-\int_{x_1(0)}^{x_2(0)}\varphi(x)\,{\rm d}x
\leq \int_{x_2(t)-tf'(d)}^{x_2(t)-tf'(c)}\varphi(x)\,{\rm d}x.
\end{equation*}
Similarly, letting $\varphi_2=\varphi$ and $\varphi_1\equiv d$ in $\eqref{c2.52}$,
and replacing $u_{2,1}$ and $u_{2,2}$ by $c\in\mathcal{U}_2(x_1,t_0)$
and $d\in\mathcal{U}_2(x_2,t_0)$ into $\eqref{c2.52}$ respectively,
we obtain that, for any $t\in [0,t_0]$,
\begin{equation*}
\int_{x_1(t)-tf'(d)}^{x_1(t)-tf'(c)}\varphi(x)\,{\rm d}x
\leq\int_{x_1(t)}^{x_2(t)}u(x,t)\,{\rm d}x-\int_{x_1(0)}^{x_2(0)}\varphi(x)\,{\rm d}x
\leq td\big(f'(d)-f'(c)\big).
\end{equation*}
This implies $\eqref{c6.11}$.

\smallskip
In particular, $\eqref{c6.12}$ follows by taking $c=d$ into $\eqref{c6.11}$.
\end{proof}

\begin{The}[Four invariants of entropy solutions]\label{pro:c6.1}
 Suppose that $\overline{u}_l$, $\underline{u}_l$, $\overline{u}_r$, and $\underline{u}_r$
 are given by $\eqref{c5.7}$.
Let $u=u(x,t)$ be the entropy solution of the Cauchy problem $\eqref{c1.1}$--$\eqref{ID}$.
Then, for any $t>0$,
\begin{equation}\label{c6.22}
\begin{cases}
 \displaystyle \mathop{\overline{\lim}}\limits_{x\rightarrow-\infty}
 \frac{1}{x}\int^x_0u(\xi,t)\,{\rm d}\xi=\overline{u}_l,\quad &
 \displaystyle \mathop{\underline{\lim}}\limits_{x\rightarrow -\infty}
 \frac{1}{x}\int^x_0u(\xi,t)\,{\rm d}\xi=\underline{u}_l,
\\[4mm]
 \displaystyle \mathop{\overline{\lim}}\limits_{x\rightarrow \infty}
 \frac{1}{x}\int^x_0u(\xi,t)\,{\rm d}\xi=\overline{u}_r, \quad &
 \displaystyle \mathop{\underline{\lim}}\limits_{x\rightarrow \infty}
 \frac{1}{x}\int^x_0u(\xi,t)\,{\rm d}\xi=\underline{u}_r.
\end{cases}
\end{equation}
\end{The}

\begin{proof}
Fix $t>0$. Since $\varphi(x)$ and $u(x,t)$ are uniformly bounded by $\|\varphi\|_{L^\infty}$,
then, for any $x\in \mathbb{R}$,
taking points $(0,t)$ and $(x,t)$ into $\eqref{c6.11}$, we obtain
\begin{equation*}
O_1(1)+\int^x_0\varphi(\xi)\,{\rm d}\xi
\leq\int^{x}_{0}u(\xi,t)\,{\rm d}\xi
\leq\int^x_0\varphi (\xi)\,{\rm d}\xi+O_2(1),
\end{equation*}
which, by multiplying $\frac{1}{x}$ and letting $x\rightarrow\pm\infty$, implies $\eqref{c6.22}$.
\end{proof}

\subsection{Locations and speeds of divides}
A divide is a global-in-time characteristic that does not interact
with any other characteristic for all time $t$.
A more detailed definition of divides is given as
\begin{Def}\label{def:c5.1}
A point $x_0$ on the $x$--axis is called a divide generation point
if there exists $c\in \mathbb{R}$ such that
$\mathcal{U}(x,t)=\{c\}$ for any point $(x,t)\in L_c(x_0):=\{x=x_0+tf'(c),\, t>0\}$.
If $x_0$ is a divide generation point,
line $L_c(x_0)$ is called a divide emitting from $x_0$ with speed $c$,
and the correspondingly speed $c$ is called a divide generation value of $x_0$.

The point set $\mathcal{D}(x_0)$ of divide generation values of $x_0$ is defined by
\begin{equation}\label{c5.0}
\mathcal{D}(x_0):=\big\{c\in \mathbb{R}\,:\, \mathcal{U}(x,t)=\{c\}
\quad {\rm for\ any}\ (x,t)\in L_c(x_0)\ {\rm and} \ t>0\big\}.
\end{equation}
\end{Def}
\noindent
From Definition $\ref{def:c5.1}$, the following three are equivalent:

\smallskip
(i) $\,x_0$ is a divide generation point.
(ii) $\,L_c(x_0)$ is a divide for some $c\in \mathbb{R}$.
(iii) $\,\mathcal{D}(x_0)\neq\varnothing$.

\smallskip
Since a divide $L_c(x_0)$ is necessarily a characteristic,
by Definition $\ref{def:c4.1}$,
\begin{equation}\label{c5.1}
\mathcal{D}(x_0)\subset \mathcal{C}(x_0)
\subset [\overline{{\rm D}}_-\Phi(x_0),\underline{{\rm D}}_+\Phi(x_0)].
\end{equation}

We first present the necessary and sufficient conditions for the divides in Proposition $\ref{pro:c5.1}$,
and a basic property of $\mathcal{D}(x_0)$ in Proposition $\ref{pro:c5.2}$.

\begin{Pro}\label{pro:c5.1}
$L_c(x_0):\, x=x_0+tf'(c)$ for $t>0$ is a divide if and only if
\begin{equation}\label{c5.1d}
\Phi(l;x_0,c)=
\int^{x_0+l}_{x_0}(\varphi (\xi)-c)\,{\rm d}\xi\geq 0
\qquad\, {\rm for\ any}\ l\in \mathbb{R}.
\end{equation}
\end{Pro}

\begin{proof}
First, if $L_c(x_0)$ is a divide,
then $\mathcal{U}(x,t)=\{c\}$ for any $(x,t)\in L_c(x_0)$ so that
\begin{equation}\label{c5.2}
E(c;x,t)-E(u;x,t)\geq 0 \qquad {\rm for\ any\ }(x,t)\in L_c(x_0)\ {\rm and}\ u\in\mathbb{R}.
\end{equation}
For any given $l_0\neq 0$, for sufficiently large $t$,
there exists $u(t)\in \mathbb{R}$ such that
$l_0=t(f'(c)-f'(u(t)))$,
which, by the strictly increasing property of $f'(u)$, implies that
$\lim_{t\rightarrow \infty} u(t)=c$
and
$u(t)\neq c$.
From $\eqref{c4.2}$, taking $l_0$ and $u=u(t)$ into $\eqref{c5.2}$,
we obtain
\begin{align*}
\int^{x_0+l_0}_{x_0}(\varphi (\xi)-c)\,{\rm d}\xi&\geq -t\int^{u(t)}_c(s-c)f''(s)\,{\rm d}s
\nonumber\\[1mm]
&=-t\,\big((u(t)-c)f'(u(t))-(f(u(t))-f(c))\big)\nonumber\\[1mm]
&=l_0(u(t)-c)\frac{f'(u(t))-f'(\hat{u}(t))}{f'(u(t))-f'(c)}
\nonumber\\[1mm]
&
=l_0\,(u(t)-c)\,\theta\ \longrightarrow 0 \qquad {\rm as}\ t\rightarrow \infty,
\end{align*}
where $\hat{u}(t)$ lies between $c$ and $u(t)$, and $\theta \in(0,1)$.
This, by the arbitrariness of $l_0$, implies $\eqref{c5.1d}$.

On the other hand, suppose that $\eqref{c5.1d}$ holds.
For any fixed $(x,t)\in L_c(x_0)$, if $u\neq c$,
then, letting $l=t(f'(c)-f'(u))$ in $\eqref{c5.1d}$, we obtain
$$
E(c;x,t)-E(u;x,t)\geq t\int^u_c(s-c)f''(s)\,{\rm d}s > 0,
$$
which means that $c$ is the unique point at which $E(\cdot\,;x,t)$ attains its maximum.
This implies that $\mathcal{U}(x,t)=\{c\}$.
By the arbitrariness of point $(x,t)\in L_c(x_0)$, $L_c(x_0)$ is a divide.
\end{proof}

\begin{Pro}\label{pro:c5.2}
If $\mathcal{D}(x_0)$ defined by $\eqref{c5.0}$ has more than one element,
then it is a closed interval{\rm :}
\begin{itemize}
\item [(i)] For any $c_{-},c_{+} \in \mathcal{D}(x_0)$ with $c_{-}<c_{+}$,
then $[c_{-},c_{+}]\subset \mathcal{D}(x_0)${\rm ;}

\vspace{2pt}
\item [(ii)] If $(c_{-},c_{+})\subset \mathcal{D}(x_0)$,
then $[c_{-},c_{+}]\subset \mathcal{D}(x_0)$.
\end{itemize}
\end{Pro}

\begin{proof} We divide the proof into two steps accordingly.

\smallskip
\noindent
{\bf 1}. Suppose that $c_{-},c_{+} \in \mathcal{D}(x_0)$ with $c_{-}<c_{+}$.
From $\eqref{c5.1d}$, for any $c\in (c_{-},c_{+})$,
\begin{equation*}
\int^{x_0+l}_{x_0}(\varphi (\xi)-c)\,{\rm d}\xi
\geq\int^{x_0+l}_{x_0}(\varphi (\xi)-c_{\pm})\,{\rm d}\xi \geq 0
\qquad\,\, {\rm if}\ \pm l\geq 0,
\end{equation*}
which, by Proposition $\ref{pro:c5.1}$, implies that $c\in \mathcal{D}(x_0)$.
Thus, $[c_{-},c_{+}]\subset \mathcal{D}(x_0)$.

\smallskip
\noindent
{\bf 2}. If $(c_{-},c_{+})\subset \mathcal{D}(x_0)$, then, from $\eqref{c5.1d}$,
for any $c\in(c_{-},c_{+})$,
$$
\int^{x_0+l}_{x_0}(\varphi (\xi)-c)\,{\rm d}\xi \geq 0
\qquad {\rm for\ any}\ l\in\mathbb{R},
$$
which, by letting $c\rightarrow c_{-}{+}$ and $c\rightarrow c_{+}{-}$ respectively,
implies that $\eqref{c5.1d}$ holds for $c=c_{-}$ and $c=c_{+}$.
This means that $c_{-},c_{+} \in \mathcal{D}(x_0)$ so that $[c_{-},c_{+}]\subset \mathcal{D}(x_0)$.
\end{proof}

From Proposition $\ref{pro:c5.2}$,
$\mathcal{D}(x_0)$ may be the empty set, a single point set, or a closed interval.
In order to determine the point set $\mathcal{D}(x_0)$ exactly,
we introduce the convex hull of $\Phi(x)=\int_0^x\varphi(\xi)\,{\rm d}\xi$ over $({-}\infty,\infty)$,
denoted by $\bar{\Phi}(x)$, as follows:
\begin{equation}\label{c5.4}
\bar{\Phi}(x)=\lim_{N\rightarrow \infty}\bar{\Phi}_N(x)\qquad {\rm for \ any}\ x\in \mathbb{R},
\end{equation}
where $\bar{\Phi}_N(x)$ is the convex-truncated function of $\Phi(x)$ over $[-N,N]$ given by
\begin{equation}\label{c5.5}
\bar{\Phi}_N(x)=
\begin{cases}
\mathop{cov}_{[-N,N]}\Phi (x) \qquad &{\rm if}\ |x|\leq N,\\[1mm]
\Phi (x)\qquad &{\rm if}\ |x|\geq N.
\end{cases}
\end{equation}
Furthermore, let $\mathcal{K}_0$ be the set of points
at which $\Phi(x)$ equals its convex hull $\bar{\Phi}(x)$, $i.e.$,
\begin{equation}\label{c5.6}
\mathcal{K}_0=\big\{x_0\in\mathbb{R}\,:\, \Phi(x_0)=\bar{\Phi}(x_0)\big\}.
\end{equation}

\smallskip
In order to keep the coherence of writing,
we list the properties of $\bar{\Phi}(x)$ in Lemmas $\ref{lem:c5.1}$--$\ref{lem:c5.3}$,
and postpone their proofs in Appendix B for completeness.

\begin{Lem}\label{lem:c5.1}
$\bar{\Phi }(x)>-\infty$ on $\mathbb{R}$ if and only if
\begin{align}\label{phicc}
\mathop{\underline{\lim}}\limits_{x\rightarrow \pm \infty}(\Phi(x)-cx)>-\infty \qquad
{\rm for\ some}\ c\in \mathbb{R}.
\end{align}
\end{Lem}

\begin{Lem}\label{lem:c5.2}
There are three cases of $\bar{\Phi}(x)${\rm :}
\begin{enumerate}
\item[{\rm (i)}] If $\overline{u}_l<\underline{u}_r$,
then $\bar{\Phi} (x)>-\infty$ on $\mathbb{R}$ and satisfies that
\begin{equation}\label{c5.8a}
\overline{u}_l\leq {\rm D}_-\bar{\Phi}(y)\leq {\rm D}_+\bar{\Phi}(y)\leq\underline{u}_r
\qquad {\rm for \ any}\ y\in \mathbb{R},
\end{equation}
and $\mathcal{K}_0\neq \varnothing$ such that
\begin{equation}\label{c5.8}
\inf_{x_0\in \mathcal{K}_0}{\rm D}_-\bar{\Phi}(x_0)=\overline{u}_l,\qquad
\sup_{x_0\in \mathcal{K}_0}{\rm D}_+\bar{\Phi}(x_0)=\underline{u}_r.
\end{equation}

\item[{\rm (ii)}] If $\overline{u}_l=\underline{u}_r$,
then $\bar{\Phi} (x)>-\infty$ on $\mathbb{R}$ if and only if
\begin{equation}\label{c5.8b}
\mathop{\underline{\lim}}_{x\rightarrow \pm \infty}(\Phi(x)-\underline{u}_rx)>-\infty.
\end{equation}
Furthermore, if $\bar{\Phi} (x)>-\infty$ on $ \mathbb{R}$, then
there exists $b_0\in \mathbb{R}$ such that $\bar{\Phi}(x)=\underline{u}_r x+b_0$.
There exist two subcases{\rm :}
\begin{itemize}
\item [(a)] $\mathcal{K}_0\neq \varnothing$ if $\Phi(x_0)=\underline{u}_r x_0+b_0$
for some $x_0\in \mathbb{R}${\rm ;}

\vspace{2pt}
\item [(b)] $\mathcal{K}_0= \varnothing$ if $\Phi(x)>\underline{u}_r x+b_0$ for any $x\in \mathbb{R}$.
\end{itemize}

\vspace{2pt}
\item[{\rm (iii)}] If $\overline{u}_l>\underline{u}_r$, then $\bar{\Phi} (x)\equiv-\infty$.
\end{enumerate}
\end{Lem}

\begin{Lem}\label{lem:c5.3}
If $\bar{\Phi}(x)>-\infty$ on $\mathbb{R}$,
$\bar{\Phi} (x)$ is convex on $\mathbb{R}$.
Furthermore, if $\mathcal{K}_0\neq \varnothing$, then

\begin{itemize}
 \item [(i)] For any $x_0\in \mathcal{K}_0$,
\begin{equation}\label{c5.10}
\begin{cases}
\displaystyle {\rm D}_+\bar{\Phi}(x_0)
=\lim\limits_{x\rightarrow x_{0}{+}}\frac{\bar{\Phi}(x)-\bar{\Phi}(x_0)}{x-x_0}
=\inf_{l>0}\Delta_l\Phi(x_0),\\[4mm]
\displaystyle {\rm D}_-\bar{\Phi}(x_0)
=\lim\limits_{x\rightarrow x_{0}{-}}\frac{\bar{\Phi}(x)-\bar{\Phi}(x_0)}{x-x_0}
=\sup_{l<0}\Delta_l\Phi(x_0),
\end{cases}
\end{equation}
where $\Delta_l\Phi(x_0)=\frac{1}{l}\big(\Phi(x_0+l)-\Phi(x_0)\big)$ for $l\neq 0$.

\vspace{2pt}
\item [(ii)] $\mathcal{K}_0$ is a closed set,
and $\mathcal{K}^c_0:=\mathbb{R}-\mathcal{K}_0$ is an open set.
The following three properties hold{\rm :}
\begin{itemize}
\item [(a)] If $ x^+_0:= \sup\mathcal{K}_0<\infty$,
then $ x^+_0\in \mathcal{K}_0$ with ${\rm D}_+\bar{\Phi}(x^+_0)= \underline{u}_r$ such that
\begin{equation}\label{c5.11}
\bar{\Phi}(x)=\Phi(x^+_0)+\underline{u}_r(x-x^+_0)\qquad {\rm for \ any}\ x\geq x^+_0.
\end{equation}
\item [(b)] If $ x^-_0:=\inf \mathcal{K}_0>-\infty$,
then $ x^-_0\in \mathcal{K}_0$ with ${\rm D}_-\bar{\Phi}(x^-_0)= \overline{u}_l$ such that
\begin{equation}\label{c5.12}
\bar{\Phi}(x)=\Phi(x^-_0)+\overline{u}_l(x-x^-_0)\qquad {\rm for \ any}\ x\leq x^-_0.
\end{equation}
\item [(c)] Let the open set $\mathcal{K}^c_0:=\cup_n (e_n,h_n)$.
If $(e_n,h_n)$ is bounded, then
\begin{equation}\label{c5.13}
\bar{\Phi}(x)={\Phi}(e_n)+{\rm D}_+\bar{\Phi}(e_n)(x-e_n)\qquad {\rm for \ any}\ x \in [e_n, h_n],
\end{equation}
where ${\rm D}_+\bar{\Phi}(e_n)={\rm D}_-\bar{\Phi}(h_n)=\frac{{\Phi}(h_n)-{\Phi}(e_n)}{h_n-e_n}$.
\end{itemize}
\end{itemize}
\end{Lem}

\smallskip
We now present the main theorems of divides.

\begin{The}[Necessary and sufficient conditions for the locations and speeds of divides]\label{the:c5.1}
Let $\mathcal{D}(x_0)$ be defined by $\eqref{c5.0}$, 
and let $\bar{\Phi}(x)$ as in $\eqref{c5.4}$--$\eqref{c5.5}$ be the convex hull of $\Phi(x)$
over $({-}\infty,\infty)$ with $\varphi(x)\in L^\infty(\mathbb{R})$.
Then
\begin{itemize}
\item [(i)] If $\bar{\Phi} (x)>-\infty$ on $\mathbb{R}$,
then $c\in \mathcal{D}(x_0)$ if and only if
\begin{equation}\label{c5.14}
\Phi(x_0)=\bar{\Phi}(x_0), \qquad {\rm D}_-\bar{\Phi}(x_0)\leq c\leq {\rm D}_+\bar{\Phi}(x_0).
\end{equation}
\item [(ii)] If $\bar{\Phi} (x)\equiv-\infty$,
then $\mathcal{D}(x_0)=\varnothing$ for any $x_0\in \mathbb{R}$.
\end{itemize}
\noindent
In the above, ${\rm D}_{+}\bar{\Phi}(x_0)$ and ${\rm D}_{-}\bar{\Phi}(x_0)$
are given by $\eqref{c5.10}$.
See {\rm Figs.} $\ref{figConvexH}$--$\ref{figConvexHK}$ for the details.
\end{The}

\begin{figure}[H]
	\begin{center}
		{\includegraphics[width=0.6\columnwidth]{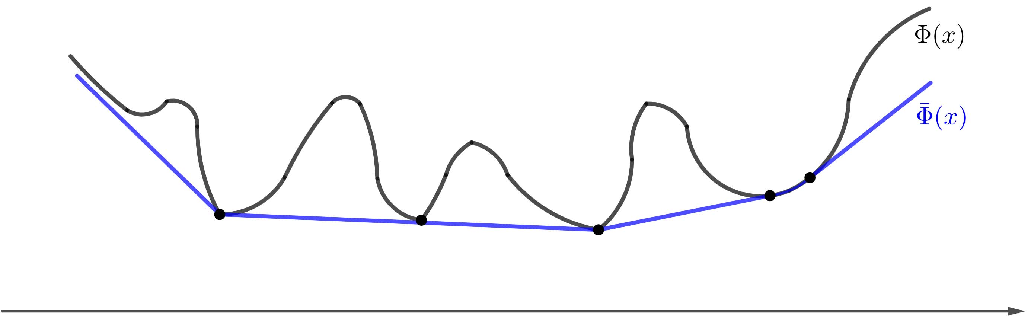}}
 \caption{The graph of $\Phi(x)$ and $\bar{\Phi}(x)$ when $\overline{u}_l<\underline{u}_r$.
		}\label{figConvexH}
	\end{center}
\end{figure}

\begin{figure}[H]
	\begin{center}
		{\includegraphics[width=0.7\columnwidth]{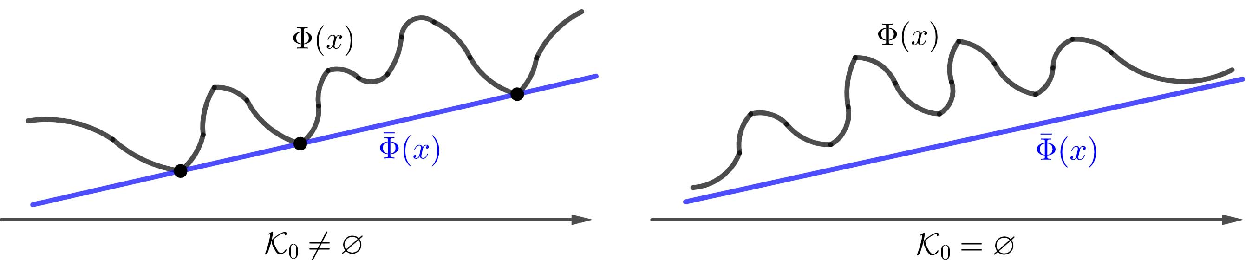}}
 \caption{The graph of $\Phi(x)$ and $\bar{\Phi}(x)$ when $\overline{u}_l=\underline{u}_r$.
		}\label{figConvexHK}
	\end{center}
\end{figure}

\begin{proof} We divide the proof into two steps accordingly.

\vspace{1pt}
\noindent
{\bf 1}. First, if $c\in \mathcal{D}(x_0)$, then $\eqref{c5.1d}$ holds.
By taking $x:=x_0+l$ into $\eqref{c5.1d}$,
\begin{equation*}
\Phi(x)\geq c(x-x_0)+\Phi(x_0)\qquad {\rm for \ any}\ x\in \mathbb{R}.
\end{equation*}

Since $\bar{\Phi}_N(x)$
is the convex hull of $\Phi(x)$ over $[-N,N]$ and $\bar{\Phi}_N(x)={\Phi}(x)$ for $|x|\geq N$, then
$\bar{\Phi}_N(x)\geq c(x-x_0)+\Phi(x_0)$ for any $x\in \mathbb{R}$,
so that $\bar{\Phi}_N(x_0)\geq\Phi(x_0)$.
Meanwhile, from $\eqref{c5.5}$, $\bar{\Phi}_N(x_0)\leq\Phi(x_0)$ for any $N>0$.
Thus, $\bar{\Phi}_N(x_0)=\Phi(x_0)$.
Letting $N\rightarrow \infty$, we obtain
\begin{equation*}
\bar{\Phi}(x_0)=\lim_{N\rightarrow \infty}\bar{\Phi}_N(x_0)=\Phi(x_0).
\end{equation*}
Furthermore, it is direct to check from $\eqref{c5.1d}$ that
\begin{equation*}
\sup_{l<0} \frac{1}{l}\big(\Phi(x_0+l)-\Phi(x_0)\big)\leq c
\leq\inf_{l>0} \frac{1}{l}\big(\Phi(x_0+l)-\Phi(x_0)\big),
\end{equation*}
which, by $\eqref{c5.10}$,
implies that ${\rm D}_-\bar{\Phi}(x_0)\leq c\leq {\rm D}_+\bar{\Phi}(x_0)$.

On the other hand, if $\bar{\Phi}(x_0)=\Phi(x_0)$
and ${\rm D}_-\bar{\Phi}(x_0)\leq c\leq {\rm D}_+\bar{\Phi}(x_0)$,
then it follows from $\eqref{c5.10}$ that $\eqref{c5.1d}$ holds, $i.e.$,
$c\in \mathcal{D}(x_0)$.

\smallskip
\noindent
{\bf 2}. Fix $x_0\in \mathbb{R}$.
Since $\bar{\Phi}(x)\equiv -\infty$, by Lemma $\ref{lem:c5.1}$,
for any fixed $c\in \mathbb{R}$,
$$
\mathop{\underline{\lim}}\limits_{x\rightarrow \infty}(\Phi(x)-cx)=-\infty \qquad
{\rm or}\quad\mathop{\underline{\lim}}\limits_{x\rightarrow -\infty}(\Phi(x)-cx)=-\infty
$$
which is to say that:
\begin{itemize}
\item [(a)] If $\mathop{\underline{\lim}}_{x\rightarrow \infty}(\Phi(x)-cx)=-\infty$,
then $\Phi(x)-cx<\Phi(x_0)-cx_0$ for sufficiently large $x>x_0$
so that $\int^x_{x_0}(\varphi (\xi)-c)\,{\rm d}\xi <0$,
which, by $\eqref{c5.1d}$, implies that $c \notin \mathcal{D}(x_0)$.

\item [(b)] If $\mathop{\underline{\lim}}_{x\rightarrow -\infty}(\Phi(x)-cx)=-\infty$,
then $\Phi(x)-cx<\Phi(x_0)-cx_0$ for sufficiently negative $x<x_0$
so that $\int^x_{x_0}(\varphi (\xi)-c)\,{\rm d}\xi <0$,
which, by $\eqref{c5.1d}$, implies that $c \notin \mathcal{D}(x_0)$.
\end{itemize}

By the arbitrariness of $c\in \mathbb{R}$, we see that $\mathcal{D}(x_0)=\varnothing$.
Then, by the arbitrariness of $x_0\in \mathbb{R}$, we conclude that
$\mathcal{D}(x_0)=\varnothing$ for any $x_0\in \mathbb{R}$.
\end{proof}

\begin{The}[Locations and speeds of divides of entropy solutions]\label{the:c5.2}
Let $\mathcal{D}(x_0)$ and $\bar{\Phi}(x)$ be defined by $\eqref{c5.0}$ and $\eqref{c5.4}$, respectively.
For the Cauchy problem $\eqref{c1.1}$--$\eqref{ID}$
with initial data function $\varphi(x)\in L^\infty(\mathbb{R})$,
\begin{itemize}
\item [(i)] If $\overline{u}_l<\underline{u}_r$,
the entropy solution $u(x,t)$ of the Cauchy problem $\eqref{c1.1}$--$\eqref{ID}$ has divides.
For any $x_0\in \mathbb{R}$, 
$\mathcal{D}(x_0)\neq \varnothing$ if and only if $\Phi(x_0)=\bar{\Phi}(x_0)$,
which implies that, in this case, 

\begin{equation}\label{c5.18}
\mathcal{D}(x_0)=[{\rm D}_-\bar{\Phi}(x_0), {\rm D}_+\bar{\Phi}(x_0)].
\end{equation}
Furthermore, there exist two cases{\rm:}
\begin{itemize}
\item [\rm (a)] If ${\rm D}_-\bar{\Phi}(x_0)={\rm D}_+\bar{\Phi}(x_0)$,
there exists a unique divide emitting from $x_0$ given by 
\begin{equation}\label{c5.19}
L_{{\rm D}_-\bar{\Phi}(x_0)}(x_0):\, x=x_0+tf'({\rm D}_-\bar{\Phi}(x_0))
\qquad {\rm for}\ t\in \mathbb{R}^+.
\end{equation}
\item[\rm (b)] If ${\rm D}_-\bar{\Phi}(x_0)<{\rm D}_+\bar{\Phi}(x_0)$,
the divides emitting from $x_0$ form 
a rarefaction wave{\rm :}
\begin{equation}\label{c5.20}
u(x,t)=(f')^{-1}(\frac{x-x_0}{t})
\end{equation}
for $x\in(x_0+tf'({\rm D}_-\bar{\Phi}(x_0)),x_0+tf'({\rm D}_+\bar{\Phi}(x_0)))$.
\end{itemize}

\item [(ii)] If $\overline{u}_l=\underline{u}_r$,
the entropy solution $u(x,t)$ of the Cauchy problem $\eqref{c1.1}$--$\eqref{ID}$ has divides
if and only if
there exists $x_0\in \mathbb{R}$ such that $\Phi(x_0)=\bar{\Phi}(x_0)$, 
which implies that, in this case, the unique divide emitting from $x_0$ is given by
\begin{equation}\label{c5.21}
L_{\underline{u}_r}(x_0):\, x=x_0+tf'(\underline{u}_r)
\qquad\,\, {\rm for }\ t\in \mathbb{R}^+.
\end{equation}

\item [(iii)] If $\overline{u}_l>\underline{u}_r$,
the entropy solution $u(x,t)$ of the Cauchy problem $\eqref{c1.1}$--$\eqref{ID}$ has no divides.
\end{itemize}
\noindent
In the above, ${\rm D}_{+}\bar{\Phi}(x_0)$ and ${\rm D}_{-}\bar{\Phi}(x_0)$
are given by $\eqref{c5.10}$.
See also {\rm Figs.} $\ref{figDSD}$--$\ref{figDSDK}$.
\end{The}

\begin{figure}[H]
	\begin{center}
		{\includegraphics[width=0.6\columnwidth]{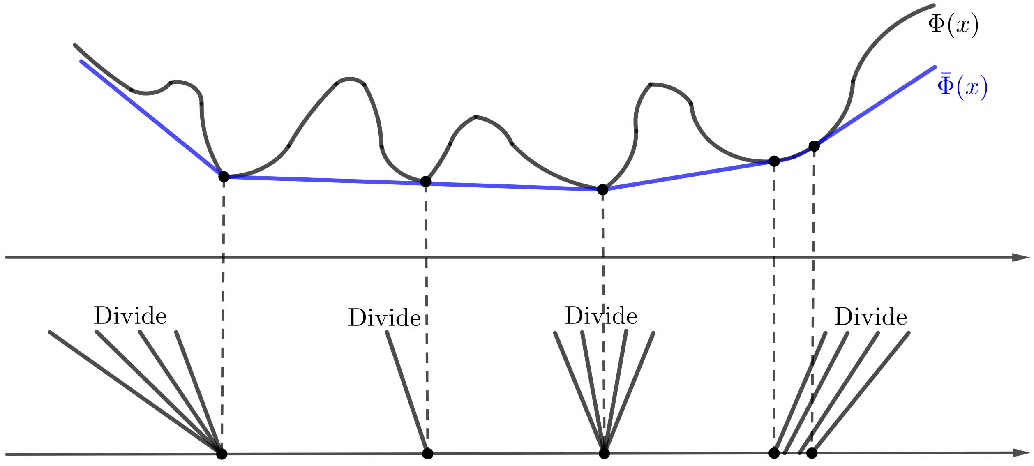}}
 \caption{Divides of entropy solutions when $\overline{u}_l<\underline{u}_r$.
		}\label{figDSD}
	\end{center}
\end{figure}

\begin{figure}[H]
	\begin{center}
		{\includegraphics[width=0.7\columnwidth]{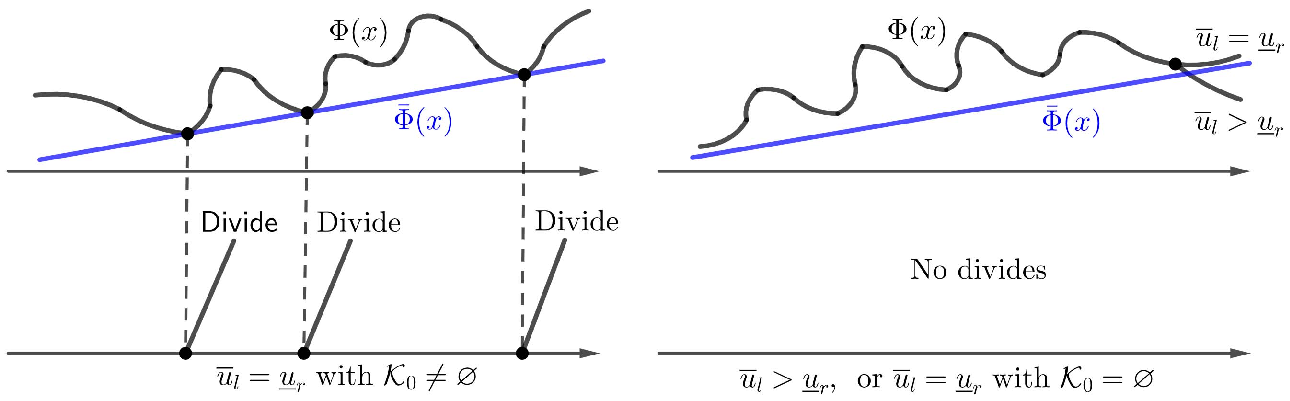}}
 \caption{Divides of entropy solutions when $\overline{u}_l\geq \underline{u}_r$.
		}\label{figDSDK}
	\end{center}
\end{figure}

\begin{proof}
We divide the proof into three steps accordingly.

\smallskip
\noindent
{\bf 1}. From Lemma $\ref{lem:c5.2}$(i), $\mathcal{K}_0\neq \varnothing$ implies that
the entropy solution $u(x,t)$ has divides.
By Theorem $\ref{the:c5.1}$(i), $\mathcal{D}(x_0)\neq \varnothing$
if and only if $\Phi(x_0)=\bar{\Phi}(x_0)$,
and $\eqref{c5.18}$ follows by $\eqref{c5.8}$ and $\eqref{c5.14}$.
Furthermore, if $\mathcal{D}(x_0)\neq \varnothing$,
$\eqref{c5.19}$--$\eqref{c5.20}$ follow by Definition $\ref{def:c5.1}$.

\smallskip
\noindent
{\bf 2}. From Lemma $\ref{lem:c5.2}$(ii) and Theorem $\ref{the:c5.1}$,
the entropy solution $u(x,t)$ has divides if and only if
there exists $x_0\in \mathbb{R}$ such that $\Phi(x_0)=\bar{\Phi}(x_0)$.
If this is the case,
from Lemma $\ref{lem:c5.2}$(ii), $\bar{\Phi}'(x)\equiv \underline{u}_r$ so that,
by Definition $\ref{def:c5.1}$, $\eqref{c5.21}$ holds.

\smallskip
\noindent
{\bf 3}. If $\overline{u}_l>\underline{u}_r$,
from Lemma $\ref{lem:c5.2}$(iii), $\bar{\Phi}(x)\equiv-\infty$ so that,
by Theorem $\ref{the:c5.1}$(ii), $\mathcal{D}(x_0)=\varnothing$ for any $x_0\in\mathbb{R}$,
which implies that the entropy solution $u(x,t)$ has no divides.
\end{proof}

\begin{Exa}[Initial data as a compact perturbation of a constant]\label{exa:c5.1}
Let the initial data function $\varphi(x)\in L^\infty(\mathbb{R})$ satisfy that
$\varphi(x)-m$ has compact support, $i.e.$,
${\rm spt}(\varphi(x)-m)\subset [-R,R]$ for some $R>0$.

It is direct to check that $\overline{u}_l=\underline{u}_r=m$,
and
$$
\mathop{\underline{\lim}}_{x\rightarrow \pm\infty}(\Phi(x)-mx) >-\infty.
$$
By {\rm Lemma} $\ref{lem:c5.2}{\rm (ii)}$,
we have
$$
\bar{\Phi}(x)=mx+b_0\qquad {\rm for}\ b_0=\inf_{x\in \mathbb{R}}\{\Phi (x)-mx\}.
$$

Since ${\rm spt}(\varphi(x)-m)\subset [-R,R]$,
$\Phi (x)-mx$ is constant if $|x|>R$ so that
there exists $x_0$ such that $\Phi (x)-mx$ attains infimum at $x_0$.
From $\eqref{c5.6}$, $\mathcal{K}_0\neq \varnothing$ is given by
\begin{equation}\label{c5.22}
\mathcal{K}_0=\big\{x_0\,:\, \Phi (x_0)-mx_0 =\min_{x\in \mathbb{R}}\{\Phi (x)-mx\}\big\}.
\end{equation}
From {\rm Theorem} $\ref{the:c5.2}{\rm(ii)}$,
all the divides of the entropy solution $u(x,t)$ are given by
\begin{equation}\label{c5.22m}
L_m(x_0):\, x=x_0+tf'(m) \qquad {\rm for}\ t\in \mathbb{R}^+,
\end{equation}
where $x_0\in \mathcal{K}_0$ with $\mathcal{K}_0$ given by $\eqref{c5.22}$.
\end{Exa}

\begin{Exa}[Periodic initial data]\label{exa:c5.2}
Let the initial data function $\varphi(x)\in L^\infty(\mathbb{R})$ be a periodic function
with the minimum positive period $p>0$ and the average $m$.

It is direct to check that
$\frac{1}{x}\int^x_0\varphi (\xi)\,{\rm d}\xi \rightarrow m$ as $x \rightarrow \pm\infty$
so that
$\overline{u}_l=\underline{u}_r=m$, and
$$
\mathop{\underline{\lim}}\limits_{x\rightarrow \pm\infty}\int^x_0(\varphi (\xi)-m)\,{\rm d}\xi =\mathop{\underline{\lim}}\limits_{x\rightarrow \pm\infty}
\int^x_{[\frac{x}{p}]\, p}(\varphi (\xi)-m)\,{\rm d}\xi
\geq -\big(\|\varphi\|_{L^\infty}+|m|\big)p>-\infty.
$$
By {\rm Lemma} $\ref{lem:c5.2}{\rm (ii)}$,
$\bar{\Phi}(x)=mx+b_0$ with $b_0=\min_{x\in [0,p]}\{\Phi (x)-mx\}$.
Then, from $\eqref{c5.6}$,
\begin{equation}\label{c5.23}
\mathcal{K}_0=\big\{x_0+np \,:\, x_0 \ {\rm minimazes}\ \Phi (x)-mx
\,\,\, {\rm on}\ [0,p],\,\, n\in \mathbb{Z}\big\}.
\end{equation}
Since $\Phi (x)-mx$ is continuous on $[0,p]$, then $\mathcal{K}_0\neq \varnothing$.
From {\rm Theorem} $\ref{the:c5.2}{\rm(ii)}$,
all the divides of the entropy solution $u(x,t)$ are given by $\eqref{c5.22m}$
for $x_0\in \mathcal{K}_0$ with $\mathcal{K}_0$ given by $\eqref{c5.23}$.
\end{Exa}

\begin{Exa}[Initial data as the $L^1$ perturbation of a constant]\label{exa:c5.3}
Let the initial data function $\varphi(x)\in L^\infty(\mathbb{R})$
satisfy that $\varphi(x)-m \in L^1(\mathbb{R})$.
It is direct to check that $\overline{u}_l=\underline{u}_r=m$ and
$$
\mathop{\underline{\lim}}\limits_{x\rightarrow \pm\infty}(\Phi(x)-mx)
=\mathop{\underline{\lim}}\limits_{x\rightarrow \pm\infty}
\int^x_0(\varphi (\xi)-m)\,{\rm d}\xi
\geq-\|\varphi(x)-m\|_{L^1(\mathbb{R})}>-\infty.
$$
By {\rm Lemma} $\ref{lem:c5.2}{\rm(ii)}$,
$\bar{\Phi}(x)=mx+b_0$ with $b_0=\inf_{x\in \mathbb{R}}\{\Phi (x)-mx\}$.
Then, from $\eqref{c5.6}$,
\begin{equation}\label{c5.22c}
\mathcal{K}_0=\big\{x_0\,:\, \Phi (x_0)-mx_0 =\inf_{x\in \mathbb{R}}\{\Phi (x)-mx\}\big\}.
\end{equation}
Then there exist two cases{\rm :}
If
$\Phi (x)-mx$ attains its infimum only at $-\infty$ or $\infty$,
then $\mathcal{K}_0=\varnothing$ so that,
by {\rm Theorem} $\ref{the:c5.2}{\rm(ii)}$,
the entropy solution $u(x,t)$ has no divides{\rm ;}
and if $\Phi (x)-mx$ attains its infimum at some finite points,
then $\mathcal{K}_0\neq \varnothing$ so that the entropy solution $u(x,t)$ has divides,
and all the divides are given by $\eqref{c5.22m}$
for $x_0\in \mathcal{K}_0$ with $\mathcal{K}_0$ given by $\eqref{c5.22c}$.
\end{Exa}

\subsection{Global structures of entropy solutions}
In this subsection,
we establish the global structures of entropy solutions of the Cauchy problem $\eqref{c1.1}$--$\eqref{ID}$.
First, we give a partition of the upper half-plane $\mathbb{R}\times \mathbb{R}^+$ as follows:

\smallskip
If $\bar{\Phi}(x)>-\infty$ with $\mathcal{K}_0\neq \varnothing$,
by the continuity of $\Phi(x)$ and $\bar{\Phi}(x)$,
$\mathcal{K}_0$ is a closed set so that $\mathcal{K}^c_0=\mathbb{R}-\mathcal{K}_0$ is open.
Thus, $\mathcal{K}^c_0$ can be expressed as a disjoint union of open intervals, $i.e.$,
\begin{equation}\label{c6.1}
\mathcal{K}^c_0:=\cup_n(e_n,h_n)
\end{equation}
with $e_n$ ({\it resp.}, $h_n$) lying in $\mathcal{K}_0$ if $e_n$ ({\it resp.}, $h_n$) is finite.

\smallskip
From Lemma $\ref{lem:c5.3}$(ii), if $(e_n,h_n)$ is a bounded interval, then
\begin{equation}\label{c6.2}
c_n:={\rm D}_+\bar{\Phi}(e_n)={\rm D}_-\bar{\Phi}(h_n)=\frac{\Phi(h_n)-\Phi(e_n)}{h_n-e_n},
\end{equation}
which, by Theorem $\ref{the:c5.2}$, implies that $L_{c_n}(e_n)$ and $L_{c_n}(h_n)$ are both divides.
From $\eqref{c6.1}$, there exists no divides in $\mathcal{H}_{(e_n,h_n)}$,
where $\mathcal{H}_{(e_n,h_n)}$ is defined by
\begin{equation}\label{c6.3}
\mathcal{H}_{(e_n,h_n)}:=
\big\{(x,t)\in \mathbb{R}\times \mathbb{R}^+\,:\, e_n+tf'(c_n)<x<h_n+tf'(c_n)\big\}.
\end{equation}
Similarly, if $x_0^+=\sup\mathcal{K}_0<\infty$,
then $L_{\underline{u}_r}(x_0^+)$ is a divide,
and there exists no divides in $\mathcal{H}_{\infty}$,
where $\mathcal{H}_{\infty}$ is defined by
\begin{equation}\label{c6.4}
\mathcal{H}_{\infty}:=
\big\{(x,t)\in \mathbb{R}\times \mathbb{R}^+\,:\, x>x^+_0+tf'(\underline{u}_r)\big\};
\end{equation}
and if $x_0^-=\inf\mathcal{K}_0>-\infty$,
then $L_{\overline{u}_l}(x_0^-)$ is a divide,
and there exists no divides in $\mathcal{H}_{-\infty}$,
where $\mathcal{H}_{-\infty}$ is defined by
\begin{equation}\label{c6.5}
\mathcal{H}_{-\infty}:
=\big\{(x,t)\in \mathbb{R}\times \mathbb{R}^+\,:\, x<x^-_0+tf'(\overline{u}_l)\big\}.
\end{equation}

Let $\mathcal{K}$ be the union of all divides of the entropy solution $u(x,t)$, $i.e.$,
\begin{equation}\label{c6.6}
\mathcal{K}:=\big\{(x,t)\in \mathbb{R}\times \mathbb{R}^+\,:\,
(x,t)\in L_c(x_0)\,\,\,\, {\rm for\ } x_0\in \mathcal{K}_0\ {\rm and}\ c\in \mathcal{D}(x_0)
\big\}.
\end{equation}
Therefore, if the entropy solution $u(x,t)$ has at least one divide, then
there is a disjoint partition of $\mathbb{R}\times \mathbb{R}^+$ given by
\begin{equation}\label{c6.7}
\mathbb{R}\times \mathbb{R}^+=\mathcal{K}\cup\big(\cup_n \mathcal{H}_{(e_n,h_n)}\big),
\end{equation}
in which $\mathcal{H}_{(e_n,h_n)}=\mathcal{H}_{\infty}$ if $h_n=\infty$,
and $\mathcal{H}_{(e_n,h_n)}=\mathcal{H}_{-\infty}$ if $e_n=-\infty$.

\smallskip
If the entropy solution $u(x,t)$ has no divides,
$\eqref{c6.7}$ can be regarded as
$
\mathbb{R}\times \mathbb{R}^+=\mathcal{H}_{({-}\infty,\infty)}
$
in the sense that $\mathcal{K}=\varnothing$,
and there exists only one $\mathcal{H}_{(e_n,h_n)}$ with $e_n=-\infty$ and $h_n=\infty$.

\smallskip
Before giving Theorem $\ref{pro:c6.2}$
on the global structures of entropy solutions,
we need the following lemma on the relation between shock curves and divides.

\begin{Lem}
\label{lem:c6.2}
Let $x=x(t)$ for $t\geq t_0$ be a shock curve of the entropy solution $u(x,t)$, $i.e.$,
$u(x(t){+},t)<u(x(t){-},t)$ for any $t>t_0$.
For any $t>t_0$, define
\begin{equation}\label{c6.15}
y^\pm(t,\tau):=x(t)+(\tau-t)f'(u(x(t){\pm},t))\qquad {\rm for}\ \tau\in[0,t).
\end{equation}
Then $y^\pm(\infty,\tau):=\lim_{t\rightarrow \infty} y^\pm(t,\tau)$ exist for any $\tau\geq 0$,
and
\begin{enumerate}
\item[{\rm(i)}] If $y^+(\infty,0)<\infty$, line $L_c(y^+(\infty,0))$, defined by
\begin{equation}\label{c6.16a}
L_c(y^+(\infty,0)):\, y^+(\infty,\tau)=y^+(\infty,0)+\tau f'(c)\qquad {\rm for}\ \tau\geq 0,
\end{equation}
is a divide emitting from $y^+(\infty,0)$ with speed $c$ satisfying
\begin{equation}\label{c6.16}
c:=\lim_{t\rightarrow \infty}u(x(t){+},t)={\rm D}_-\bar{\Phi}(y^+(\infty,0)).
\end{equation}
Let $v^+(t)$ be uniquely determined by $y^+(\infty,0)=x(t)-tf'(v^+(t))$ for large $t$, then
\begin{equation}\label{c6.17}
0\leq\int^{y^\pm(t,0)}_{y^+(\infty,0)}(\varphi(\xi)-c)\,{\rm d}\xi
\leq (v^+(t)-c)\big(y^\pm(t,0)-y^+(\infty,0)\big).
\end{equation}

\item[{\rm(ii)}] If $y^-(\infty,0)>-\infty$, line $L_d(y^-(\infty,0))$, defined by
\begin{equation}\label{c6.18a}
L_d(y^-(\infty,0)):\, y^-(\infty,\tau)=y^-(\infty,0)+\tau f'(d)\qquad {\rm for}\ \tau\geq 0,
\end{equation}
is a divide emitting from $y^-(\infty,0)$ with speed $d$ satisfying
\begin{equation}\label{c6.18}
d:=\lim_{t\rightarrow \infty}u(x(t){-},t)={\rm D}_+\bar{\Phi}(y^-(\infty,0)).
\end{equation}
Let $v^-(t)$ be uniquely determined by $y^-(\infty,0)=x(t)-tf'(v^-(t))$ for large $t$, then
\begin{equation}\label{c6.19}
0\leq\int^{y^\pm(t,0)}_{y^-(\infty,0)}(\varphi(\xi)-d)\,{\rm d}\xi
\leq (v^-(t)-d)\big(y^\pm(t,0)-y^-(\infty,0)\big).
\end{equation}
\end{enumerate}
\noindent
In the above,
${\rm D}_{+}\bar{\Phi}(x_0)$ and ${\rm D}_{-}\bar{\Phi}(x_0)$
are given by $\eqref{c5.10}$.
\end{Lem}

\begin{proof}
From $\eqref{c3.5}$, $y^-(\cdot,\tau)$ $(${\it resp.},
$y^+(\cdot,\tau)$$)$ is nonincreasing $(${\it resp.}, nondecreasing$)$ for any fixed $\tau\geq 0$.
Then, by Lemma $\ref{lem:c2.4}$, for any $t_2>t_1>t_0$,
$$
y^-(t_2,\tau)\leq y^-(t_1,\tau)<y^+(t_1,\tau)\leq y^+(t_2,\tau),
$$
so that $\lim_{t\rightarrow \infty}y^\pm(t,\tau)$ exist for any $\tau\geq 0$
with limit $y^+(\infty,\tau)$ $(${\it resp.}, $y^-(\infty,\tau)$$)$
being a finite number or $\infty $ $(${\it resp.}, $-\infty$$)$.

\noindent
{\bf 1}. If $y^+(\infty,0)<\infty$, it follows from $\eqref{c6.15}$ that
$y^+(t,\tau)-y^+(t,0)=\tau f'(u(x(t){+},t))$.
Since $u(x(t){+},t)$ is uniformly bounded,
then $y^+(\infty,\tau)<\infty$ for any $\tau>0$ so that
\begin{equation*}
y^+(\infty,\tau)-y^+(\infty,0)
=\lim_{t\rightarrow \infty}\tau f'(u(x(t){+},t))
=\tau f'(\lim_{t\rightarrow \infty}u(x(t){+},t)),
\end{equation*}
which, by the strictly increasing property of $f'(u)$, implies
\begin{equation}\label{c6.21}
c=\lim_{t\rightarrow \infty}u(x(t){+},t)=(f')^{-1}(\tfrac{y^+(\infty,\tau)-y^+(\infty,0)}{\tau}).
\end{equation}

\smallskip
{\bf (a)}. Since $(y^+(t,\tau),\tau)$ lies on the characteristic: $ l_{u(x(t){+},t)}(x(t),t)$,
then from Lemma $\ref{lem:c2.5}$,
$\mathcal{U}(y^+(t,\tau),\tau)=\big\{u(x(t){+},t)\big\}$ for any $\tau \in (0,t)$
so that, for any $\tau \in (0,t)$,
\begin{equation}\label{c6.21b}
E(u(x(t){+},t);y^+(t,\tau),\tau)-E(u;y^+(t,\tau),\tau)\geq 0
\qquad {\rm for \ any}\ u\in \mathbb{R}.
\end{equation}
Then, for any fixed $\tau>0$, it follows from $\eqref{c6.21}$--$\eqref{c6.21b}$ that,
for any $u \in \mathbb{R}$,
\begin{align*}
&\ E(c;y^+(\infty,\tau),\tau)-E(u;y^+(\infty,\tau),\tau)\\[1mm]
&=-\int^{y^+(\infty,\tau)-\tau f'(c)}_{y^+(\infty,\tau)-\tau f'(u)}\varphi(\xi)\,{\rm d}\xi
-\tau\int^c_us f''(s)\,{\rm d}s\\[1mm]
&=\lim_{t\rightarrow \infty}
\Big({-}\int^{y^+(t,\tau)-\tau f'(u(x(t){+},t))}_{y^+(t,\tau)-\tau f'(u)}\varphi(\xi)\,{\rm d}\xi
-\tau\int^{u(x(t){+},t)}_us f''(s)\,{\rm d}s\Big)\\[1mm]
&=\lim_{t\rightarrow \infty}\big(E(u(x(t){+},t);y^+(t,\tau),\tau)-E(u;y^+(t,\tau),\tau)\big)
\geq 0,
\end{align*}
which implies that $c\in\mathcal{U}(y^+(\infty,\tau),\tau)$ for any $\tau>0$.

\smallskip
From Lemma $\ref{lem:c2.5}$, along line $y^+(\infty,\tau)=y^+(\infty,0)+\tau f'(c)$,
$u(y^+(\infty,\tau){\pm},\tau)\equiv c$ for any $\tau>0$,
so that, $L_c(y^+(\infty,0))$ is a divide.
By Theorem $\ref{the:c5.1}$, $c={\rm D}_-\bar{\Phi}(y^+(\infty,0))$ so that
$\eqref{c6.16}$ follows from $\eqref{c6.21}$.

\smallskip
{\bf (b)}.
Furthermore, from $\eqref{c6.15}$ and $y^\pm(\infty,\tau)=\lim_{t\rightarrow \infty} y^\pm(t,\tau)$,
for sufficiently large $t$,
\begin{align*}
\frac{x(t)-y^+(\infty,0)}{t}
=\frac{y^+(t,0)-y^+(\infty,0)}{t}+f'(u(x(t){+},t))
\in (f'({-}\infty),f'(\infty)),
\end{align*}
which, by the strictly increasing property of $f'(u)$, implies that, for large $t$,
there exists a unique $v^+(t)$ such that $y^+(\infty,0)=x(t)-tf'(v^+(t))$.
Then, from $u^\pm(t):=u(x(t){\pm},t)\in \mathcal{U}(x(t),t)$,
\begin{align}\label{c6.21c}
0&\leq E(u^\pm(t);x(t),t)-E(v^+(t);x(t),t)\nonumber\\[1mm]
&=-\int^{x(t)-tf'(u^\pm(t))}_{x(t)-tf'(v^+(t))}\varphi(\xi)\,{\rm d}\xi
+t\int^{v^+(t)}_{u^\pm(t)}sf''(s)\,{\rm d}s\nonumber\\[1mm]
&=-\int^{y^\pm(t,0)}_{y^+(\infty,0)}(\varphi(\xi)-c)\,{\rm d}\xi
+t\int^{v^+(t)}_{u^\pm(t)}(s-c)f''(s)\,{\rm d}s.
\end{align}
Since $L_c(y^+(\infty,0))$ in $\eqref{c6.16a}$
is a divide emitting from $y^+(\infty,0)$ with speed $c$, from $\eqref{c5.1d}$,
$\int^{y^\pm(t,0)}_{y^+(\infty,0)}(\varphi(\xi)-c)\,{\rm d}\xi\geq 0$ for any $t>0$
and, by the convexity of $f(u)$,
\begin{equation*}
f(v^+(t))\geq f(u^\pm(t))+f'(u^\pm(t))(v^+(t)-u^\pm(t)).
\end{equation*}
Then $\eqref{c6.17}$ is inferred by $\eqref{c6.21c}$ as follows:
\begin{align*}
0 &\leq\int^{y^\pm(t,0)}_{y^+(\infty,0)}(\varphi(\xi)-c)\,{\rm d}\xi
\leq t\int^{v^+(t)}_{u^\pm(t)}(s-c)f''(s)\,{\rm d}s\\[1mm]
&=t\big((v^+(t)-c)f'(v^+(t))-(u^\pm(t)-c)f'(u^\pm(t))+f(u^\pm(t))-f(v^+(t))\big)\\[1mm]
&\leq t\big(v^+(t)(f'(v^+(t))-f'(u^\pm(t)))-c(f'(v^+(t))-f'(u^\pm(t)))\big)\\[1mm]
&=t(v^+(t)-c)\big(f'(v^+(t))-f'(u^\pm(t))\big)
\\[1mm]
&
=(v^+(t)-c)\big(y^\pm(t,0)-y^+(\infty,0)\big).
\end{align*}

\smallskip
\noindent
{\bf 2}. Similar to Step 1 above, it can be checked that,
if $y^-(\infty,0)>-\infty$,
line $L_d(y^-(\infty,0))$ in $\eqref{c6.18a}$ is a divide emitting from
$y^-(\infty,0)$ with $d={\rm D}_+\bar{\Phi}(y^-(\infty,0))$,
and $\eqref{c6.18}$--$\eqref{c6.19}$ hold.
\end{proof}

\begin{The}[Global dynamic patterns of entropy solutions]\label{pro:c6.2}
Let $u=u(x,t)$ be the entropy solution of the Cauchy problem $\eqref{c1.1}$--$\eqref{ID}$
with initial data in $L^\infty$. 
Then the following statements hold{\rm :}
\begin{itemize}
\item [(i)] If $\overline{u}_l<\underline{u}_r$,
or $\,\overline{u}_l=\underline{u}_r$ with $\mathcal{K}_0\neq \varnothing$,
then the entropy solution $u(x,t)$ possesses divides. 
The necessary and sufficient conditions for the locations and speeds of divides are given by {\rm Theorem} $\ref{the:c5.1}$ 
and all the divides of 
$u(x,t)$ are given by {\rm Theorem} $\ref{the:c5.2}$.

Furthermore, $\mathbb{R}\times \mathbb{R}^+$ can be partitioned into
$\mathbb{R}\times \mathbb{R}^+=\mathcal{K}\cup\big(\cup_n \mathcal{H}_{(e_n,h_n)}\big)$ as in $\eqref{c6.7}$.
Then
$u(x,t)$ possesses no shocks in $\mathcal{K}${\rm;}
any two shocks in  each $\mathcal{H}_{(e_n,h_n)}$ $(${\it resp.}, $\mathcal{H}_{\infty}$, or $\mathcal{H}_{-\infty})$ must coincide each other in a finite time, 
in which the shock set in each $\mathcal{H}_{(e_n,h_n)}$ $(${\it resp.}, $\mathcal{H}_{\infty}$, or $\mathcal{H}_{-\infty})$ is path-connected.

Moreover, the entropy solution $u(x,t)$ turns into a single shock asymptotically
in each $\mathcal{H}_{(e_n,h_n)}$ $(${\it resp.}, $\mathcal{H}_{\infty}$, or $\mathcal{H}_{-\infty})$, 
and the fine structures of 
$u(x,t)$ on the shock set 
in each $\mathcal{H}_{(e_n,h_n)}$ $(${\it resp.}, $\mathcal{H}_{\infty}$, or $\mathcal{H}_{-\infty})$
are given by {\rm Theorem} $\ref{the:c4.2}$.

\vspace{2pt}
\item [(ii)] If $\overline{u}_l>\underline{u}_r$,
or $\overline{u}_l=\underline{u}_r$ with $\mathcal{K}_0=\varnothing$,
the entropy solution has no divides and 
any two shocks in $\mathbb{R}\times\mathbb{R}^+$ must coincide each other in a finite time,
in which the shock set in $\mathbb{R}\times\mathbb{R}^+$ is path-connected. 
Moreover, the entropy solution $u(x,t)$ turns into a single shock asymptotically in $\mathbb{R}\times\mathbb{R}^+$, 
and the fine structures of 
$u(x,t)$ on the shock set in $\mathbb{R}\times\mathbb{R}^+$
are given by {\rm Theorem} $\ref{the:c4.2}$.
\end{itemize}
\noindent
In the above, $\overline{u}_l$ and $\underline{u}_r$, and $\mathcal{K}_0$
are given by $\eqref{c5.7}$ and $\eqref{c5.6}${\rm;}
and $\mathcal{H}_{(e_n,h_n)}$ and $\mathcal{H}_{\pm\infty}$
are given by $\eqref{c6.3}$--$\eqref{c6.5}$.
See {\rm Figs.} $\ref{figGSSD}$--$\ref{figGSSDKabab}$ for the details.
\end{The}

\begin{proof} We divide the proof into two steps accordingly.

\smallskip
\noindent
{\bf 1}. Suppose that $\overline{u}_l<\underline{u}_r$,
or $\overline{u}_l=\underline{u}_r$ with $\mathcal{K}_0\neq \varnothing$.
From Theorem $\ref{the:c5.2}$, the entropy solution has at least one divide.
If there exists $\mathcal{H}_{(e_n,h_n)}$,
then there exists no divide in $\mathcal{H}_{(e_n,h_n)}$ so that,
the forward generalized characteristic of each point $(x,t)\in\mathcal{H}_{(e_n,h_n)}$
must become a shock in a finite time.
This means that 
the shock set in $\mathcal{H}_{(e_n,h_n)}$,
denoted by $\Gamma(\mathcal{H}_{(e_n,h_n)})$, is not empty.

\vspace{2pt}
From Lemma $\ref{pro:c3.2}$,
for any two points $(x_i,t_i)\in \Gamma(\mathcal{H}_{(e_n,h_n)})$ for $i=1,2$,
the forward generalized characteristics $X(x_i,t_i)$ for $i=1,2$ are shock curves so that
$X(x_i,t_i)\subset \Gamma(\mathcal{H}_{(e_n,h_n)})$.
Then, by Lemma $\ref{lem:c6.2}$,
$X(x_1,t_1)$ and $X(x_2,t_2)$ both possess the properties that
$y^-(\infty,0)=e_n$ and $y^+(\infty,0)=h_n$.
This means that the backward characteristic triangles,
belonging to the points on the forward generalized characteristics $X(x_1,t_1)$ and $X(x_2,t_2)$ respectively,
must intersect each other at some time $t_3>\max\{t_1,t_2\}$.
Thus, by Proposition $\ref{pro:c3.1}$, these two backward characteristic triangles
must coincide with each other.
By the uniqueness of the forward generalized characteristic in Lemma $\ref{pro:c3.2}$,
$X(x_1,t_1)$ and $X(x_2,t_2)$ coincide with each other for $ t\geq t_3$.
This means that 
any two shocks in $\mathcal{H}_{(e_n,h_n)}$ must coincide with each other in a finite time,
and points $(x_1,t_1)$ and $(x_2,t_2)$ can be connected by the curve
$$
\big\{X(x_1,t_1)\,:\, t_1\leq t\leq t_3\big\}\cup \big\{X(x_2,t_2)\,:\, t_2\leq t\leq t_3\big\}
\subset \Gamma(\mathcal{H}_{(e_n,h_n)}),
$$
which, by the arbitrariness of $(x_i,t_i)$ for $i=1,2$, implies that
$\Gamma(\mathcal{H}_{(e_n,h_n)})$ is path-connected. 
Therefore, 
entropy solution $u(x,t)$ turns into a single shock asymptotically
in $\mathcal{H}_{(e_n,h_n)}$.
For the fine structures of entropy solution $u(x,t)$ 
on the shock set in $\mathcal{H}_{(e_n,h_n)}$, 
see {\rm Theorem} $\ref{the:c4.2}$.

\vspace{2pt}
Similar to the case of $\mathcal{H}_{(e_n,h_n)}$,
the case of $\mathcal{H}_{\infty}$ follows by the facts that
there exists no divide in $\mathcal{H}_{\infty}$
and, from Lemma $\ref{lem:c6.2}$,
$y^-(\infty,0)=x_0^+$ and $y^+(\infty,0)=\infty$
hold for any forward generalized characteristic $X(x_1,t_1)$ with $(x_1,t_1)\in \mathcal{H}_{\infty}$.
The case of $\mathcal{H}_{-\infty}$ follows by the facts that
there exists no divide in $\mathcal{H}_{-\infty}$,
and from Lemma $\ref{lem:c6.2}$,
$y^-(\infty,0)=-\infty$ and $y^+(\infty,0)=x_0^-$
hold for any forward generalized characteristic $X(x_1,t_1)$ with $(x_1,t_1)\in \mathcal{H}_{-\infty}$.

\smallskip
\noindent
{\bf 2}. Suppose that $\overline{u}_l>\underline{u}_r$,
or $\overline{u}_l=\underline{u}_r$ with $\mathcal{K}_0=\varnothing$.
By Theorem $\ref{the:c5.2}$, the entropy solution has no divide.
From Lemma $\ref{lem:c6.2}$,
$y^-(\infty,0)=-\infty$ and $y^+(\infty,0)=\infty$
hold for any forward generalized characteristic $X(x_1,t_1)$ with $(x_1,t_1)\in \mathbb{R}\times\mathbb{R}^+$.
Then, for any two points $(x_i,t_i)\in \mathbb{R}\times\mathbb{R}^+$ for $i=1,2$,
the backward characteristic triangles,
belonging to the points on the forward generalized characteristics $X(x_i,t_i)$ for $i=1,2$,
intersect each other at some time $t_3>\max\{t_1,t_2\}$,
which, by Proposition $\ref{pro:c3.1}$, implies that
these two backward characteristic triangles coincide with each other.

Let $\Gamma(\mathbb{R}\times\mathbb{R}^+)$ be the shock set in the whole $\mathbb{R}\times\mathbb{R}^+$.
By the uniqueness of the forward generalized characteristic in Lemma $\ref{pro:c3.2}$,
$X(x_1,t_1)$ and $X(x_2,t_2)$ coincide with each other for $ t\geq t_3$,
which means that 
any two shocks in $\mathbb{R}\times\mathbb{R}^+$ must coincide with each other in a finite time,
and points $(x_1,t_1)$ and $(x_2,t_2)$ can be connected by the curve
$$\big\{X(x_1,t_1)\,:\, t_1\leq t\leq t_3\big\}\cup \big\{X(x_2,t_2)\,:\, t_2\leq t\leq t_3\big\}
\subset \Gamma(\mathbb{R}\times\mathbb{R}^+),$$
which, by the arbitrariness of $(x_i,t_i)$ for $i=1,2$, implies that
$\Gamma(\mathbb{R}\times\mathbb{R}^+)$ is path-connected.
Therefore, 
entropy solution $u(x,t)$ turns into a single shock asymptotically
in $\mathbb{R}\times\mathbb{R}^+$.
For the fine structures of entropy solution $u(x,t)$ 
on the shock set in $\mathbb{R}\times\mathbb{R}^+$, 
see {\rm Theorem} $\ref{the:c4.2}$.
\end{proof}

\begin{Rem}
From \S $3$--\S $6$,
{\it we can roughly solve the Cauchy problem $\eqref{c1.1}$--$\eqref{ID}$
with initial data $\varphi(x)\in L^{\infty}(\mathbb{R})$
nearly at a level of solving the Riemann problem}
in the following sense{\rm :}
Combining {\rm Fig.} $\ref{figPhiFab}$ and {\rm Figs.} $\ref{figGSSD}$--$\ref{figGSSDKabab}$ with
{\rm Figs.} $\ref{figPhiFabab}$--$\ref{figCCP}$,
we can obtain a pretty clear image of the entropy solution
only by using the information of the graph of $\Phi(x)=\int_0^x\varphi(\xi)\,{\rm d}\xi$,
including the initial waves for the Cauchy problem,
the path-connected branches of the shock sets,
the divides, and the rough profile, {\it etc.}.
This shows the power of the new solution formula $\eqref{c2.50}$ with $\eqref{c2.1}$ in some sense.
\end{Rem}

\section{Asymptotic Behaviors of Entropy Solutions with Initial Data in $L^\infty$}
In this section,
we focus on the asymptotic behaviors of entropy solutions with initial data in $L^\infty$
respectively in the $L^\infty$--norm and the $L^p_{{\rm loc}}$--norm.
In \S 7.1, by using the results on invariants, divides, 
and global structures of entropy solutions in \S 6,
the asymptotic behaviors of entropy solutions in the $L^\infty$--norm are established
by proving that the asymptotic profile is either a rarefaction-constant solution
(as a function only consisting of some constant regions,
centered rarefaction waves, and/or rarefaction regions)
or a single shock in Theorem $\ref{pro:c6.3}$,
and obtaining the corresponding decay rates 
in Theorems $\ref{the:c6.2}$--$\ref{the:c6.3}$ with Corollary $\ref{cor:c6.1}$;
in \S 7.2, by introducing the generalized $N$--waves,
the asymptotic behaviors of entropy solutions in the $L^p_{{\rm loc}}$--norm are established
by proving that the entropy solutions decay to the generalized $N$--waves
and obtaining the corresponding decay rates in Theorem $\ref{the:n.1}$ with Lemma $\ref{lem:n.1}$.

\vspace{1pt}
By Lemma $\ref{lem:c5.3}$, if $\bar{\Phi}(x)>-\infty$, then $\bar{\Phi}(x)$ is a convex function on $\mathbb{R}$.
Let $\tilde{u}=\tilde{u}(x,t)$ is the entropy solution of the following Cauchy problem
\begin{equation}\label{c6.10}
\begin{cases}
\tilde{u}_t+f(\tilde{u})_x=0\qquad {\rm for\ } (x,t)\in \mathbb{R}\times\mathbb{R}^+,\\[1mm]
\tilde{u}(x,0)=\bar{\Phi}'(x).
\end{cases}
\end{equation}
Denote the convex hull of $\bar{\Phi}(x)$ over $({-}\infty,\infty)$ by $\bar{\bar{\Phi}}(x)$. 
Since $\bar{\Phi}(x)$ is convex, then 
$\bar{\Phi}(x)\equiv\bar{\bar{\Phi}}(x)$,
which, by Theorem $\ref{the:c4.0}$ and $\eqref{c5.18}$, implies that, for any $x_0\in \mathbb{R}$,
\begin{equation*}
\mathcal{D}(x_0)=[{\rm D}_-\bar{\bar{\Phi}}(x_0),{\rm D}_+\bar{\bar{\Phi}}(x_0)]
=[{\rm D}_-\bar{\Phi}(x_0),{\rm D}_+\bar{\Phi}(x_0)]=\mathcal{C}(x_0).
\end{equation*}
Then every characteristic of solution $\tilde{u}(x,t)$ of the Cauchy problem $\eqref{c6.10}$
is a divide
so that $\tilde{u}(x,t)$ is a rarefaction-constant solution.
Since $\bar{\Phi}'(x)$ is nondecreasing,
then $\bar{\Phi}'(x)$ has at most countably discontinuous points,
at which a rarefaction wave is generated, and there exists a unique
divide emitting from the point at which $\bar{\Phi}'(x)$ is continuous.
Therefore, solution $\tilde{u}(x,t)$ possesses at most countable rarefaction waves.

Furthermore, if $\mathcal{K}_0=\varnothing$, 
it follows from Lemma $\ref{lem:c5.2}$(ii) that $\tilde{u}(x,t)\equiv \underline{u}_r$;
and if $\mathcal{K}_0\neq \varnothing$, $\tilde{u}(x,t)$ can be determined as follows:
For $\mathbb{R}\times \mathbb{R}^+=\mathcal{K}\cup\big(\cup_n \mathcal{H}_{(e_n,h_n)}\big)$
as in $\eqref{c6.7}$, 
it follows from $\eqref{c5.19}$--$\eqref{c5.20}$ that, for $(x,t)\in \mathcal{K}$,
\begin{equation}\label{kk00}
\tilde{u}(x,t)=\big(f'\big)^{-1}(\frac{x-x_0}{t})
\qquad {\rm with}\ x_0\in \mathcal{K}_0;
\end{equation}
and it follows from Lemma $\ref{lem:c5.3}$(ii) that,
\begin{align}\label{tidleuH}
\tilde{u}(x,t)=
\begin{cases}
c_n\quad &{\rm if}\ (x,t)\in \mathcal{H}_{(e_n,h_n)},\\[1mm]
\underline{u}_r\quad &{\rm if}\ (x,t)\in \mathcal{H}_{\infty},\\[1mm]
\overline{u}_l\quad &{\rm if}\ (x,t)\in \mathcal{H}_{{-}\infty}.
\end{cases}
\end{align}
In the above, $\overline{u}_l$ and $\underline{u}_r$ are given by $\eqref{c5.7}${\rm ;}
$\mathcal{K}_0$ and $\mathcal{K}$ are given by $\eqref{c5.6}$ and $\eqref{c6.6}$, respectively;
and $\mathcal{H}_{(e_n,h_n)}$ with $c_n$ and $\mathcal{H}_{\pm\infty}$ are given by $\eqref{c6.2}$--$\eqref{c6.5}$.

\subsection{Decay of entropy solutions in the $L^\infty$--norm} 
We first have the following lemma on the asymptotic behaviors of the entropy solutions
which possess at least one divide.

\begin{Lem}
\label{the:c6.1}
Suppose that the initial data function $\varphi(x)\in L^{\infty}(\mathbb{R})$ satisfying that
$\overline{u}_l<\underline{u}_r$,
or $\overline{u}_l=\underline{u}_r$ with $\mathcal{K}_0\neq \varnothing$.
Let $u=u(x,t)$ and $\tilde{u}=\tilde{u}(x,t)$ be the entropy solutions
of the Cauchy problem $\eqref{c1.1}$--$\eqref{ID}$ 
and the Cauchy problem $\eqref{c6.10}$, respectively. 
Then
\begin{equation}\label{c6.25}
u(x,t)=\tilde{u}(x,t) \qquad {\rm for \ any}\ (x,t)\in \mathcal{K},
\end{equation}
and the entropy solution $u=u(x,t)$ has at most countable rarefaction waves on $\mathcal{K}$.

Furthermore, the rarefaction-constant solution $\tilde{u}(x,t)$ is the asymptotic profile
of the entropy solution $u=u(x,t)$.
More specifically, the following statements hold
$($see also {\rm Fig.} $\ref{figABSD}$$)${\rm :}

\begin{itemize}
 \item [(i)] If there exists $\mathcal{H}_{(e_n,h_n)}=:\mathcal{H}_{(e,h)}$ in $\eqref{c6.3}$
with $c_n=:c$ given by $\eqref{c6.2}$,
define $e(t):=e+tf'(c)$ and $h(t):=h+tf'(c)$.
Then $\tilde{u}(x,t)\equiv c$ on $\mathcal{H}_{(e,h)}$,
\begin{equation}\label{c6.26}
\int^{h(t)}_{e(t)}u(x,t)\,{\rm d}x=\int^h_e\varphi(x)\,{\rm d}x=c(h-e)
\qquad {\rm for\ any}\ t>0;
\end{equation}
and, for any forward generalized characteristic $X(x_0,t_0):\, x=x(t)$ for $t\geq t_0$
with $(x_0,t_0)\in \mathcal{H}_{(e,h)}$, as $t\rightarrow \infty$,
\begin{align}
&\qquad\quad f'(u(x,t))-f'(c)=
\begin{cases}
\displaystyle\frac{x-e(t)}{t}+o(1)\, t^{-1}\quad &{\rm uniformly \ for}\ e(t)<x<x(t),\\[4mm]
\displaystyle\frac{x-h(t)}{t}+o(1)\, t^{-1}\quad &{\rm uniformly \ for}\ x(t)<x<h(t),
\end{cases}\label{c6.27}\\[2mm]
&\qquad\quad \lim_{t\rightarrow\infty}u(x(t){\pm},t)=c,
\qquad \lim_{t\rightarrow\infty}\dfrac{x(t)}{t}=f'(c).\label{c6.28}
\end{align}
Moreover, if $f''(u)=(N+o(1))|u-c|^\alpha$ as $u\rightarrow c$ for $\alpha\geq 0$ and $N>0$,
then, as $t\rightarrow \infty$,
\begin{equation}\label{c6.29}
|u(x,t)-c|\leq(C+o(1))\Big(\frac{h-e}{2}\Big)^{\frac{1}{1+\alpha}}\, t^{-\frac{1}{1+\alpha}}
\end{equation}
holds uniformly in $x$ with $(x,t)\in \mathcal{H}_{(e,h)}$, and
\begin{equation}\label{c6.30}
\lim_{t\rightarrow\infty}\theta(t)
=\frac{1}{2},\qquad
\lim_{t\rightarrow\infty}\big(x(t)-tf'(c)\big)=\frac{e+h}{2},
\end{equation}
where $C=\big(\frac{1+\alpha}{N}\big)^{\frac{1}{1+\alpha}}$ and
$\theta(t)$ is given by
$$\theta(t):=\frac{f'(u(x(t){-},t))-f'(c)}{f'(u(x(t){-},t))-f'(u(x(t){+},t))}.$$

\smallskip
\item [(ii)] If $x_0^+= \sup\mathcal{K}_0<\infty$,
$\tilde{u}(x,t)\equiv \underline{u}_r$ on $\mathcal{H}_{\infty}$ as in $\eqref{c6.4}$.
Define $x_0^+(t):=x_0^++tf'(\underline{u}_r)$.
Then, for any forward generalized characteristic $X(x_0,t_0):\, x=x(t)$ for $t\geq t_0$
with $(x_0,t_0)\in \mathcal{H}_{\infty}$,
as $t\rightarrow \infty$,
\begin{align}
&\qquad\quad f'(u(x,t))-f'(\underline{u}_r)=\frac{x-x_0^+(t)}{t}+o(1)\, t^{-1}\quad\,\,
{\rm uniformly \ for}\ x_0^+(t)<x<x(t),\label{c6.31}\\[2mm]
&\qquad\quad \lim_{t\rightarrow\infty}u(x(t){\pm},t)=\underline{u}_r,
\qquad \lim_{t\rightarrow\infty}\dfrac{x(t)}{t}=f'(\underline{u}_r).\label{c6.32}
\end{align}
Furthermore, define $l^+(t):=y^+(t,0)-x_0^+$ with $y^+(t,0)$ given by $\eqref{c6.15}$. Then
\begin{equation}\label{c6.33}
\lim_{t\rightarrow \infty}
\frac{1}{l^+(t)}\int^{x_0^++l^+(t)}_{x_0^+}\varphi (\xi)\,{\rm d}\xi=\underline{u}_r.
\end{equation}

\item [(iii)] If $x_0^-= \inf\mathcal{K}_0>-\infty$,
$\tilde{u}(x,t)\equiv \overline{u}_l$ on $\mathcal{H}_{-\infty}$ as in $\eqref{c6.5}$.
Define $x_0^-(t):=x_0^-+tf'(\overline{u}_l)$.
Then, for any forward generalized characteristic $X(x_0,t_0):\, x=x(t)$ for $t\geq t_0$
with $(x_0,t_0)\in \mathcal{H}_{-\infty}$, as $t\rightarrow \infty$,
\begin{align}
&\qquad\quad f'(u(x,t))-f'(\overline{u}_l)=\frac{x-x_0^-(t)}{t}+o(1)\, t^{-1}\quad\,\,
{\rm uniformly \ for}\ x(t)<x<x_0^-(t),\label{c6.34}\\[2mm]
&\qquad\quad\lim_{t\rightarrow\infty}u(x(t){\pm},t)=\overline{u}_l,\qquad
\lim_{t\rightarrow\infty}\dfrac{x(t)}{t}=f'(\overline{u}_l).\label{c6.35}
\end{align}
Furthermore, define $l^-(t):=y^-(t,0)-x_0^-$ with $y^-(t,0)$ given by $\eqref{c6.15}$. Then
\begin{equation}\label{c6.36}
\lim_{t\rightarrow \infty}
\frac{1}{l^-(t)}\int^{x_0^-+l^-(t)}_{x_0^-}\varphi (\xi)\,{\rm d}\xi=\overline{u}_l.
\end{equation}
\end{itemize}
\noindent
In the above,
$\overline{u}_l$ and $\underline{u}_r$, $\mathcal{K}_0$, and $\mathcal{K}$
are given by $\eqref{c5.7}$, $\eqref{c5.6}$, and $\eqref{c6.6}$, respectively.
\end{Lem}

\begin{figure}[H]
	\begin{center}
		{\includegraphics[width=0.6\columnwidth]{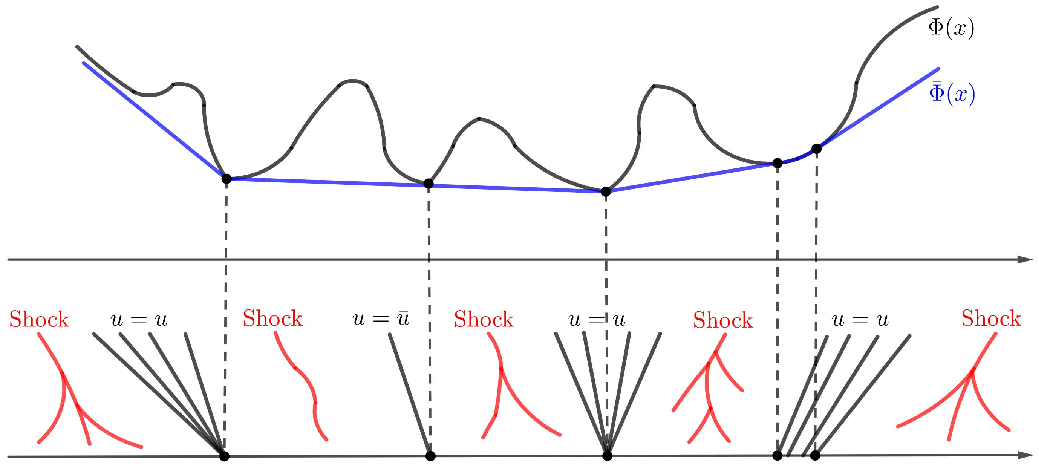}}
 \caption{Asymptotic behaviors of entropy solutions possessing divides.}\label{figABSD}
	\end{center}
\end{figure}

\begin{proof}
From Lemma $\ref{lem:c5.2}$, $\mathcal{K}_0\neq \varnothing$ is equivalent to
that $\overline{u}_l<\underline{u}_r$,
or $\overline{u}_l=\underline{u}_r$ with $\mathcal{K}_0\neq \varnothing$.
Since $\bar{\Phi}(x)$ is convex on $({-}\infty,\infty)$,
then $\bar{\Phi}(x)\equiv\bar{\bar{\Phi}}(x)$ so that,
for any $x_0\in \mathcal{K}_0$,
$$
\Phi(x_0)=\bar{\Phi}(x_0)=\bar{\bar{\Phi}}(x_0),\qquad
[{\rm D}_-\bar{\Phi}(x_0),{\rm D}_+\bar{\Phi}(x_0)]
=[{\rm D}_-\bar{\bar{\Phi}}(x_0),{\rm D}_+\bar{\bar{\Phi}}(x_0)],
$$
which means that each divide of the entropy solution $u(x,t)$ is a divide
of the rarefaction-constant solution $\tilde{u}(x,t)$.
From the definition of $\mathcal{K}$ in $\eqref{c6.6}$, $\eqref{c6.25}$ holds so that
the entropy solution $u=u(x,t)$ has at most countable rarefaction waves on $\mathcal{K}$.

\smallskip
\noindent
{\bf 1}. Since $e,h \in \mathcal{K}_0$ and $c= {\rm D}_+\bar{\Phi}(e)$$={\rm D}_-\bar{\Phi}(h)$,
it follows from $\eqref{c5.13}$ and $\eqref{c6.10}$ that $\tilde{u}(x,t)\equiv c$ on $\mathcal{H}_{(e,h)}$.
From Theorem $\ref{the:c5.1}$,
$e(t)=e+tf'(c)$ and $h(t)=h+tf'(c)$ for $t>0$
are both divides, which, by Proposition $\ref{pro:c5.1}$, implies
that
$\int^h_e(\varphi(x)-c)\,{\rm d}x\geq 0$ and $\int^e_h(\varphi(x)-c)\,{\rm d}x\geq 0$
so that $\int^h_e(\varphi(x)-c)\,{\rm d}x=0$.
Thus, $\eqref{c6.26}$ follows from $\eqref{c6.12}$ immediately.

Since there exists no divide in $\mathcal{H}_{(e,h)}$,
any forward generalized characteristic $X(x_0,t_0):\, x=x(t)$ for $t\geq t_0$ with $(x_0,t_0)\in \mathcal{H}_{(e,h)}$
must become a shock curve when $t>t_1$ for some $t_1\geq t_0$ large enough.
Then, by Lemma $\ref{lem:c6.2}$,
\begin{equation}\label{eh0}
y^-(\infty,0)=e,\qquad y^+(\infty,0)=h.
\end{equation}
From Lemma $\ref{lem:c2.4}$, for any $(x,t)\in \mathcal{H}_{(e,h)}$ with $e(t)<x<x(t)$,
$x-tf'(u(x,t))$ lies between $e=e(t)-tf'(c)$ and $y^-(t,0)$
so that, as $t\rightarrow \infty$,
\begin{equation*}
0\leq x-e(t)-t\big(f'(u(x,t))-f'(c)\big)\leq y^-(t,0)-e\ \longrightarrow 0,
\end{equation*}
which implies that $\eqref{c6.27}$ holds for the case that
$(x,t)\in \mathcal{H}_{(e,h)}$ with $e(t)<x<x(t)$.
Similarly, $\eqref{c6.27}$ holds for the case that
$(x,t)\in \mathcal{H}_{(e,h)}$ with $x(t)<x<h(t)$.

Since $f'(u)$ is strictly increasing,
it follows from $\eqref{c6.27}$ that $\lim_{t\rightarrow\infty}u(x(t){\pm},t)=c$.
From $x(t)\in (e(t), h(t))$ for any $t\geq t_0$,
$$
f'(c)=\lim_{t \rightarrow \infty}\frac{e(t)}{t}
\leq \lim_{t \rightarrow \infty}\frac{x(t)}{t}
\leq \lim_{t \rightarrow \infty}\frac{h(t)}{t}=f'(c),
$$
which yields $\eqref{c6.28}$.

If, in addition, $f''(u)=(N+o(1))|u-c|^\alpha$ as $u\rightarrow c$ for some $\alpha\geq 0$ and $N>0$,
from $\eqref{c6.27}$ and $\eqref{aa2}$,
we first have
\begin{align}\label{c6.29op}
|u(x,t)-c|
&=\Big(\frac{1+\alpha}{N+o(1)}\Big)^{\frac{1}{1+\alpha}} \,
\big|f'(u(x,t))-f'(c)\big|^{\frac{1}{1+\alpha}}
\leq (C+o(1))\,(h-e)^{\frac{1}{1+\alpha}}
\, t^{-\frac{1}{1+\alpha}}.
\end{align}

\smallskip
We now give the proof of $\eqref{c6.30}$, which is divided into three steps.

\smallskip
{\bf (a).} Denote $u^\pm(t):=u(x(t){\pm},t)$.
From $\eqref{c6.15}$,
$y^+(t,0)-y^-(t,0)=t\big(f'(u^-(t))-f'(u^+(t))\big)$
so that, by the definition of $\theta(t)$,
$
\theta(t)\big(y^+(t,0)-y^-(t,0)\big)=t\big(f'(u^-(t))-f'(c)\big).
$
Then, by Lemma $\ref{lem:c6.2}$, from $\eqref{eh0}$,
$y^-(t,0)=e+o(1)$ and $y^+(t,0)=h+o(1)$ as $t\rightarrow \infty$, and hence
\begin{align}\label{c6.37}
x(t)-tf'(c)&=y^-(t,0)+t\big(f'(u^-(t))-f'(c)\big)\nonumber\\[1mm]
&=y^-(t,0)+\theta(t)\big(y^+(t,0)-y^-(t,0)\big)\nonumber\\[1mm]
&=\theta(t)\, y^+(t,0)+(1-\theta(t))\, y^-(t,0)
=e+\theta(t)(h-e)+o(1),
\end{align}
which, by $x(t)-tf'(c)\in (e,h)$, implies
\begin{equation}\label{c6.38}
0\leq\mathop{\underline{\lim}}\limits_{t\rightarrow \infty}\theta(t)\leq \mathop{\overline{\lim}}\limits_{t\rightarrow \infty}\theta(t)\leq 1.
\end{equation}

{\bf (b).} From $\eqref{c6.37}$--$\eqref{c6.38}$,
to prove $\eqref{c6.30}$, it suffices to show that
$\lim_{t\rightarrow \infty}\theta(t)= \frac{1}{2}.$

In fact, since $u^\pm(t)\in \mathcal{U}(x(t),t)$, then
$E(u^-(t);x(t),t)-E(u^+(t);x(t),t)=0$.
Thus, by the definition of $E(u;x,t)$ in $\eqref{c2.1}$,
$$
-\int^{x(t)-tf'(u^-(t))}_{x(t)-tf'(u^+(t))}(\varphi(\xi)-c)\,{\rm d}\xi
-t\int^{u^-(t)}_{u^+(t)}(s-c)f''(s)\,{\rm d}s=0,
$$
which, by $\eqref{c6.15}$, implies
\begin{equation}\label{c6.39}
\int^{y^+(t,0)}_{y^-(t,0)}(\varphi(\xi)-c)\,{\rm d}\xi
=t\int^{u^-(t)}_{u^+(t)}(s-c)f''(s)\,{\rm d}s.
\end{equation}

We claim: As $t\rightarrow \infty$,
\begin{align}
&t\int^{u^-(t)}_{u^+(t)}(s-c)f''(s)\,{\rm d}s
=O(1)\,\Big(|\theta(t)|^{\frac{2+\alpha}{1+\alpha}}-|1-\theta(t)|^{\frac{2+\alpha}{1+\alpha}}\Big)\,
t^{-\frac{1}{1+\alpha}},\label{c6.40}\\[1mm]
&\int^{y^+(t,0)}_{y^-(t,0)}(\varphi(\xi)-c)\,{\rm d}\xi= o(1)\, t^{-\frac{1}{1+\alpha}}.\label{c6.41}
\end{align}
Then, if $\eqref{c6.40}$--$\eqref{c6.41}$ hold,
it follows from $\eqref{c6.38}$--$\eqref{c6.39}$ that
$\lim_{t\rightarrow \infty}\theta(t)= \frac{1}{2}$.

\smallskip
{\bf (c).} We now prove $\eqref{c6.40}$--$\eqref{c6.41}$.
Since $f''(u)=(N+o(1))|u-c|^\alpha$ as $u\rightarrow c$ for some $\alpha\geq 0$ and $N>0$,
from $\eqref{c6.15}$ and $\eqref{eh0}$,
$$
f'(u^-(t))-f'(u^+(t))=\frac{y^+(t,0)-y^-(t,0)}{t}=\frac{h-e+o(1)}{t},
$$
which, by $\eqref{aa2}$--$\eqref{aa3}$, implies
\begin{align*}
t\int^{u^-(t)}_{u^+(t)}(s-c)f''(s)\,{\rm d}s
&=\frac{N+o(1)}{2+\alpha}\Big(|u^-(t)-c|^{2+\alpha}-|u^+(t)-c|^{2+\alpha}\Big)\, t\\
&=O(1)\, \Big(|f'(u^-(t))-f'(c)|^{\frac{2+\alpha}{1+\alpha}}-|f'(u^+(t))-f'(c)|^{\frac{2+\alpha}{1+\alpha}}\Big)\, t\\[1mm]
&=O(1)\, \Big(\frac{h-e+o(1)}{t}\Big)^{\frac{2+\alpha}{1+\alpha}}\,
\Big(|\theta(t)|^{\frac{2+\alpha}{1+\alpha}}-|1-\theta(t)|^{\frac{2+\alpha}{1+\alpha}}\Big)\, t\\[1mm]
&=O(1)\, \Big(|\theta(t)|^{\frac{2+\alpha}{1+\alpha}}-|1-\theta(t)|^{\frac{2+\alpha}{1+\alpha}}\Big)\,
t^{-\frac{1}{1+\alpha}}.
\end{align*}
This means that $\eqref{c6.40}$ holds.

From $\eqref{c6.17}$ and $\eqref{c6.19}$,
for sufficiently large $t$, there exists a unique $v^{\pm}(t)\in\mathbb{R}$
such that
$e=x(t)-tf'(v^-(t))$ and $h=x(t)-tf'(v^+(t)).$
Then, from $\eqref{c6.15}$,
$$
f'(v^-(t))-f'(u^-(t))=\frac{y^-(t,0)-e}{t},
\qquad f'(v^+(t))-f'(u^+(t))=\frac{y^+(t,0)-h}{t},
$$
which, by $\eqref{eh0}$, implies that
$f'(v^\pm(t))-f'(u^\pm(t))=o(1)\, t^{-1}.$

Since $L_c(e)$ is a divide,
then, by $\eqref{c6.29op}$ and $\eqref{aa2}$,
it follows from $\eqref{c6.19}$ that
\begin{align*}
 0\leq \int^{y^-(t,0)}_e(\varphi(\xi)-c)\,{\rm d}\xi
&\leq o(1)\, \big(|v^-(t)-u^-(t)|+|u^-(t)-c|\big)\nonumber\\
&=o(1)\, \big(O(1)|f'(v^-(t))-f'(u^-(t))|^{\frac{1}{1+\alpha}}+|u^-(t)-c|\big)\nonumber\\[1mm]
&=o(1)\, \big(O(1)o(1)\,t^{-\frac{1}{1+\alpha}}+O(1)\,t^{-\frac{1}{1+\alpha}}\big)
= o(1)\, t^{-\frac{1}{1+\alpha}}.
\end{align*}
Similarly, we have
\begin{equation*}
0\leq \int^{y^+(t,0)}_h(\varphi(\xi)-c)\,{\rm d}\xi= o(1)\, t^{-\frac{1}{1+\alpha}}.
\end{equation*}
Therefore, it follows from $\eqref{c6.26}$ that
$$
\int^{y^+(t,0)}_{y^-(t,0)}(\varphi(\xi)-c)\,{\rm d}\xi
=\bigg(\int^{y^+(t,0)}_h+\int^h_e+\int^e_{y^-(t,0)}\bigg)(\varphi(\xi)-c)\,{\rm d}\xi
= o(1)\, t^{-\frac{1}{1+\alpha}},
$$
which yields $\eqref{c6.41}$, and hence $\eqref{c6.30}$ holds.

Finally, similar to $\eqref{c6.29op}$, it follows from
$\eqref{c6.27}$ and $\eqref{c6.30}$ that
\begin{align*}
|u(x,t)-c|
& \leq (C+o(1))\, \max\Big\{\Big|\frac{x(t)-e(t)}{t}\Big|^{\frac{1}{1+\alpha}},
\,\Big|\frac{x(t)-h(t)}{t}\Big|^{\frac{1}{1+\alpha}}\Big\}
\nonumber\\[2mm]
&
=(C+o(1))\,\Big(\frac{h-e}{2}\Big)^{\frac{1}{1+\alpha}}
\, t^{-\frac{1}{1+\alpha}}.
\end{align*}

\smallskip
\noindent
{\bf 2}. By the definition of $\mathcal{H}_{\infty}$ in $\eqref{c6.4}$,
$L_{\underline{u}_r}(x_0^+):\, x=x_0^++tf'(\underline{u}_r)$ for $t>0$ is a divide.
From $\eqref{c5.11}$ and $\eqref{c6.10}$,
$\tilde{u}(x,t)\equiv \underline{u}_r$ on $\mathcal{H}_{\infty}$.
Since there exists no divide in $\mathcal{H}_{\infty}$,
any forward generalized characteristic $X(x_0,t_0): x=x(t)$ for $t\geq t_0$ with $(x_0,t_0)\in \mathcal{H}_{\infty}$
must become a shock curve when $t>t_1$ for some $t_1$ large enough, $i.e.$,
$u(x(t){-},t)>u(x(t){+},t)$ for any $t>t_1$.

\smallskip
Denote $u^\pm(t):=u(x(t){\pm},t)$.
By Lemma $\ref{lem:c6.2}$,
$y^-(\infty,0)=x_0^+$ and $y^+(\infty,0)=\infty,$
so that
\begin{equation}\label{c6.45}
\mathop{\overline{\lim}}\limits_{t\rightarrow \infty}u^+(t)
\leq \lim_{t\rightarrow \infty}u^-(t)=\underline{u}_r.
\end{equation}

Similar to the proof of $\eqref{c6.27}$,
it is direct to check
$\eqref{c6.31}$.
To prove $\eqref{c6.32}$, it suffices to show:
$\mathop{\underline{\lim}}_{t\rightarrow \infty}u^+(t)\geq \underline{u}_r.$
In fact, since $u^\pm(t)\in \mathcal{U}(x(t),t)$,
by the definition of $E(u;x,t)$ in $\eqref{c2.1}$,
$$
-\int^{x(t)-tf'(u^-(t))}_{x(t)-tf'(u^+(t))}\varphi(\xi)\,{\rm d}\xi
-t\int^{u^-(t)}_{u^+(t)}sf''(s)\,{\rm d}s=0,
$$
which, by $\eqref{c6.15}$, implies
$$
\int^{y^+(t,0)}_{y^-(t,0)}(\varphi(\xi)-\underline{u}_r)\,{\rm d}\xi
=t\int^{u^-(t)}_{u^+(t)}(s-\underline{u}_r)f''(s)\,{\rm d}s.
$$
Then, by the definition of $\rho(u,v)$ in $\eqref{aa5a}$,
\begin{equation}\label{c6.46}
\frac{1}{y^+(t,0)-y^-(t,0)}\int^{y^+(t,0)}_{y^-(t,0)}(\varphi(\xi)-\underline{u}_r)\,{\rm d}\xi
=\rho(u^-(t),u^+(t))-\underline{u}_r,
\end{equation}
which, by the definition of $\underline{u}_r$ in $\eqref{c5.7}$, implies that
$\mathop{\underline{\lim}}_{t\rightarrow \infty} \rho(u^-(t),u^+(t))=\underline{u}_r$.
Since solution $u\in L^\infty$, there exists a sequence $\{t_n\}$
such that
$\lim_{n\rightarrow \infty}u^+(t_n)
=\mathop{\underline{\lim}}_{t \rightarrow \infty} u^+(t)$.
Thus, by $\eqref{c6.45}$ and $\rho(u,v)\in C(\mathbb{R}\times\mathbb{R})$, we obtain
\begin{equation*}
\rho(\underline{u}_r,\mathop{\underline{\lim}}\limits_{t\rightarrow \infty} u^+(t))
=\lim_{n\rightarrow \infty}\rho(u^-(t_n),u^+(t_n))
\geq \mathop{\underline{\lim}}\limits_{t\rightarrow \infty} \rho(u^-(t),u^+(t))
=\underline{u}_r.
\end{equation*}
Since $\rho(\underline{u}_r,v)$ is strictly increasing in $v\in \mathbb{R}$ with $\rho(\underline{u}_r,\underline{u}_r)=\underline{u}_r$,
then
$\mathop{\underline{\lim}}_{t\rightarrow \infty} u^+(t)\geq\underline{u}_r$,
which, by $\eqref{c6.45}$, implies that
$\lim_{t \rightarrow \infty} u^\pm(t)=\underline{u}_r$.

\smallskip
From $y^-(\infty,0)=x_0^+$
and $\lim_{t \rightarrow \infty} u^\pm(t)=\underline{u}_r$,
it is direct to check from $x(t)=y^-(t,0)+tf'(u^-(t))$ that $\eqref{c6.32}$ holds.
Since $\lim_{t \rightarrow \infty} u^\pm(t)=\underline{u}_r$
and $\rho(u,v)\in C(\mathbb{R}\times\mathbb{R})$,
$\eqref{c6.33}$ follows by letting $t\rightarrow \infty$ in $\eqref{c6.46}$.

\smallskip
\noindent
{\bf 3}. By the same arguments as in Step 2,
it can be checked that $\eqref{c6.34}$--$\eqref{c6.35}$ and $\eqref{c6.36}$ hold.
\end{proof}

\begin{Rem}
From {\rm Lemma} $\ref{the:c6.1}$, we have
\begin{itemize}
\item [(i)] For $\mathcal{H}_{(e_n,h_n)}$, by $\eqref{c6.30}$,
$u(x(t){-},t)>c>u(x(t){+},t)$ with $\lim_{t\rightarrow \infty}u(x(t){\pm},t)=c$.
It is interesting to see that $u(x(t){-},t)$ and $u(x(t){+},t)$ decay to $c$ with the same rate
in the sense that $\lim_{t\rightarrow\infty}\theta(t)=\frac{1}{2}$,
and any forward generalized characteristic $x(t)$ in $\mathcal{H}_{(e_n,h_n)}$ takes
the straight line $x=\frac{e_n+h_n}{2}+tf'(c)$ as its asymptotic line,
where $c={\rm D}_+\bar{\Phi}(e_n)={\rm D}_-\bar{\Phi}(h_n)$.

\item[(ii)] For $\mathcal{H}_{\infty}$, from $\eqref{c6.32}$,
$\lim_{t\rightarrow \infty}u(x(t){\pm},t)=\underline{u}_r$.
It is interesting to see that the limit of $\frac{\Phi(x)}{x}$ as $x\rightarrow\infty$
may not exist generically,
but from $\eqref{c6.33}$, for any forward generalized characteristic $x(t)$ in $\mathcal{H}_{\infty}$,
the limit of $\frac{\Phi(y^+(t,0))}{y^+(t,0)}$ always exists as $t\rightarrow\infty$.
It is similar for $\mathcal{H}_{-\infty}$.
\end{itemize}
\end{Rem}

\smallskip
In general, for the initial data function in $L^\infty$,
the entropy solution may not possess some uniform decay in $x$ as $t\rightarrow\infty$.
In order to obtain the uniform decay in $x$ of the entropy solution to its asymptotic profile,
we need the asymptotic behavior of the initial data functions
as $x\rightarrow \pm\infty$.
From $\eqref{c5.7}$, it can be checked that,
as $x\rightarrow \infty$,
\begin{equation}\label{c6.49}
\underline{u}_r + o_r(1;y)
\leq\frac{1}{x}\int^{y+x}_y\varphi (\xi)\,{\rm d}\xi
\leq \overline{u}_r +o_r(1;y),
\end{equation}
where $o_r(1;y)$ represents an infinitesimal quantity that depends on $y$;
and, as $x\rightarrow {-}\infty$,
\begin{equation}\label{c6.51}
\underline{u}_l + o_l(1;y)
\leq\frac{1}{x}\int^{y+x}_y\varphi (\xi)\,{\rm d}\xi
\leq \overline{u}_l +o_l(1;y).
\end{equation}
where $o_l(1;y)$ represents an infinitesimal quantity that depends on $y$.

Since $u(x,t)\in \mathcal{U}(x,t)$,
then
$$
E(u(x,t);x,t)-E(u;x,t)\geq 0\qquad\, {\rm for\ any}\ u\in\mathbb{R}.
$$
It is direct to check from $\eqref{c2.1}$ that
$$
-\int^{x-tf'(u(x,t))}_{x-tf'(u)}(\varphi(\xi)-u)\,{\rm d}\xi
-t\int^{u(x,t)}_{u}(s-u)f''(s)\,{\rm d}s\geq 0,
$$
which is equivalent to that,
for any $u\in \mathbb{R}$,
\begin{equation}\label{c6.53}
0\leq \int^{u(x,t)}_{u}(s-u)f''(s)\,{\rm d}s
\leq \frac{1}{t}\int_{x-tf'(u(x,t))}^{x-tf'(u)}(\varphi(\xi)-u)\,{\rm d}\xi.
\end{equation}

\smallskip
We now present the following lemma on the asymptotic behaviors of the entropy solutions which possess no divides
and the asymptotic behaviors of entropy solutions in $\mathcal{H}_{\pm\infty}$ (if it exists) when the entropy solutions possess at least one divide.

\begin{Lem}\label{lem:c6.3}
Let $u=u(x,t)$ be the entropy solution of the Cauchy problem $\eqref{c1.1}$--$\eqref{ID}$
with initial data function $\varphi(x)\in L^{\infty}(\mathbb{R})$,
and let $x(t)$ for $t\in [0,\infty)$ be a forward generalized characteristic starting from $x(0)=x_0$
along which the entropy solution 
is discontinuous for sufficiently large $t$.
Assume that \eqref{c6.49}--\eqref{c6.51} 
hold uniformly in $y$ for $y\geq 0$ and $y\leq 0$, respectively.
Then
\begin{enumerate}
\item [{\rm(i)}]
If $\lim_{t\rightarrow \infty}y^+(t,0)=\infty$,
then, as $t\rightarrow \infty$,
\begin{equation}\label{c6.54}
\underline{u}_r+o(1)\leq u(x\pm,t)\leq\overline{u}_r+o(1)
\qquad {\rm uniformly\ for}\ x>x(t).
\end{equation}
Furthermore, if $\lim_{t\rightarrow \infty}y^-(t,0)>-\infty$,
then the forward generalized characteristic $x(t)$ lies in $\mathcal{H}_{\infty}$,
and $\eqref{c6.54}$ holds uniformly in $x$ with $(x,t)\in\mathcal{H}_{\infty}$.

\vspace{1pt}
\item [{\rm(ii)}]
If $\lim_{t\rightarrow \infty}y^-(t,0)=-\infty$,
then, as $t\rightarrow \infty$,
\begin{equation}\label{c6.55}
\underline{u}_l+o(1)\leq u(x\pm,t)\leq\overline{u}_l+o(1)
\qquad {\rm uniformly \ for}\ x<x(t).
\end{equation}
Furthermore, if $\lim_{t\rightarrow \infty}y^+(t,0)<\infty$,
then the forward generalized characteristic $x(t)$ lies in $\mathcal{H}_{-\infty}$,
and $\eqref{c6.55}$ holds uniformly in $x$ with $(x,t)\in\mathcal{H}_{-\infty}$.
\end{enumerate}
\noindent
In the above, $\overline{u}_r$, $\underline{u}_r$, $\overline{u}_l$, and $\underline{u}_l$
are given by $\eqref{c5.7}${\rm ;}
and $\mathcal{H}_{\pm\infty}$ are given by $\eqref{c6.4}$--$\eqref{c6.5}$.
\end{Lem}

\begin{proof} We divide the proof into two steps accordingly.

\vspace{1pt}
\noindent
{\bf 1}. Assume that $x>x(t)$.
Denote $y_\pm(x,t):=x-tf'(u(x\pm,t))$ and
\begin{equation}\label{c6.56a}
\overline{y}_r(x,t):=x-tf'(\overline{u}_r),\qquad
\underline{y}_r(x,t):=x-tf'(\underline{u}_r).
\end{equation}

We first prove the right-hand side of $\eqref{c6.54}$ by contradiction.
Otherwise, if the right-hand side of $\eqref{c6.54}$ does not hold uniformly in $x$ with $x>x(t)$,
then there exist $\delta_0>0$ and a sequence $(x_n,t_n)$ in the region of $x>x(t)$ with $\lim_{n\rightarrow\infty}t_n=\infty$
such that
\begin{equation*}
u(x_{n}{\pm},t_n)>\overline{u}_r+\delta_0
\qquad {\rm for\ sufficiently\ large}\ n.
\end{equation*}
Since $f'(u)$ is strictly increasing, this implies that
\begin{equation}\label{c6.56c}
\bar{l}_n:=\overline{y}_r(x_n,t_n)-y_\pm(x_n,t_n)
=t_n\big(f'(u(x_{n}{\pm},t_n))-f'(\overline{u}_r)\big)\ \rightarrow \infty
\qquad\mbox{as $n\rightarrow \infty$}.
\end{equation}
By $\eqref{aa6}$,
$\rho(u,\overline{u}_r)$ is strictly increasing in $u$,
then
there exists $\varepsilon_0=\varepsilon(\delta_0)>0$ such that
\begin{equation*}
\rho(u(x_{n}{\pm},t_n),\overline{u}_r)-\overline{u}_r
=\rho(u(x_{n}{\pm},t_n),\overline{u}_r)-\rho(\overline{u}_r,\overline{u}_r)
>\varepsilon_0.
\end{equation*}
From $\eqref{c6.56a}$--$\eqref{c6.56c}$, $\overline{y}_r(x_n,t_n)>y_\pm(x_n,t_n)\geq x_0$.
Since $\eqref{c6.49}$ 
holds uniformly in $y$ for $y\geq x_0$,
then
\begin{equation}\label{c6.57}
\frac{1}{\bar{l}_n}\int_{y_\pm(x_n,t_n)}^{y_\pm(x_n,t_n)+\bar{l}_n}
(\varphi(\xi)-\overline{u}_r)\,{\rm d}\xi
\leq 
o_r(1;y_\pm(x_n,t_n))
\ \rightarrow 0
\qquad \mbox{as $n\rightarrow\infty$}.
\end{equation}
Since $u(x_{n}{\pm},t_n)\in \mathcal{U}(x_n,t_n)$, then
$E(u(x_{n}{\pm},t_n);x_n,t_n)-E(\overline{u}_r;x_n,t_n)\geq 0$.
By taking $\overline{u}_r$ and $(x_n,t_n)$ in $\eqref{c6.53}$,
$$
0\leq \int^{u(x_{n}{\pm},t_n)}_{\overline{u}_r}(s-\overline{u}_r)f''(s)\,{\rm d}s
\leq \frac{1}{t_n}\int_{y_\pm(x_n,t_n)}^{\overline{y}_r(x_n,t_n)}
(\varphi(\xi)-\overline{u}_r)\,{\rm d}\xi,
$$
which, by $\eqref{c6.57}$ and $\eqref{aa5a}$, implies that
\begin{equation*}
0\leq \rho(u(x_{n}{\pm},t_n),\overline{u}_r)-\overline{u}_r
\leq \frac{1}{\bar{l}_n}\int_{y_\pm(x_n,t_n)}^{\overline{y}_r(x_n,t_n)}
(\varphi(\xi)-\overline{u}_r)\,{\rm d}\xi \
\rightarrow 0
\qquad \mbox{as $n\rightarrow\infty$}.
\end{equation*}
This contradicts to that $\rho(u(x_{n}{\pm},t_n),\overline{u}_r)-\overline{u}_r>\varepsilon_0$.

Furthermore, if $\lim_{t\rightarrow \infty}y^-(t,0)>-\infty$,
by Lemma $\ref{lem:c6.2}$,
it follows from $\lim_{t\rightarrow \infty}y^+(t,0)=\infty$ that
$x(t)$ lies in $\mathcal{H}_{\infty}$.
Thus, $\underline{y}_r(x,t)>x^+_0$ and $y_\pm(x,t)\geq x^+_0$ for any $(x,t)\in\mathcal{H}_{\infty}$.
Similarly, it is direct to check that
the right-hand side of $\eqref{c6.54}$ holds uniformly in $x$ for $(x,t)\in\mathcal{H}_{\infty}$.

\smallskip
We now prove the left-hand side of $\eqref{c6.54}$.
There are two cases:

\smallskip
{\bf (a).} If $\lim_{t\rightarrow \infty}y^-(t,0)>-\infty$,
then $x(t)$ lies in $\mathcal{H}_{\infty}$ so that
$$\underline{y}_r(x,t)>x^+_0, \quad y_\pm(x,t)\geq x^+_0
\qquad {\rm for\ any}\ (x,t)\in\mathcal{H}_{\infty}.$$
Otherwise, if the left-hand side of $\eqref{c6.54}$ does not hold uniformly in $x$ with $(x,t)\in\mathcal{H}_{\infty}$,
there exist $\delta_1>0$ and a sequence $(x_n,t_n)$ with $\lim_{n\rightarrow\infty}t_n=\infty$
such that
\begin{equation*}
u(x_{n}{\pm},t_n)<\underline{u}_r-\delta_1
\qquad {\rm for\ sufficiently\ large}\ n,
\end{equation*}
which, by the strictly increasing property of $f'(u)$, implies that
\begin{equation}\label{c6.56cb}
\underline{l}_n:=y_\pm(x_n,t_n)-\underline{y}_r(x_n,t_n)
=t_n\big(f'(\underline{u}_r)-f'(u(x_{n}{\pm},t_n))\big)\ \rightarrow\infty
\qquad\mbox{as $n\rightarrow\infty$}.
\end{equation}
By $\eqref{aa6}$,
$\rho(u,\underline{u}_r)$ is strictly increasing in $u$,
then
there exists $\varepsilon_1=\varepsilon(\delta_1)>0$ such that
\begin{equation*}
\rho(u(x_{n}{\pm},t_n),\underline{u}_r)-\underline{u}_r
=\rho(u(x_{n}{\pm},t_n),\underline{u}_r)-\rho(\underline{u}_r,\underline{u}_r)<-\varepsilon_1.
\end{equation*}
From $\eqref{c6.56a}$ and
$\eqref{c6.56cb}$, $y_\pm(x_n,t_n)>\underline{y}_r(x_n,t_n)>x^+_0$.
Since $\eqref{c6.49}$ 
holds uniformly in $y$ for $y\geq x_0$,
then
\begin{equation}\label{c6.59}
\frac{1}{\underline{l}_n}\int_{\underline{y}_r(x_n,t_n)}^{\underline{y}_r(x_n,t_n)+\underline{l}_n}(\varphi(\xi)-\underline{u}_r)\,{\rm d}\xi
\geq 
o_r(1;\underline{y}_r(x_n,t_n))
\ \rightarrow 0
 \qquad \mbox{as $n\rightarrow\infty$}.
\end{equation}
Since $u(x_{n}{\pm},t_n)\in \mathcal{U}(x_n,t_n)$, then
$E(u(x_{n}{\pm},t_n);x_n,t_n)-E(\underline{u}_r;x_n,t_n)\geq 0$.
By taking $\underline{u}_r$ and $(x_n,t_n)$ in $\eqref{c6.53}$,
$$
0\geq -\int^{u(x_{n}{\pm},t_n)}_{\underline{u}_r}(s-\underline{u}_r)f''(s)\,{\rm d}s
\geq \frac{1}{t_n}\int^{y_\pm(x_n,t_n)}_{\underline{y}_r(x_n,t_n)}
(\varphi(\xi)-\underline{u}_r)\,{\rm d}\xi,
$$
which, by $\eqref{c6.59}$ and $\eqref{aa5a}$, implies that
\begin{equation*}
0\geq\rho(u(x_{n}{\pm},t_n),\underline{u}_r)-\underline{u}_r
\geq \frac{1}{\underline{l}_n}\int^{y_\pm(x_n,t_n)}_{\underline{y}_r(x_n,t_n)}
(\varphi(\xi)-\underline{u}_r)\,{\rm d}\xi \rightarrow 0
 \qquad \mbox{as $n\rightarrow\infty$}.
\end{equation*}
This contradicts to that $\rho(u(x_{n}{\pm},t_n),\underline{u}_r)-\underline{u}_r<-\varepsilon_1$.

\smallskip
{\bf (b).} If $\lim_{t\rightarrow \infty}y^-(t,0)=-\infty$,
by combining with $\lim_{t\rightarrow \infty}y^+(t,0)=\infty$,
the entropy solution has no divides so that,
by Theorem $\ref{the:c5.2}$,
the entropy solution satisfies that,
any two shocks in $\mathbb{R}\times\mathbb{R}^+$ 
must coincide with each other in a finite time.

Denote $\xi(t):=x_0+\lambda t$ with $\lambda$ satisfying that
\begin{equation}\label{lambda}
\lambda=\frac{f(\overline{u}_l)-f(\underline{u}_r)}{\overline{u}_l-\underline{u}_r}
\quad {\rm if\ }\overline{u}_l>\underline{u}_r,\qquad
\lambda=f'(\underline{u}_r) \quad {\rm if\ }\overline{u}_l=\underline{u}_r.
\end{equation}
Then, for $x>x(t)$, there are two subcases:

\smallskip
{\bf (b1).} If $x>\max\{x(t),\xi(t)\}$,
then $y_\pm(x,t)\geq x_0$ and $\underline{y}_r(x,t)>\xi(t)-tf'(\underline{u}_r)\geq x_0$.
Similar to {\bf (a)},
the left-hand side of $\eqref{c6.54}$ holds uniformly in $x$ with $x>\max\{x(t),\xi(t)\}$.

\smallskip
{\bf (b2).} If $x(t)<x\leq\xi(t)$,
then
$$y_\pm(x,t)\geq x(t)-tf'(u(x(t){+},t))=y^+(t,0)>x_0\qquad {\rm for\ large}\ t.$$
If $x=\xi (t)$, we set $(\xi(\kappa),\kappa):=(\xi(t),t)$;
and if $x(t)<x<\xi(t)$, by Lemma $\ref{lem:c2.4}$,
the characteristic line
$l_{u(x\pm,t)}(x,t):\, \xi=x+(\tau-t) f'(u(x\pm,t))$ for $\tau\in (0,t)$
must intersect with the straight line $\xi(\tau)=x_0+\lambda\tau$
at some point $(\xi(\kappa),\kappa)$ with $\kappa\in (0,t)$.
Then
$\kappa$ satisfies that
\begin{equation}\label{c6.61}
x_0+\kappa \big(\lambda-f'(u(x\pm,t))\big)=x-tf'(u(x\pm,t))=y_\pm(x,t)\geq y^+(t,0),
\end{equation}
which, by $u(x,t)\in L^\infty(\mathbb{R}\times \mathbb{R}^+)$ and $y^+(\infty,0)=\infty$, implies that,
as $t\rightarrow \infty$,
\begin{equation*}
\kappa\rightarrow \infty\qquad {\rm uniformly\ in\ } x \ {\rm for}\ x\in (x(t),\xi(t)].
\end{equation*}

Let $y_\pm(\xi(\kappa),\kappa):=\xi(\kappa)-\kappa f'(u(x\pm,t))$.
Then
\begin{equation}\label{c6.61c}
y_\pm(\xi(\kappa),\kappa)=y_\pm(x,t)>y^+(t,0)\ \rightarrow \infty \qquad\mbox{as $t\rightarrow\infty$}.
\end{equation}
By Lemma $\ref{lem:c2.5}$, $u(x\pm,t)\in \mathcal{U}(\xi(\kappa),\kappa)$ so that
$E(u(x\pm,t);\xi(\kappa),\kappa)-E(\underline{u}_r;\xi(\kappa),\kappa)\geq 0,$
which, by $\eqref{c6.53}$, implies
\begin{equation}\label{c6.62}
0\geq -\int^{u(x\pm,t)}_{\underline{u}_r}(s-\underline{u}_r)f''(s)\,{\rm d}s
\geq \frac{1}{\kappa}\int^{y_\pm(\xi(\kappa),\kappa)}_{\underline{y}_r(\xi(\kappa),\kappa)}
(\varphi(\xi)-\underline{u}_r)\,{\rm d}\xi.
\end{equation}
From $\eqref{lambda}$, $\underline{y}_r(\xi(\kappa),\kappa)\geq x_0$.
By the same arguments as {\bf (a)},
it is direct to check from $\eqref{c6.62}$ that
the left-hand side of $\eqref{c6.54}$ holds uniformly in $x$ with $x(t)<x\leq\xi(t)$.

Therefore, $\eqref{c6.54}$ holds uniformly in $x$ with $x>x(t)$ for the case of {\bf (b)}.

\smallskip
\noindent
{\bf 2}. By the same argument as in Step 1,
it can be checked from $\eqref{c6.51}$--$\eqref{c6.53}$ that (ii) is true.
\end{proof}

\smallskip
If $\overline{u}_r=\underline{u}_r$ and $\overline{u}_l=\underline{u}_l$,
similar to \eqref{c6.49}--\eqref{c6.51}
it is direct to check that,
as $x\rightarrow \infty$,
\begin{equation}\label{c6.63}
\frac{1}{x}\int^{y+x}_y\varphi (\xi)\,{\rm d}\xi
= \underline{u}_r + o_r(1;y), 
\end{equation}
where $o_r(1;y)$ represents an infinitesimal quantity that depends on $y$;
and, as $x\rightarrow {-}\infty$,
\begin{equation}\label{c6.65}
\frac{1}{x}\int^{y+x}_y\varphi (\xi)\,{\rm d}\xi
=\overline{u}_l+ o_l(1;y),
\end{equation}
where $o_l(1;y)$ represents an infinitesimal quantity that depends on $y$.

\begin{The}[Asymptotic profiles in the $L^\infty$--norm]\label{pro:c6.3}
Let $u=u(x,t)$ and $\tilde{u}=\tilde{u}(x,t)$ be the entropy solutions
of the Cauchy problem $\eqref{c1.1}$--$\eqref{ID}$ and the Cauchy problem $\eqref{c6.10}$, respectively.
Assume that $\eqref{c6.63}$ and $\eqref{c6.65}$ hold uniformly in $y$ for $y\geq 0$ and $y\leq 0$, respectively.
Then
\begin{enumerate}
\item[{\rm(i)}] When $\overline{u}_l<\underline{u}_r$,
or $\overline{u}_l=\underline{u}_r$ with $\mathcal{K}_0\neq \varnothing$.
The asymptotic profile of the entropy solution $u=u(x,t)$ in the $L^\infty$--norm is the rarefaction-constant solution $\tilde{u}=\tilde{u}(x,t)$ given by \eqref{kk00}--\eqref{tidleuH} as follows{\rm :}
For $\mathbb{R}\times \mathbb{R}^+=\mathcal{K}\cup\big(\cup_n \mathcal{H}_{(e_n,h_n)}\big)$
as in $\eqref{c6.7}$, 

\vspace{2pt}
\begin{itemize}
\item[(a)] In $\mathcal{K}$, $u(x,t)\equiv \tilde{u}(x,t)$ 
and can be expressed by
\begin{equation}\label{kk0}
u(x,t)=\tilde{u}(x,t)=\big(f'\big)^{-1}(\frac{x-x_0}{t})
\qquad {\rm with}\ x_0\in \mathcal{K}_0.
\end{equation}

\item [(b)] If there exists $\mathcal{H}_{(e_n,h_n)}$ with finite $e_n$ and $h_n$,
then, as $t\rightarrow \infty$,
\begin{equation}\label{c6.67cn}
u(x\pm,t)\cong\tilde{u}(x,t)=c_n
\qquad {{\rm uniformly\ in}}\ x \ {{\rm for}}\ (x,t)\in \mathcal{H}_{(e_n,h_n)}.
\end{equation}

\item [(c)] If there exists $\mathcal{H}_{(e_n,h_n)}$ with finite $e_n$ and $h_n=\infty$,
then, as $t\rightarrow \infty$,
\begin{equation}\label{c6.67}
u(x\pm,t)\cong\tilde{u}(x,t)=\underline{u}_r
\qquad \mbox{\rm uniformly in $x$ for $(x,t)\in \mathcal{H}_{\infty}$}.
\end{equation}

\item [(d)] If there exists $\mathcal{H}_{(e_n,h_n)}$ with finite $h_n$ and $e_n={-}\infty$,
then, as $t\rightarrow \infty$,
\begin{equation}\label{c6.68}
u(x\pm,t)\cong\tilde{u}(x,t)=\overline{u}_l
\qquad \mbox{\rm uniformly in $x$ for $(x,t)\in \mathcal{H}_{-\infty}$}.
\end{equation}
\end{itemize}

\item[{\rm(ii)}] When $\overline{u}_l>\underline{u}_r$,
or $\overline{u}_l=\underline{u}_r$ with $\mathcal{K}_0= \varnothing$.
The entropy solution $u=u(x,t)$ will turn into a single shock asymptotically in 
$\mathbb{R}\times\mathbb{R}^+$.
There exist two cases{\rm :}

\begin{itemize}
\item [(a)] If $\overline{u}_l=\underline{u}_r$,
then, as $t\rightarrow \infty$,
\begin{equation}\label{c6.69}
u(x\pm,t)=\underline{u}_r+o(1)
\qquad \mbox{\rm uniformly in $x$ for $(x,t)\in \mathbb{R}\times \mathbb{R}^+$}.
\end{equation}

\item [(b)] If $\overline{u}_l>\underline{u}_r$, 
then, as $t\rightarrow \infty$,
\begin{align}\label{c6.70}
\begin{cases}
\displaystyle u(x\pm,t)=\overline{u}_l+o(1) \quad &\mbox{\rm uniformly in $x$ for $x<x(t)$}, \\[1mm]
\displaystyle u(x\pm,t)=\underline{u}_r+o(1)\quad &\mbox{\rm uniformly in $x$ for $x>x(t)$},
\end{cases}
\end{align}
where $x=x(t)$ for $t\in[0,\infty)$ is a forward generalized characteristic 
emitting from any point on the $x$--axis.
\end{itemize}
\end{enumerate}
\noindent
In the above, $\overline{u}_l$ and $\underline{u}_r$ are given by $\eqref{c5.7}${\rm ;}
$\mathcal{K}_0$ and $\mathcal{K}$ are given by $\eqref{c5.6}$ and $\eqref{c6.6}$, respectively{\rm;}
and $\mathcal{H}_{(e_n,h_n)}$ with $c_n$ and $\mathcal{H}_{\pm\infty}$ are given by $\eqref{c6.2}$--$\eqref{c6.5}$.
\end{The}

\begin{proof}
The proof is divided into two steps accordingly.

\smallskip
\noindent
{\bf 1.} If $\overline{u}_l<\underline{u}_r$,
or $\overline{u}_l=\underline{u}_r$ with $\mathcal{K}_0\neq \varnothing$, 
by Lemma $\ref{lem:c5.2}$, 
$\bar{\Phi}(x)>-\infty$ and is convex so that,
the Cauchy problem $\eqref{c6.10}$ admits a unique rarefaction-constant solution
$\tilde{u}=\tilde{u}(x,t)$.
Thus, it follows from $\eqref{c6.25}$ that $u(x,t)\equiv \tilde{u}(x,t)$ on $\mathcal{K}$, and $\eqref{kk0}$ follows from $\eqref{kk00}$. 
Furthermore, according to $\eqref{tidleuH}$, $\eqref{c6.67cn}$ follows from $\eqref{c6.27}$,
and $\eqref{c6.67}$--$\eqref{c6.68}$ follow immediately by taking
$\overline{u}_r=\underline{u}_r$ and $\overline{u}_l=\underline{u}_l$ into Lemma $\ref{lem:c6.3}$.

\smallskip
\noindent
{\bf 2.} If $\overline{u}_l<\underline{u}_r$,
or $\overline{u}_l=\underline{u}_r$ with $\mathcal{K}_0\neq \varnothing$, according to Theorem $\ref{pro:c6.2}$(ii),
entropy solution $u(x,t)$ possesses no divides and turns into a single shock asymptotically in $\mathbb{R}\times\mathbb{R}^+$. 
Thus, any forward generalized characteristic will turn into a shock in a finite time.
Furthermore, $\eqref{c6.69}$--$\eqref{c6.70}$ follow immediately by taking
$\overline{u}_r=\underline{u}_r$ and $\overline{u}_l=\underline{u}_l$ into Lemma $\ref{lem:c6.3}$.
\end{proof}

\smallskip
In order to obtain the decay rates of the entropy solution to its asymptotic profile,
we need the growth rates of the initial data functions $\varphi(x)\in L^\infty(\mathbb{R})$
as $x\rightarrow \pm\infty$.
More precisely, 
when $\overline{u}_r=\underline{u}_r$ and $\overline{u}_l=\underline{u}_l$,
according to $\eqref{c6.63}$--$\eqref{c6.65}$,
we assume that,
for some $\gamma_r,\gamma_l\in [0,1)$,
as $x\rightarrow \infty$,
\begin{equation}\label{c6.71}
\int^{y+x}_y\varphi (\xi)\,{\rm d}\xi
= \underline{u}_r x+ O(1)\, x^{\gamma_r} 
\qquad\quad\,\;{\rm uniformly\ in}\ y\geq 0;
\end{equation}
and, as $x\rightarrow {-}\infty$,
\begin{equation}\label{c6.72}
\int^{y+x}_y\varphi (\xi)\,{\rm d}\xi
=\overline{u}_l x+O(1)\, ({-}x)^{\gamma_l} 
\qquad\, {\rm uniformly\ in}\ y\leq 0.
\end{equation}

\begin{The}[Decay rates in the $L^\infty$--norm]\label{the:c6.2}
Let $u=u(x,t)$ and $\tilde{u}=\tilde{u}(x,t)$ be the entropy solutions
of the Cauchy problem $\eqref{c1.1}$--$\eqref{ID}$
and the Cauchy problem $\eqref{c6.10}$, respectively.
Assume that 
$\eqref{c6.71}$ and 
$\eqref{c6.72}$ hold
for some $\gamma_r\in [0,1)$ and $\gamma_l\in [0,1)$, respectively. Then

\begin{enumerate}
\item[{\rm(i)}] When $\overline{u}_l<\underline{u}_r$,
or $\overline{u}_l=\underline{u}_r$ with $\mathcal{K}_0\neq \varnothing$.
For $\mathbb{R}\times \mathbb{R}^+=\mathcal{K}\cup\big(\cup_n \mathcal{H}_{(e_n,h_n)}\big)$
as in $\eqref{c6.7}$,  

\vspace{2pt}
\begin{itemize}
\item[(a)] In $\mathcal{K}$, 
$\eqref{kk0}$ holds{\rm;} especially, $u(x,t)-\tilde{u}(x,t)\equiv 0$.

\vspace{1pt}
\item [(b)] If there exists $\mathcal{H}_{(e_n,h_n)}$ with finite $e_n$ and $h_n$,
then, as $t\rightarrow \infty$,
\begin{equation}\label{c6.82ch}
|u(x,t)-\tilde{u}(x,t)|=
|u(x,t)-c_n|
\lessapprox C_n
\Big(\frac{h_n-e_n}{2}\Big)^{\frac{1}{1+\alpha_n}} \, t^{-\frac{1}{1+\alpha_n}}
\end{equation}
holds uniformly in $x$ with $(x,t)\in \mathcal{H}_{(e_n,h_n)}$.

\vspace{1pt}
\item [(c)] If there exists $\mathcal{H}_{(e_n,h_n)}$ with finite $e_n$ and $h_n=\infty$,
then, as $t\rightarrow \infty$,
\begin{equation}\label{c6.73}
|u(x,t)-\tilde{u}(x,t)|=|u(x,t)-\underline{u}_r|
\lesssim t^{-\frac{1-\gamma_r}{2-\gamma_r+\alpha(1-\gamma_r)}}
\end{equation}
holds uniformly in $x$ with $(x,t)\in \mathcal{H}_{\infty}$.

\vspace{1pt}
\item [(d)] If there exists $\mathcal{H}_{(e_n,h_n)}$ with finite $h_n$ and $e_n={-}\infty$,
then, as $t\rightarrow \infty$,
\begin{equation}\label{c6.74}
|u(x,t)-\tilde{u}(x,t)|=|u(x,t)-\overline{u}_l|
\lesssim t^{-\frac{1-\gamma_l}{2-\gamma_l+\beta(1-\gamma_l)}}
\end{equation}
holds uniformly in $x$ with $(x,t)\in \mathcal{H}_{{-}\infty}$.
\end{itemize}

\item[{\rm(ii)}] When $\overline{u}_l>\underline{u}_r$,
or $\overline{u}_l=\underline{u}_r$ with $\mathcal{K}_0= \varnothing$.

\vspace{2pt}
\begin{itemize}
\item [(a)] If $\overline{u}_l=\underline{u}_r$,
then, as $t\rightarrow \infty$,
\begin{equation}\label{c6.75}
\,\,\,\,\,\,\,\,\,\,\,\,\,\,\,\,\,\,\,\,
|u(x,t)-\underline{u}_r|
\lesssim
t^{-\frac{1-\gamma}{2-\gamma+\alpha(1-\gamma)}}
\qquad {\rm uniformly\ in }\ x \ {\rm for\ } (x,t)\in \mathbb{R}\times\mathbb{R}^+,
\end{equation}
where $\gamma:=\max \{\gamma_r,\gamma_l\}$.

\vspace{1pt}
\item [(b)] If $\overline{u}_l>\underline{u}_r$, 
then, as $t\rightarrow \infty$,
\begin{align}\label{c6.76}
\,\,\,\,\,\,\,\,\,\,\,\,
\begin{cases}
\displaystyle|u(x,t)-\overline{u}_l|
\lesssim
t^{-\frac{1-\gamma_l}{2-\gamma_l+\beta(1-\gamma_l)}} \quad &{\rm uniformly \ for\ }\ x<x(t), \\[1mm]
\displaystyle|u(x,t)-\underline{u}_r|
\lesssim
t^{-\frac{1-\gamma_r}{2-\gamma_r+\alpha(1-\gamma_r)}}\quad &{\rm uniformly \ for\ }\ x>x(t),
\end{cases}
\end{align}
where $x=x(t)$ for $t\in[0,\infty)$ is a forward generalized characteristic 
emitting from any point on the $x$--axis.
\end{itemize}
\end{enumerate}
\noindent
In the above, $\overline{u}_l$ and $\underline{u}_r$ are given by $\eqref{c5.7}${\rm ;}
$\mathcal{K}_0$ and $\mathcal{K}$ are given by $\eqref{c5.6}$ and $\eqref{c6.6}$, respectively{\rm;}
$\mathcal{H}_{(e_n,h_n)}$ with $c_n$ and $\mathcal{H}_{\pm\infty}$ are given by $\eqref{c6.2}$--$\eqref{c6.5}${\rm;}
and parameters $\alpha\geq0$ and $\beta\geq 0$ satisfy
$$\lim_{u\rightarrow \underline{u}_r}\frac{f''(u)}{|u-\underline{u}_r|^{\alpha}}=N_r>0,\qquad
\lim_{u\rightarrow \overline{u}_l}\frac{f''(u)}{|u-\overline{u}_l|^{\beta}}=N_l>0.$$
\end{The}

\begin{proof}
By Theorem $\ref{pro:c6.3}$, it suffices to check the decay rates in different cases.

\smallskip
\noindent
{\bf 1}. Suppose that $\overline{u}_l<\underline{u}_r$,
or $\overline{u}_l=\underline{u}_r$ with $\mathcal{K}_0\neq \varnothing$.
$\eqref{c6.82ch}$ follows from $\eqref{c6.29}$ directly.

\smallskip
{\bf (a).} If $x_0^+=\sup\mathcal{K}_0<\infty$,
then $\mathcal{H}_{\infty}$ exists and $\eqref{c6.67}$ holds.
For any $(x,t)\in \mathcal{H}_{\infty}$, we let
\begin{equation}\label{c6.77a}
y_\pm(x,t)=x-tf'(u(x\pm,t)),\qquad y_r(x,t)=x-tf'(\underline{u}_r).
\end{equation}
Then
$y_\pm(x,t)\geq x_0^+$ and $y_r(x,t)>x_0^+$ for any $(x,t)\in \mathcal{H}_{\infty}.$

Since $u(x\pm,t)\in \mathcal{U}(x,t)$, then
$E(u(x\pm,t);x,t)-E(\underline{u}_r;x,t)\geq 0$.
By taking $\underline{u}_r$ into $\eqref{c6.53}$,
\begin{equation}\label{c6.77}
0\leq \int^{u(x\pm,t)}_{\underline{u}_r}(s-\underline{u}_r)f''(s)\,{\rm d}s
\leq \frac{1}{t}\int_{y_\pm(x,t)}^{y_r(x,t)}(\varphi(\xi)-\underline{u}_r)\,{\rm d}\xi.
\end{equation}

\noindent
Denote
$l(x,t):=\big|y_r(x,t)-y_\pm(x,t)\big|=t\big|f'(u(x\pm,t))-f'(\underline{u}_r)\big|.$
Then
$\eqref{c6.71}$ implies
\begin{align*}
\frac{1}{t}\int_{y_\pm(x,t)}^{y_r(x,t)}(\varphi(\xi)-\underline{u}_r)\,{\rm d}\xi \leq
\begin{cases}
L \|\varphi(x)-\underline{u}_r\|_{L^\infty}\, t^{-1}=O(1)\, t^{-1}\quad &{\rm if} \ l(x,t)\leq L, \\[1mm]
O(1)\, \big(t|f'(u(x\pm,t))-f'(\underline{u}_r)|\big)^{\gamma_r}\,t^{-1}\quad &{\rm if}\ l(x,t)> L.
\end{cases}
\end{align*}
Therefore, by using $\eqref{c6.77a}$
and $\eqref{aa2}$--$\eqref{aa3}$,
it is direct to check from $\eqref{c6.77}$ that
\begin{equation}\label{c6.78}
\begin{cases}
\displaystyle|u(x\pm,t)-\underline{u}_r|\leq O(1)\,
t^{-\frac{1}{2+\alpha}}\quad &{\rm if}\ l(x,t)\leq L,\\[1mm]
\displaystyle|u(x\pm,t)-\underline{u}_r|\leq O(1)\,
t^{-\frac{1-{\gamma_r}}{2-{\gamma_r}+\alpha(1-{\gamma_r})}}\quad &{\rm if}\ l(x,t)> L.
\end{cases}
\end{equation}
Since $\frac{1-{\gamma_r}}{2-{\gamma_r}+\alpha(1-{\gamma_r})}\leq \frac{1}{2+\alpha}$
for $\alpha\geq 0$ and ${\gamma_r}\in[0,1)$,
$\eqref{c6.78}$ implies $\eqref{c6.73}$.

\smallskip
{\bf (b).} Similar to {\bf (a)},
it can be checked from $\eqref{c6.65}$, $\eqref{c6.68}$, and $\eqref{c6.72}$ that $\eqref{c6.74}$ holds.

\smallskip
\noindent
{\bf 2}. Suppose that $\overline{u}_l>\underline{u}_r$,
or $\overline{u}_l=\underline{u}_r$ with $\mathcal{K}_0= \varnothing$.

\smallskip
{\bf (a).} If $\overline{u}_l=\underline{u}_r$,
it is direct to check from $\eqref{c6.69}$ and $\eqref{c6.71}$--$\eqref{c6.72}$ that,
 as $x\rightarrow\pm\infty $,
$$
\Big|\int^{y+x}_y(\varphi (\xi) -\underline{u}_r )\,{\rm d}\xi\Big|\leq O(1)\, |x|^{\gamma}
\qquad {\rm uniformly\ in }\ y\in ({-}\infty,\infty).
$$
By the same arguments as those for $\eqref{c6.77a}$--$\eqref{c6.78}$,
 it can be checked that $\eqref{c6.75}$ holds.

\smallskip
{\bf (b).} If $\overline{u}_l>\underline{u}_r$.
Denote $\xi(t):=x_0+\lambda t$ with
$\lambda=\frac{f(\overline{u}_l)-f(\underline{u}_r)}{\overline{u}_l-\underline{u}_r}$.
There are two cases:

\smallskip
{\bf (b1).} For the case that $x>x(t)$, by Lemma $\ref{lem:c6.3}$, there are two subcases:

If $x>\max\{x(t),\xi(t)\}$, then $y_\pm(x,t)\geq x_0$ and $y_r(x,t)>x_0$.
Since $\eqref{c6.71}$ holds for $\gamma_r\in[0,1)$,
similar to $\eqref{c6.77a}$--$\eqref{c6.78}$,
it can be checked that $\eqref{c6.76}$ holds for $x>\max\{x(t),\xi(t)\}$.

If $x(t)<x\leq\xi(t)$,
by the same argument as for the case of {\bf (b2)} in the proof of Lemma $\ref{lem:c6.3}$,
from $\eqref{c6.61c}$ and $\eqref{c6.77a}$,
$y_\pm(\xi(\kappa),\kappa)\geq x_0$ and $y_r(\xi(\kappa),\kappa)>x_0$ so that $\eqref{c6.62}$ holds.
Applying the same argument as that for $\eqref{c6.77a}$--$\eqref{c6.78}$ to $\eqref{c6.62}$,
we obtain that, as $t\rightarrow\infty$,
\begin{equation*}
|u(x\pm,t)-\underline{u}_r|
\leq O(1)\, \kappa^{-\frac{1-\gamma_r}{2-\gamma_r+\alpha(1-\gamma_r)}}
\qquad {\rm uniformly\ in}\ x\ {\rm with}\ x(t)<x\leq\xi(t).
\end{equation*}
Thus, to prove $\eqref{c6.76}$,
it suffices to show that, as $t\rightarrow\infty$,
\begin{equation}\label{c6.80a}
\frac{\kappa}{t}=1+o(1)\qquad {\rm uniformly\ in\ } x \ {\rm for} \ x(t)<x\leq\xi(t).
\end{equation}
In fact, if $x=\xi(t)$, then $\kappa=t$.
If $x(t)<x<\xi(t)$, from $\eqref{c6.61}$,
\begin{equation}\label{c6.80c}
0>(1-\frac{\kappa}{t})\big(f'(u(x\pm,t))-\lambda\big)
=\frac{x-\xi(t)}{t}>\frac{x(t)-\xi(t)}{t}=\frac{x(t)-x_0}{t}-\lambda.
\end{equation}
From $\eqref{c6.70}$, we obtain
$$
\lim\limits_{t\rightarrow\infty}\Big(\frac{x(t)}{t}-\lambda\Big)=
\lim\limits_{t\rightarrow\infty}
\frac{1}{t}\int_{t_0}^t\Big(\frac{f(u^-(\tau))-f(u^+(\tau))}{u^-(\tau)-u^+(\tau)}
-\frac{f(\overline{u}_l)-f(\underline{u}_r)}{\overline{u}_l-\underline{u}_r}\Big)
\,{\rm d}\tau=0,
$$
so that, by $\eqref{c6.80c}$, as $t\rightarrow \infty$,
$\frac{x-\xi(t)}{t}\longrightarrow 0$ uniformly in $x$ for $x(t)<x\leq\xi(t),$
which, by combining with $\eqref{c6.70}$ and $\eqref{c6.80c}$, yields $\eqref{c6.80a}$.

\smallskip
{\bf (b2).} Similar to {\bf (b1)},
it can be checked that $\eqref{c6.76}$ holds for the case that $x<x(t)$.

\smallskip
Up to now, we have completed the proof of Theorem $\ref{the:c6.2}$.
\end{proof}

\smallskip
If $\varphi(x)\in L^{\infty}(\mathbb{R})$ satisfies that
$\varphi(x)-\underline{u}_r\in L^p([0,\infty))$ with $p\in [1,\infty)$,
then $\overline{u}_r=\underline{u}_r$.
By the H\"{o}lder inequality,
\begin{equation}\label{c6.81}
\big|\int_{y}^{y+x}(\varphi(\xi)-\underline{u}_r)\,{\rm d}\xi\big|
\leq \|\varphi-\underline{u}_r\|_{L^p}\, |x|^{1-\frac{1}{p}},
\end{equation}
which implies that $\eqref{c6.71}$ holds for $\gamma_r=1-\frac{1}{p}$.
Similarly, if $\varphi(x)\in L^{\infty}(\mathbb{R})$ satisfies that
$\varphi(x)-\overline{u}_l\in L^q(({-}\infty,0])$ with $q\in [1,\infty)$,
then $\overline{u}_l=\underline{u}_l$,
and $\eqref{c6.72}$ holds for $\gamma_l=1-\frac{1}{q}$.

\smallskip
Applying Theorem $\ref{the:c6.2}$ to the $L^{\infty}$ initial data function $\varphi(x)$
satisfying that
$\varphi(x)-\underline{u}_r\in L^p([0,\infty))$
and $\varphi(x)-\overline{u}_l\in L^q(({-}\infty,0])$ with $p,q\in [1,\infty)$,
we obtain the following theorem.

\begin{The}[Decay rates in the $L^\infty$--norm]\label{the:c6.3}
Let $u=u(x,t)$ be the entropy solution of the Cauchy problem $\eqref{c1.1}$--$\eqref{ID}$
with initial data function $\varphi(x)\in L^{\infty}(\mathbb{R})$ satisfying that
$\varphi(x)-\underline{u}_r\in L^p([0,\infty))$
and $\varphi(x)-\overline{u}_l\in L^q(({-}\infty,0])$ with $p,q\in [1,\infty)$,
and let $\tilde{u}=\tilde{u}(x,t)$ be the entropy solutions
of the Cauchy problem $\eqref{c6.10}$ as in \eqref{kk00}--\eqref{tidleuH}.
Then
\begin{itemize}
\item [{\rm (i)}] When $\overline{u}_l<\underline{u}_r$,
or $\overline{u}_l=\underline{u}_r$ with $\mathcal{K}_0\neq \varnothing$.
For $\mathbb{R}\times \mathbb{R}^+=\mathcal{K}\cup\big(\cup_n \mathcal{H}_{(e_n,h_n)}\big)$
as in $\eqref{c6.7}$, 

\vspace{2pt}
\begin{enumerate}
\item[{\rm(a)}] In $\mathcal{K}$, 
$\eqref{kk0}$ holds{\rm;} especially, $u(x,t)-\tilde{u}(x,t)\equiv 0$.

\vspace{1pt}
\item [{\rm (b)}] If there exists $\mathcal{H}_{(e_n,h_n)}$ with finite $e_n$ and $h_n$,
then, as $t\rightarrow \infty$,
\begin{equation}\label{c6.82}
|u(x,t)-\tilde{u}(x,t)|=|u(x,t)-c_n|
\lessapprox C_n
\Big(\frac{h_n-e_n}{2}\Big)^{\frac{1}{1+\alpha_n}} \, t^{-\frac{1}{1+\alpha_n}}
\end{equation}
holds uniformly in $x$ with $(x,t)\in \mathcal{H}_{(e_n,h_n)}$.

\vspace{1pt}
\item [{\rm (c)}] If there exists $\mathcal{H}_{(e_n,h_n)}$ with finite $e_n$ and $h_n=\infty$,
then, as $t\rightarrow \infty$,
\begin{equation}\label{c6.83}
|u(x,t)-\tilde{u}(x,t)|=|u(x,t)-\underline{u}_r|
\lessapprox C_r
\|\varphi-\underline{u}_r\|_{L^p([x_0^+,\infty))}^{\frac{p}{p+1+\alpha}}\,
t^{-\frac{1}{p+1+\alpha}}
\end{equation}
holds uniformly in $x$ with $(x,t)\in \mathcal{H}_{\infty}$.

\vspace{1pt}
\item [{\rm (d)}] If there exists $\mathcal{H}_{(e_n,h_n)}$ with finite $h_n$ and $e_n={-}\infty$,
then, as $t\rightarrow \infty$,
\begin{equation}\label{c6.84}
|u(x,t)-\tilde{u}(x,t)|=|u(x,t)-\overline{u}_l|
\lessapprox C_l
\|\varphi-\overline{u}_l\|_{L^q(({-}\infty,x_0^-])}^{\frac{q}{q+1+\beta}}\,
t^{-\frac{1}{q+1+\beta}}
\end{equation}
holds uniformly in $x$ with $(x,t)\in \mathcal{H}_{-\infty}$.
\end{enumerate}

\smallskip
\item [{\rm (ii)}] When $\overline{u}_l>\underline{u}_r$,
or $\overline{u}_l=\underline{u}_r$ with $\mathcal{K}_0= \varnothing$.

\vspace{2pt}
\begin{enumerate}
\item [{\rm (a)}] If $\overline{u}_l=\underline{u}_r$, then, as $t\rightarrow \infty$,
\begin{equation}\label{c6.85}
|u(x,t)-\underline{u}_r|
\lessapprox C_r
\|\varphi-\underline{u}_r\|_{L^{s}(\Omega)}^{\frac{s}{s+1+\alpha}}\,
t^{-\frac{1}{s+1+\alpha}}
\end{equation}
holds uniformly in $x$ with $(x,t)\in \mathbb{R}\times \mathbb{R}^+$,
where $s:=\max\{p,q\}$,  $\Omega=[0,\infty)$ if $p>q$,
$\Omega=({-}\infty,0]$ if $p<q$, and $\Omega=\mathbb{R}$ if $p=q$.

\vspace{1pt}
\item [{\rm (b)}] If $\overline{u}_l>\underline{u}_r$,
then, as $t\rightarrow \infty$,
\begin{align}\label{c6.86}
\,\,\,\,\,\,\,\,\,\,\,\,\,\,\,\,\,\,\,\,\,\,\,\,\,\,\,
\begin{cases}
|u(x,t)-\overline{u}_l|
\lessapprox C_l
\|\varphi-\overline{u}_l\|_{L^q(({-}\infty,x_0])}^{\frac{q}{q+1+\beta}}\,
t^{-\frac{1}{q+1+\beta}}
\ &{\rm uniformly \ for\ } x<x(t), \\[3mm]
|u(x,t)-\underline{u}_r|
\lessapprox C_r
\|\varphi-\underline{u}_r\|_{L^p([x_0,\infty))}^{\frac{p}{p+1+\alpha}}\,
t^{-\frac{1}{p+1+\alpha}}
\ &{\rm uniformly \ for\ } x>x(t).
\end{cases}
\end{align}
where $x=x(t)$ for $t\in[0,\infty)$ is a forward generalized characteristic 
emitting from any point $x_0$ on the $x$--axis.
\end{enumerate}
\end{itemize}
\noindent
In the above,
$\overline{u}_l$ and $\underline{u}_r$ are given by $\eqref{c5.7}${\rm ;}
$\mathcal{K}_0$ and $\mathcal{K}$ are given by $\eqref{c5.6}$ and $\eqref{c6.6}$, respectively{\rm;}
$\mathcal{H}_{(e_n,h_n)}$ with $c_n$ and $\mathcal{H}_{\pm\infty}$ are given by $\eqref{c6.2}$--$\eqref{c6.5}${\rm;}
and constants $C_n,C_r,C_l$ and parameters $\alpha_n,\alpha,\beta$ are given by

\smallskip
${\rm (a)}$ $\,C_n=\big(\frac{1+\alpha_n}{N_n}\big)^{\frac{1}{1+\alpha_n}}$
with $\alpha_n\geq0$ satisfying that
 $\lim_{u\rightarrow c_n}\frac{f''(u)}{|u-c_n|^{\alpha_n}}=N_n>0$.

\smallskip
${\rm (b)}$ $\, C_r=
\big(\frac{2+\alpha}{1+\alpha}\big)^{\frac{p}{p+1+\alpha}}\big(\frac{1+\alpha}{N_r}\big)^{\frac{1}{p+1+\alpha}}$
and $C_l=
\big(\frac{2+\beta}{1+\beta}\big)^{\frac{q}{q+1+\beta}}\big(\frac{1+\beta}{N_l}\big)^{\frac{1}{q+1+\beta}}$
with $\alpha\geq0$ and $\beta\geq0$ satisfying that
$\lim\limits_{u\rightarrow \underline{u}_r}\frac{f''(u)}{|u-\underline{u}_r|^\alpha}=N_r>0$ and
$\lim\limits_{u\rightarrow \overline{u}_l}\frac{f''(u)}{|u-\overline{u}_l|^\beta}=N_l>0$, respectively.
\end{The}

\begin{proof}
By $\eqref{c6.81}$,
Theorem $\ref{the:c6.3}$ follows from Theorem $\ref{the:c6.2}$ by
taking $\gamma_r=1-\frac{1}{p}$ and $\gamma_l=1-\frac{1}{q}$,
except to check the constants in $\eqref{c6.83}$--$\eqref{c6.86}$.

\smallskip
We now check constant $C_r$ in $\eqref{c6.83}$ as an example.
All the other constants in $\eqref{c6.84}$--$\eqref{c6.86}$ can be checked similarly.
Suppose that $x_0^+=\sup\mathcal{K}_0<\infty$.
For any $(x,t)\in \mathcal{H}_{\infty}$,
since $u(x\pm,t)\in \mathcal{U}(x,t)$, then
$E(u(x\pm,t);x,t)-E(\underline{u}_r;x,t)\geq 0$
so that, by taking $\underline{u}_r$ into $\eqref{c6.53}$,
\begin{equation}\label{c6.87}
0\leq \int^{u(x\pm,t)}_{\underline{u}_r}(s-\underline{u}_r)f''(s)\,{\rm d}s
\leq \frac{1}{t}\int_{x-tf'(u(x\pm,t))}^{x-tf'(\underline{u}_r)}
(\varphi(\xi)-\underline{u}_r)\,{\rm d}\xi.
\end{equation}
From $\mathcal{H}_{\infty}$ in $\eqref{c6.4}$,
$L_{\underline{u}_r}(x_0^+):x=x_0^++tf'(\underline{u}_r)$ is a divide so that,
$$x-tf'(u(x\pm,t))\geq x_0^+,\quad x-tf'(\underline{u}_r)>x_0^+
\qquad {\rm for\ any}\ (x,t)\in \mathcal{H}_{\infty}.$$
Since $\varphi(x)\in L^\infty(\mathbb{R})$ satisfying that
$\varphi(x)-\underline{u}_r\in L^p([0,\infty))$ with $p\in [1,\infty)$,
then $\varphi(x)-\underline{u}_r\in L^p([x_0^+,\infty))$ with $p\in [1,\infty)$.
Thus, by the H\"{o}lder inequality,
\begin{equation*}
\int_{x-tf'(u(x\pm,t))}^{x-tf'(\underline{u}_r)}(\varphi(\xi)-\underline{u}_r)\,{\rm d}\xi
\leq \|\varphi-\underline{u}_r\|_{L^p([x_0^+,\infty))}\,
\big(t|f'(u(x\pm,t))-f'(\underline{u}_r)|\big)^{1-\frac{1}{p}},
\end{equation*}
which, by $\eqref{c6.87}$ and $\eqref{aa2}$--$\eqref{aa3}$, implies that,
for $u(x\pm,t)\neq \underline{u}_r$,
\begin{equation*}
0\leq \dfrac{\frac{N_r}{2+\alpha}+o(1)}{\big(\frac{N_r}{1+\alpha}\big)^{1-\frac{1}{p}}+o(1)}\,
\big|u(x\pm,t)-\underline{u}_r\big|^{\frac{p+1+\alpha}{p}}
\leq \|\varphi-\underline{u}_r\|_{L^p([x_0^+,\infty))}\, t^{-\frac{1}{p}}.
\end{equation*}
This, by a simple calculation, yields constant $C_r$ in $\eqref{c6.83}$.

\smallskip
Up to now, we have proved Theorem $\ref{the:c6.3}$.
\end{proof}

If $\mathcal{K}_0\neq \varnothing$, define $L(\mathcal{K}_0)$ as follows:
If there exists a bounded interval $(e_n,h_n)\subset\mathcal{K}^c_0$ in $\eqref{c6.1}$,
\begin{equation}\label{c6.91}
L(\mathcal{K}_0):=\sup_n\big\{h_n-e_n\,:\, (e_n,h_n)\subset \mathcal{K}^c_0 \ {\rm is\ a \ bounded \ interval}\big\};
\end{equation}
and, if there exists no bounded interval $(e_n,h_n)\subset\mathcal{K}^c_0$,
we set $L(\mathcal{K}_0)=0$.

\begin{Cor}[Asymptotic behaviors in the $L^\infty$--norm for the case of flux functions of uniform convexity]\label{cor:c6.1}
Suppose that $f''(u)\geq c_0>0$ on $u\in\mathbb{R}$.
Let $u=u(x,t)$ be the entropy solution of the Cauchy problem $\eqref{c1.1}$--$\eqref{ID}$
with initial data function $\varphi(x)\in L^\infty(\mathbb{R})$ satisfying that
$\varphi(x)-\overline{u}_l\in L^p(({-}\infty,0])$
and $\varphi(x)-\underline{u}_r\in L^p([0,\infty))$ with $p\in [1,\infty)$.
Then
\begin{enumerate}
\item [{\rm (i)}] If $\overline{u}_l<\underline{u}_r$,
or $\overline{u}_l=\underline{u}_r$ with $\mathcal{K}_0\neq\varnothing$.
Then $\eqref{c6.25}$ holds, $\eqref{c6.82}$ holds with $C_n=c_0^{-1}$ and $\alpha_n=0$,
and $\eqref{c6.83}$--$\eqref{c6.84}$ hold with
$C_r=C_l=\big(2^{p}c_0^{-1}\big)^{\frac{1}{p+1}}$ and $\alpha=\beta=0.$
Furthermore, if $L(\mathcal{K}_0)$ in $\eqref{c6.91}$ satisfies that $L(\mathcal{K}_0)<\infty$,
then, as $t\rightarrow \infty$,
\begin{equation}\label{c6.92}
|u(x,t)-\tilde{u}(x,t)|
\lessapprox \tilde{C}\,
t^{-\frac{1}{p+1}}
\qquad {\rm uniformly\ in}\ x\ {\rm with}\ (x,t)\in \mathbb{R}\times \mathbb{R}^+,
\end{equation}
where $\tilde{C}$ is given by
\begin{align*}
\tilde{C}=\begin{cases}
\big(2^{p}c_0^{-1}P_r\big)^{\frac{1}{p+1}}
\quad &{\rm if}\ \sup\mathcal{K}_0<\infty,\ \inf\mathcal{K}_0=-\infty,\\[1mm]
\big(2^{p}c_0^{-1}P_l\big)^{\frac{1}{p+1}}
\quad &{\rm if}\ \sup\mathcal{K}_0=\infty,\ \inf\mathcal{K}_0>-\infty,\\[1mm]
\big(2^{p}c_0^{-1}\max\{P_r,P_l\}\big)^{\frac{1}{p+1}}
\quad &{\rm if}\ \sup\mathcal{K}_0<\infty,\ \inf\mathcal{K}_0>-\infty
\end{cases}
\end{align*}
for $P_r=\|\varphi(x)-\underline{u}_r\|^p_{L^p([x_0^+, \infty))}$
and $P_l=\|\varphi(x)-\overline{u}_l\|^p_{L^p(({-}\infty, x_0^-])}$.

\vspace{2pt}
In particular, if $L(\mathcal{K}_0)<\infty$ with $\sup\mathcal{K}_0=\infty$ and $\inf\mathcal{K}_0=-\infty$,
then, as $t\rightarrow\infty$,
\begin{equation}\label{c6.95}
\,\,\,\,\,\,\,\,\,\,\,\,\,\,\,\,\,\,\,\,\,\,\,\,
|u(x,t)-\tilde{u}(x,t)|
\lessapprox
(2c_0)^{-1} L(\mathcal{K}_0)\, t^{-1}
\quad {\rm uniformly\ in}\ x\ {\rm with}\ (x,t)\in \mathbb{R}\times \mathbb{R}^+.
\end{equation}

\item [{\rm (ii)}] If $\overline{u}_l=\underline{u}_r$ with $\mathcal{K}_0=\varnothing$,
then, as $t\rightarrow \infty$,
\begin{equation}\label{c6.93}
|u(x,t)-\underline{u}_r|
\lessapprox \tilde{C}_r\,
t^{-\frac{1}{p+1}}
\qquad {\rm uniformly\ in}\ x\ {\rm with}\ (x,t)\in \mathbb{R}\times \mathbb{R}^+,
\end{equation}
where $\tilde{C}_r=\big(2^p f''(\underline{u}_r)^{-1}\|\varphi(x)-\underline{u}_r\|^p_{L^p(\mathbb{R})}\big)^{\frac{1}{p+1}}$.

\vspace{1pt}
\item [{\rm (iii)}] If $\overline{u}_l>\underline{u}_r$, then,
for any forward generalized characteristic $X(x_0,0): x=x(t)$ for $t\geq 0$ emitting from $x_0$ on the $x$--axis,
as $t\rightarrow \infty$,
\begin{equation}\label{c6.94}
\begin{cases}
\displaystyle|u(x,t)-\overline{u}_l|
\lessapprox \bar{C}_l\,
t^{-\frac{1}{p+1}} \quad &{\rm uniformly \ for\ }\ x<x(t), \\[1mm]
\displaystyle|u(x,t)-\underline{u}_r|
\lessapprox \bar{C}_r\,
t^{-\frac{1}{p+1}}\quad &{\rm uniformly \ for\ }\ x>x(t),
\end{cases}
\end{equation}
where $\bar{C}_l$ and $\bar{C}_r$ are given by
\begin{equation*}
\begin{cases}
\bar{C}_l=\big(2^p f''(\overline{u}_l)^{-1}
\|\varphi(x)-\overline{u}_l\|^p_{L^p(({-}\infty,x_0])}\big)^{\frac{1}{p+1}}, \\[2mm]
\bar{C}_r=\big(2^p f''(\underline{u}_r)^{-1}
\|\varphi(x)-\underline{u}_r\|^p_{L^p([x_0, \infty))}\big)^{\frac{1}{p+1}}.
\end{cases}
\end{equation*}
\end{enumerate}
\noindent
In the above,
$\overline{u}_r$, $\underline{u}_r$, $\overline{u}_l$, and $\underline{u}_l$
are given by $\eqref{c5.7}${\rm ;}
and $\mathcal{K}_0$ is given by $\eqref{c5.6}$.
\end{Cor}

\begin{proof}  We divide the proof into three steps accordingly.

\vspace{1pt}
\noindent
{\bf 1}. Since $\mathcal{K}_0\neq \varnothing$,
from Lemma $\ref{lem:c5.3}$ and the definition of $\tilde{u}=\tilde{u}(x,t)$ in $\eqref{c6.10}$,
$$\tilde{u}(x,t)\equiv {\rm D}_+\bar{\Phi}(e_n)\quad {\rm on}\ \mathcal{H}_{(e_n,h_n)},\qquad
\tilde{u}(x,t)\equiv {\rm D}_\pm\bar{\Phi}(x_0^\pm)\quad {\rm on}\ \mathcal{H}_{\pm\infty},$$
if the related case appears.
From Theorem $\ref{the:c6.1}$, $u(x,t)=\tilde{u}(x,t)$ on $\mathcal{K}$.

Since $f''(u)\geq c_0>0$ on $\mathbb{R}$,
it is direct to check from Theorem $\ref{the:c6.3}$ that
$\eqref{c6.82}$ holds with $C_n=c_0^{-1}$ and $\alpha_n=0$,
and $\eqref{c6.83}$--$\eqref{c6.84}$ hold
with $C_r=C_l=(2^{p}c_0^{-1})^{\frac{1}{p+1}}$ and $\alpha=\beta=0$.

For the case that $L(\mathcal{K}_0)<\infty$,
since $\varphi(x)-\overline{u}_l\in L^p(({-}\infty,0])$
and $\varphi(x)-\underline{u}_r\in L^p([0,\infty))$,
it follows from $\eqref{c6.82}$--$\eqref{c6.84}$ that $\eqref{c6.92}$ holds.

In particular, for the case that $L(\mathcal{K}_0)<\infty$ with $\sup\mathcal{K}_0=\infty$
and $\inf\mathcal{K}_0=-\infty$,
both $\mathcal{H}_{\infty}$ and $\mathcal{H}_{-\infty}$ do not appear.
Since $f''(u)\geq c_0>0$ on $\mathbb{R}$, then $\alpha_n\equiv 0$
and $C_n=f''(c_n)^{-1}\leq c_0^{-1}$ in $\eqref{c6.82}$,
which, by $L(\mathcal{K}_0)<\infty$,
yields $\eqref{c6.95}$.

\vspace{2pt}
\noindent
{\bf 2}. $\eqref{c6.93}$ follows by taking $q=p$ and $\alpha=0$ in $\eqref{c6.85}$.

\vspace{2pt}
\noindent
{\bf 3}. $\eqref{c6.94}$ follows by taking $q=p$ and $\alpha=\beta=0$ in $\eqref{c6.86}$.
\end{proof}

\begin{Rem}
Suppose that $f''(u)\geq c_0>0$ as in {\rm Corollary} $\ref{cor:c6.1}$. Then

\begin{itemize}
\item [(i)] For {\rm Example} $\ref{exa:c5.1}$,
the initial data function $\varphi(x)\in L^\infty(\mathbb{R})$ satisfies that
$\varphi(x)-m$ has compact support, $i.e.$,
${\rm spt} (\varphi(x)-m)\subset [-R,R]$ for some $R>0$.
Then $\mathcal{K}_0\neq \varnothing$.

If, for any $\,x\in [-R,R]$,
$$
\int_0^x(\varphi(\xi)-m)\,{\rm d}\xi\geq \int_0^R(\varphi(\xi)-m)\,{\rm d}\xi=\int_0^{-R}(\varphi(\xi)-m)\,{\rm d}\xi,
$$
then $\sup\mathcal{K}_0=\infty$ and $\inf\mathcal{K}_0=-\infty$ with $L(\mathcal{K}_0)\leq 2R$ so that,
from $\eqref{c6.95}$, the entropy solution $u=u(x,t)$ decays to
$\tilde{u}(x,t)\equiv m$ in $L^\infty(\mathbb{R})$ with decay rate
$(c_0)^{-1}R\,t^{-1}$.

Otherwise,
$\sup\mathcal{K}_0<\infty$ and/or $\inf\mathcal{K}_0<-\infty$ with $L(\mathcal{K}_0)\leq 2R$ so that,
from $\eqref{c6.92}$, the entropy solution $u=u(x,t)$ decays to
$\tilde{u}(x,t)\equiv m$ in $L^\infty(\mathbb{R})$ with decay rate $t^{-\frac{1}{2}}$.

\vspace{1pt}
\item [(ii)] For {\rm Example} $\ref{exa:c5.2}$,
the initial data function $\varphi(x)\in L^\infty(\mathbb{R})$ is a periodic function
with the minimum positive period $p>0$ and its average $m$ over the period.
Thus,
$\sup\mathcal{K}_0=\infty$ and $\inf\mathcal{K}_0=-\infty$ with $L(\mathcal{K}_0)\leq p$ so that,
from $\eqref{c6.95}$, the entropy solution $u=u(x,t)$ decays to
$\tilde{u}(x,t)\equiv m$ in $L^\infty(\mathbb{R})$ with decay rate $(2c_0)^{-1}p\,t^{-1}$.

\vspace{1pt}
\item [(iii)] For {\rm Example} $\ref{exa:c5.3}$,
initial data function $\varphi(x)\in L^\infty(\mathbb{R})$ satisfies that
$\varphi(x)-m\in L^1(\mathbb{R})$, then $\overline{u}_l=\underline{u}_r=m$.
In general, from $\eqref{c6.92}$--$\eqref{c6.93}$,
the entropy solution $u=u(x,t)$ decays to $\tilde{u}(x,t)\equiv m$ in $L^\infty(\mathbb{R})$
with decay rate $t^{-\frac{1}{2}}$.
For the special case that $\mathcal{K}_0\neq\varnothing$ and $L(\mathcal{K}_0)<\infty$ with $\sup\mathcal{K}_0=\infty$ and $\inf\mathcal{K}_0=-\infty$,
it follows from $\eqref{c6.95}$ that the entropy solution $u=u(x,t)$ decays to
$\tilde{u}(x,t)\equiv m$ in $L^\infty(\mathbb{R})$ with the decay rate $t^{-1}$.
\end{itemize}
\end{Rem}

\subsection{Generalized $N$--waves and decay of entropy solutions in the $L^p_{{\rm loc}}$--norm}
We now consider the decay of entropy solutions in $L^p_{{\rm loc}}$--norm.
We first introduce the notion of generalized $N$--waves,
and then prove that entropy solutions decay to the generalized $N$--waves in the $L^p_{{\rm loc}}$--norm,
when the entropy solutions possess at least one divide.

\begin{Def}[Generalized $N$--waves]\label{def:n.1}
Let the initial data function $\varphi(x)\in L^\infty(\mathbb{R})$ satisfies that
$\mathcal{K}_0\neq\varnothing$ for $\mathcal{K}_0$ given by $\eqref{c5.6}$.
The generalized $N$--wave $w=w(x,t)$ is defined in forms by $($see {\rm Fig.} $\ref{figNwavef}$$)$
\begin{equation}\label{n.0}
w(x,t)=\big(f'\big)^{-1}\big(\frac{x-\xi}{t}\big)\qquad {\rm for}\ \xi\in\mathcal{K}_0,
\end{equation}
expressed by
\begin{enumerate}
\item [{\rm(i)}] For the case that $(x,t)\in \mathcal{K}$ as in $\eqref{c6.6}$,
\begin{equation}\label{n.4}
w(x,t)=\big(f'\big)^{-1}\big(\frac{x-x_0}{t}\big)\qquad {\rm if}\ (x,t)\in L_c(x_0)
\end{equation}
for all divides $L_c(x_0)\subset \mathcal{K}$, $i.e.$,
$x_0\in \mathcal{K}_0$ and $c\in \mathcal{D}(x_0)$.

\vspace{1pt}
\item [{\rm(ii)}] If there exists $\mathcal{H}_{(e_n,h_n)}$ as in $\eqref{c6.3}$
with $c_n$ given by $\eqref{c6.2}$,
and $e_n(t):=e_n+tf'(c_n)$ and $h_n(t):=h_n+tf'(c_n)$,
then, for any chosen forward generalized characteristic $x_n(t)$ for $t\geq0$ in $\mathcal{H}_{(e_n,h_n)}$,
\begin{equation}\label{n.1}
w(x,t)=
\begin{cases}
\displaystyle \big(f'\big)^{-1}\big(\frac{x-e_n}{t}\big)\quad &{\rm if}\ e_n(t)<x<x_n(t),\\[4mm]
\displaystyle \big(f'\big)^{-1}\big(\frac{x-h_n}{t}\big)\quad &{\rm if}\ x_n(t)<x<h_n(t).
\end{cases}
\end{equation}

\item [{\rm(iii)}] If $x_0^+=\sup\mathcal{K}_0<\infty$,
then $\mathcal{H}_{\infty}$ as in $\eqref{c6.4}$ exists
with $\underline{u}_r={\rm D}_+\bar{\Phi}(x_0^+)$ and
$x_0^+(t)=x_0^++tf'(\underline{u}_r)$.
For any chosen forward generalized characteristic $x_+(t)$ for $t\geq0$ in $\mathcal{H}_{\infty}$,
\begin{equation}\label{n.2}
w(x,t)=\big(f'\big)^{-1}\big(\frac{x-x_0^+}{t}\big)\qquad {\rm if}\ x_0^+(t)<x<x_+(t).
\end{equation}

\item [{\rm(iv)}] If $x_0^-=\inf\mathcal{K}_0>-\infty$,
then $\mathcal{H}_{-\infty}$ as in $\eqref{c6.5}$ exists
with $\overline{u}_l={\rm D}_-\bar{\Phi}(x_0^-)$ and
$x_0^-(t)=x_0^-+tf'(\overline{u}_l)$.
For any chosen forward generalized characteristic $x_-(t)$ for $t\geq0$ in $\mathcal{H}_{-\infty}$,
\begin{equation}\label{n.3}
w(x,t)=\big(f'\big)^{-1}\big(\frac{x-x_0^-}{t}\big)\qquad {\rm if}\ x_-(t)<x<x_0^-(t).
\end{equation}
\end{enumerate}
In the above,
$\mathcal{D}(x_0)$ and ${\rm D}_{\pm}\bar{\Phi}(x_0)$
are given by $\eqref{c5.0}$ and $\eqref{c5.10}$, respectively.
\end{Def}

\begin{figure}[H]
	\begin{center}
		{\includegraphics[width=0.7\columnwidth]{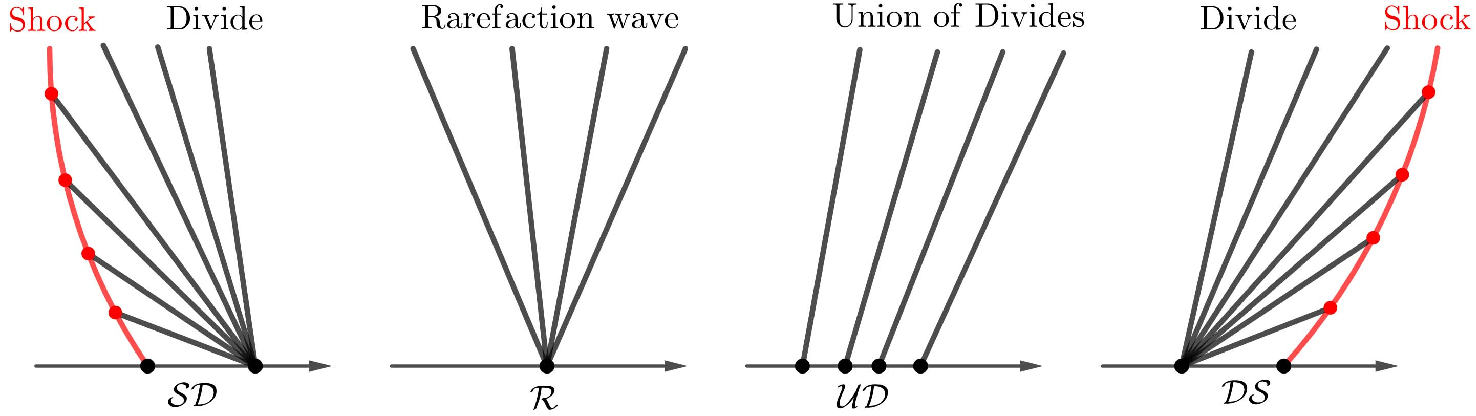}}
 \caption{The four types of generalized $N$-waves:
 $\mathcal{SD}$, $\mathcal{R}$, $\mathcal{UD}$, and $\mathcal{DS}$.
		}\label{figNwavef}
	\end{center}
\end{figure}

From Definition $\ref{def:n.1}$,
the generalized $N$--wave $w=w(x,t)$ is the union of the four types of waves
$\mathcal{SD}$, $\mathcal{R}$, $\mathcal{UD}$, and $\mathcal{DS}$
as showed in {\rm Fig.} $\ref{figNwavef}$,
all of which are bounded by two divides, or by a shock curve and a divide.
From $\eqref{c5.1}$, for any $x_0\in\mathcal{K}_0$,
\begin{equation*}
\mathcal{D}(x_0)=[{\rm D}_-\bar{\Phi}(x_0),{\rm D}_+\bar{\Phi}(x_0)]\subset \mathcal{C}(x_0)
\subset [\overline{{\rm D}}_-\Phi(x_0),\underline{{\rm D}}_+\Phi(x_0)].
\end{equation*}

\begin{Rem}[Detailed description of generalized $N$--waves]
\label{rem:n.1}
From $\eqref{n.0}$, the form of the generalized $N$--wave $w(x,t)$ looks like
a {\rm ``}generalized rarefaction wave{\rm ''} emitting from the point set $\mathcal{K}_0$,
instead of a single point for the case of rarefaction waves.
For the cases of $\mathcal{R}$ and $\mathcal{UD}$,
from $\eqref{kk0}$ and $\eqref{n.4}$,
\begin{equation}\label{n.9}
w(x,t)=\tilde{u}(x,t)=u(x,t) \qquad \ {\rm on} \ \mathcal{K}.
\end{equation}
For the cases of $\mathcal{SD}$ and $\mathcal{DS}$,
we have
\begin{enumerate}
\item [{\rm(i)}] For the case that $\mathcal{SD}\subset\mathcal{H}_{(e_n,h_n)}$,
bounded by a shock $x_n(t)$ in $\mathcal{H}_{(e_n,h_n)}$ from the left
and a divide $L_{c_n}(h_n)$ from the right,
by Theorem $\ref{the:c4.1}$, there are three subcases{\rm :}
\begin{itemize}
\item[(a)] If $\overline{{\rm D}}_-\Phi(h_n)<{\rm D}_-\bar{\Phi}(h_n)$
with $\overline{{\rm D}}_-\Phi(h_n)\notin\mathcal{C}(h_n)$,
there exists a shock $\mathcal{S}$ emitting from $h_n$ and lying on the left of divide $h_n(t)$.
Since any two forward generalized characteristics in $\mathcal{H}_{(e_n,h_n)}$
coincide with each other on $t>t_n$
for sufficiently large $t_n>0$, then
\begin{equation}\label{n.5}
w(x,t)=u(x,t) \qquad \ \mbox{\rm if  $x_n(t)<x<h_n(t)$ and $t\geq t_n$}.
\end{equation}

\item[(b)] If $\overline{{\rm D}}_-\Phi(h_n)<{\rm D}_-\bar{\Phi}(h_n)$
with $\overline{{\rm D}}_-\Phi(h_n)\in\mathcal{C}(h_n)$,
then the forward generalized characteristic of points on the characteristic segment
emitting from $h_n$
with speed $f'(\overline{{\rm D}}_-\Phi(h_n))$
coincides with $x_n(t)$ for large time so that $\eqref{n.5}$ holds.

\item[(c)] If $\overline{{\rm D}}_-\Phi(h_n)={\rm D}_-\bar{\Phi}(h_n)\in\mathcal{C}(h_n)$,
then, for any $x\in(x_n(t),h_n(t))$,
$$x-tf'(u(x{\pm},t))<h_n=x-tf'(w(x,t)),$$
which, by the strictly increasing property of $f'(u)$, implies
\begin{equation*}
w(x,t)<u(x,t) \qquad \ {\rm if} \ x_n(t)<x<h_n(t).
\end{equation*}
\end{itemize}
The case of $\mathcal{SD}\subset\mathcal{H}_{-\infty}$ is the same
via replacing $h_n\in\mathcal{K}_0$ by $x_0^-\in\mathcal{K}_0$, $etc.$.

\item [{\rm(ii)}] Similar to $\mathcal{SD}$,
for $\mathcal{DS}$ with $e_n\in\mathcal{K}_0$ $($or $x_0^+\in\mathcal{K}_0$$)$,
there are three subcases{\rm:}
\begin{itemize}
\item[(d)] If $\underline{{\rm D}}_+\Phi(e_n)>{\rm D}_+\bar{\Phi}(e_n)$
with $\underline{{\rm D}}_+\Phi(e_n)\notin\mathcal{C}(e_n)$,
for sufficiently large $t'_n>0$,
\begin{equation}\label{n.7}
w(x,t)=u(x,t) \qquad \ {\rm if} \ e_n(t)<x<x_n(t), \ t\geq t'_n.
\end{equation}

\item[(e)] If $\underline{{\rm D}}_+\Phi(e_n)>{\rm D}_+\bar{\Phi}(e_n)$
with $\underline{{\rm D}}_+\Phi(e_n)\in\mathcal{C}(e_n)$, then $\eqref{n.7}$ holds.

\item[(f)] If $\underline{{\rm D}}_+\Phi(e_n)={\rm D}_+\bar{\Phi}(e_n)\in\mathcal{C}(h_n)$,
then
\begin{equation*}
w(x,t)>u(x,t) \qquad \ {\rm if} \ e_n(t)<x<x_n(t).
\end{equation*}
\end{itemize}
\end{enumerate}
See also {\rm Fig.} $\ref{figNwaved}$.
\end{Rem}

\begin{figure}[H]
	\begin{center}
		{\includegraphics[width=0.6\columnwidth]{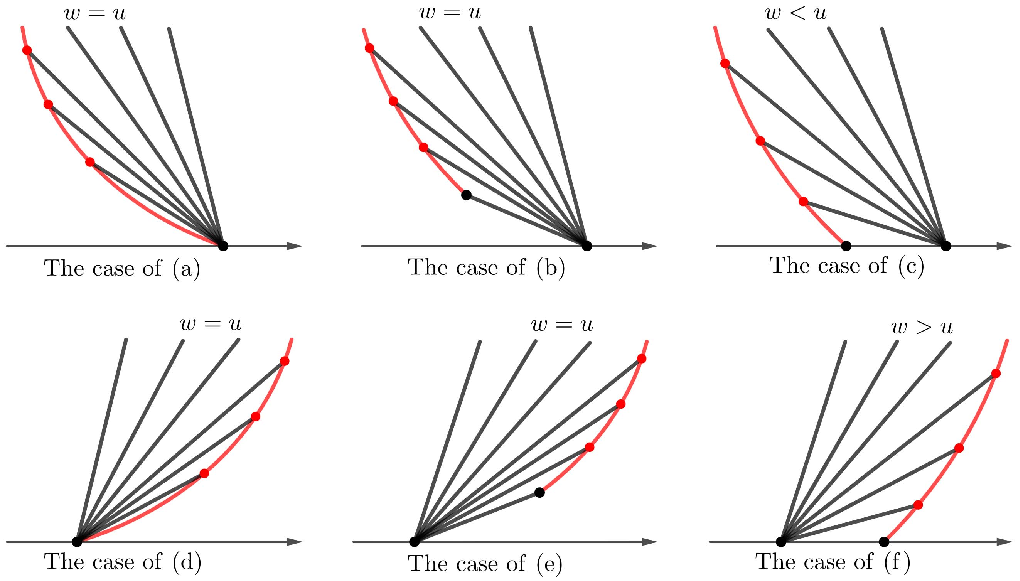}}
 \caption{The detailed description of generalized $N$--waves as in Remark $\ref{rem:n.1}$.
		}\label{figNwaved}
	\end{center}
\end{figure}

According to Lemma $\ref{lem:c6.2}$,
\begin{equation*}
\left\{\begin{array}{lll}
y^-(\infty,0)=e_n,\ &y^+(\infty,0)=h_n \qquad& {\rm for} \ x_n(t)\in \mathcal{H}_{(e_n,h_n)}, \\[1mm]
y^-(\infty,0)=x_0^+,\ &y^+(\infty,0)=\infty \qquad&{\rm for} \ x_+(t)\in \mathcal{H}_{\infty},\\[1mm]
y^-(\infty,0)=-\infty,\ &y^+(\infty,0)=x_0^- \qquad &{\rm for} \ x_-(t)\in \mathcal{H}_{-\infty}.
\end{array}\right.
\end{equation*}

For the case that $x_n(t)\in \mathcal{H}_{(e_n,h_n)}$,
$f'(v^\pm(t))-f'(u^\pm(t))=o(1)\,t^{-1}$ as $t\rightarrow \infty$,
where $v^+(t)$ and $v^-(t)$ are given by $\eqref{c6.17}$ and $\eqref{c6.19}$, respectively.
Then, from $\eqref{c6.29}$ and $\eqref{aa2}$,
if $\eqref{aa1}$ holds with $c=c_n$ and $\alpha=\alpha_n$, then
\begin{align}\label{pmn2}
|v^\pm(t)-c_n|
&\leq|v^\pm(t)-u^\pm(t)|+|u^\pm(t)-c_n|\nonumber\\[1mm]
&\leq o(1)t^{-\frac{1}{1+\alpha_n}}+O(1)t^{-\frac{1}{1+\alpha_n}}
=O(1)\,t^{-\frac{1}{1+\alpha_n}},
\end{align}
which, by $\eqref{c6.17}$ and $\eqref{c6.19}$, implies that,
as $t\rightarrow\infty$,
\begin{equation*}
\begin{cases}
\displaystyle\int_{e_n}^{y^-(t,0)}(\varphi(\xi)-c_n)\,{\rm d}\xi
\leq O(1)\,t^{-\frac{1}{1+\alpha_n}} \big(y^-(t,0)-e_n\big),\\[4mm]
\displaystyle\int_{h_n}^{y^+(t,0)}(\varphi(\xi)-c_n)\,{\rm d}\xi
\leq O(1)\,t^{-\frac{1}{1+\alpha_n}} \big(h_n-y^+(t,0)\big).
\end{cases}
\end{equation*}
Since $\Phi(x)>\bar{\Phi}(x)$ on $(e_n,h_n)$ with
$\Phi(e_n)=\bar{\Phi}(e_n)$ and $\Phi(h_n)=\bar{\Phi}(h_n)$,
then, from $\eqref{c5.13}$,
\begin{equation*}
\begin{cases}
\displaystyle\int_{e_n}^{x}(\varphi(\xi)-c_n)\,{\rm d}\xi=\Phi(x)-\bar{\Phi}(x)>0
\qquad {\rm for\ any}\ x\in(e_n,h_n),\\[4mm]
\displaystyle\int_{h_n}^{x}(\varphi(\xi)-c_n)\,{\rm d}\xi=\Phi(x)-\bar{\Phi}(x)>0
\qquad {\rm for\ any}\ x\in(e_n,h_n).
\end{cases}
\end{equation*}
Therefore, we can assume that, for some $\sigma_n>0$,
\begin{equation}\label{n.12}
\begin{cases}
\displaystyle\int_{e_n}^{x}(\varphi(\xi)-c_n)\,{\rm d}\xi
\geq O(1)\, |x-e_n|^{1+\sigma_n}\qquad {\rm as}\ x\rightarrow e_n{+},\\[4mm]
\displaystyle\int_{h_n}^{x}(\varphi(\xi)-c_n)\,{\rm d}\xi
\geq O(1)\, |x-h_n|^{1+\sigma_n}\qquad {\rm as}\ x\rightarrow h_n{-}.
\end{cases}
\end{equation}

Similarly, for the case that $x_+(t)\in \mathcal{H}_{\infty}$,
if $\eqref{aa1}$ holds with $c=\underline{u}_r$ and $\alpha=\alpha_r$,
and $\varphi(x)-\underline{u}_r\in L^p([0,\infty))$, then
\begin{equation*}
0\leq\int_{x_0^+}^{y^-(t,0)}(\varphi(\xi)-\underline{u}_r)\,{\rm d}\xi
\leq O(1)\,t^{-\frac{1}{p+1+\alpha_r}} \big(y^-(t,0)-x_0^+\big),
\end{equation*}
so that we can assume that, for some $\sigma_r>0$,
\begin{equation}\label{n.13a}
\int_{x_0^+}^{x}(\varphi(\xi)-\underline{u}_r)\,{\rm d}\xi
\geq O(1)\, |x-x_0^+|^{1+\sigma_r}\qquad {\rm as}\ x\rightarrow x_0^+{+};
\end{equation}
and, for the case that $x_-(t)\in \mathcal{H}_{-\infty}$,
if $\eqref{aa1}$ holds with $c=\overline{u}_l$ and $\alpha=\alpha_l$, and $\varphi(x)-\overline{u}_l\in L^p(({-}\infty,0])$,
then
\begin{equation*}
0\leq\int_{x_0^-}^{y^+(t,0)}(\varphi(\xi)-\overline{u}_l)\,{\rm d}\xi
\leq O(1)\,t^{-\frac{1}{p+1+\alpha_l}} \big(x_0^--y^+(t,0)\big),
\end{equation*}
so that we can assume that, for some $\sigma_l>0$,
\begin{equation}\label{n.14a}
\int_{x_0^-}^{x}(\varphi(\xi)-\overline{u}_l)\,{\rm d}\xi
\geq O(1)\, |x-x_0^-|^{1+\sigma_l}\qquad {\rm as}\ x\rightarrow x_0^-{-}.
\end{equation}

From $\eqref{n.9}$, $w(x,t)=u(x,t)$ for the case of $\mathcal{R}$ and $\mathcal{UD}$.
We now give the $L^p_{{\rm loc}}$--norm decay of entropy solutions
to the generalized $N$--waves for the cases of $\mathcal{SD}$ and $\mathcal{DS}$.

\begin{Lem}[Generalized $N$--waves as asymptotic profiles in the $L^p_{\rm loc}$--norm]\label{lem:n.1}
Suppose that $\overline{u}_l<\underline{u}_r$,
or $\overline{u}_l=\underline{u}_r$ with $\mathcal{K}_0\neq \varnothing$.
Let $u=u(x,t)$ be the entropy solution of the Cauchy problem $\eqref{c1.1}$--$\eqref{ID}$
with initial data function $\varphi(x)\in L^{\infty}(\mathbb{R})$ satisfying that
$\varphi(x)-\underline{u}_r\in L^p([0,\infty))$
and $\varphi(x)-\overline{u}_l\in L^p(({-}\infty,0])$ for $p\in[1,\infty)$,
and let the generalized $N$--wave $w=w(x,t)$ be given by {\rm Definition} $\ref{def:n.1}$.
Then $u(x,t)\equiv w(x,t)$ on $\mathcal{K}$, 
and the following statements hold{\rm :}

\begin{enumerate}
\item[{\rm(i)}] If there exists $\mathcal{H}_{(e_n,h_n)}$ as in $\eqref{c6.3}$,
then, for any $q>0$,
as $t\rightarrow\infty$,
\begin{equation}\label{n.15}
\|f'(u(\cdot,t))-f'(w(\cdot,t))\|_{L^q([e_n(t),\,h_n(t)])}
\lessapprox
\Big(\frac{h_n-e_n}{2}\Big)^{\frac{1}{q}} \, \varepsilon_{n,q}(t)\, t^{-1}.
\end{equation}
Furthermore, if $\eqref{n.12}$ holds for some $\sigma_n>0$, then,
as $t\rightarrow\infty$,
\begin{equation}\label{n.15a}
\|f'(u(\cdot,t))-f'(w(\cdot,t))\|_{L^q([e_n(t),\,h_n(t)])}
\lesssim t^{-1-\frac{1}{\sigma_n(1+\alpha_n)}}.
\end{equation}

\item[{\rm(ii)}] If $x_0^+=\sup\mathcal{K}_0<\infty$,
 for any forward generalized characteristic $x_+(t)$ in $\mathcal{H}_{\infty}$ as in $\eqref{c6.4}$ and for any $q>\frac{p}{p+1+\alpha_r}$,
 as $t\rightarrow\infty$,
\begin{equation}\label{n.16}
\|f'(u(\cdot,t))-f'(w(\cdot,t))\|_{L^q([x_0^+(t),\,x_+(t)])}
\lesssim
\varepsilon_r(t)\, t^{-1+\frac{p/q}{p+1+\alpha_r}}.
\end{equation}
Furthermore, if $\eqref{n.13a}$ holds for some $\sigma_r>0$, then, for any $q>\frac{p}{p+1+\alpha_r}$,
as $t\rightarrow\infty$,
\begin{equation}\label{n.16a}
\|f'(u(\cdot,t))-f'(w(\cdot,t))\|_{L^q([x_0^+(t),\,x_+(t)])}
\lesssim
t^{-1+\frac{p/q}{p+1+\alpha_r}-\frac{1/\sigma_r}{p+1+\alpha_r}}.
\end{equation}

\item[{\rm(iii)}] If $x_0^-=\inf\mathcal{K}_0>-\infty$,
 for any forward generalized characteristic $x_-(t)$ in $\mathcal{H}_{-\infty}$
 as in $\eqref{c6.5}$ and for any $q>\frac{p}{p+1+\alpha_l}$,
 as $t\rightarrow\infty$,
\begin{equation}\label{n.17}
\|f'(u(\cdot,t))-f'(w(\cdot,t))\|_{L^q([x_-(t),\,x_0^-(t)])}
\lesssim
\varepsilon_l(t)\,
t^{-1+\frac{p/q}{p+1+\alpha_l}}.
\end{equation}
Furthermore, if $\eqref{n.14a}$ holds for some $\sigma_l>0$, then, for any $q>\frac{p}{p+1+\alpha_l}$,
as $t\rightarrow\infty$,
\begin{equation}\label{n.17a}
\|f'(u(\cdot,t))-f'(w(\cdot,t))\|_{L^q([x_-(t),\,x_0^-(t)])}
\lesssim
t^{-1+\frac{p/q}{p+1+\alpha_l}-\frac{1/\sigma_l}{p+1+\alpha_l}}.
\end{equation}
\end{enumerate}
In the above,
$\overline{u}_r$, $\underline{u}_r$, $\overline{u}_l$, and $\underline{u}_l$
are given by $\eqref{c5.7}${\rm ;}
$\mathcal{K}_0$ is given by $\eqref{c5.6}${\rm ;}
$\alpha_n$, $\alpha_r$, and $\alpha_l$ are given by $\eqref{aa1}$
with $c=c_n$, $c=\underline{u}_r$, and $c=\overline{u}_l$, respectively{\rm ;}
and
$\varepsilon_{n,q}(t):=\big((\varepsilon^-_n(t))^q+(\varepsilon^+_n(t))^q\big)^{\frac{1}{q}}$ with
$\varepsilon^-_n(t):=y^-_n(t,0)-e_n$ and $\varepsilon^+_n(t):=h_n-y^+_n(t,0)$,
$\varepsilon_r(t):=y^-_r(t,0)-x_0^+$, and $\varepsilon_l(t):=x_0^--y^+_l(t,0)$,
for $y^{\pm}_n(t,0)$, $y^-_r(t,0)$, and $y^+_l(t,0)$ given by $\eqref{c6.15}$ with
the forward generalized characteristics $x_n(t)$ in $\mathcal{H}_{(e_n,h_n)}$,
$x_+(t)$ in $\mathcal{H}_{\infty}$,
and $x_-(t)$ in $\mathcal{H}_{-\infty}$, respectively.
\end{Lem}

\begin{proof}
Denote $y(x,t):=x-tf'(u(x,t))$ for any $(x,t)\in \mathbb{R}\times\mathbb{R}^+$.
From $\eqref{kk0}$ and $\eqref{n.4}$, 
$u(x,t)\equiv w(x,t)$ on $\mathcal{K}$.
The remaining proof is divided into three steps accordingly.

\smallskip
\noindent
{\bf 1}. By Lemma $\ref{lem:c6.2}$,
for any forward generalized characteristic $x_n(t)$ in $\mathcal{H}_{(e_n,h_n)}$,
as $t\rightarrow\infty$,
\begin{equation*}
\varepsilon^-_n(t)=y^-_n(t,0)-e_n\ \longrightarrow 0,
\qquad\varepsilon^+_n(t)=h_n-y^+_n(t,0)\ \longrightarrow 0.
\end{equation*}
From Lemma $\ref{lem:c2.4}$,
by $\eqref{n.1}$,
for any $(x,t)\in\mathcal{H}_{(e_n,h_n)}$ with $e_n(t)<x<x_n(t)$,
\begin{equation*}
|f'(u(x,t))-f'(w(x,t))|
=\Big|\frac{y(x,t)-e_n}{t}\Big|
\leq\Big|\frac{y^-_n(t,0)-e_n}{t}\Big|
=\frac{\varepsilon^-_n(t)}{t};
\end{equation*}
and, for any $(x,t)\in\mathcal{H}_{(e_n,h_n)}$ with $x_n(t)<x<h_n(t)$,
\begin{equation*}
|f'(u(x,t))-f'(w(x,t))|=\Big|\frac{y(x,t)-h_n}{t}\Big|
\leq\Big|\frac{y^+_n(t,0)-h_n}{t}\Big|=\frac{\varepsilon^+_n(t)}{t}.
\end{equation*}
From $\eqref{c6.30}$, $x_n(t)-tf'(c_n)=\frac{h_n+e_n}{2}+o(1)$ as $t\rightarrow\infty$.
Then
\begin{align*}
&\ \|f'(u(\cdot,t))-f'(w(\cdot,t))\|^q_{L^q([e_n(t),h_n(t)])}\\
&=\int_{e_n(t)}^{x_n(t)}|f'(u(x,t))-f'(w(x,t))|^q\,{\rm d}x
+\int_{x_n(t)}^{h_n(t)}|f'(u(x,t))-f'(w(x,t))|^q\,{\rm d}x\\[1mm]
&\leq \Big(\frac{\varepsilon^-_n(t)}{t}\Big)^q\,
\Big(\frac{h_n-e_n}{2}+o(1)\Big)+\Big(\frac{\varepsilon^+_n(t)}{t}\Big)^q\,
\Big(\frac{h_n-e_n}{2}+o(1)\Big)\\[1mm]
&=(1+o(1))\frac{h_n-e_n}{2}\big((\varepsilon^-_n(t))^q+(\varepsilon^+_n(t))^q\big)\, t^{-q}\\[1mm]
&=(1+o(1))\frac{h_n-e_n}{2}(\varepsilon_{n,q}(t))^q\, t^{-q},
\end{align*}
which yields $\eqref{n.15}$.

Furthermore, if $\eqref{n.12}$ holds for some $\sigma_n>0$,
it follows from $\eqref{c6.19}$ and $\eqref{pmn2}$ that
\begin{align}\label{n.190}
O(1)(\varepsilon^-_n(t))^{1+\sigma_n}
\leq\int_{e_n}^{y^-_n(t,0)}(\varphi(\xi)-c_n)\,{\rm d}\xi
\leq \big(v^-(t)-c_n\big)\,\varepsilon^-_n(t)
\leq O(1)\,t^{-\frac{1}{1+\alpha_n}}\, \varepsilon^-_n(t),
\end{align}
which implies
\begin{equation*}
0\leq\varepsilon^-_n(t)=y^-_n(t,0)-e_n\leq O(1)\, t^{-\frac{1}{\sigma_n(1+\alpha_n)}}.
\end{equation*}
Similarly, we have
\begin{equation*}
0\leq\varepsilon^+_n(t)=h_n-y^+_n(t,0)\leq O(1)\, t^{-\frac{1}{\sigma_n(1+\alpha_n)}}.
\end{equation*}
Therefore, $\eqref{n.15a}$ follows from $\eqref{n.15}$.

\smallskip
\noindent
{\bf 2}. Suppose $x_0^+=\sup\mathcal{K}_0<\infty$.
By Lemma $\ref{lem:c6.2}$, for any forward generalized characteristic $x_+(t)$ in $\mathcal{H}_{\infty}$,
\begin{equation*}
\varepsilon_r(t)=y^-_r(t,0)-x_0^+\ \rightarrow 0\qquad\, \mbox{as $t\rightarrow\infty$}.
\end{equation*}
From Lemma $\ref{lem:c2.4}$, by $\eqref{n.2}$,
for any $(x,t)\in\mathcal{H}_{\infty}$ with $x_0^+(t)<x<x_+(t)$,
\begin{equation*}
|f'(u(x,t))-f'(w(x,t))|
=\Big|\frac{y(x,t)-x_0^+}{t}\Big|
\leq\Big|\frac{y^-_r(t,0)-x_0^+}{t}\Big|
=\frac{\varepsilon_r(t)}{t}.
\end{equation*}
Since $\varphi(x)-\underline{u}_r\in L^p([0,\infty))$ for $p\in[1,\infty)$,
then, from $\eqref{c6.83}$ and $\eqref{aa2}$,
$$
0<x_+(t)-x_0^+(t)
\leq y^-_r(t,0)-x_0^++t\big|f'(u^-(t))-f'(\underline{u}_r)\big|
\leq O(1)\, t^{\frac{p}{p+1+\alpha_r}},$$
so that,
as $t\rightarrow \infty$,
\begin{align*}
\|f'(u(\cdot,t))-f'(w(\cdot,t))\|_{L^q([x_0^+(t),x_+(t)])}
&=\Big(\int_{x_0^+(t)}^{x_+(t)}|f'(u(x,t))-f'(w(x,t))|^q\,{\rm d}x\Big)^{\frac{1}{q}}\\[1mm]
&\leq \frac{\varepsilon_r(t)}{t}\, \big(x_+(t)-x_0^+(t)\big)^{\frac{1}{q}}
\leq O(1)\, \varepsilon_r(t)\,t^{-1+\frac{p/q}{p+1+\alpha_r}},
\end{align*}
which yields $\eqref{n.16}$.
Furthermore, if $\eqref{n.13a}$ holds for some $\sigma_r>0$,
similar to $\eqref{n.190}$,
 \begin{equation*}
\varepsilon_r(t)=y^-_r(t,0)-x_0^+\leq O(1)\, t^{-\frac{1/\sigma_r}{p+1+\alpha_r}},
\end{equation*}
which, by taking into $\eqref{n.16}$, implies $\eqref{n.16a}$.

\smallskip
\noindent
{\bf 3}.
Suppose $x_0^-=\inf\mathcal{K}_0>-\infty$.
By Lemma $\ref{lem:c6.2}$,
for any forward generalized characteristic $x_-(t)$ in $\mathcal{H}_{-\infty}$,
\begin{equation}\label{n.22a}
\varepsilon_l(t)=x_0^--y^+_l(t,0)\ \rightarrow 0\qquad\, \mbox{as $t\rightarrow\infty$}.
\end{equation}
From Lemma $\ref{lem:c2.4}$, by $\eqref{n.3}$,
for any $(x,t)\in\mathcal{H}_{-\infty}$ with $x_-(t)<x<x_0^-(t)$,
\begin{equation}\label{n.22}
|f'(u(x,t))-f'(w(x,t))|=\Big|\frac{y(x,t)-x_0^-}{t}\Big|\leq\Big|\frac{y^+_l(t,0)-x_0^-}{t}\Big|=\frac{\varepsilon_l(t)}{t}.
\end{equation}
Since $\varphi(x)-\overline{u}_l\in L^p(({-}\infty,0])$ for $p\in[1,\infty)$,
from $\eqref{c6.84}$, $\eqref{n.22a}$, and $\eqref{aa2}$,
$$
0< x_0^-(t)-x_-(t)\leq x_0^--y^+_l(t,0)+t\big|f'(u^+(t))-f'(\overline{u}_l)\big|\leq O(1)\, t^{\frac{p}{p+1+\alpha_l}},
$$
so that, by $\eqref{n.22}$, as $t\rightarrow \infty$,
\begin{align*}
\|f'(u(\cdot,t))-f'(w(\cdot,t))\|_{L^q([x_-(t),x_0^-(t)])}
&=\Big(\int_{x_-(t)}^{x_0^-(t)}|f'(u(x,t))-
f'(w(x,t))|^q\,{\rm d}x\Big)^{\frac{1}{q}}\\[1mm]
&\leq \frac{\varepsilon_l(t)}{t}\, \big(x_0^-(t)-x_-(t)\big)^{\frac{1}{q}}
\leq O(1)\, \varepsilon_l(t)\,
t^{-1+\frac{p/q}{p+1+\alpha_l}},
\end{align*}
which yields $\eqref{n.17}$.
Furthermore, if $\eqref{n.14a}$ holds for some $\sigma_r>0$,
similar to $\eqref{n.190}$,
\begin{equation*}
\varepsilon_l(t)=x_0^--y^+_l(t,0)\leq O(1)\, t^{-\frac{1/\sigma_l}{p+1+\alpha_l}},
\end{equation*}
which, by taking into $\eqref{n.17}$, implies $\eqref{n.17a}$.
\end{proof}

\begin{The}[Decay rates in the $L^p_{\rm loc}$--norm]\label{the:n.1}
Suppose that 
$\overline{u}_l<\underline{u}_r$,
or $\overline{u}_l=\underline{u}_r$ with $\mathcal{K}_0\neq \varnothing$.
Let $u=u(x,t)$ be the entropy solution of the Cauchy problem $\eqref{c1.1}$--$\eqref{ID}$
with initial data function $\varphi(x)\in L^{\infty}(\mathbb{R})$ satisfying that
$\varphi(x)-\underline{u}_r\in L^p([0,\infty))$
and $\varphi(x)-\overline{u}_l\in L^p(({-}\infty,0])$ for $p\in[1,\infty)$,
and let the generalized $N$--wave
$w=w(x,t)$ be given by {\rm Definition} $\ref{def:n.1}$. 
Then the following statements hold{\rm :}
\begin{enumerate}
\item[{\rm(i)}] For any two divides $e(t)$ and $h(t)$ with $e(t)< h(t)$ and for any $q>0$,
as $t\rightarrow\infty$,
\begin{equation}\label{n.24}
\|f'(u(\cdot,t))-f'(w(\cdot,t))\|_{L^q([e(t),\,h(t)])}\leq o(1)\, t^{-1}.
\end{equation}

\item[{\rm(ii)}] If $\inf\mathcal{K}_0>-\infty$ and/or $\sup\mathcal{K}_0<\infty$,
for any two forward generalized characteristics $x_1(t)\in \mathcal{H}_{-\infty}$
and/or $x_2(t)\in \mathcal{H}_{\infty}$ with $x_1(t)< x_2(t)$
and for any $q>\frac{p}{p+1+\alpha}$,
as $t\rightarrow\infty$,
\begin{equation}\label{n.25}
\|f'(u(\cdot,t))-f'(w(\cdot,t))\|_{L^q([x_1(t),\,x_2(t)])}\leq o(1)\,
t^{-1+\frac{p/q}{p+1+\alpha}},
\end{equation}
In particular, if $x_0^+=\sup\mathcal{K}_0<\infty$
with ${\rm D}_+\bar{\Phi}(x_0^+)<\underline{{\rm D}}_+\Phi(x_0^+)$
and/or $x_0^-=\inf\mathcal{K}_0>-\infty$
with $\overline{{\rm D}}_-\Phi(x_0^-)<{\rm D}_-\bar{\Phi}(x_0^-)$,
then, for any $q>0$,
as $t\rightarrow\infty$,
\begin{equation}\label{n.25a}
\|f'(u(\cdot,t))-f'(w(\cdot,t))\|_{L^q([x_1(t),\,x_2(t)])}\leq o(1)\, t^{-1}.
\end{equation}
\end{enumerate}
In the above,
$\overline{u}_r$, $\underline{u}_r$, $\overline{u}_l$, and $\underline{u}_l$
are given by $\eqref{c5.7}${\rm ;}
 $\mathcal{K}_0$ and $\mathcal{H}_{\pm\infty}$
are given by $\eqref{c5.6}$ and $\eqref{c6.4}$--$\eqref{c6.5}$, respectively{\rm ;}
$\underline{{\rm D}}_+\Phi(x)$ and $\overline{{\rm D}}_-\Phi(x)$,
and ${\rm D}_{\pm}\bar{\Phi}(x_0)$
are given by $\eqref{c4.1}$ and $\eqref{c5.10}$, respectively{\rm ;}
and $\alpha$ is given by
\begin{equation*}
\alpha=\left\{\begin{array}{lll}
\alpha_l \quad&{\rm if}\ x_1(t)\in \mathcal{H}_{-\infty},&x_2(t)\in \mathcal{H}_{(e_n,\,h_n)},\\[1mm]
\alpha_r \quad&{\rm if}\ x_1(t)\in \mathcal{H}_{(e_n,\,h_n)},&x_2(t)\in \mathcal{H}_{\infty},\\[1mm]
\min\{\alpha_l,\alpha_r\} \quad &{\rm if}\ x_1(t)\in \mathcal{H}_{-\infty},&x_2(t)\in \mathcal{H}_{\infty}
\end{array}\right.
\end{equation*}
for $\alpha_r$ and $\alpha_l$ given by $\eqref{aa1}$ with $c=\underline{u}_r$ and $c=\overline{u}_l$, respectively.
See also {\rm Fig.} $\ref{figNwavth}$.
\end{The}

\begin{proof}
It is direct to see that $\eqref{n.24}$ follows from $\eqref{n.15}$ and $u(x,t)=w(x,t)$ on $\mathcal{K}$ in $\eqref{n.9}$.

Furthermore,
$\eqref{n.25}$ follows by combining $\eqref{n.15}$ and $\eqref{n.16}$ with $\eqref{n.17}$.

By Remark $\ref{rem:n.1}$, if $x_0^+=\sup\mathcal{K}_0<\infty$
with ${\rm D}_+\bar{\Phi}(x_0^+)<\underline{{\rm D}}_+\Phi(x_0^+)$,
there exists sufficiently large $t_+>0$ such that
$u(x,t)=w(x,t)$ for $x_0^+(t)<x<x_2(t)$ with $t>t_+$.
If $x_0^-=\inf\mathcal{K}_0>-\infty$
with $\overline{{\rm D}}_-\Phi(x_0^-)<{\rm D}_-\bar{\Phi}(x_0^-)$,
there exists sufficiently large $t_->0$ such that
$u(x,t)=w(x,t)$ for $x_1(t)<x<x_0^-(t)$ with $t>t_-.$
Therefore, it follows from $\eqref{n.24}$ that $\eqref{n.25a}$ holds.
\end{proof}

For the initial data function $\varphi(x)\in L^{\infty}(\mathbb{R})$ satisfying that
\begin{equation}\label{ssxx}
\varphi(x)\equiv u_l\quad {\rm if}\ x<s_-,\qquad\,\,\varphi(x)\equiv u_r\quad {\rm if}\ x>s_+
\end{equation}
with $s_-<s_+$, we have
$$
\Phi(x)-\Phi(s_-)\equiv u_l(x-s_-)\quad {\rm if}\ x\leq s_-, \qquad\,\,
\Phi(x)-\Phi(s_+)\equiv u_r(x-s_+)\quad {\rm if}\ x\geq s_+.
$$
Thus, it follows from Lemmas $\ref{lem:c5.1}$--$\ref{lem:c5.2}$ that,
$\bar{\Phi}(x)>-\infty$ with $\mathcal{K}_0\neq \varnothing$ if $u_l\leq u_r$,
and $\bar{\Phi}(x)=-\infty$ with $\mathcal{K}_0=\varnothing$ if $u_l> u_r$.

Furthermore, from Lemma $\ref{lem:c5.3}$, if $u_l\leq u_r$,
\begin{equation}\label{ssxx1}
\left\{\begin{array}{llll}
({-}\infty,s_-]\subset \mathcal{K}_0 &{\rm if}\ s_-\in \mathcal{K}_0,
   \quad &({-}\infty,s_-]\cap \mathcal{K}_0=\varnothing &{\rm if}\ s_-\notin \mathcal{K}_0,\\[2mm]
[s_+,\infty)\subset \mathcal{K}_0 &{\rm if}\ s_+\in \mathcal{K}_0,
  \quad &[s_+,\infty)\cap \mathcal{K}_0=\varnothing &{\rm if}\ s_+\notin \mathcal{K}_0.
\end{array}\right.
\end{equation}

\begin{Cor}[Decay in the $L^p$--norm for some special initial data]
Let $u=u(x,t)$ be the entropy solution of the Cauchy problem $\eqref{c1.1}$--$\eqref{ID}$
for the initial data function $\varphi(x)\in L^{\infty}(\mathbb{R})$ satisfying
$\eqref{ssxx}$ with $u_l\leq u_r$.
\begin{enumerate}
\item[{\rm (i)}] If $s_+\in \mathcal{K}_0$ and $s_-\in \mathcal{K}_0$,
then, for any $q>0$,
as $t\rightarrow\infty$,
\begin{equation}\label{n.26}
\|f'(u(\cdot,t))-f'(w(\cdot,t))\|_{L^q(\mathbb{R})}\leq o(1)\, t^{-1}.
\end{equation}

\item[{\rm (ii)}] If $s_-\notin \mathcal{K}_0$ and/or $s_+\notin \mathcal{K}_0$,
for any chosen forward generalized characteristics $x_\pm(t)$ emitting from $s_\pm$,
define
\begin{equation*}
\bar{w}(x,t):=\begin{cases}
u_l\quad &{\rm if}\ x<x_-(t),\\
w(x,t)\quad &{\rm if}\ x_-(t)<x<x_+(t),\\
u_r\quad &{\rm if}\ x>x_+(t).
\end{cases}
\end{equation*}
Then, for any $q>\frac{p}{p+1+\alpha}$,
as $t\rightarrow\infty$,
\begin{equation}\label{n.27}
\|f'(u(\cdot,t))-f'(\bar{w}(\cdot,t))\|_{L^q(\mathbb{R})}\leq o(1)\,
t^{-1+\frac{p/q}{p+1+\alpha}}.
\end{equation}
In particular, if $s_+\notin \mathcal{K}_0$
with ${\rm D}_+\bar{\Phi}(x_0^+)<\underline{{\rm D}}_+\Phi(x_0^+)$
and/or $s_-\notin \mathcal{K}_0$
with $\overline{{\rm D}}_-\Phi(x_0^-)<{\rm D}_-\bar{\Phi}(x_0^-)$,
then, for any $q>0$,
as $t\rightarrow\infty$,
\begin{equation}\label{n.28}
\|f'(u(\cdot,t))-f'(\bar{w}(\cdot,t))\|_{L^q(\mathbb{R})}\leq o(1)\, t^{-1}.
\end{equation}
\end{enumerate}
In the above,
$\mathcal{K}_0$ is given by $\eqref{c5.6}${\rm ;}
$\underline{{\rm D}}_+\Phi(x)$ and $\overline{{\rm D}}_-\Phi(x)$,
and ${\rm D}_{\pm}\bar{\Phi}(x_0)$
are given by $\eqref{c4.1}$ and $\eqref{c5.10}$, respectively{\rm ;}
and $\alpha$ is given by
\begin{equation*}
\alpha=\left\{\begin{array}{lll}
\alpha_l \quad&{\rm if}\ s_-\notin \mathcal{K}_0,&s_+\in \mathcal{K}_0,\\
\alpha_r \quad&{\rm if}\ s_-\in \mathcal{K}_0,&s_+\notin \mathcal{K}_0,\\
\min\{\alpha_l,\alpha_r\} \quad &{\rm if}\ s_-\notin \mathcal{K}_0,&s_+\notin \mathcal{K}_0
\end{array}\right.
\end{equation*}
for $\alpha_r$ and $\alpha_l$ given by $\eqref{aa1}$ with $c=\underline{u}_r$ and $c=\overline{u}_l$, respectively.
\end{Cor}

\begin{proof}
If $s_+\in \mathcal{K}_0$ and $s_-\in \mathcal{K}_0$,
by $\eqref{ssxx1}$, $u(x,t)=w(x,t)\equiv u_r$ on $\{x>s_++tf'(u_r)\}\subset \mathcal{K}$ and
$u(x,t)=w(x,t)\equiv u_l$ on $\{x<s_-+tf'(u_l)\}\subset \mathcal{K}$ so that
$\eqref{n.26}$ follows by $\eqref{n.24}$.

\smallskip
If $s_+\notin \mathcal{K}_0$, from $\eqref{ssxx1}$, we see that $x_0^+<s_+$ so that,
by Lemma $\ref{lem:c6.2}$,
$$
y^+(t,0)=x_+(t)-tf'(u(x_+(t){+},t))\ \rightarrow \infty \qquad\,\, {\rm as}\ t\rightarrow \infty,
$$
which means that there exists sufficiently large $t_+$ such that $y^+(t,0)>s_+$ for $t>t_+$.
From $\eqref{ssxx}$, by Lemmas $\ref{lem:c2.4}$--$\ref{lem:c2.5}$,
$u(x,t)\equiv u_r$ for $x>x_+(t)$ with $t>t_+.$
Similarly, if $s_-\notin \mathcal{K}_0$, then
$u(x,t)\equiv u_l$ for $x<x_-(t)$ with $t>t_-.$
Thus, we obtain that, when $t>\max\{t_+,t_-\}$,
$u(x,t)=\bar{w}(x,t)$ for $x>x_+(t)\ {\rm or}\ x<x_-(t),$
which, by combining with $\eqref{n.25}$, yields $\eqref{n.27}$.

\smallskip
Finally, it follows from $\eqref{n.25a}$ that $\eqref{n.28}$ holds.
\end{proof}

\begin{Rem}
The generalized $N$--waves provide a
good approximation of entropy solutions
which possess at least one divide.
In detail,
the generalized $N$--wave $w(x,t)$ exactly equals to the entropy solution $u(x,t)$ for large time
in five of the seven cases in {\rm Remark} $\ref{rem:n.1}$.
For the remaining two cases,
according to {\rm Lemma} $\ref{lem:n.1}$ and {\rm Theorem} $\ref{the:n.1}$,
{\it the entropy solutions can decay to generalized $N$--waves with rates
{\rm(}even far more{\rm)} greater than $t^{-1}$}.
In particular, from {\rm Remark} $\ref{rem:n.1}$,
for the period initial data function satisfying that
there emits a rarefaction wave initially at any point of $\mathcal{K}_0$,
the generalized $N$--wave exactly equals to the entropy solution for large enough time.
As an example, considering the Burgers' equation $u_t+(\frac{1}{2}u^2)_x=0$
with the initial data function $\varphi(x)=(|\sin x|)'$,
it is direct to see that, when $t>\frac{\pi}{2}$, 
the entropy solution equals to the generalized $N$--wave, 
for which the forward generalized characteristics $x_n(t)$ are chosen to emit from points $x_0=n\pi$ for $n\in \mathbb{Z}$ with the same initial speed $1$.
\end{Rem}

\section{
Extension of the New Formula
to More General 
Hyperbolic Conservation Laws}
The new formula for entropy solutions can be generalized to the Cauchy problem
for more general scalar hyperbolic conservation law:
\begin{eqnarray}
&& U(u)_t+F(u)_x=0\qquad\,\,
{\rm for}\ (x,t) \in \mathbb{R} \times \mathbb{R}^+,\label{c7.1}\\[1mm]
&&u|_{t=0}=\varphi(x), \label{c7.1b}
\end{eqnarray}
where the flux pair $(U(u), F(u))\in C^2(\mathbb{R})$ satisfies
\begin{equation}\label{c7.1c}
U'(u)>0 \,\,\,\,\, {\rm on}\ u\in \mathbb{R},
\qquad\,\, H(u):=\frac{F'(u)}{U'(u)} \,\, {\rm is \ strictly\ increasing},
\end{equation}
and the initial data function $\varphi(x) \in L_{{\rm loc}}^1(\mathbb{R})$ satisfies
\begin{equation}\label{c7.19}
\mathop{\overline{\lim}}\limits_{s\rightarrow \pm \infty}\frac{\varphi (x-t H(s))}{s}<1
\qquad {\rm locally\ uniform\ in\ } (x,t).
\end{equation}
Condition \eqref{c7.1c} for $H(u)$ is equivalent to
\begin{equation}\label{c1.2b}
H'(u)\ge 0 \,\,\,\,\mbox{on $u \in \mathbb{R}$},\qquad\,\,\, \mathcal{L}\{u\,:\, H'(u)=0\}=0,
\end{equation}
where $\mathcal{L}$ is the Lebesgue measure, which means that $H(u)$ is not uniformly increasing.

The new formula for entropy solutions of the Cauchy problem \eqref{c7.1}--\eqref{c7.1b}
can be introduced as follows:
For any fixed $(x,t)\in \mathbb{R}\times \mathbb{R}^+$,
\begin{equation}\label{c7.2}
E(u;x,t)=t\int_0^u H'(s)\big(U(\varphi(x-tH(s)))-U(s)\big)\,{\rm d}s.
\end{equation}
From $\eqref{c7.19}$, for any $(x,t)\in \mathbb{R}\times \mathbb{R}^+$,
there exist $s^{\pm}(x,t)$ such that
\begin{equation*}
\varphi (x-tH(u))<u\quad {\rm if}\ u>s^+(x,t),\qquad\,\,\,
\varphi (x-tH(u))>u\quad {\rm if}\ u<s^-(x,t).
\end{equation*}
This, by $\eqref{c7.1c}$, implies that, if $u>s^+(x,t)$, then
$$H'(u)\big(U(\varphi(x-t H(u)))-U(u)\big)\leq 0$$
so that $E(\cdot\,;x,t)$ is strictly decreasing for $u>s^+(x,t)$;
and if $u<s^-(x,t)$, then
$$H'(u)\big(U(\varphi(x-t H(u)))-U(u)\big)\geq 0$$
so that $E(\cdot\,;x,t)$ is strictly increasing for $u<s^-(x,t)$.
This means that $E(\cdot\,;x,t)$ attains its maximum on $[s^-(x,t),s^+(x,t)]$.

\smallskip
Let $\mathcal{U}(x,t)$ be the set of points
at which $E(u;x,t)$ attains its maximum in $u\in \mathbb{R}$, $i.e.$,
\begin{equation}\label{c7.3}
\mathcal{U}(x,t)=\big\{w\in \mathbb{R}\,:\, E(w;x,t)=\max\limits_{v\in \mathbb{R}}E(v;x,t)\big\}.
\end{equation}
The definition of $u^\pm(x,t)$ is given by
\begin{equation}\label{c7.4}
u^-(x,t)=\sup\,\mathcal{U}(x,t), \qquad u^+(x,t)=\inf\,\mathcal{U}(x,t).
\end{equation}
Similar to $\eqref{c2.5}$, for any $(x,t)\in \mathbb{R}\times \mathbb{R}^+$, we define
\begin{equation}\label{c2.5h}
\hat{E}(x,t):=
\max_{v\in \mathbb{R}}E(v;x,t)-\int_0^{x-tH(0)}U(\varphi(\xi))\,{\rm d}\xi.
\end{equation}

\smallskip
Similar to $\eqref{c2.59}$--$\eqref{c2.58}$,
for any fixed $\tau>0$, consider the following the Cauchy problem:
\begin{equation}\label{c2.59h}
\begin{cases}
U(v)_t+F(v)_x=0 \qquad & {\rm for}\ x\in \mathbb{R},\ t>\tau, \\[2mm]
v(x,t)|_{t=\tau}=u(x,\tau)\qquad & {\rm for}\ x\in \mathbb{R},\ t=\tau,
\end{cases}
\end{equation}
where $u(x,\tau)$ is obtained by restricting solution $u(x,t)$
of the Cauchy problem \eqref{c7.1}--\eqref{c7.1b} to the line: $t=\tau$.
Similar to $\eqref{c7.2}$--$\eqref{c2.5h}$,
we define that, for any $(x,t)\in \mathbb{R}\times \mathbb{R}^+$ with $t>\tau$,
\begin{equation}\label{c2.57h}
E(v;x,t;\tau):=(t-\tau)\int_0^v H'(s)\big(U(u(x-(t-\tau)H(s),\tau))-U(s)\big)\,{\rm d}s,
\end{equation}
and
\begin{equation}\label{c2.58h}
\begin{cases}
\displaystyle\mathcal{U}(x,t;\tau)
:=\big\{v\in \mathbb{R}\,:\, E(v;x,t;\tau)=\max_{w\in\mathbb{R}} E(w;x,t;\tau)\big\}, \\[3mm]
\displaystyle v^-(x,t):=\sup\,\mathcal{U}(x,t;\tau),
\qquad v^+(x,t):=\inf\, \mathcal{U}(x,t;\tau),\\[2mm]
\displaystyle \hat{E}(x,t;\tau):=\max_{w\in\mathbb{R}} E(w;x,t;\tau)+\hat{E}(x-(t-\tau)H(0),\tau).
\end{cases}
\end{equation}

With these, we can obtain the main theorem for the Cauchy problem \eqref{c7.1}--\eqref{c7.1b}
similar to Theorem $\ref{the:mt}$ and Corollaries $\ref{cor:c2.1}$--$\ref{cor:c2.2}$.
From $\eqref{c7.19}$, the intervals:
$[s^-(x,t),s^+(x,t)]$ are contained in the same bounded interval locally in $(x,t)$.
Then $u^\pm(x,t)$ defined by $\eqref{c7.4}$ are locally uniformly bounded,
which, instead of $\eqref{c2.7}$, can be used to prove
the local properties similar to Lemmas $\ref{lem:c2.3}$--$\ref{lem:c2.5}$ as follows:
\begin{itemize}
\item[(i)] For any $t>0$ and $x_1,x_2\in \mathbb{R}$,
\begin{equation*}
\int^{x_2}_{x_1}U(u^\pm(x,t))\,{\rm d}x=\hat{E}(x_1,t)-\hat{E}(x_2,t).
\end{equation*}
\item[(ii)] For any $x\in \mathbb{R}$ and $t_2>t_1>0$,
\begin{equation*}
\int^{t_2}_{t_1}F(u^\pm(x,t))\,{\rm d}t=\hat{E}(x,t_2)-\hat{E}(x,t_1).
\end{equation*}

\item[(iii)] For any $t>0$ and $x,x'\in \mathbb{R}$ with $x<x'$,
\begin{equation*}
x-tH(u^+(x,t))\leq x' -tH(u^-(x',t)).
\end{equation*}

\item [(iv)] Lemma $\ref{lem:c2.5}$ holds via replacing $f'(c)$ by $H(c)$ in $\eqref{c2.43}$.
\end{itemize}

\smallskip
By the same arguments as the proofs of Theorem $\ref{the:mt}$ and Corollaries $\ref{cor:c2.1}$--$\ref{cor:c2.2}$,
these local properties of $u^\pm(x,t)$ and $\hat{E}(x,t)$ validate
the following main theorem for the Cauchy problem \eqref{c7.1}--\eqref{c7.1b}.

\begin{The}[Formula for entropy solutions of the Cauchy problem $\eqref{c7.1}$--$\eqref{c7.1b}$]
Suppose that the flux pair $(U(u), F(u))$ satisfies $\eqref{c7.1c}$
and the initial data function $\varphi(x)\in L_{{\rm loc}}^1(\mathbb{R})$ satisfies $\eqref{c7.19}$.
Let $u=u(x,t)$ on $\mathbb{R}\times \mathbb{R}^+$ be defined by
\begin{equation}\label{c7.8}
u(x,t):=u^+(x,t)=\inf \big\{w\in \mathbb{R}\,:\, E(w;x,t)=\max_{v\in \mathbb{R}}E(v;x,t)\big\}
\end{equation}
with $E(u;x,t)$ given by $\eqref{c7.2}$.
Then $u=u(x,t)$,
defined by the solution formula $\eqref{c7.8}$,
is the unique entropy solution of the Cauchy problem $\eqref{c7.1}$--$\eqref{c7.1b}$ in the sense that
\begin{equation*}
H(u(x_2,t))-H(u(x_1,t))\leq \frac{x_2-x_1}{t} \qquad {\rm for\ any\ } x_2>x_1\ {\rm and}\ t>0.
\end{equation*}
Furthermore, the entropy solution $u=u(x,t)$ possesses both the left- and right-traces pointwise
and satisfies
\begin{equation*}
u(x{-},t)=u^-(x,t),\qquad
u(x{+},t)=u^+(x,t)\qquad\, {\rm on}\ (x,t)\in \mathbb{R} \times \mathbb{R}^+;
\end{equation*}
and, for each $t>0$, $u(x,t)=u(x{-},t)=u(x{+},t)$ holds except
at most a countable set of $(x,t)$ at which $u(x{-},t)>u(x{+},t)$ and satisfies
\begin{equation*}
H(u(\cdot,t))\in {\rm BV}_{{\rm loc}}(\mathbb{R}) \qquad {\rm for\ any }\ t>0.
\end{equation*}
Moreover, the results in {\rm Corollary} $\ref{cor:c2.1}$ still hold
if $\varphi_i(x)$,  $u_i(x,t)$, and $f'(u)$ in $\eqref{c2.51}$
are replaced by  $U(\varphi_i(x))$, $U(u_i(x,t))$, and $H(u)$, respectively.
Also, the results in {\rm Corollary} $\ref{cor:c2.2}$ still hold.
\end{The}

To elucidate the fine properties of entropy solutions of the Cauchy problem \eqref{c7.1}--\eqref{c7.1b} as in \S 3--\S 7,
we restrict to the case that
the initial data function $\varphi(x) \in L^\infty(\mathbb{R})$.
For this purpose,
we present the following adjustments to the corresponding symbols and notions:

\begin{itemize}
\item [(a)] Replace $f'(u)$ and $f''(u)$ by $H(u)$ and $H'(u)$, respectively.

\vspace{1pt}
\item [(b)] Replace all the integrals of $\varphi(x)$ and $u(x,t)$ by
the corresponding integrals of $U(\varphi(x))$ and $U(u(x,t))$,
especially for replacing $\Phi(x)=\int^x_{0} \varphi(\xi)\,{\rm d}\xi$
by $\Phi(x)=\int^x_{0} U(\varphi(\xi))\,{\rm d}\xi$
along with its Dini derivatives and convex hull $\bar{\Phi}(x)$ over $(-\infty,\infty)$.

\vspace{1pt}
\item [(c)] Replace $\Phi(l;x_0,c)$ and $F(l;t,c)$ defined in $\eqref{c4.2b}$ by
\begin{align*}
\qquad\quad\ \Phi(l;x_0,c)=\!\int^{x_0+l}_{x_0}(U(\varphi(\xi))-U(c))\,{\rm d}\xi,\quad
F(l;t,c)={-}t\!\int^{u(l;t,c)}_c(U(s)-U(c))H'(s)\,{\rm d}s.
\end{align*}

\item [(d)] The four invariants
$\overline{u}_l$, $\underline{u}_l$, $\overline{u}_r$, and $\underline{u}_r$,
defined by $\eqref{c5.7}$, are replaced by
\begin{equation*}
\qquad\quad\begin{cases}
U(\overline{u}_l):=\displaystyle
\mathop{\overline{\lim}}\limits_{x\rightarrow -\infty}
\frac{1}{x}\int^x_0U(\varphi (\xi))\,{\rm d}\xi,
\quad
&U(\underline{u}_l):=\displaystyle
\mathop{\underline{\lim}}\limits_{x\rightarrow -\infty}
\frac{1}{x}\int^x_0U(\varphi (\xi))\,{\rm d}\xi,\\[4mm]
U(\overline{u}_r):=\displaystyle
\mathop{\overline{\lim}}\limits_{x\rightarrow \infty}
\frac{1}{x}\int^x_0U(\varphi (\xi))\,{\rm d}\xi,
\quad
&U(\underline{u}_r):=\displaystyle
\mathop{\underline{\lim}}\limits_{x\rightarrow \infty}
\frac{1}{x}\int^x_0U(\varphi (\xi))\,{\rm d}\xi.
\end{cases}
\end{equation*}

\item [(e)] Similar to $\eqref{c6.10}$,
let $\tilde{u}=\tilde{u}(x,t)$ be the rarefaction-constant solution
of the Cauchy problem for \eqref{c7.1} with the initial data function $\bar{\Phi}'(x)$.
\end{itemize}

By the same arguments as \S 3--\S 7,
these adjustments can validate the following main theorem
on the fine properties of entropy solutions of the Cauchy problem \eqref{c7.1}--\eqref{c7.1b}.

\begin{The}[Fine properties for entropy solutions of the Cauchy problem $\eqref{c7.1}$--$\eqref{c7.1b}$]
Suppose that the flux pair $(U(u), F(u))$ satisfies $\eqref{c7.1c}$
and the initial data function $\varphi(x)\in L^\infty(\mathbb{R})$.
Then the unique entropy solution $u=u(x,t)$,
defined by $\eqref{c7.8}$,
of the Cauchy problem \eqref{c7.1}--\eqref{c7.1b} possesses the following fine properties{\rm :}

\begin{itemize}
\item [(i)]
The results in {\rm Theorems} $\ref{pro:c3.3}$--$\ref{pro:c3.4}$
with {\rm Proposition} $\ref{pro:c3.1}$ and {\rm Lemma} $\ref{pro:c3.2}$ still hold{\rm ;}
and the results in {\rm Theorem} $\ref{pro:c3.6}$ still hold if
the left- and right-derivatives of shock curves as in $\eqref{c3.36}$--$\eqref{c3.37}$ are replaced by
\begin{equation*}
\qquad\quad \frac{{\rm d}^+x(t)}{{\rm d}t}
=\frac{F(u(x(t){-},t))-F(u(x(t){+},t))}{U(u(x(t){-},t))-U(u(x(t){+},t))},\qquad
\frac{{\rm d}^-x(t)}{{\rm d}t}=\frac{F(d_n)-F(c_n)}{U(d_n)-U(c_n)}.
\end{equation*}

\item [(ii)]
The results in {\rm Lemma} $\ref{lem:c4.1}$ still hold{\rm ;}
the results in {\rm Propositions} $\ref{pro:c4.1}$--$\ref{pro:c4.4}$ still hold
if $\eqref{c4.3}$ is replaced by
\begin{equation*}
\overline{{\rm D}}_- \Phi(x_0)\leq U(c)\leq\underline{{\rm D}}_+ \Phi(x_0)
\end{equation*}
and
$\eqref{c4.3a}$ is replaced by
\begin{equation*}
(\overline{{\rm D}}_- \Phi(x_0),\underline{{\rm D}}_+ \Phi(x_0))\subset
U(\mathcal{C}(x_0))\subset
[\overline{{\rm D}}_- \Phi(x_0),\underline{{\rm D}}_+ \Phi(x_0)];
\end{equation*}
and
the results in {\rm Theorems} $\ref{the:c4.0}$--$\ref{the:c4.1}$, $\ref{the:c4.2}$,
and $\ref{lem:plc1}$--$\ref{the:elc0}$ with {\rm Lemma} $\ref{lem:plc0}$ also hold.

\item [(iii)]
The results in {\rm Theorems} $\ref{the:c4.3}$
and $\ref{the:dsw1}$ with {\rm Lemma} $\ref{lem:fsw2}$ still hold.

\item [(iv)]
The results in {\rm Lemma} $\ref{lem:c6.1}$ still hold
if $tc$ and $td$ in $\eqref{c6.11}$ are replaced by
$tU(c)$ and $tU(d)$, respectively{\rm ;}
the results in {\rm Theorem} $\ref{pro:c6.1}$ still hold
if $\overline{u}_l$, $\underline{u}_l$, $\overline{u}_r$, and $\underline{u}_r$
in $\eqref{c6.22}$ are replaced by
$U(\overline{u}_l)$, $U(\underline{u}_l)$, $U(\overline{u}_r)$, and $U(\underline{u}_r)$, respectively{\rm ;}
and the results in {\rm Propositions} $\ref{pro:c5.1}$--$\ref{pro:c5.2}$
and {\rm Theorems} $\ref{the:c5.1}$--$\ref{the:c5.2}$ still hold
if $\eqref{c5.14}$ is replaced by
\begin{equation*}
\Phi(x_0)=\bar{\Phi}(x_0), \qquad {\rm D}_-\bar{\Phi}(x_0)\leq U(c)\leq {\rm D}_+\bar{\Phi}(x_0).
\end{equation*}
Moreover,
the results in {\rm Lemma} $\ref{lem:c6.2}$ still hold
if $v^+(t)-c$ in $\eqref{c6.17}$
and $v^-(t)-d$ in $\eqref{c6.19}$ are replaced by
$U(v^+(t))-U(c)$ and $U(v^-(t))-U(d)$ respectively,
and $\eqref{c6.16}$ and $\eqref{c6.18}$ are replaced by
$$
U(c)=\lim_{t\rightarrow \infty}U(u(x(t){+},t))={\rm D}_-\bar{\Phi}(y^+(\infty,0))
$$
and
$$
U(d)=\lim_{t\rightarrow \infty}U(u(x(t){-},t))={\rm D}_+\bar{\Phi}(y^-(\infty,0))
$$
respectively{\rm ;}
and
the results in {\rm Theorem} $\ref{pro:c6.2}$ also hold.

\item [(v)]
The results in {\rm Lemmas} $\ref{the:c6.1}$--$\ref{lem:c6.3}$
and {\rm Theorems} $\ref{pro:c6.3}$--$\ref{the:c6.2}$ hold{\rm ;}
the results in {\rm Theorem} $\ref{the:c6.3}$ still hold
if
$u(x,t)-\overline{u}_l$ and $u(x,t)-\underline{u}_r$ are replaced by
$U(u(x,t))-U(\overline{u}_l)$ and $U(u(x,t))-U(\underline{u}_r)$ respectively,
and
$\varphi(x)-\overline{u}_l$ and $\varphi(x)-\underline{u}_r$ are replaced by
$U(\varphi(x))-U(\overline{u}_l)$ and $U(\varphi(x))-U(\underline{u}_r)$ respectively{\rm ;}
and the results in
{\rm Lemma} $\ref{lem:n.1}$ and {\rm Theorem} $\ref{the:n.1}$ still hold
for the generalized $N$--waves given by {\rm Definition} $\ref{def:n.1}$
if
$\varphi(x)-\overline{u}_l$ and $\varphi(x)-\underline{u}_r$ are replaced by
$U(\varphi(x))-U(\overline{u}_l)$ and $U(\varphi(x))-U(\underline{u}_r)$ respectively.
\end{itemize}
\end{The}

\smallskip
In addition, for Appendix A, $\rho(u,v)$ in $\eqref{aa5a}$ becomes
$$
\rho(u,v)=\frac{\int^u_vU(s)H'(s)\,{\rm d}s}{\int^u_vH'(s)\,{\rm d}s}\quad {\rm if}\ u\neq v,\qquad \rho(v,v)=U(v)\quad {\rm if }\ u=v.
$$
Then the corresponding results similar to $\eqref{aa5}$--$\eqref{aa6}$ hold.
Furthermore, for any fixed $c\in\mathbb{R}$,
if
$$
H'(u)=(N_\pm+o(1))\, |u-c|^{\alpha\pm}\qquad {\rm as}\ u\rightarrow c{\pm}
$$
holds for some $\alpha_\pm \geq 0$ and $N_\pm>0$,
then $\eqref{aa2}$ still holds if $f'(u)$ is replaced by $H(u)$,
and $\eqref{aa3}$ becomes
$$
\int^u_c\big(U(s)-U(c)\big)H'(s)\,{\rm d}s
=\frac{N_\pm U'(c)+o(1)}{2+\alpha_\pm}\, |u-c|^{2+\alpha_\pm}
\qquad\, {\rm as}\ u\rightarrow c{\pm};
$$
moreover, for $\eqref{aa6a}$, it is direct to check that, as $u\rightarrow c{\pm}$,
\begin{equation*}
\rho(u,c)-U(c)=(1+o(1))\,\frac{1+\alpha_\pm}{2+\alpha_\pm}\,(U(u)-U(c)).
\end{equation*}

For Appendix B,
{\rm Lemmas} $\ref{lem:c5.1}$--$\ref{lem:c5.3}$ still hold if
$\overline{u}_l$ and $\underline{u}_r$ are replaced by $U(\overline{u}_l)$ and $U(\underline{u}_r)$.

\medskip
\appendix
\section{Properties of the Flux Functions of Convexity Degeneracy}
Suppose that the flux function $f(u)$ satisfies $\eqref{c1.2}$.
For $(u,v)\in \mathbb{R}\times\mathbb{R}$, we define
\begin{equation}\label{aa5a}
\rho(u,v):=\dfrac{\int^u_vsf''(s)\,{\rm d}s}{\int^u_vf''(s)\,{\rm d}s} \quad
\mbox{if}\ u\neq v,
\qquad \rho(u,v):=v \quad \ \mbox{if}\ u=v.
\end{equation}
Then
$\rho(u,v)=\rho(v,u)\in C(\mathbb{R}\times\mathbb{R})$ satisfies
\begin{equation}\label{aa5}
\lim_{u,v\rightarrow w}\rho(u,v)=\rho(w,w)=w \qquad {\rm for\ any}\ w\in\mathbb{R};
\end{equation}
and, for any fixed $v\in\mathbb{R}$,
$\rho(u,v)=\rho(v,u)$ is strictly increasing in $u\in \mathbb{R}$
and satisfies
\begin{equation}\label{aa6}
\partial_1\rho(u,v)=\partial_2\rho(v,u)=
f''(u)\frac{u-\rho(u,v)}{f'(u)-f'(v)} \geq 0
\qquad {\rm if}\ u\neq v.
\end{equation}

Furthermore, for any fixed $c\in \mathbb{R}$, if
\begin{equation}\label{aa1}
f''(u)=(N_\pm+o(1))\, |u-c|^{\alpha_\pm} \qquad {\rm as}\ u\rightarrow c{\pm}
\end{equation}
holds for some $\alpha_\pm \geq0$ and $N_\pm>0$,
then
\begin{itemize}
\item [(i)] As $u\rightarrow c{\pm}$,
\begin{equation}\label{aa2}
|f'(u)-f'(c)|=\frac{N_\pm+o(1)}{1+\alpha_\pm}\, |u-c|^{1+\alpha_\pm}.
\end{equation}

\item [(ii)] As $u\rightarrow c{\pm}$,
\begin{equation}\label{aa3}
\int^u_c(s-c)f''(s)\,{\rm d}s=\frac{N_\pm+o(1)}{2+\alpha_\pm}\, |u-c|^{2+\alpha_\pm}.
\end{equation}

\item [(iii)] As $u\rightarrow c{\pm}$,
\begin{equation}\label{aa6a}
\rho(u,c)-c=(1+o(1))\,\frac{1+\alpha_\pm}{2+\alpha_\pm}\,(u-c).
\end{equation}
\end{itemize}

\begin{proof}
Clearly, $\rho(u,v)=\rho(v,u)$. Since $f'(u)$ is strictly increasing in $u\in \mathbb{R}$,
then it is direct to see that $\rho(u,v)$ is continuous at any point $(u,v)$ with $u\neq v$.
Moreover, for any $w\in \mathbb{R}$,
\begin{align*}
|\rho(u,v)-\rho(w,w)|
&=\bigg|\frac{\int^u_v(s-v)f''(s)\,{\rm d}s}{f'(u)-f'(v)}+v-w\bigg|\\[1mm]
&\leq\bigg|\frac{(u-v)f'(u)-(f(u)-f(v))}{f'(u)-f'(v)}\bigg|+|v-w|\\[1mm]
&=\frac{f'(u)-f'(\xi)}{f'(u)-f'(v)}|u-v|+|v-w|
\leq|u-v|+|v-w|,
\end{align*}
where $\xi$ lies between $u$ and $v$.
Thus, $\eqref{aa5}$ holds so that
$\rho(u,v)$ is continuous at $(w,w)$ for any $w\in \mathbb{R}$.
Therefore, $\rho(u,v) \in C(\mathbb{R}\times \mathbb{R})$.

To prove $\eqref{aa6}$, fix $v\in\mathbb{R}$. For any $u\neq v$,
it is direct to check that
$$\partial_1\rho(u,v)=\partial_2\rho(v,u)=
f''(u)\,\frac{u-\rho(u,v)}{f'(u)-f'(v)}
=f''(u)\,\frac{\int^u_v(u-s)f''(s)\,{\rm d}s}{\big(\int^u_v f''(s)\,{\rm d}s\big)^2}
\geq 0,$$
which, by the strictly increasing property of $f'(u)$, implies that
$\rho(u,v)$ is strictly increasing in $u$.

\smallskip
Furthermore, if $\eqref{aa1}$ holds, then

\begin{itemize}
\item[(i)] Letting $u\rightarrow c{\pm}$,
\begin{align*}
|f'(u)-f'(c)|&=\int^1_0f''(\theta u+(1-\theta)c)\,{\rm d}\theta\, |u-c|\\[1mm]
&=\int^1_0(N_\pm+o(1))\, |\theta(u-c)|^{\alpha_\pm} \,{\rm d}\theta\, |u-c|
=\frac{N_\pm+o(1)}{1+\alpha_\pm}\, |u-c|^{1+\alpha_\pm}.
\end{align*}

\item[(ii)] Letting $u\rightarrow c{\pm}$, if $u>c$,
\begin{align*}
\int^u_c(s-c)f''(s)\,{\rm d}s= \int^u_c(s-c)(N_\pm+o(1))(s-c)^{\alpha_\pm} \,{\rm d}s
=\frac{N_\pm+o(1)}{2+\alpha_\pm}(u-c)^{2+\alpha_\pm};
\end{align*}
and, if $u<c$,
\begin{align*}
\int^u_c(s-c)f''(s)\,{\rm d}s= \int^c_u(c-s)(N_\pm+o(1))(c-s)^{\alpha_\pm} \,{\rm d}s
=\frac{N_\pm+o(1)}{2+\alpha_\pm}(c-u)^{2+\alpha_\pm}.
\end{align*}

\item[(iii)] From $\eqref{aa2}$--$\eqref{aa3}$, as $u\rightarrow c{\pm}$,
$$
\frac{\rho(u,c)-\rho(c,c)}{u-c}
=\dfrac{\int^u_c(s-c)f''(s)\,{\rm d}s}{(u-c)(f'(u)-f'(c))}
=\dfrac{\frac{N_\pm+o(1)}{2+\alpha_\pm}|u-c|^{2+\alpha_\pm}}
{\frac{N_\pm+o(1)}{1+\alpha_\pm}|u-c|^{2+\alpha_\pm}}
=(1+o(1))\frac{1+\alpha_\pm}{2+\alpha_\pm},
$$
which implies $\eqref{aa6a}$.
\end{itemize}

In summary, we have proved the properties of the flux functions of convexity degeneracy.
\end{proof}

\section{Proofs of Lemmas $\ref{lem:c5.1}$--$\ref{lem:c5.3}$}
From $\eqref{c5.4}$--$\eqref{c5.5}$, $\Phi(x)$ is
continuous on $({-}\infty,\infty)$,
and the convex hull $\bar{\Phi}_N(x)$ of $\Phi(x)$ on $[-N,N]$ satisfies that
$\Phi(x)\geq \bar{\Phi}_N(x)$ on $[-N,N]$, $\bar{\Phi}_N(\pm N)=\Phi(\pm N),$
and
$$
\Phi(x)\geq \bar{\Phi}_1(x)\geq\cdots\geq\bar{\Phi}_N(x)\geq\bar{\Phi}_{N+1}(x)\geq \cdots
\qquad\,\, \mbox{for any $x\in \mathbb{R}$}.
$$
Then $\bar{\Phi}_N(x)$ converges to $\bar{\Phi}(x)$ pointwise on $[-\infty,\infty)$.

\smallskip
We now give the proofs of Lemmas $\ref{lem:c5.1}$--$\ref{lem:c5.3}$.

\begin{proof}[Proof of {\rm Lemma} $\ref{lem:c5.1}$.]
The proof is divided into two steps.

\smallskip
\noindent
{\bf 1.} Suppose that $\bar{\Phi}(x)>-\infty$,
we now prove $\eqref{phicc}$.
By the definition of $\bar{\Phi}_N(x)$ in $\eqref{c5.5}$,
for any fixed $x_1,x_2\in \mathbb{R}$, if $N>\max\{|x_1|,|x_2|\}$,
\begin{equation*}
\bar{\Phi}_N(\frac{x_1+x_2}{2})\leq\frac{\bar{\Phi}_N(x_1)+\bar{\Phi}_N(x_2)}{2},
\end{equation*}
which, by letting $N\rightarrow\infty$, implies that $\bar{\Phi}(x)$ is a convex function, $i.e.$,
for any $x_1,x_2\in \mathbb{R}$,
\begin{equation}\label{ab1}
\bar{\Phi}(\frac{x_1+x_2}{2})\leq\frac{\bar{\Phi}(x_1)+\bar{\Phi}(x_2)}{2}.
\end{equation}

From $\eqref{ab1}$, the left- and right-derivatives ${\rm D}_{\pm}\bar{\Phi}(x)$ of $\bar{\Phi}(x)$
exist pointwise on $({-}\infty,\infty)$.
Denote $c:={\rm D}_-\bar{\Phi}(0)$. Then
$\bar{\Phi}(x)\geq\bar{\Phi}(0)+cx$ for any $x\in \mathbb{R}.$
From ${\Phi}(x)\geq\bar{\Phi}_N(x)\geq\bar{\Phi}(x)$,
\begin{equation*}
\Phi(x)\geq\bar{\Phi}(0)+cx \qquad {\rm for\ any}\ x\in \mathbb{R},
\end{equation*}
which implies
\begin{equation}\label{ab3a}
\mathop{\underline{\lim}}_{x\rightarrow \pm\infty}(\Phi(x)-cx)\geq \bar{\Phi}(0)>-\infty.
\end{equation}

\smallskip
\noindent
{\bf 2.}
On the other hand, suppose that $\eqref{ab3a}$ holds for some $c\in \mathbb{R}$.
Then there exists $b_0\in \mathbb{R}$ such that
$\mathop{\underline{\lim}}_{x\rightarrow \pm\infty}(\Phi(x)-cx)\geq b_0.$
Thus, for $N$ sufficiently large,
\begin{equation*}
\Phi(x)-c x\geq b_0-1 \qquad {\rm for\ any}\ |x|> N.
\end{equation*}

Denote the minimum of $\Phi(x)-cx$ on $[-N,N]$ by $b_1$ and $b:=\min\{b_0-1,b_1\}$.
Then
\begin{equation}\label{ab4}
\Phi(x)\geq cx+b \qquad{\rm for\ any}\ x \in \mathbb{R},
\end{equation}
so that, by $\eqref{c5.5}$,
$
\bar{\Phi}_N(x)\geq cx+b $ for any $ x \in \mathbb{R}.
$
By letting $N\rightarrow\infty$, this implies that
\begin{equation*}
\bar{\Phi}(x)\geq cx+b>-\infty\qquad {\rm for\ any}\ x\in \mathbb{R},
\end{equation*}
as desired.
\end{proof}

\begin{proof}[Proof of {\rm Lemma} $\ref{lem:c5.2}$.]  We divide the proof into three steps accordingly.

\smallskip
\noindent
{\bf 1}. If $\overline{u}_l$ and $\underline{u}_r$ in $\eqref{c5.7}$ satisfy that $\overline{u}_l<\underline{u}_r$,
then, for any $c\in (\overline{u}_l,\underline{u}_r)$,
\begin{equation}\label{ab5}
\begin{cases}
 \displaystyle\lim\limits_{x\rightarrow \infty}(\Phi(x)-cx)
 =\lim\limits_{x\rightarrow \infty}x\, \Big(\frac{\Phi(x)}{x}-c\Big)=\infty,
\\[4mm]
 \displaystyle\lim\limits_{x\rightarrow -\infty}(\Phi(x)-cx)
 =\lim\limits_{x\rightarrow -\infty}({-}x)\, \Big(c-\frac{\Phi(x)}{x}\Big)=\infty,
\end{cases}
\end{equation}
which, by Lemma $\ref{lem:c5.1}$, implies that $\bar{\Phi}(x)>-\infty$ on $x \in \mathbb{R}$.

From $\eqref{ab5}$, there exists $x_0\in \mathbb{R}$ such that
${\Phi}(x)-cx\geq {\Phi}(x_0)-cx_0$ for any $x\in \mathbb{R},$
which, by $\eqref{c5.5}$, implies that
$\bar{\Phi}_N(x)\geq cx+( {\Phi}(x_0)-cx_0)$ for any $x\in \mathbb{R}.$
Then, by letting $N\rightarrow\infty$,
\begin{equation}\label{ab6}
\bar{\Phi}(x)\geq cx+( {\Phi}(x_0)-cx_0) \qquad {\rm for\ any}\ x\in \mathbb{R},
\end{equation}
which means that
$\bar{\Phi}(x_0)\geq cx_0+ ({\Phi}(x_0)-cx_0)=\Phi(x_0).$
Since ${\Phi}(x)\geq \bar{\Phi}(x)$ on $\mathbb{R}$,
then ${\Phi}(x_0)= \bar{\Phi}(x_0)$ so that $\mathcal{K}_0\neq\varnothing$.

\smallskip
Since ${\Phi}(x)\geq \bar{\Phi}(x)$ on $ \mathbb{R}$, then,
for any fixed $y\in \mathbb{R}$,
\begin{equation}\label{ab7}
\frac{\bar{\Phi}(x)-\bar{\Phi}(y)}{x-y}\leq \frac{{\Phi}(x)-\bar{\Phi}(y)}{x-y}
\ \ {\rm if}\ x>y, \qquad
\frac{\bar{\Phi}(x)-\bar{\Phi}(y)}{x-y}\geq \frac{{\Phi}(x)-\bar{\Phi}(y)}{x-y}
\ \ {\rm if}\ x<y.
\end{equation}
Then it follows from $\eqref{ab1}$ that
$\frac{\bar{\Phi}(x)-\bar{\Phi}(y)}{x-y}$ is nondecreasing in $x\in \mathbb{R}$
so that, by $\eqref{ab7}$,
\begin{align}\label{ab8}
 {\rm D}_+\bar{\Phi}(y)
 &=\lim\limits_{x\rightarrow y{+}}\frac{\bar{\Phi}(x)-\bar{\Phi}(y)}{x-y}
\leq\lim\limits_{x\rightarrow \infty}\frac{\bar{\Phi}(x)-\bar{\Phi}(y)}{x-y}
\leq\mathop{\underline{\lim}}\limits_{x\rightarrow \infty}
\frac{{\Phi}(x)-\bar{\Phi}(y)}{x-y}\nonumber\\[1mm]
&=\mathop{\underline{\lim}}\limits_{x\rightarrow \infty}\frac{{\Phi}(x)-\bar{\Phi}(y)}{x-y}\,
\mathop{\underline{\lim}}\limits_{x\rightarrow \infty}\frac{x-y}{x}
+\mathop{\underline{\lim}}\limits_{x\rightarrow \infty}\frac{\bar{\Phi}(y)}{x}\nonumber\\[1mm]
&\leq\mathop{\underline{\lim}}\limits_{x\rightarrow \infty}
\Big(\frac{{\Phi}(x)-\bar{\Phi}(y)}{x-y}\,\frac{x-y}{x}+\frac{\bar{\Phi}(y)}{x}\Big)
=\mathop{\underline{\lim}}\limits_{x\rightarrow \infty}\frac{{\Phi}(x)}{x}
=\underline{u}_r.
\end{align}
Similarly, it follows from $\eqref{ab7}$ that
${\rm D}_-\bar{\Phi}(y)\geq \overline{u}_l$ for any $y\in \mathbb{R}$.
This implies $\eqref{c5.8a}$.

To prove $\eqref{c5.8}$, from $\eqref{c5.8a}$, we first have
\begin{equation}\label{ab9}
\inf_{x_0\in \mathcal{K}_0}{\rm D}_-\bar{\Phi}(x_0)\geq \overline{u}_l,
\qquad \sup_{x_0\in \mathcal{K}_0}{\rm D}_+\bar{\Phi}(x_0)\leq \underline{u}_r.
\end{equation}
On the other hand, similar to the arguments as those for $\eqref{ab5}$--$\eqref{ab6}$,
for any $c\in (\overline{u}_l,\underline{u}_r)$,
there exists $x_0\in \mathcal{K}_0$ such that
$\bar{\Phi}(x)\geq cx+\bar{\Phi}(x_0)-cx_0$ for any $x\in \mathbb{R},$
which implies
$$
{\rm D}_-\bar{\Phi}(x_0)
=\lim\limits_{x\rightarrow x_{0}{-}}\frac{\bar{\Phi}(x)-\bar{\Phi}(x_0)}{x-x_0}\leq c, \qquad
{\rm D}_+\bar{\Phi}(x_0)
=\lim\limits_{x\rightarrow x_{0}{+}}\frac{\bar{\Phi}(x)-\bar{\Phi}(x_0)}{x-x_0}\geq c.
$$
By the arbitrariness of $c\in (\overline{u}_l,\underline{u}_r)$, we obtain
\begin{equation}\label{ab10}
\inf_{x_0\in \mathcal{K}_0}{\rm D}_-\bar{\Phi}(x_0)\leq \overline{u}_l,
\qquad \sup_{x_0\in \mathcal{K}_0}{\rm D}_+\bar{\Phi}(x_0)\geq \underline{u}_r.
\end{equation}
Combining $\eqref{ab9}$ with $\eqref{ab10}$, $\eqref{c5.8}$ holds.

\smallskip
\noindent
{\bf 2}. If $\eqref{c5.8b}$ holds, it follows from Lemma $\ref{lem:c5.1}$ that
$\bar{\Phi}(x)>-\infty$ on $ x\in \mathbb{R}$.

On the other hand, suppose that $\bar{\Phi}(x)>-\infty$ on $ x\in \mathbb{R}$.
From Lemma $\ref{lem:c5.1}$,
there exists $c\in \mathbb{R}$ such that
$\mathop{\underline{\lim}}_{x\rightarrow \pm\infty}(\Phi(x)-cx)>-\infty$,
and hence $\eqref{ab4}$ holds.
Then, to prove $\eqref{c5.8b}$, it suffices to show that $c=\underline{u}_r$.
In fact, from $\eqref{ab4}$,
$$\frac{\Phi(x)}{x}\geq \frac{b}{x}+c\quad {\rm if}\ x>0,
\qquad \frac{\Phi(x)}{x}\leq \frac{b}{x}+c \quad {\rm if }\ x<0,$$
which implies
$$
\overline{u}_l=\mathop{\overline{\lim}}\limits_{x\rightarrow -\infty}\frac{\Phi(x)}{x}\leq c\leq\mathop{\underline{\lim}}\limits_{x\rightarrow \infty}\frac{\Phi(x)}{x}=\underline{u}_r.
$$
This, by $\overline{u}_l=\underline{u}_r$, yields that $c=\underline{u}_r$.

By the same arguments as those for $\eqref{ab7}$--$\eqref{ab8}$,
it can be checked that
${\rm D}_-\bar{\Phi}(y)={\rm D}_+\bar{\Phi}(y)=\underline{u}_r$ for any $y\in \mathbb{R},$
which implies that
$\bar{\Phi}(x)=\underline{u}_r\, x+b_0$ for some $b_0\in \mathbb{R}.$
Thus, if there exists $x_0\in \mathbb{R}$ such that $\Phi(x_0)=\underline{u}_r x_0+b_0$,
then $\mathcal{K}_0\neq \varnothing$;
and, if $\Phi(x)>\underline{u}_r x+b_0$ for any $x\in \mathbb{R}$,
then $\mathcal{K}_0= \varnothing$.

\smallskip
\noindent
{\bf 3}. By the definition of $\overline{u}_l$ and $\underline{u}_r$,
for any $c>\underline{u}_r$,
$$
\mathop{\underline{\lim}}\limits_{x\rightarrow \infty}(\Phi(x)-cx)
=\mathop{\underline{\lim}}\limits_{x\rightarrow \infty}x\Big(\frac{\Phi(x)}{x}-c\Big)=-\infty,
$$
and, for any $c<\overline{u}_l$,
$$
\mathop{\underline{\lim}}\limits_{x\rightarrow -\infty}(\Phi(x)-cx)
=\mathop{\underline{\lim}}\limits_{x\rightarrow -\infty}({-}x)\Big(c-\frac{\Phi(x)}{x}\Big)=-\infty.
$$
Then it follows from $\overline{u}_l>\underline{u}_r$ that, for any $c\in \mathbb{R}$,
$$
\mathop{\underline{\lim}}\limits_{x\rightarrow \infty}(\Phi(x)-cx)=-\infty
\qquad {\rm or}\qquad
\mathop{\underline{\lim}}\limits_{x\rightarrow -\infty}(\Phi(x)-cx)=-\infty,
$$
which, by Lemma $\ref{lem:c5.1}$, implies that $\bar{\Phi}(x)\equiv-\infty$ on $\mathbb{R}$.

\smallskip
Up to now, we have proved Lemma $\ref{lem:c5.2}$.
\end{proof}

\begin{proof}[Proof of {\rm Lemma} $\ref{lem:c5.3}$.]
Suppose that $\bar{\Phi}(x)>-\infty$ on $x\in \mathbb{R}$.
From $\eqref{ab1}$, $\bar{\Phi}(x)$ is convex.

\smallskip
\noindent
{\bf 1}. For any fixed $x_0 \in\mathcal{K}_0$, $\Phi(x_0)=\bar{\Phi}(x_0)$.
Since $\Phi(x)\geq\bar{\Phi}_N(x)\geq\bar{\Phi}(x)$, then
\begin{equation}\label{ab11}
{\Phi}(x_0)=\bar{\Phi}_N(x_0)=\bar{\Phi}(x_0) \qquad {\rm for\ any}\ N.
\end{equation}

\smallskip
{\bf (a).}
We first prove that,
\begin{equation}\label{ab12}
\inf_{l>0}\Delta_l\bar{\Phi}_N(x_0)=\inf_{l>0}\Delta_l{\Phi}(x_0),\qquad \sup_{l<0}\Delta_l\bar{\Phi}_N(x_0)=\sup_{l<0}\Delta_l{\Phi}(x_0).
\end{equation}
In fact, since $\Phi(x)\geq\bar{\Phi}_N(x)$, by $\eqref{ab11}$,
for any fixed $N$,
$$
\frac{{\Phi}(x_0+l)-{\Phi}(x_0)}{l}\geq\frac{\bar{\Phi}_N(x_0+l)-\bar{\Phi}_N(x_0)}{l}
\qquad {\rm for\ any}\ l>0,
$$
which implies
\begin{equation}\label{ab13}
\inf_{l>0}\Delta_l{\Phi}(x_0)\geq\inf_{l>0}\Delta_l\bar{\Phi}_N(x_0).
\end{equation}
Furthermore, for any $l>0$, if $x_0+l\geq N$,
by $\eqref{c5.5}$, $\Phi(x_0+l)=\bar{\Phi}_N(x_0+l)$ so that, by $\eqref{ab11}$,
\begin{equation}\label{ab14}
\frac{\bar{\Phi}_N(x_0+l)-\bar{\Phi}_N(x_0)}{l}
=\frac{{\Phi}(x_0+l)-{\Phi}(x_0)}{l}\geq\inf_{l>0}\Delta_l{\Phi}(x_0);
\end{equation}
and, if $x_0+l<N$, $\eqref{ab14}$ holds when $\Phi(x_0+l)=\bar{\Phi}_N(x_0+l)$.
For the case that $\Phi(x_0+l)>\bar{\Phi}_N(x_0+l)$ with $x_0+l<N$,
by the definition of convex hull and $\eqref{ab11}$,
there exists $\bar{l}>0$ with $x_0+\bar{l}\leq N$ such that
\begin{equation}\label{ab15}
\frac{\bar{\Phi}_N(x_0+l)-\bar{\Phi}_N(x_0)}{l}
=\frac{{\Phi}(x_0+\bar{l})-{\Phi}(x_0)}{\bar{l}}
\geq\inf_{l>0}\Delta_l{\Phi}(x_0).
\end{equation}
Then it follows from $\eqref{ab14}$--$\eqref{ab15}$ that
\begin{equation}\label{ab16}
\inf_{l>0}\Delta_l\bar{\Phi}_N(x_0)\geq \inf_{l>0}\Delta_l{\Phi}(x_0).
\end{equation}
Combining $\eqref{ab13}$ with $\eqref{ab16}$, it follows that $\inf_{l>0}\Delta_l\bar{\Phi}_N(x_0)=\inf_{l>0}\Delta_l{\Phi}(x_0)$ in $\eqref{ab12}$.
Similarly, it can be checked that
$\sup_{l<0}\Delta_l\bar{\Phi}_N(x_0)=\sup_{l<0}\Delta_l{\Phi}(x_0)$.
Hence, $\eqref{ab12}$ holds.

\smallskip
{\bf (b).} We now prove $\eqref{c5.10}$.
Since $\bar{\Phi}(x)$ is convex on $({-}\infty,\infty)$,
then $\Delta_l\bar{\Phi}(x_0)$ is nondecreasing so that
$$
{\rm D}_+\bar{\Phi}(x_0)=\inf_{l>0}\Delta_l\bar{\Phi}(x_0),\qquad\,\, {\rm D}_-\bar{\Phi}(x_0)=\sup_{l<0}\Delta_l\bar{\Phi}(x_0).
$$
Suppose that  $x_0\in \mathcal{K}_0$.
Since $\Phi(x)\geq \bar{\Phi}_N(x)\geq \bar{\Phi}(x)$ on $\mathbb{R}$,
it is direct to check from $\eqref{ab11}$ that
$$
\Delta_l\bar{\Phi}(x_0)\leq \Delta_l\bar{\Phi}_N(x_0)\quad {\rm for }\ l>0,
\qquad\,\, \Delta_l\bar{\Phi}(x_0)\geq \Delta_l\bar{\Phi}_N(x_0)\quad {\rm for }\ l<0,
$$
which, by $\eqref{ab12}$, implies
\begin{equation}\label{ab17}
\begin{cases}
\displaystyle {\rm D}_+\bar{\Phi}(x_0)=\inf_{l>0}\Delta_l\bar{\Phi}(x_0)
\leq \inf_{l>0}\Delta_l\bar{\Phi}_N(x_0)=\inf_{l>0}\Delta_l{\Phi}(x_0),\\[4mm]
\displaystyle {\rm D}_-\bar{\Phi}(x_0)=\sup_{l<0}\Delta_l\bar{\Phi}(x_0)
\geq \sup_{l<0}\Delta_l\bar{\Phi}_N(x_0)=\sup_{l<0}\Delta_l{\Phi}(x_0).
\end{cases}
\end{equation}

Then, to prove $\eqref{c5.10}$,
it suffices to prove the inverse inequalities in $\eqref{ab17}$.
In fact, for any given $R>|x_0|$,
$\{\bar{\Phi}_N(x)\}_N$ is a nonincreasing sequence of continuous functions on $[-R,R]$
and converges to $\bar{\Phi}(x)$ pointwise so that
$\bar{\Phi}_N(x)$ converges to $\bar{\Phi}(x)$ uniformly on $[-R,R]$.
Thus, for any fixed $\varepsilon>0$, there exists $N_0>0$ such that,
for any $N>N_0$,
$$
0\leq\bar{\Phi}_N(x)-\bar{\Phi}(x)<\varepsilon \qquad {\rm for\ any}\ x\in [-R,R],
$$
which, by $\eqref{ab12}$, implies that, if $x\in (x_0, R)$, then
$$
\frac{\varepsilon+\bar{\Phi}(x)-\bar{\Phi}(x_0)}{x-x_0}
>\frac{\bar{\Phi}_N(x)-\bar{\Phi}_N(x_0)}{x-x_0}
\geq\inf_{l>0}\Delta_l\bar{\Phi}_N(x_0)
=\inf_{l>0}\Delta_l{\Phi}(x_0),
$$
and, if $x\in ({-}R, x_0)$, then
$$
\frac{\varepsilon+\bar{\Phi}(x)-\bar{\Phi}(x_0)}{x-x_0}
<\frac{\bar{\Phi}_N(x)-\bar{\Phi}_N(x_0)}{x-x_0}
\leq\sup_{l<0}\Delta_l\bar{\Phi}_N(x_0)
=\sup_{l<0}\Delta_l{\Phi}(x_0).
$$
By the arbitrariness of $\varepsilon>0$, we obtain
\begin{equation*}
\begin{cases}
\displaystyle \frac{\bar{\Phi}(x)-\bar{\Phi}(x_0)}{x-x_0}
\geq\inf_{l>0}\Delta_l{\Phi}(x_0) \quad &{\rm if}\ x\in (x_0,R),\\[4mm]
 \displaystyle\frac{\bar{\Phi}(x)-\bar{\Phi}(x_0)}{x-x_0}
 \leq\sup_{l<0}\Delta_l{\Phi}(x_0) \quad &{\rm if}\ x\in ({-}R,x_0),
\end{cases}
\end{equation*}
which, by letting $x\rightarrow x_{0}{+}$ and $x\rightarrow x_{0}{-}$ respectively, implies
\begin{equation}\label{ab19}
\begin{cases}
\displaystyle {\rm D}_+\bar{\Phi}(x_0)
=\lim\limits_{x\rightarrow x_{0}{+}}\frac{\bar{\Phi}(x)-\bar{\Phi}(x_0)}{x-x_0}\geq\inf_{l>0}\Delta_l{\Phi}(x_0),\\[4mm]
\displaystyle{\rm D}_-\bar{\Phi}(x_0)
=\lim\limits_{x\rightarrow x_{0}{-}}\frac{\bar{\Phi}(x)-\bar{\Phi}(x_0)}{x-x_0}\leq\sup_{l<0}\Delta_l{\Phi}(x_0).
\end{cases}
\end{equation}

We combine $\eqref{ab17}$ with $\eqref{ab19}$ to conclude $\eqref{c5.10}$.

\smallskip
\noindent
{\bf 2}.
Since $\Phi(x)$ and $\bar{\Phi}(x)$ are both continuous functions,
$\mathcal{K}_0$ in $\eqref{c5.6}$ is a closed set so that
$\mathcal{K}^c_0=\mathbb{R}-\mathcal{K}_0$ is an open set.

\smallskip
{\bf (a).} Suppose that $x^+_0=\sup \mathcal{K}_0<\infty$.
Since $\mathcal{K}_0$ is a closed set, then $x^+_0\in \mathcal{K}_0$.
Since ${\rm D}_+\bar{\Phi}(x)$ is nondecreasing,
it follows from $\eqref{c5.8}$ that ${\rm D}_+\bar{\Phi}(x^+_0)=\underline{u}_r$.

By the nondecreasing of
$\Delta_l\bar{\Phi}(x^+_0)=\frac{1}{l}\big(\bar{\Phi}(x^+_0+l)-{\Phi}(x^+_0)\big)$ in $l\neq0$,
for $l>0$,
\begin{align*}
\underline{u}_r
&={\rm D}_+\bar{\Phi}(x^+_0)\leq \Delta_l\bar{\Phi}(x^+_0)
\leq \lim_{l\rightarrow \infty} \Delta_l\bar{\Phi}(x^+_0)\nonumber\\[1mm]
&=\lim_{l\rightarrow \infty}\Big(\frac{\bar{\Phi}(x^+_0+l)}{x^+_0+l} \,
\frac{x^+_0+l}{l}-\frac{\bar{\Phi}(x^+_0)}{l}\Big)
=\lim_{l\rightarrow \infty}\frac{\bar{\Phi}(x^+_0+l)}{x^+_0+l}
\leq\mathop{\underline{\lim}}\limits_{l\rightarrow \infty}\frac{{\Phi}(x^+_0+l)}{x^+_0+l}
=\underline{u}_r,
\end{align*}
which implies that, for any $x=x_0^++l > x_0^+$,
$$
\frac{\bar{\Phi}(x)-\Phi(x_0^+)}{x-x_0^+}
=\frac{\bar{\Phi}(x)-\bar{\Phi}(x_0^+)}{x-x_0^+}
=\Delta_l\bar{\Phi}(x^+_0)=\underline{u}_r.
$$
This implies $\eqref{c5.11}$.

\smallskip
{\bf (b).} Suppose that $x^-_0=\inf \mathcal{K}_0>-\infty$.
Similar to {\bf (a)},
then $x^-_0\in \mathcal{K}_0$ and ${\rm D}_-\bar{\Phi}(x^-_0)=\overline{u}_l$.

By the nondecreasing of
$\Delta_l\bar{\Phi}(x^-_0)=\frac{1}{l}\big(\bar{\Phi}(x^-_0+l)-{\Phi}(x^-_0)\big)$ in $l\neq0$,
for $l<0$,
\begin{align*}
\overline{u}_l
&={\rm D}_-\bar{\Phi}(x^-_0) \geq\Delta_l\bar{\Phi}(x^-_0)
\geq \lim_{l\rightarrow -\infty} \Delta_l\bar{\Phi}(x^-_0)\nonumber\\[1mm]
&=\lim_{l\rightarrow -\infty}\Big(\frac{\bar{\Phi}(x^-_0+l)}{x^-_0+l} \,
\frac{x^-_0+l}{l}-\frac{\bar{\Phi}(x^-_0)}{l}\Big)
=\lim_{l\rightarrow -\infty} \frac{\bar{\Phi}(x^-_0+l)}{x^-_0+l}
\nonumber\\[1mm]
&
\geq\mathop{\overline{\lim}}_{l\rightarrow -\infty} \frac{{\Phi}(x^-_0+l)}{x^-_0+l}
=\overline{u}_l,
\end{align*}
which implies that, for any $x=x_0^-+l < x_0^-$,
$$
\frac{\bar{\Phi}(x)-\Phi(x_0^-)}{x-x_0^-}
=\frac{\bar{\Phi}(x)-\bar{\Phi}(x_0^-)}{x-x_0^-}
=\Delta_l\bar{\Phi}(x^-_0)=\overline{u}_l.
$$
This implies $\eqref{c5.12}$.

\smallskip
{\bf (c).} Since the bounded interval $(e_n,h_n)$ contains in
$\mathcal{K}^c_0$ with $e_n,h_n \in \mathcal{K}_0$, then
\begin{equation}\label{ab22}
\begin{cases}
{\Phi}(x)>\bar{\Phi}(x) \qquad\quad {\rm for\ any}\ x\in (e_n,h_n),\\[1mm]
{\Phi}(e_n)=\bar{\Phi}(e_n),\quad {\Phi}(h_n)=\bar{\Phi}(h_n),
\end{cases}
\end{equation}
which, by ${\Phi}(x)\geq\bar{\Phi}_N(x)\geq\bar{\Phi}(x)$,
implies that, for any $N>0$,
\begin{equation*}
{\Phi}(e_n)=\bar{\Phi}_N(e_n)=\bar{\Phi}(e_n),
\qquad {\Phi}(h_n)=\bar{\Phi}_N(h_n)=\bar{\Phi}(h_n).
\end{equation*}
Then it follows from $\eqref{c5.5}$ that,
if $N>\max\{|e_n|,|h_n|\}$,
\begin{equation*}
\bar{\Phi}_N(x)\big|_{[e_n,\,h_n]}=\mathop{\rm cov}_{[e_n,\,h_n]}\Phi(x)
\qquad {\rm for\ any}\ x\in [e_n, h_n],
\end{equation*}
which, by letting $N\rightarrow\infty$, implies
\begin{equation}\label{ab25}
\bar{\Phi}(x)\big|_{[e_n,\,h_n]}
=\lim_{N\rightarrow \infty}\bar{\Phi}_N(x)\big|_{[e_n,\,h_n]}
=\mathop{\rm cov}_{[e_n,\,h_n]}\Phi(x)
\qquad {\rm for\ any}\ x\in [e_n,h_n].
\end{equation}
Combining $\eqref{ab22}$ with $\eqref{ab25}$, we obtain
\begin{equation*}
\begin{cases}
\displaystyle{\Phi}(x)>\mathop{\rm cov}\limits_{[e_n,\,h_n]}\Phi(x)
\qquad\quad {\rm for\ any}\ x\in (e_n,\,h_n),\\[4mm]
\displaystyle{\Phi}(e_n)=\mathop{\rm cov}\limits_{[e_n,\,h_n]}\Phi(e_n),
\quad {\Phi}(h_n)=\mathop{\rm cov}\limits_{[e_n,\,h_n]}\Phi(h_n),
\end{cases}
\end{equation*}
which, by the definition of the convex hull and $\eqref{ab25}$, implies that,
for any $x\in [e_n,h_n]$,
\begin{align*}
\bar{\Phi}(x)=\mathop{\rm cov}\limits_{[e_n,\,h_n]}\Phi(x)
&={\Phi}(e_n)+\frac{{\Phi}(h_n)-{\Phi}(e_n)}{h_n-e_n}(x-e_n)
={\Phi}(e_n)+{\rm D}_+\bar{\Phi}(e_n)(x-e_n).
\end{align*}
This means that $\eqref{c5.13}$ holds.

\smallskip
To sum up, we have completed the proof of Lemma $\ref{lem:c5.3}$.
\end{proof}

\medskip
~\\ \textbf{Acknowledgements}.
The research of Gaowei Cao was supported in part
by the National Natural Science Foundation of China No.11701551
and the China Scholarship Council No.20200491 0200.
The research of Gui-Qiang G. Chen was supported in part
by the UK Engineering and Physical Sciences Research Council Award
EP/L015811/1, EP/V008854, and EP/V051121/1.
The research of Xiaozhou Yang was supported in part
by the National Natural Science Foundation
of China No.11471332.
For the purpose of open access, the authors have applied a CC BY public copyright license
to any Author Accepted Manuscript (AAM) version
arising from this submission.

\bigskip
\medskip
\noindent{\bf Conflict of Interest:} The authors declare that they have no conflict of interest.
The authors also declare that this manuscript has not been previously published,
and will not be submitted elsewhere before your decision.

\bigskip
\noindent{\bf Data availability:} Data sharing is not applicable to this article as no datasets were generated or analyzed during the current study.

\end{document}